\documentclass[10pt]{amsart}

\usepackage{graphicx}        
\usepackage{amscd,amsmath,amssymb,amsfonts,enumerate}
\usepackage[cmtip, all]{xy}
\usepackage{hyperref}
\usepackage{comment}
\usepackage{mathrsfs}
\usepackage{enumitem}
\usepackage{mathrsfs}
\setlist[enumerate,1]{label={(\arabic*)}}
\setlist[enumerate,2]{label={(\alph*)}}

\newtheorem{theorem}{Theorem}[section]
\newtheorem{proposition}[theorem]{Proposition}
\newtheorem{lemma}[theorem]{Lemma}
\newtheorem{corollary}[theorem]{Corollary}

\newtheorem{hyp}{Hypothesis}

\theoremstyle{definition}
\newtheorem{definition}[theorem]{Definition}
\theoremstyle{remark}
\newtheorem{remark}[theorem]{Remark}
\newtheorem{assumption}[theorem]{Assumption}

\newtheorem{BA}[theorem]{Basic Assumptions}
\newtheorem{construction}[theorem]{Construction}

\numberwithin{equation}{section}

\newcommand{\CH}{{\rm CH}}

\newcommand{\red}{{\rm red}}
\newcommand{\bal}{{\rm bal}}
\DeclareMathOperator{\codim}{codim}

\newcommand{\Pic}{{\rm Pic}}

\newcommand{\Hom}{{\rm Hom}}
\newcommand{\im}{{\rm im}}
\DeclareMathOperator{\ims}{im}
\newcommand{\Spec}{{\rm Spec\,}}
\newcommand{\Spf}{{\rm Spf\,}}

\DeclareMathOperator{\Char}{char}

\newcommand{\ksep}{k^{\operatorname{sep}}}
\newcommand{\Tr}{{\text{Tr}}}

\newcommand{\Gal}{{\rm Gal}}

\newcommand{\0}{\emptyset}
\newcommand{\sA}{{\mathcal A}}
\newcommand{\sB}{{\mathcal B}}
\newcommand{\sC}{{\mathcal C}}
\newcommand{\sD}{{\mathcal D}}

\newcommand{\sI}{{\mathcal I}}
\newcommand{\sJ}{{\mathcal J}}

\newcommand{\sL}{{\mathcal L}}
\newcommand{\sM}{{\mathcal M}}
\newcommand{\sN}{{\mathcal N}}
\newcommand{\sO}{{\mathcal O}}
\newcommand{\sP}{{\mathcal P}}

\newcommand{\sS}{{\mathcal S}}

\newcommand{\sU}{{\mathcal U}}

\newcommand{\A}{{\mathbb A}}

\newcommand{\C}{{\mathbb C}}
\newcommand{\N}{{\mathbb N}}
\renewcommand{\P}{{\mathbb P}}

\newcommand{\R}{{\mathbb R}}
\newcommand{\T}{{\mathbb T}}

\newcommand{\Z}{{\mathbb Z}}

\newcommand{\mfS}{\mathfrak{S}}

\renewcommand{\det}{\operatorname{det}}

\newcommand{\id}{{\operatorname{\rm Id}}}

\newcommand{\<}{\langle}
\renewcommand{\>}{\rangle}

\newcommand{\Div}{{\operatorname{div}}}

\newcommand{\del}{\partial}

\renewcommand{\max}{{\operatorname{\rm max}}}

\newcommand{\Sym}{{\operatorname{Sym}}}

\newcommand{\Ext}{{\operatorname{Ext}}}

\newcommand{\Bl}{\text{Bl}}

\newcommand{\GW}{{\operatorname{GW}}}

\newcommand{\Aut}{{\operatorname{Aut}}}

\newcommand{\trip}{{\operatorname{trip}}}
\newcommand{\tac}{{\operatorname{tac}}}
\newcommand{\cusp}{{\operatorname{cusp}}}
\newcommand{\good}{{\operatorname{good}}}
\newcommand{\Res}{{\operatorname{Res}}}

\newcommand{\PGL}{\operatorname{PGL}}
\newcommand{\bir}{{\operatorname{bir}}}
\newcommand{\birf}{{\operatorname{birf}}}
\newcommand{\odp}{{\operatorname{odp}}}
\newcommand{\unr}{{\operatorname{unr}}}
\newcommand{\tor}{{\operatorname{tor}}}

\newcommand{\Deg}{\text{deg}}
\newcommand{\ev}{{ev}}

\newcommand{\ind}[1]{}

\newcommand{\inp}[1]{}

\newcommand{\res}{\text{res}}

\newcommand{\coker}{{\operatorname{coker}}}

\newcommand{\disc}{{\operatorname{disc}}}

\newcommand{\Hilb}{{\operatorname{Hilb}}}

\newcommand{\bMbir}[2]{\overline{M_{#1,#2}^\bir}}
\newcommand{\uc}[1]{\bar X_{0,#1}}

\newcommand{\M}{{\bar{M}}}
\newcommand{\calO}{{ \mathcal O}}
\newcommand{\calD}{{ \mathcal D}}
\newcommand{\ddp}{{\operatorname{tac}}}
\newcommand{\dpl}{{\calD}}
\newcommand{\LiftDeg}{\tilde{D}}
\newcommand{\Liftpi}{\tilde{\pi}}
\newcommand{\LiftS}{\tilde{S}}
\newcommand{\LiftA}{\tilde{A}}
\newcommand{\Liftev}{\widetilde{\ev}}
\newcommand{\Liftdpl}{{\tilde{\calD}}}
\newcommand{\dplodp}{{\dpl^\odp}}
\newcommand{\dplcusp}{{\dpl_\cusp}}
\newcommand{\dpltac}{{\dpl_\tac}}
\newcommand{\dpltrip}{{\dpl_\trip}}
\newcommand{\f}{{\mathfrak{f}}}
\newcommand{\I}{{\mathcal{I}}}
\newcommand{\tf}{{\tilde p}}
\newcommand{\bM}[1]{\M_{0,#1}(S,D)}
\newcommand{\fc}{\mathscr{F}}
\newcommand{\cI}{\mathcal{I}}
\newcommand{\cV}{\mathcal{V}}

\DeclareMathOperator{\ochar}{char}

\begin{document}

\pagestyle{plain}

\title{A relative orientation for the moduli space of stable maps to a del Pezzo surface}

\author{Jesse Leo Kass}

\address{Current: J.~L.~Kass, Dept.~of Mathematics, University of California, Santa Cruz, 1156 High Street, Santa Cruz, CA 95064, United States of America}
\email{jelkass@ucsc.edu}
\urladdr{https://www.math.ucsc.edu/people/faculty.php?uid=jelkass}

\author{Marc Levine}

\address{Current: M.~Levine, University of Duisburg-Essen, Germany}
\email{marc.levine@uni-due.de}
\urladdr{https://www.esaga.uni-due.de/marc.levine/}

\author{Jake P. Solomon}

\address{Current: J.~P.~Solomon, Institute of Mathematics, Hebrew University, Givat Ram Jerusalem, 91904, Israel}
\email{jake@math.huji.ac.il}
\urladdr{http://www.ma.huji.ac.il/~jake/}

\author{Kirsten Wickelgren}

\address{Current: K.~Wickelgren, Department of Mathematics, Duke University, 120 Science Drive
Room 117 Physics, Box 90320, Durham, NC 27708-0320, USA}
\email{kirsten.wickelgren@duke.edu}
\urladdr{https://services.math.duke.edu/~kgw/}


\date{March 2026}

\begin{abstract}
We prove orientation results for evaluation maps of moduli spaces of rational stable maps to del Pezzo surfaces over a field, both in characteristic $0$ and in positive characteristic. These results and the theory of degree developed in a sequel produce quadratically enriched counts of rational curves over non-algebraically closed fields of characteristic not $2$ or $3$. Orientations are constructed in two steps. First, the ramification locus of the evaluation map is shown to be the divisor in the moduli space of stable maps where image curves have a cusp. Second, this divisor is related to the discriminant of a branched cover of the moduli space given generically by pairs of points on the universal curve with the same image.  
\end{abstract}
\maketitle

\tableofcontents

\section{Introduction}

An {\em orientation} of a map $f: X \to Y$ of smooth schemes $X$ and $Y$ over a field $k$ is defined to be the data of a line bundle $\sL$ on $X$ and an isomorphism $\sL^{\otimes 2} \cong \omega_f$, where $\omega_f \cong \Hom(f^* \det T^*Y, \det T^*X)$ denotes the relative canonical bundle. An orientation of $f$ is viewed as a relative orientation of $X$ over $Y$. For example, for $k = \mathbb{R}$, an orientation of $f: X \to Y$ gives the data of a topological orientation on the real manifold $f^{-1}(y)(\R)$ for $y$ a regular value of $f.$

By a del Pezzo surface, we mean a smooth, projective surface $S$ over a field $k$ whose inverse canonical class $-K_S$ is ample. Examples of interest include blow-ups of $\P^2$ at fewer than $9$ points, $\P^1 \times \P^1$, and cubic surfaces. The degree of $S$ is $d_S = K_S \cdot K_S$. Let $D$ be an effective Cartier divisor on $S$. A pointed rational map of degree $D$ on $S$ is a map $u: \P \to S$ from an arithmetic genus $0$ curve $\P$ with at worst nodal singularities to $S$ such that $u_* [\P] = D$ in $\CH^1(S)$ along with marked points $p_1,\ldots,p_n$ of the smooth locus of $\P$. Such a map is {\em stable} if it has finitely many automorphisms. There is a moduli stack $\bar{M}_{0,n}(S, D)$ parametrizing rational stable maps $(u: \P \to S, p_1,\ldots, p_n)$ of degree $D$. See \cite{AO}. This moduli stack is discussed further in Section~\ref{subsection:ModuliStacksDefs}. Define the evaluation map 
\[
\ev: \bar{M}_{0,n}(S, D) \to S^n
\]
by taking the class of $(u:\P\to S, (p_1,\ldots, p_n))$ to $(u(p_1), \ldots, u(p_n))$. This paper constructs an orientation of $\ev$ away from the preimage of a set $A \subset S^n$ of codimension $\geq 2$ under appropriate hypotheses. First consider the case where the characteristic of $k$ is $0$.

\begin{hyp}\label{hyp:basic}
Assume that $D$ is not an $m$-fold multiple of a $-1$-curve for $m>1$. Moreover, assume that $d_S\ge 4$, or $d_S=3$ and $d:= -K_S \cdot D\neq 6$, or $d_S = 2$ and $d\ge 7$.
\end{hyp} 

\begin{theorem}\label{thm:introchar0}
Suppose $k$ is a field of characteristic zero and that $(S,D)$ satisfies Hypothesis~\ref{hyp:basic}. Let $n=d-1$. Then there is a codimension $\geq 2$ closed subset $A$ of $S^n$ such that 
\[
\ev\vert_{\ev^{-1}(S^n \setminus A)}: \bar{M}_{0,n}(S, D) \setminus \ev^{-1}(A) \to S^n \setminus A
\] admits an orientation.
\end{theorem}

The closed subset $A$ is constructed in Theorem~\ref{prop:Good}, and the orientation is constructed in Theorem~\ref{thm:Orient1}. 

In positive characteristic, we lift $\ev$ to a map over a complete discrete valuation ring $\Lambda$ with residue field $k$,
\[
\Liftev: \bar{M}_{0,n}(\LiftS, \LiftDeg) \to \LiftS^{n},
\]
and orient $\Liftev$ away from the inverse image of a codimension $\geq 2$ subset of $\LiftS^n$ under additional hypotheses which we know describe.

Let $M^{\bir}_{0,n}(S, D) \subset \bar{M}_{0,n}(S, D)$ denote the locus of stable maps that are birational onto their images with irreducible domain curves. See Definition~\ref{df:Mbir}. For $\P$ irreducible and thus smooth, a stable map $u: \P \to S$ over an algebraically closed field is {\em unramified} if $df: f^* T^*S \to T^* \P$ is surjective. 
 
 \begin{hyp}\label{hyp:pc}
In addition to Hypothesis~\ref{hyp:basic}, assume $k$ is perfect of characteristic not $2$ or~$3$. If $d_S =2$, assume additionally that for every effective $D' \in Pic(S)$, there is a geometric point $u$ in each irreducible component of $M^\bir_0(S, D')$ with $u$ unramified. 
\end{hyp}

The existence of unramified maps as in Hypothesis~\ref{hyp:pc} for $d_S \geq 3$ is shown in Appendix~\ref{Section:unramified_maps_in_any_char} following arguments of \cite{BLRT}. 

\begin{theorem}\label{Intro:thm:orient:ev:pc}
Suppose $(S,D)$ satisfies Hypothesis~\ref{hyp:pc}. Let $n=d-1$. Then there is a codimension $\geq 2$ closed subset $\LiftA \subset \LiftS^n$ such that 
\[
\Liftev\vert_{\Liftev^{-1}(\LiftS^n \setminus \LiftA)}: \bar{M}_{0,n}(\LiftS, \LiftDeg) \setminus \Liftev^{-1}(\LiftA) \to \LiftS^n \setminus \LiftA
\] admits an orientation.
\end{theorem}

Theorem~\ref{Intro:thm:orient:ev:pc} is shown as Theorem~\ref{thm:Orient_pos_char_1} and Theorem~\ref{thm:hyp:pc}. See also Construction~\ref{const:MixCharGood} for the construction of $\LiftA$.

Orienting $\ev$ enables one to define an appropriate notion of the degree of $\ev$, and we do this in \cite{degree}. The relevant notion comes from Morel and Voevodsky's $\A^1$-homotopy theory \cite{morelvoevodsky1998}. Under appropriate hypotheses on $f: X \to Y$, the degree may be computed as a weighted count of the points of $f^{-1}(y)$ for a general point $y$. The weights are no longer integers but elements of the Grothendieck--Witt group $\GW(k(y))$, defined to be the group completion of isomorphism classes of symmetric, nondegenerate bilinear forms over $k(y)$. The Grothendieck--Witt group appears here from Morel's calculation of stable $\pi_{0,0}$ of the sphere spectrum in $\A^1$-homotopy theory and is compatible with the topological degree after passing to real and complex realizations. 

The degree of $\ev$ resulting from Theorems~\ref{thm:introchar0} and \ref{Intro:thm:orient:ev:pc} when $k=\C$ is a certain Gromov--Witten invariant \cite{Caporaso_Harris_CoutingPlaneCurves} \cite{Gromov}\cite{Kontsevich-Manin} \cite{LiTian} \cite{McDuff-Salamon} \cite{Ruan-Tian}. When $k = \R$, it contains the additional information of a correponding Welschinger invariant \cite{Welschinger-invtsReal4mflds}, or open Gromov-Witten invariant~\cite{Solomon-thesis}. The open Gromov-Witten invariants of~\cite{Solomon-thesis} were defined as the degree of a relatively oriented evaluation map. See also~\cite{Cho}. We show in~\cite{degree} that the degree of $\ev$ is given by a weighted count of the stable maps in the fiber over a chosen tuple of points. Each stable map through the chosen tuple of points is given a weight in $\GW(k)$ connected to the field of definition of the curve and the fields of definitions of the tangent directions at the nodes. Thus, the degree of $\ev$ is a quadratically enriched curve enumerating invariant. It is a genus $0$ Gromov--Witten invariant recording arithmetically interesting information over $k$ about the curves interpolating a generally chosen set of points.

To allow the interpolated points to have nontrivial residue field extensions, we consider twists
\[
\ev_{\sigma}:  \bar{M}_{0,n}(S, D)_{\sigma} \to \prod_{i=1}^r \Res_{L_i/k} S 
\] 
of $\ev$ for $\sigma = (L_1,\ldots,L_r)$ with $L_i$ a finite separable extension of $k$ and $\sum_{i=1}^r L_i = n =d-1$. When $k$ has characteristic zero and Hypothesis~\ref{hyp:basic} holds, we construct an orientation of $\ev_\sigma$ away from the preimage of a subset of $\prod_{i=1}^r \Res_{L_i/k} S$ of codimension $\geq 2.$ When Hypothesis~\ref{hyp:pc} holds, we construct an analogous orientation for a lifting of $\ev_\sigma$ to a map over a complete discrete valuation ring $\Lambda$ with residue field $k.$ See Sections~\ref{Section:twisting_deg_map_char0} and~\ref{subsection:twists_of_evaluation_map_positive_char}. 
Thus, for each of these twists, we are able to define a degree~\cite{degree} and a resulting quadratically enriched Gromov--Witten invariant. This invariant counts genus $0$ degree $D$ curves on $S$ passing through generally chosen points with residue fields $L_1,\ldots,L_r$ when such points exist. For example, in \cite[Section 8.1]{degree} we compute the quadratically enriched Gromov--Witten invariant corresponding to the degree of the twisted evaluation map $\ev_{\sigma}$ for $D$ equal to the anticanonical class $D=-K_S$ and any $\sigma$. For $a$ in $k^*$, let $\langle a \rangle$ in $\GW(k)$ denote the class of the bilinear form $k \times k \to k$ mapping $(x,y)$ to $axy$. When $D=-K_S$, we have that the quadratically enriched Gromov--Witten invariant $\deg \ev_{\sigma}$ equals
\[
\deg \ev_{\sigma} = \langle -1 \rangle\chi^{\A^1}(S)  + \langle 1 \rangle +  \Tr_{k(\sigma)/k} \langle 1 \rangle
\] the sum of the trace form of $\sigma = \prod_{i=1}^r L_i$ with $\langle 1 \rangle$ and a multiple of the $\mathbb{A}^1$-Euler characteristic $\chi^{\A^1}(S) $. This can be computed explicitly with \cite{LLV-Eulerchar} giving results such as 
\[
\deg \ev_{\sigma} = \langle 5 \rangle+4(\langle 1\rangle+\langle -1\rangle)+\langle 1 \rangle+\Tr_{k(\sigma)/k}\langle 1\rangle 
\] for $S$ given by the cubic surface $xy^2+y^2z+z^2w+w^2x = 0$  \cite[Section 8.2]{degree}. We are able to recover the full range of Welschinger or open Gromov--Witten invariants for $(S,D)$ over $k =\R$ under the above hypotheses.

There has been considerable progress towards computing these invariants since this work became available. Andr\'es Jaramillo Puentes and Sabrina Pauli computed the degree of the untwisted evaluation map $\ev$ over toric surfaces via a tropical correspondence theorem \cite{JPP-quadenr}, building on their previous work on a quadratically enriched tropical B\'ezout theorem \cite{PuentesPauli-Bezout}. Jaramillo Puentes, Hannah Markwig,  Pauli, and Felix R\"ohrle extend these results to give a tropical count for $L_i$ of degree $2$ or $1$ in \cite{JPMPR-quadenr} computing the degree of the twisted evaluation maps for quadratic twists of toric surfaces. They also found Caporasso--Harris style recursions for untwisted evaluation maps in \cite{JPMPR-tropPlaneCurves}. Erwan Brugall\'e and the fourth named author compute the change in the quadratically enriched Gromov--Witten invariants under degeneration to a weak del Pezzo surface in \cite{BW-ABQ}, giving a wall-crossing formula. Brugall\'e, Johannes Rau, and the fourth named author give a conjectural computation of the invariants for all $k$-rational surfaces by means of an integrality property of Witt-invariants \cite{BRW-WWI}, and prove it in degree $> 5$ by combining the wall-crossing formula of \cite{BW-ABQ} and the toric quadratic results of \cite{JPMPR-quadenr}. Other quadratic or $\A^1$-enrichements of enumerative results are found in, e.g. \cite{KWA1degree} \cite{Levine-EC} \cite{cubicsurface} \cite{Levine-Witt} \cite{Pauli-quadratic_types_quintic_3-fold} \cite{McKean-Bezout} \cite{PajwaniPal-YZ} \cite{Levine-IntrinsicStable} \cite{LV-DTP13}.

Restrictions of the twists $\ev_{\sigma}$ to certain dense opens are pulled back from a symmetrized evaluation map $\ev_\mfS$ which maps to the quotient $\Sym^n_0S$ of the complement of the pairwise diagonals in $S^n$ by the symmetric group on $n$-letters. Orientation results are obtained for $\ev_\mfS$ in Section~\ref{Section:symmetrized_moduli_space_kchar0} in characteristic $0$ and in Section~\ref{subsection:pos_char_symm_ev_map} in positive characteristic. Relations between degrees of $\ev_\sigma$ and $\ev_\mfS$ are given in~\cite[Section 5]{degree}.

The main steps in our construction of relative orientations are as follows. Let $D_{\cusp}$ (respectively $D_{\tac}$) denote the Cartier divisor on $M^{\bir}_{0,n}(S, D)$ defined as the closure of $(f:\P\to S, (p_1,\ldots, p_n))$ such that $f(\P)$ has one simple cusp (respectively tacnode) and nodes, but no other singularities. See Definition~\ref{def:cusp_tac_trip} and Lemma~\ref{lem:locallyclosed_Dcusp_tac_trip}. 

\begin{theorem}\label{into:thm:ram(ev)}
Suppose $k$ is a field of characteristic $0$ and $(S,D)$ satisfies Hypothesis~\ref{hyp:basic}. Then, there exists a codimension $\geq 2$ subscheme $A \subset S^n$ such that $\ev|_{ev^{-1}(S^n \setminus A)}$ is a map between smooth schemes that is \'etale  on the complement of $D_{cusp}$ with differential vanishing to order $1$ along $D_{cusp}$.
\end{theorem}

With $A$ as in Theorem~\ref{into:thm:ram(ev)}, let 
\[
\bar{M}_{0,n}(S,D)^\good:=\bar{M}_{0,n}(S,D) \setminus \ev^{-1}(A) = ev^{-1}(S^n \setminus A).
\]
Let $\uc{n}^\good \to \M_{0,n}(S,D)^\good$ denote the pullback of the universal curve $\bar{M}_{0,n+1}(S, D) \to \bar{M}_{0,n}(S,D)$ to $\M_{0,n}(S,D)^\good$. In Section~\ref{Section:dpl}, we define a closed subscheme in the product of universal curves,
\[
\dpl  \subset \uc{n}^\good \times_{\bar{M}_{0,n}(S, D)^\good} \uc{n}^\good,
\] 
called the {\em double point locus}. By construction $\dpl$ comes with a projection map $\pi:\dpl \to \bar{M}_{0,n}(S, D)$. Over a point $(f:\P^1\to S, (p_1,\ldots, p_n)) \in M_{0,n}(S, D)$ such that $f(\P^1)$ has only ordinary double points, the fiber of $\pi$ consists of pairs of points $x_1,x_2 \in \P^1$ such that $f(x_1) = f(x_2).$

Let $\dplcusp \subset \uc{n}^\good \times_{\bar{M}_{0,n}(S, D)^\good} \uc{n}^\good$ denote the locus of points $(f,x,x)$ where $f : \P^1 \to S$ is a map and $x \in \P^1$ is such that $f(x)$ is a simple cusp of $f(\P^1)$. The locus $\dplcusp$ is naturally a subscheme of the double point locus $\dpl$ as proven in Lemma~\ref{lm:dcuspdpl}. Let $\dpltac \subset \dpl$ denote the locus $(f,x_1,x_2)$ where $f : \P^1 \to S$ is a map and $x_1,x_2 \in \P^1$ are such that $f(x_1)= f(x_2)$ is a point where $f(\P^1)$ has a simple tacnode.

\begin{theorem}\label{into:thm:discpi}
Under the assumptions of Theorem~\ref{into:thm:ram(ev)}, we can choose $A$ such that the double point locus $\dpl$ is smooth and the map $\pi : \dpl \to \M_{0, n}(S,D)^\good$ is finite, flat and generically \'etale.  The ramification of $\pi$ is supported on $\dplcusp$ and $\dpltac$, where it is simply ramified, and the divisor of the discriminant is given
\[
\Div~\disc_\pi=1\cdot D_\cusp+2\cdot D_\tac.
\]
\end{theorem}
The definition of the discriminant is recalled in \eqref{eq:def:disc}. Theorem~\ref{into:thm:discpi} is proven as Corollary~\ref{Cor:dpl_sm_ram_pi} and Theorem~\ref{prop:Disc}.

Theorems~\ref{into:thm:ram(ev)} and ~\ref{into:thm:discpi} combine to give an orientation of $\ev$ as follows. Discriminant bundles are canonically isomorphic to the square of a line bundle. Thus, Theorem~\ref{into:thm:discpi} identifies $\sO_{\bar{M}^\good_{0,n}(S,D)}(D_\cusp)$ canonically as the square of a line bundle, and Theorem~\ref{into:thm:ram(ev)} identifies $\sO_{\bar{M}^\good_{0,n}(S,D)}(D_\cusp)$ with the relative canonical bundle of $\ev\vert_{\bar{M}^\good_{0,n}(S,D)}$. This orientation is given in Theorem~\ref{thm:Orient1}. In \cite{degree}, it is referred to as the {\em double point orientation}.

\subsection{Acknowledgements} This paper has been a long time in coming. While many of the ideas were present in 2018, understanding enough of the geometry of the moduli stack (and coping with non-mathematical realities) led to a drawn-out battle and we sincerely thank the mathematicians and organizations who supported us during this time.
We thank Rahul Pandharipande, Sho Tanimoto and Ilya Tyomkin, for helpful discussions.
ML is supported by the ERC Grant QUADAG: this paper is part of a project that has received funding from the European Research Council (ERC) under the European Union's Horizon 2020 research and innovation programme (grant agreement No. 832833).\\ 
\includegraphics[scale=0.08]{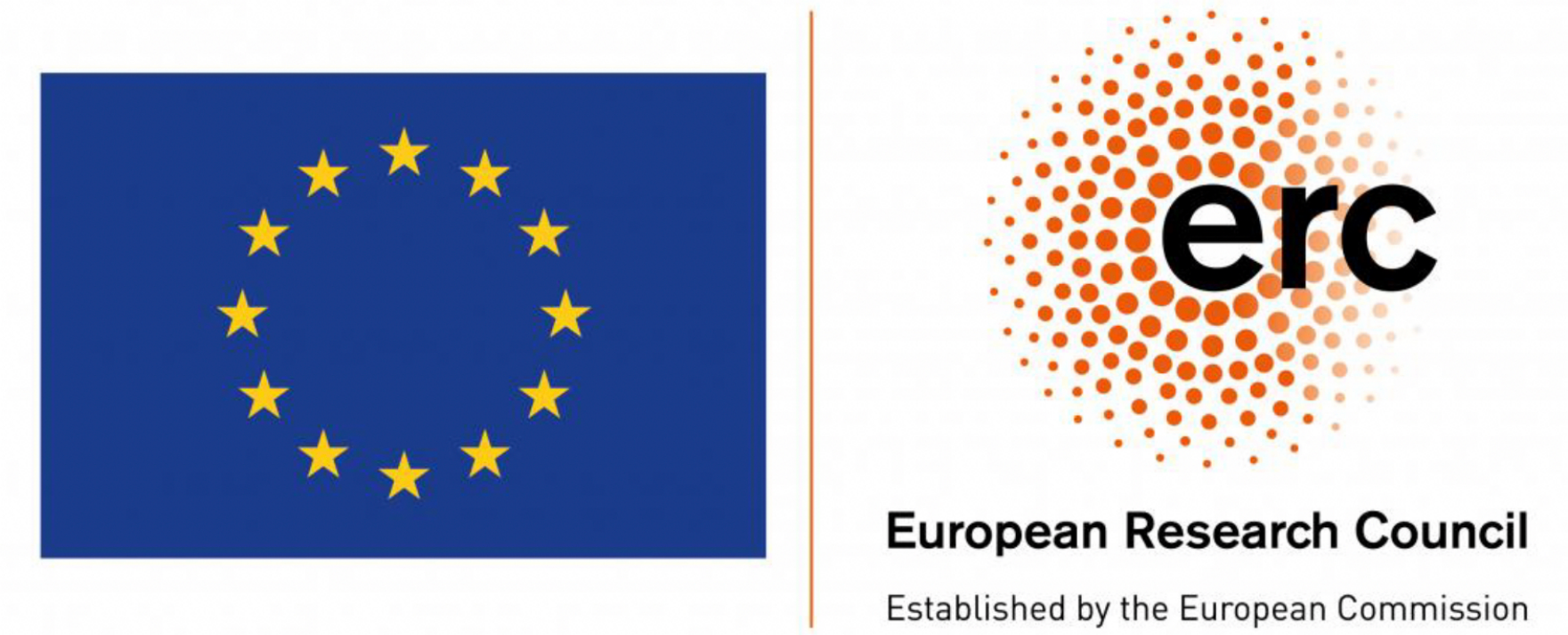}\\
JS was partially supported by ERC Starting Grant 337560 as well as ISF Grants 569/18 and~1127/22. KW was partially supported by National Science Foundation Awards DMS-1552730, DMS-2001890, DMS-2103838, and DMS-2405191.

\section{Rational curves on a del Pezzo surface}
\subsection{Definitions of moduli spaces of stable maps to del Pezzo surfaces}\label{subsection:ModuliStacksDefs}

We will be using the moduli space of pointed, stable maps, and we set up notation and references for this now. Let $\sO$ be a Noetherian ring. Let $\sS\to \Spec\sO$ be a smooth projective $\sO$-scheme and let $\sD$ be an effective relative Cartier divisor on $\sS$.

Let $\tilde{M}_{0,n}(\sS, \sD)$ denote the scheme of morphisms $f:\P^1\to \sS$ with $f_*([\P^1])\in |\sD|$, together with $n$ disjoint points (i.e. disjoint sections from the base scheme) $p_1,\ldots,p_n$ of $\P^1$ with its natural $\PGL_2$-action. Let

\[
d:=\Deg(-\sD\cdot K_{\sS})
\] in $\CH^0(\sO)$ denote the degree of $\sD$ with respect to $-K_{\sS}$ and suppose that $d$ is everywhere greater than $0$. (Intersection numbers are locally constant. See for example  \cite[B18]{Kleinman_The_Picard_Scheme}.) This implies that each $f:\P^1\to \sS$ with $f_*([\P^1])\in |\sD|$ is a stable map (in the sense of having finite automorphisms -- see \cite[p. 91(6)]{AO}).

 This stability gives a natural map $\tilde{M}_{0,n}(\sS, \sD)\to M_{0,n}(\sS, \sD)$ to a moduli stack $M_{0,n}(\sS, \sD)$ of stable maps given by the quotient by $\PGL_2$. $M_{0,n}(\sS, \sD)$ is an open substack of a compactified moduli stack $\bar{M}_{0,n}(\sS, \sD)$ of $n$-pointed, stable maps of a genus zero curve to $\sS$, in the curve class $\sD$. See \cite[Theorem 2.8 and comment p.90]{AO} for more information. We use the notation $\tilde{M}_{0}(\sS, \sD)$, $M_{0}(\sS, \sD)$, and $\bar{M}_{0}(\sS, \sD)$ to denote $\tilde{M}_{0,n}(\sS, \sD)$, $M_{0,n}(\sS, \sD)$, and $\bar{M}_{0,n}(\sS, \sD)$, respectively, with $n=0$.

The moduli stack $\bar{M}_{0,n}(\sS, \sD)$ is a proper (in particular, separated) algebraic stack over $\sO$ with projective coarse moduli space. $\bar{M}_{0,n}(\sS, \sD)$ is constructed by representing the functor of morphisms of $n$-pointed curves as a quasi-projective subscheme $\tilde{\bar{M}}_{0,n}(\sS,\sD)$ of a suitable Hilbert scheme and then defining $\bar{M}_{0,n}(\sS, \sD)$ as the quotient stack of this scheme by the natural $\PGL_N$-action (for suitable $N$). In particular, over an open substack $V$ with trivial groupoid structure, $\tilde{\bar{M}}_{0,n}(\sS,\sD)$ has a free and stable $\PGL_N$-action, so $V$ is isomorphic to its image in the coarse moduli space. In particular, $V$ is a quasi-projective scheme over $\sO$.

We will be interested in the case where $\sS$ is a del Pezzo surface. The references \cite[Chap III, \S3]{Kollar} and \cite[Chap. 8]{Dolgachev} contain pertinent information on del Pezzo surfaces, and we now give a definition and some description.

\begin{definition} A $\sO$-scheme $\sS$ is a {\em del Pezzo surface over $\sO$} if $\sS$ is smooth of relative dimension $2$ and projective over $\sO$ and the anti-canonical sheaf $-K_{\sS}$ is relatively ample. The {\em degree} $d_{\sS}$ of a del Pezzo surface $\sS$ is the self-intersection $K_{\sS} \cdot K_{\sS}$.
\end{definition}

\subsection{Over a field}

Now let $k$ be a perfect base field and let $S$ be a del Pezzo surface over $k$ of degree $d_S$. Over the algebraic closure of $k$, we may represent $S$ as the blow-up of $\P^2$ at $9-d_S$ points or as $S=\P^1\times\P^1$ (in this case, $d_S=8$). In case $d_S\ge 3$, the anti-canonical divisor $-K_S$ is very ample and in case $d_S=2$, the linear system $|-K_S|$ defines a finite, 2-1 cover of $\P^2$. We call $S$ a {\em general} del Pezzo of degree $d_S$ if $d_S\ge 5$ or if $S$ is the blow-up of $\P^2$ at $9-d_S$ ``general'' points (that is, an assertion about $S$ is true for all sets of points outside a closed algebraic subset of $S^{9-d_S}$).

\subsection{Properties of rational curves}

We will need to examine the geometry of $M_{0,n}(S,D)$ and $\bar{M}_{0,n}(S, D)$ at some length. To do this, we need to recall quite a number of definitions of properties of rational curves on $S$ which affect the geometry of their neighborhoods in these moduli spaces.

For later use in \S\ref{section:positive_characteristic}, in this section we use a separated Noetherian scheme $B$ as base-scheme.  We fix a del Pezzo surface $S\to B$ over $B$, endowed with relative effective Cartier divisor $D$. Following \S\ref{subsection:ModuliStacksDefs}, we have the moduli stack  $M_{0,n}(S, D)\to B$.

Let $C\to B$ be a separated morphism of a noetherian scheme $C$ to  $B$ and let $f:\P^1_C\to S_C$ be a morphism.  We say that $f$ is non-constant, resp. separable, if for each geometric point $x\to C$, the base-change $f_x:\P^1_x\to S_x$ is non-constant, resp. separable to the image curve $f_x(\P^1_x)\subset S_x$. For $f:\P^1_C\to S_C$ non-constant and separable, we have the normal sheaf $\sN_f$ defined by the exactness of the sheaf sequence
\begin{equation}\label{eqn:NormalSheaf}
0\to T_{\P^1_T/T}\to f^*T_S\to \sN_f\to 0
\end{equation}
When additionally, $f_*([\P^1_C])\in |D_C|$, we have that $d:=\Deg(-D\cdot K_S)$ is the degree of the determinant of $f^* T_{S_C}$ over $C$. If $C=\Spec K$ for some field $K$, we say that $f:\P^1_K\to S_K$ is {\em defined over $K$}; to save notation, we call this a morphism $f:\P^1\to S$, defined over $K$. Similarly, we write $f_*([\P^1])\in |D|$ for $f_*([\P^1_K])\in |D_K|$.

\begin{remark}\label{rmk:computation_normal_sheaf}
Suppose $f:\P^1\to S$ is a nonconstant and separable morphism defined over a field   $F$ and $f_*([\P^1])\in |D|$. Then the sequence \eqref{eqn:NormalSheaf} yields $\sN_f\cong \sO_{\P^1}(m)\oplus \sN_f^\tor$ with $m=d-2-\dim_FH^0(\P^1, \sN_f^\tor)$ where $\sN_f^{\tor}$ denotes the torsion subsheaf of $\sN_f$.
\end{remark}

\begin{definition}\label{df:free}
For $f:\P^1\to S$ defined over some field $F$,  we call $f$ {\em free} if $\sN_f$ is generated by global sections and $H^1(\P^1, \sN_f)=0$; equivalently, $H^1(\P^1, \sN_f(-1))=0$ (see \cite[Chap. II, Definition 3.1]{Kollar}). In general, we call $f:\P^1_C\to S_C$ free if $f_x$ is free for all geometric points $x\to C$.
\end{definition}

For $f:\P^1\to S$ defined over some field $F$, we let $\sN_f^\tor\subset \sN_f$ be the torsion subsheaf of $\sN_f$. We have $\sN_f\cong \sO_{\P^1}(m)\oplus \sN_f^\tor$, where $m=d-2-\ell$ and $\ell$ the length of $\sN_f^\tor$.

For a morphism $f:\P^1\to S$ defined over $F$, we write $H^i(\P^1, \sN_f)$ for the $F$-vector space $H^i(\P^1_F, \sN_f)$ and drop the subscript $F$ in other situations if the context makes the meaning clear.

Let $(\P, p_*)$ be a semi-stable genus zero curve with $n$ marked points and let $f:\P\to S$ be a stable map with (reduced) image curve $C=f(\P)\subset S$, defined over some field $F$. We call $f$ birational if $f:\P\to C$ is a birational map of curves, that is, $f^*$ is an isomorphism on total quotient rings; $f$ is {\em non-birational} if $f$ is not birational.

\begin{definition}\label{df:Mbir}
Let $M^\bir_{0,n}(S,D) \subset M_{0,n}(S,D)$ be the open subscheme with geometric points $[(f,p_*)]$ such that $f : \P^1 \to f(\P^1)$ is birational. Let $\bMbir{0}{n}(S,D)\subset \bar M_{0,n}(S,D)$ denote the closure of $M_{0,n}(S,D)^\bir.$
\end{definition}

Since a birational map has no automorphisms, $M^\bir_{0,n}(S,D)$ is in fact an open subscheme of the moduli stack of stable maps $M_{0,n}(S, D)$.

\begin{definition}
We let $M^\birf_{0,n}(S, D)\subset M_{0,n}(S, D)$ be the open subscheme with geometric points $[(f,p_*)]$ such that $f:\P^1\to f(\P^1)$ is birational and free.
\end{definition}

 \begin{definition}\label{def:M^unr}\label{def:unramified}
 We say that a map $f: \P \to S$ from a genus $0$ curve $\P$ is {\em unramified} if $f:\P \to f(\P)$ is an unramified map of relative curves. For $\P$ smooth, this is equivalent to the induced map on cotangent spaces $df: f^* T^*S \to T^* \P$ being surjective. Let $M^\unr_{0,n}(S, D)$ represent those $[(f,p)]$ in $M_{0,n}(S, D)$ such that $f:\P\to f(\P)$ is unramified.
 \end{definition}

\begin{remark}\label{rmk:computation_normal_sheaf_for_unramified}
Suppose $f:\P^1\to S$ is unramified. Then we have the exact sheaf sequence \ref{eqn:NormalSheaf}, $\sN_f^\tor \cong 0$, and $\sN_f\cong \sO_{\P^1}(d-2)$.
\end{remark}

In addition, letting $\tilde{M}^\birf_{0,n}(S,D)$, $\tilde{M}^\unr_{0,n}(S, D)$ be the corresponding subschemes of $\tilde{M}_{0,n}(S, D)$, the map $\tilde{M}^\birf_{0,n}(S,D)\to M_{0,n}^{\birf}(S,D)$ is a $\PGL_2$-bundle; for $F$ an algebraically closed field, we will often identify an $[f]\in M^\birf_{0,n}(S,D)(F)$ with a choice of lifting $f\in \tilde{M}^\birf_{0,n}(S,D)(F)$, leaving the context to make the meaning clear.

\begin{lemma}\label{lm:unramified_maps_are_birational}
$M^\unr_{0,n}(S, D)$ consists of birational maps.
\end{lemma}
\begin{proof}
Since birationality can be detected after base change to the algebraic closure, it suffices to show this for geometric points of $M^\unr_{0,n}(S, D)$.
Let $f : (\P^1,p_*) \to S$ be a geometric point of $M^\unr_{0,n}(S, D).$
By the universal property of normalization, the map $f : \P^1 \to f(\P^1)$ factors through the normalization $\nu: Z \to f(\P^1).$ So, let $g : \P^1 \to Z$ be the map such that $\nu \circ g = f.$ By L\"uroth's theorem, $Z \cong \P^1.$ Since $f$ is unramified, the differential $df_q : T_q \P^1 \to T_{f(q)}S$ is injective for all $q \in \P^1.$ It follows that $dg_q : T_q \P^1 \to T_{g(q)} Z$ is injective for all $q \in \P^1$ and thus $g$ is unramified. Since $g$ is a non-constant map of curves, whose codomain is normal, $g$ is flat \cite[Tag 0CCK]{stacks-project}. Since $g$ is flat, unramified, and of finite presentation, $g$ is \'etale \cite[02GV, 02G3]{stacks-project}.  Since the \'etale fundamental group of $\P^1$ is trivial (even in characteristic $p>0$) \cite[Th\'eor\`eme 2.6 Expos\'ee X SGA1]{sga1}, $g$ is an isomorphism.
\end{proof}

Our desired orientation of $\ev$ will be described in terms of singularities of the image curve $f(\P^1)$ at $(f:\P^1\to S, (p_1,\ldots, p_n))$ in $M_{0,n}(S,D)$, so we define certain singularities now. Suppose $f:\P^1\to S$ is an unramified map defined over an algebraically closed field. If the preimage of any point of $S$ consists of at most two points and for all points with two inverse images $p_1$ and $p_2$, the subspaces $df(T_{\P^1, p_i})$ are distinct for $i=1,2$, then the image curve $f(\P^1)$ has only {\em ordinary double points}. We can extend this notion to apply to a map $f: \P \to S$ over an algebraically closed field, where $\P$ has potentially multiple components. We say that $f$ has only {\em ordinary double points} if $\P \to f(\P)$ is unramified, and if the map from the normalization $\coprod^n \P^1 \to S$ satisfies the property that the preimage of any point of $S$ consists of at most two points and for all points with two inverse images $p_1$ and $p_2$, the subspaces $df(T_{\P^1, p_i})$ are distinct for $i=1,2$.

  \begin{definition}\label{df:Modp}
 Let $M^\odp_{0,n}(S, D)$ represent those $(f,p_*)$ in $M_{0,n}(S, D)$ such that $f$ is unramified, and over every geometric point of the base, $f(\P^1)$ has only ordinary double points. Dropping the assumption that the genus $0$ curve $\P$ be smooth, let $\M^{\odp}_{0,n}(S,D)$ represent those $f$ in $\M_{0,n}(S,D)$ such that $f : \P \to f(\P)$ is unramified, and has only ordinary double points over every geometric point of the base.
 \end{definition}

\begin{definition}\label{def:cusp_tac_trip}
	Let $f:\P\to S$ be a geometric point of $M_{0,n}^\bir(S,D)$.   We say that $f$ has a {\em cusp or worse} if there is   point $p \in \P$ such that $Tf(p) \colon T_{p}\P \to T_{f(p)}S$ is the zero map; we say that  $f$ has a {\em cusp} if in addition $f^{-1}(f(p))=\{p\}$. We  say a geometric point  $f$ of $M_0^\unr(S,D)$ has a {\em tacnode or worse} if  there are points $p\neq q \in \P$ such that $f(p) = f(q)$, and $T_f( T_{p}\P) = T_{f}(T_{q}S)$; if in addition $f^{-1}(f(p))=\{p,q\}$ we say $f$ has a tacnode. We say that a geometric point $f\in M_0^\unr(S,D)$ has a {\em $m$-fold point or worse} if  there are $m$  points in $\P$, $p_1,\ldots, p_m$,  with $f(p_i)=f(p_j)$ for all $i,j$; we say that $f$ has an $m$-fold point if in addition $f^{-1}(f(p_1))=\{p_1,\ldots, p_m\}$. For $m=3$, we use the term {\em triple point} instead of $m$-fold point.  
	
A cusp at $p\in \P$ is {\em ordinary} if $\dim_{\overline{k}(p)}(\Omega_{\P,p}/f^*\Omega_{S,f(p)})=1$ and $f^{-1}(f(p))=\{p\}$.   An $m$-fold point is  ordinary if the images $df(T_{p_i} \P)$ in of the tangent spaces  in $T_{S, f(p_i)}$ are pairwise distinct and $f^{-1}(f(p_1))=\{p_1,\ldots, p_m\}$. A tacnode is ordinary if  $f^{-1}(f(p))=\{p,q\}$, and there are generators $x, y$ for the maximal ideal in the complete local ring $\widehat{\sO}_{S, f(p)}$ such that the image curve $f(\P)$ has defining equation $y(y-x^2)\in \widehat{\sO}_{S, f(p)}$,
\end{definition}

Forgetting the last marked point defines a proper morphism
\[
\pi_n:X_{0,n}: = M_{0,n+1}(S, D) \to M_{0,n}(S, D)
\] 
from the universal curve. Evaluation on the $(n+1)$st point gives a map $f:=\ev_{n+1}: M_{0,n+1}(S, D) \to S$. As usual, we write 
\[
\pi:X_0: = M_{0,1}(S, D) \to M_{0}(S, D)
\] 
in case $n=0$, and let $\pi_\unr:X_0^{\unr}\to M_0^\unr(S,D)$ be the restriction over  
 $M_0^\unr(S,D)$.  
 
 Let $\Delta_S\subset S\times_kS$, $\Delta_{X^\unr_0}\subset X^\unr_0\times_{M^\unr_{0}(S, D)}X^\unr_0$  be the diagonals. Define the locally closed subset $\dpl^\unr$ of $X^\unr_0\times_{M^\unr_{0}(S, D)}X^\unr_0$ by
 \[
 \dpl^\unr:=(f\times f)^{-1}(\Delta_S)\setminus \Delta_{X^\unr_0}
 \]
 
 \begin{lemma}\label{lemma_unramified_doublelocus} $\dpl^\unr$ is closed in $X^\unr_0\times_{M^\unr_{0}(S, D)}X^\unr_0$.
 \end{lemma}

 \begin{proof} Let $\overline{\dpl}^\unr$ be the closure of $\dpl^\unr$ and suppose that 
 $\overline{\dpl}^\unr\setminus {\dpl}^\unr$ is non-empty, equivalently, there is a point $(p,p)\in \overline{\dpl}^\unr\cap\Delta_{X^\unr_0}$.Using the valuative criterion for properness, this means there is a complete discrete valuation ring $\sO$, with generic point $\eta$, closed point $a$ and parameter $t$,  and a map $g:\Spec\sO\to \overline{\dpl}^\unr$ with $g(\eta)\in \dpl^\unr$ and $g(x)\in \Delta_{X^\unr_0}$. We may assume that the residue field $\kappa$ of $\sO$ is algebraically closed; after making a base-change to $\kappa$ and changing notation, we may assume $\kappa=k$.   
 
 The projection $\Spec \sO\to M^\unr_{0}(S, D)$ gives us an unramified $\sO$-morphism $F_\sO:\P^1_\sO\to S_\sO$,  together with two sections $s_1, s_2: \Spec\sO\to \P^1_\sO$ such that $f_\sO(s_1(\eta))=f_\sO(s_2(\eta))$. Since $S$ is separated over $k$, we have $f_\sO\circ s_1=f_\sO\circ s_2$.  Let $q\in S(\sO)$ be the $\sO$-point $f_\sO\circ s_1=f_\sO\circ s_2$. Consider the completion $\sO_{S,q}^\wedge$ of $\sO_{S_\sO}$ along $q$. Since $\sO$ is complete and local, and $S$ is smooth over $k$, we can write the maximal ideal of  $\sO_{S,q}^\wedge$ as $(y_1,y_2)$ and we have $\sO_{S,q}^\wedge=\sO[[y_1,y_2]]:=\lim_{\leftarrow, n} \sO[y_1,y_2]/((t, y_1,y_2)^n)$. Similarly, we have the $k$-point $p=s_1(a)=s_2(a)\in \P^1(k)$, and we may assume that $p=0\in \A^1:=\P^1\setminus\{(0:1)\}$, with standard coordinate $x:=X_1/X_0$. We pass to the completion of $\sO[x]$ at $(a,0)$, which we identify with $\sO[[x]]$.
 
 Thus on $\Spec \sO[[x]]$, $f_\sO$ is given by two element $f_i:=f_\sO^*(y_i)\in \sO[[x]]$,  with $f_i\equiv0\mod x$,  $i=1,2$. Similarly,  the sections $s_i$ are given by   $s_i\in\sO$,  with  $s_i\equiv0\mod t$, and $f_i(s_1)=f_i(s_2)$, $i=1,2$. Translating on  $\Spec \sO[[x]]$ by $s_1$, we may assume that $s_1=0$. Since $s_1\neq s_2$, $s_2$ is not zero, and we may write $s_2=a_nt^n\mod t^{n+1}$ with $a_n\in k\neq0$. 
 
 Since $f_i\equiv0\mod x$, we may write $f_i(x)=xh_i(x)$, $i=1,2$ for some $h_i\in \sO[[x]]$.  Since $f_i(s_2)=0$, and $s_2\neq0$, we also have $h_i(s_2)=0$, so $h_i$ is divisible by $x-s_2$ and $f_i$ is thus divisible by $x^2-xs_2$. But then the image$\bar{f}_i(x)\in k[[x]]$ under the quotient map $\sO[[x]]\to k[[x]]$  is divisible by $x^2$, so 
 \[
 \frac{d\bar{f}_i}{dx}(0)=0
 \]
 and thus $f_k:\P^1\to S$  is ramified at $(1:0)$, contrary to our assumption that $\Spec\sO$  maps to $M^\unr_{0}(S, D)$.
 \end{proof} 
 We now return to our usual setting over the field $k$.

\begin{lemma}\label{lem:locallyclosed_Dcusp_tac_trip}\ 
\begin{enumerate}
\item The locus of stable maps with a cusp or worse is a closed subset of $M^\bir_0(S, D)$. \item The locus of stable maps with a tacnode or worse is a closed subset of  $M^\unr_0(S, D)$.  
\item For each $m\ge3$, the locus of stable maps with an $m$-fold point or worse is 
is a closed subset of  $M^\unr_0(S, D)$.  
\end{enumerate}
\end{lemma}

\begin{proof} We have the maps
\[
f \times f:  X_0^{\unr} \times_{M^{\unr}_0(S, D)} X_0^{\unr} \to S \times S
\]
\[
f \times f \times f: X_0^{\unr} \times_{M^{\unr}_0(S, D)} X_0^{\unr} \times_{M^{\unr}_0(S, D)} X_0^{\unr} \to S \times S \times S
\]
Let $\Delta{(f)}$ denote the inverse image of $\Delta_{S}$ under $f \times f $. Then $\Delta(f)$   is closed and contains $\Delta_{X_{0}^{\unr}}$; by Lemma~\ref{lemma_unramified_doublelocus}, we have the closed subset $\dpl^\unr:=\Delta(f)\setminus\Delta_{X_{0}^{\unr}}$ of $X_0^{\unr} \times_{M^{\unr}_0(S, D)} X_0^{\unr}$. Clearly 
$\dpl^\unr$ parametrizes unramified maps $f:\P\to S$ in $M_0(S, D)$ together with a pair of points $p\neq q\in \P$ such that $f(p)=f(q)$. 
 
 Similarly, for $m\ge3$ an integer, we have the $m$-fold fiber product $(X_0^{\unr})^{\times_{M_0^{\unr}m}}$ and the morphism
 \[
 f^{\times m}:  (X_{0}^{\unr})^{\times_{M_0^{\unr}m}}\to S^m.
 \]
 Let $\Delta^{(m)}_{S}\subset S^m$ denote the (small) diagonal and let  $\Delta^{(m)}(f):=(f^{\times m})^{-1}(\Delta^{(m)}_{S})$, a closed subset of $(X_{0}^{\unr})^{\times_{M_0^{\unr}m}}$.  For $1\le i<j\le m$, let $\Delta_{X^\unr_0, i,j}\subset (X_{0,n}^{\unr})^{\times_{M_0^{\unr}m}}$ denote the $i,j$-diagonal. It follows from repeated applications of Lemma~\ref{lemma_unramified_doublelocus} that 
 \[
 (m-\text{fold})^\unr:=\Delta^{(m)}(f)\setminus \cup_{1\le i<j\le m}\Delta_{X^\unr_0, i,j}
 \]
 is a closed subset of $(X_{0,n}^{\unr})^{\times_{M_0^{\unr}m}}$. 

Since the projection 
 \[
 \pi^{(m)}_{\unr}:(X_0^{\unr})^{\times_{M^{\unr}_0(S, D)}m}\to M^{\unr}_0(S, D)
 \]
 is proper, we have the closed subset 
 \[
 D^{\unr}_{m-\text{fold}}:=\pi^{(m)}_{\unr}((m-\text{fold})^\unr)
 \]
 of $M^{\unr}_0(S, D)$, parametrizing those  $f\in M^{\unr}_0(S, D)$ having an $m$-fold point or worse.  

For the case of a tacnode, let $p:\P(T_S)\to S$ be the projectivization of the tangent bundle of $S$ and let 
\[
T_2:\dpl^\unr\to \P(T_S)\times_S\P(T_S)
\]
be the map sending $(f:\P\to S, p,q)$ to the pair of  lines $(df(T_{\P,p}), df(T_{\P,q})$, viewed as a pair of points in $ \P(T_S)$. Let $\dpl^\unr_\ddp\subset \dpl^\unr=T_2^{-1}(\Delta_{\P(T_S)})$, a closed subset of $\dpl^\unr$, hence also closed in $(X_0^{\unr})^{\times_{M^{\unr}_0(S, D)}2}$. Letting 
\[
 \pi^{(2)}_{\unr}:(X_0^{\unr})^{\times_{M^{\unr}_0(S, D)}2}\to M^{\unr}_0(S, D)
 \]
 be the projection, we see  as above that 
 \[
 D^\unr_\tac:=\pi^{(2)}_{\unr}(\dpl_\ddp\setminus  \Delta_{X_0^{\unr}})
 \]
 is a closed subset of $M_0^\unr(S,D)$ that parametrizes maps $f\in M_0^\unr(S,D)$ with a tacnode or worse.
  
 Finally, for the case of a cusp, we consider the universal curve $\pi:X^\bir_{0,1}\to M^\bir_{0,1}(S,D)$ with section $s:M^\bir_{0,1}(S,D)\to X^\bir_{0,1}$.  Consider the universal map over $ M^\bir_{0,1}(S, D)$, 
 \[
 F:X^\bir_{0,1}\to S\times_kM^\bir_{0,1}(S, D)
 \]
and let $R\subset X^\bir_{0,1}$ be the support of the cokernel of the map 
 \[
 dF:F^*(p_1^*\Omega_S)\to \Omega_{X^\bir_{0,1}/M^\bir_{0,1}(S, D)}  
 \]
 Let $\bar{R}_{0,1}:=\pi(R\cap s(M_{0,1}(S,D)))$, a closed subset of $M^\bir_{0,1}(S,D)$, and let $\bar{R}$ be the image of $\bar{R}_{0,1}$ under the projection $M^\bir_{0,1}(S,D)\to M^\bir_{0}(S,D)$. Noting that $\pi_\bir:M^\bir_{0,1}(S,D)\to M^\bir_{0}(S,D)$ is the universal curve over $M^\bir_{0}(S,D)$, so $\pi_\bir$ is proper, and thus $\bar{R}$ is closed in $M^\bir_{0}(S,D)$. 
 \end{proof}
 
 Relying on Lemma~\ref{lem:locallyclosed_Dcusp_tac_trip}, we make the following definition.
\begin{definition}\label{df:Dcusp_Dtac_Dtrip}
We let  $Z_\cusp\subset M^\bir_0(S,D)$ be the closed subset of stable maps  with a cusp or worse. We let $Z_\tac\subset M^\unr_0(S,D)$, resp. $Z_\trip\subset M^\unr_0(S,D)$ be the closed  set of stable maps  with a tacnode or worse, resp.~a triple point or worse.
\end{definition}

\begin{lemma}\label{lm:odp_in_unr_in_M_open} We have open subschemes
\[
M^\odp_{0,n}(S, D)\subset M^\unr_{0,n}(S, D)\subset M_{0,n}(S, D)
\]
\end{lemma}

\begin{proof}
Forgetting the last point defines a proper morphism
\[
\pi:X_{0,n}: = M_{0,n+1}(S, D) \to M_{0,n}(S, D)
\] from the universal curve. Evaluation on the $(n+1)$st point gives a map $f:=\ev_{n+1}: M_{0,n+1}(S, D) \to S$, which in turn induces a map of coherent sheaves $df: f^* \Omega_S \to \Omega_{M_{0,n+1}(S, D)/M_{0,n}(S, D)}$. The cokernel of $df$ has closed support on $X_{0,n}$, whence closed image under $\pi$. The open complement  in $M_{0,n}(S, D)$ of this image is $M^\unr_{0,n}(S, D)$ by Definition~\ref{def:unramified}.

Geometric points of the complement of $M^\odp_{0}(S, D)$ in $M^\unr_{0}(S, D)$ are $f: \P^1 \to S$ with either three distinct points $p_1,p_2,p_3$ such that $f(p_1) = f(p_2) = f(p_3)$ or two distinct points $p_1,p_2$ with $f(p_1) = f(p_2)$ and $f_* T_{p_1} \P^1 = f_* T_{p_2} \P^1 $. These are closed conditions as in Lemma~\ref{lem:locallyclosed_Dcusp_tac_trip}. 
\end{proof}

\subsection{Some geometry of moduli stacks of birational and/or unramified maps}
Here are two fundamental results.

\begin{theorem}[G\"ottsche-Pandharipande \hbox{\cite[Theorem 4.1]{GoettschePand}}]\label{thm:GP}  Suppose $k$ is algebraically closed and of characteristic zero,  $n=d-1\ge 1$ and $S$ is a general del Pezzo of degree $d_S$. Let  $N_{D, S}$ be the Gromov-Witten invariant counting the number of rational curves in the curve class $D$ passing through $n$ general points of $S$ and suppose $N_{D,S}>0$. Then $N_{D, S}$  is equal to the number of integral rational curves $C$ in the curve class $D$
passing through general points $p_1,\ldots, p_n$ of $S$. Moreover, for each such $C$, the associated morphism $f:\P^1\to S$ with image $C$ is unramified.
\end{theorem}
This result can be interpreted as follows:  let $$\ev:\bar{M}_{0,n}(S,D)\to S^n$$
\[
(f:\P\to S, (p_1,\ldots, p_n)) \mapsto (f(p_1),\ldots, f(p_n))\in S^n
\] denote the evaluation map, where  $\P$ is a semi-stable genus 0 curve with $n$ distinct points $p_1,\ldots, p_n$ and $f:(\P, p_1,\ldots, p_n)\to S$ is a stable  map. For $S$ general, if $N_{D, S}>0$, then $\ev:\bar{M}_0(S,D)\to S^n$ is surjective and \'etale over a dense open subset $U$ of $S^n$, moreover, for each $p_*:=(p_1,\ldots, p_n)\in U$, we have $\ev^{-1}(p_*)\subset M^\unr_0(S, D)$.

\begin{theorem}[Testa \hbox{\cite{Testa}}]\label{thm:Testa} Suppose $k$ is algebraically closed and of characteristic zero and that $d_S\ge2$. Then $\overline{M^{\bir}_0}(S, D)$ is empty or is irreducible  of dimension $d-1$. 
\end{theorem}

For results of this kind in positive characteristic, see \cite{BLRT}.

Recall that $\sN_f$ denotes the normal sheaf as defined by \eqref{eqn:NormalSheaf}.

\begin{lemma}\label{lem:dimMod_noMarked}
Suppose that $f$ is a geometric point of $M_{0}(S, D)$ such that $f:\P^1\to S$ is a birational to the image curve $f(\P^1)$ and $H^1(\P^1, \sN_f)=0$, for instance, $f$
 a geometric point of $M_0^\birf(S, D)$ or $f$ unramified.  Then $M_0(S,D)$ is a smooth scheme over $k$ of dimension $d-1$ at $f$.
\end{lemma}

\begin{proof}
Note that if $f$ is unramified then $\sN_{f} \cong \sO(d- 2)$ (Remark~\ref{rmk:computation_normal_sheaf}) and $H^1(\P^1, \sN_{f})=0$, so we may assume $f$ to be birational and $H^1(\P^1, \sN_f)=0$.

Since $f$ is birational, $f$ has no automorphisms, so
 $M_0(S,D)$ is a  $k$-scheme in a neighborhood of $f$. Since $H^1(\P^1, \sN_f)=0$ and the morphism $f:\P^1\to S$ has 

 no automorphisms, then by standard deformation theory,  $M_0(S,D)$ is smooth over $k$  at $f$, and the tangent space at $f$ is isomorphic to $H^0(\P^1, \sN_f)$ $$H^0(\P^1, \sN_f) \cong T_fM_0(S, D).$$ By Remark~\ref{rmk:computation_normal_sheaf}, $\sN_f\cong \sO_{\P^1}(m)\oplus \sN_f^\tor$ with $m=d-2-\dim_FH^0(\P^1, \sN_f^\tor)$, where $F$ denotes the field of definition of $F$. Thus $H^0(\P^1, \sN_f)$ has dimension $d-1$ over $F$, which also proves that $M_0(S,D)$ is a smooth scheme over $k$ of dimension $d-1$ at $f$ as claimed.
\end{proof}

\begin{remark}\label{rem:TfMbar=hypercohomology_complex_N_fp}\label{rem:moduli_interpretation_dev}

 More generally, consider a geometric point $(f,p_1,\ldots,p_n)$ of $M_{0,n}(S,D)$ where $f: \P^1 \to S$ is a stable map and $p_1,\ldots, p_n$ are marked points on the domain curve $\P^1$. There is a canonical isomorphism $$ T_fM_{0,n}(S,D) \cong \mathbb{H}^1(\P^1, T\P^1(-\sum_i p_i) \stackrel{df}{\to} f^* TS )$$ identifying the tangent space $ T_fM_{0,n}(S,D)$ with the hypercohomology of $\P^1$ with coefficients in the two-term complex $T\P^1(-\sum_i p_i) \stackrel{Tf}{\to} f^* TS $, where $T\P^1(-\sum_i p_i)$ is the sheaf of those tangent vector fields vanishing at the $p_i$. See \cite[p. 175]{CoxKatz-Mirror_symmetry}.

When $f$ is birational, the map $df: T\P^1(-\sum_i p_i) \to f^* TS$ is injective, and there is a canonical quasi-isomorphism between $T\P^1(-\sum_i p_i) \to f^* TS$ and the sheaf $\sN_{f,p}$ defined by $$0 \to T_{\P^1}(-\sum p_i) \to f^* T_S \to \sN_{f,p} \to 0.$$
 \end{remark}

\begin{lemma}\label{rem:DimModuli}
Suppose that $f$ is a geometric point of $M_{0,n}(S, D)$ with field of definition $F$ such that $f:\P^1\to S$ is birational and $H^1(\P^1, \sN_{f,p})=0$. Then $M_{0,n}(S,D)$ is smooth at $f$ of dimension $d-1+n$ and there is a canonical isomorphism $$ T_f M_0(S, D) \cong H^0(\P^1, \sN_{f,p}).$$
\end{lemma}

\begin{proof}
By Remark~\ref{rem:TfMbar=hypercohomology_complex_N_fp}, there is a canonical quasi-isomorphism between $\sN_{f,p}$ and $T\P^1(-\sum_i p_i) \to f^* TS$. Since $H^1(\P^1, \sN_{f,p})=0$, it follows from \cite[p. 175]{CoxKatz-Mirror_symmetry} and standard deformation theory that $\bar{M}_{0,n}(S,D)$ is smooth at $f$ and $ T_f\bar{M}_0(S, D) \cong H^0(\P^1, \sN_{f,p}).$ Thus the dimension of $\bar{M}_{0,n}(S,D)$ at $f$ is $\dim_F H^0(\P^1, \sN_{f,p})$. We have that $\dim H^0(\P^1, \sN_{f,p}) = d-1+n$ by the calculation $\sN_{f,p}\cong \sO_{\P^1}(m)\oplus \sN_{f,p}^\tor$ with $m=n+d-2-\dim_FH^0(\P^1, \sN_{f,p}^\tor)$ similarly to the above.

\end{proof}

\begin{proposition}\label{Mbar_smooth_unramified_map_from_two_comp_reducible_transverse_tangent_directions}
Let $f : (\P,p_1,\ldots,p_n) \to S$ be a point of $\bar{M}_{0,n}(S,D)$ satisfying the following conditions.
\begin{enumerate}
\item
$\P=\P_1\cup\P_2$, with $\P_i\cong \P^1$ and $\P_1 \cap \P_2 = \{p\}$.
\item
$f$ is unramified and $f|_{\P_1}$ is transversal to $f|_{\P_2}$ at $p.$
\end{enumerate}
Then $\bar{M}_{0,n}(S,D)$ is smooth at $f$ and has dimension $d -1 + n.$
\end{proposition}
\begin{proof}

Let $\sC$ denote the mapping cone of $f^*\Omega_S \to \Omega_\P(\sum_{i = 1}^n p_i)$, or equivalently $\sC$ is the two-term complex
\[
\sC = f^*\Omega_S \to \Omega_\P(\sum_{i = 1}^n p_i)
\]
By \cite[p. 175]{CoxKatz-Mirror_symmetry} the tangent space of $\bar{M}_{0,n}(S,D)$ at $f$ is $\Ext^1_\P(\sC,\sO_\P)$ and the obstructions are $\Ext^2_\P(\sC,\sO_\P)$. It follows from stability~\cite[p. 175]{CoxKatz-Mirror_symmetry} that
\[
\Ext^0(\sC, \sO_{\P}) = 0.
\]
We show that
\[
\dim \Ext^1_\P(\sC,\sO_\P) = d-1+n, \qquad \dim\Ext^2_\P(\sC,\sO_\P) = 0.
\]

By definition of $\sC$, there is long exact sequence
\begin{multline}\label{eq:lescone}
0\to  \Ext^0(\sC, \sO_{\P})\to  \Ext^0(\Omega_\P(\sum_{i = 1}^n p_i), \sO_{\P})\to \Ext^0(f^*\Omega_S, \sO_{\P}) \\
\to \Ext^1(\sC, \sO_{\P})\to \Ext^1(\Omega_\P(\sum_{i = 1}^n p_i), \sO_{\P}) \to \Ext^1(f^*\Omega_S, \sO_{\P}) \\
\to \Ext^2(\sC, \sO_{\P}) \to \Ext^2(\Omega_\P(\sum_{i = 1}^n p_i), \sO_{\P}) \to\Ext^2(f^*\Omega_S, \sO_{\P})\to \ldots 
\end{multline}
We show
\begin{equation}\label{eq:extos}
\dim \Ext^1(f^*\Omega_S, \sO_{\P}) = 0, \qquad \dim \Ext^0(f^*\Omega_S, \sO_{\P}) = d + 2.
\end{equation}
Indeed, since $f^* \Omega_S$ is locally free,
\[
\Ext^i(f^*\Omega_S, \sO_{\P}) \cong H^i(\P, f^* T_S)
\]
for all $i$. 

Let $i_j : \P_j \to \P$ denote the inclusion and let $f_j = f\circ i_j.$ Let $D_j = (f_j)_*([\P^1])$ and let $d_j = -K_S \cdot D_j.$ There is a short exact sequence
\[
0 \to f^*T_S \to (i_1)_* i_1^* f^*T_S \oplus (i_2)_* i_2^* f^*T_S \to (i_p)_*(i_p)^*f^*T_S \to 0.
\]
Since $i_j$ is affine,
\[
H^k((i_1)_* i_1^* f^*T_S \oplus (i_2)_* i_2^* f^*T_S) = H^k(\P_1,f_1^*T_S) \oplus H^k(\P_2,f_2^*T_S), \qquad k = 0,1.
\]
So, by the long exact sequence in cohomology, it suffices to show that
\begin{equation}\label{eq:h1TS}
H^1(\P_i,f_i^*T_S) = 0, \qquad \dim H^0(\P_i,f_i^*T_S) = d_i + 2,
\end{equation}
and that the map
\begin{equation}\label{eq:H0surj}
H^0(\P_1,f_1^*T_S) \oplus H^0(\P_2,f_2^*T_S) \to T_{S,f(p)}
\end{equation}
is surjective. To prove~\eqref{eq:h1TS}, consider the exact sequence
\[
0 \to T_{\P^1} \to f_i^*T_S \to \sN_{f_i} \to 0.
\]
Observe that $\sN_{f_i} \cong \sO(d_i-2).$ Moreover, since $S$ is del-Pezzo, $d_i = -K_S \cdot D_i > 0.$ Thus
\[
\dim H^0(\sN_{f_i}) = d_i-1, \qquad  \dim H^1(\sN_{f_i}) = 0.
\]
Since
\[
\dim H^0(T_{\P^1}) = 3, \qquad
\dim H^1(T_{\P^1}) = 0,
\]
equation~\eqref{eq:h1TS} follows.

To prove the surjectivity of~\eqref{eq:H0surj}, consider the commutative diagram
\[
\xymatrix{
H^0(T_{\P^1}) \ar[r]\ar[d]^{df_i} & T_{\P^1,p} \ar[d]^{(df_i)_p}\\
H^0(f_i^*T_S) \ar[r] & T_{S,f_i(p)}.
}
\]
Since the upper horizontal arrow is surjective, the image of the bottom horizontal arrow contains $(df_i)_p(T_{\P^1,p}).$ Since $f_1$ and $f_2$ are transversal at $p,$ the surjectivity of~\eqref{eq:H0surj} follows.

Next, we calculate $\Ext^k(\Omega_\P(\sum_{i = 1}^n p_i), \sO_{\P})$. Indeed, since the dualizing sheaf $\varpi$ of $\P$ is a line bundle, Serre duality gives
\[
\Ext^k(\Omega_\P(\sum_{i = 1}^n p_i), \sO_{\P}) \cong \Ext^k(\Omega_\P(\sum_{i = 1}^n p_i)\otimes \varpi, \varpi) \cong \Ext^{1-k}(\sO_\P, \Omega_\P(\sum_{i = 1}^n p_i)\otimes \varpi).
\]
This shows that $\Ext^2(\Omega_\P(\sum_{i = 1}^n p_i), \sO_{\P}) = 0.$ It follows from the exact sequence~\eqref{eq:lescone} and~\eqref{eq:extos} that $\Ext^2(\sC, \sO_{\P}) = 0$ and $\bar{M}_{0,n}(S,D)$ is smooth at $f$ as claimed.

On the other hand,
\[
\chi(\Omega_\P(\sum_{i = 1}^n p_i)\otimes \varpi) = \sum_{k = 0}^1 (-1)^k \dim H^k(\Omega_\P(\sum_{i = 1}^n p_i)\otimes \varpi)
\]
is constant in flat families.

We smooth $\P$ to a flat family  $\tilde{\P}\to\Spec k[[t]]$ with smooth generic fiber $\P^1_{k((t))}$ and with sections $\mathfrak{p}_i$ reducing to $p_i$ over $\Spec k$, $i=1,\ldots, n$. More precisely, $\tilde{\P}$ is projective over $\Spec k[[t]]$, $\tilde{\P}\setminus\{p\}\to \Spec k[[t]]$ is smooth, and an open neighborhood of $p$ in $\tilde{\P}$ is isomorphic to $\Spec k[[t]][x,y]/(xy-t)$ as $k[[t]]$-scheme.  

An easy computation shows that $\Omega_{\tilde{\P}/k[[t]]}$ is flat over $k[[t]]$;  since $\tilde{\P}\to\Spec k[[t]]$ is an lci morphism, the relative dualizing sheaf $\varpi_{\tilde{\P}/k[[t]]}$ is an invertible sheaf, hence is also flat over $k[[t]]$. Since the sheaf Euler characteristic is locally constant in flat, proper families, we have

\begin{multline*}
\sum_{k = 0}^1 (-1)^k \dim H^k(\Omega_\P(\sum_{i = 1}^n p_i)\otimes \varpi) \\=
\sum_{k = 0}^1 (-1)^k \dim H^k(\P^1_{k((t))}, \Omega_{\P^1_{k((t))}/k((t))}(\sum_{i = 1}^n \mathfrak{p}_{i, k((t))})\otimes \Omega_{\P^1_{k((t))}/k((t))})\\
=
\sum_{k = 0}^1 (-1)^k \dim H^k(\P^1_{k((t))}, \sO_{\P^1_{k((t))}}(n-4)) = n-3.
\end{multline*}

It follows from the exact sequence~\eqref{eq:lescone} and~\eqref{eq:extos} that
\[
\dim \Ext^1_\P(\sC,\sO_\P) = d-1+n.
\]
as claimed.

\end{proof}

 \begin{remark}\label{rem:mod_int_dev_kernel_cokernel_withNfSES}
Let $(f,p_1,\ldots,p_n)$ be a geometric point of $M_{0,n}(S,D)$ such that $f: \P^1 \to S$ is birational and $H^1(\P^1, \sN_{f,p})=0$. Let $F$ be the field of definition of $(f,p_1,\ldots,p_n)$. Suppose additionally that $$df \otimes F: T_{\P^1,p_i} \otimes F \hookrightarrow T_{S,q_i} \otimes F$$ is injective for all $i$. For example, if $f$ could be a geometric point of $M^\unr_{0,n}(S, D).$ Then there is an additional description of the kernel and cokernel of $d\ev$ in terms of the exact sequence $$0\to \sN_f(-\sum_ip_i)\to \sN_f\to \oplus_i (f^*T_{S, q_i} \otimes F)/df(T_{\P^1, p_i} \otimes F)\to 0,$$ where $q_i = f(p_i)$. We give this description now. Applying the snake lemma to the map of short exact sequences \begin{equation}\label{NfpNfcomp_map_SES}\xymatrix{0\ar[r] & \ar[d] T_{\P^1}(-\sum_ip_i)\ar[r] &  \ar[d]_1 f^*T_{S} \ar[r] & \ar[d] N_{f,p}\ar[r]& 0 \\
0\ar[r] & \T_{\P^1} \ar[r] & f^*T_{S} \ar[r] &N_f \ar[r]& 0} \end{equation} defines a canonical isomorphism $$\ker (N_{f,p} \to \sN_f) \stackrel{\cong}{\to} \coker (T_{\P^1}(-\sum_ip_i) \to T_{\P^1}) \cong \oplus_i \mathcal{O}_{p_i}$$ where $\mathcal{O}_{p_i} := (p_i)_* \mathcal{O}_F$ is the pushforward of the the structure sheaf of the points $p_i$.  We also deduce from \eqref{NfpNfcomp_map_SES} that $N_{f,p} \to \sN_f$ is surjective, giving the short exact sequence $$ 0 \to \oplus_i \mathcal{O}_{p_i} \to N_{f,p} \to \sN_f \to 0.$$ This gives rise to the map of short exact sequences 
$$ 
\xymatrixcolsep{10pt}
\xymatrix{0 \ar[r] & H^0(\oplus_i \mathcal{O}_{p_i}) \ar[d] \ar[r] & H^0(N_{f,p}) \ar[r] \ar[d]_{d\ev_f}& H^0(N_f) \ar[r] \ar[d] & 0  \\  0 \ar[r] & \oplus_i df(T_{\P^1, p_i} \otimes F)\ar[r] & \oplus_i  (f^*T_{S, q_i} \otimes F) \ar[r] & \oplus_i (f^*T_{S, q_i} \otimes F)/df(T_{\P^1, p_i} \otimes F) \ar[r] & 0 .}
$$ 
As the left vertical map is an isomorphism, the snake lemma gives canonical isomorphisms $$ \ker d\ev_f \cong \ker (H^0(N_f) \to \oplus_i (f^*T_{S, q_i} \otimes F)/df(T_{\P^1, p_i} \otimes F)) \cong H^0(\sN_f(-\sum_ip_i))$$ and $$ \coker ~d\ev_f \cong \coker (H^0(N_f) \to \oplus_i (f^*T_{S, q_i} \otimes F)/df(T_{\P^1, p_i} \otimes F))\cong H^1(\sN_f(-\sum_ip_i)).$$ 
\end{remark}

\begin{definition}\label{def:ramification_index}
Let $F$ be an algebraically closed extension field of $k$. For $f:\P^1_F\to S_F$ a  morphism and $p\in \P^1(F)$, choose a uniformizing parameter $t_p\in \mathfrak{m}_p\subset \sO_{\P^1, p}$ and coordinates $(x, y)$ at $q=f(p )$. Define the integer $e_p\ge0$ by  $f^*(x,y)\sO_{\P^1, p}=(t^{e_p})$; we call $e_p$ the {\em ramification index} of $f$ at $p$.
\end{definition}

\begin{definition}\label{def:torsion_index}
 If  $f^*(x,y)\sO_{\P^1, p}=(t^{e_p})$, then after a linear of coordinates, we may assume that $f^*(x)=ut^{e_p}$, $f^*(y)=vt^{e_p+r}$ with $u, v\in \sO_{\P^1,p}^\times$ and $r>0$. Thus $N_f^\tor\otimes\sO_{\P^1, p}\cong F^{t_p}$, with $t_p\ge e_p-1$, with equality if $e_p$ is prime to the characteristic.  Let $t(f)=\sum_{p\in \P^1}t_p$. We call $t(f)$ the {\em torsion index} of $f$. For $f : \P \to S$ a possibly reducible stable map, we define the torsion index $t(f)$ to be the sum of torsion indices of the restrictions of $f$ to each of the irreducible components. 
 \end{definition}

Remark~\ref{rmk:computation_normal_sheaf_for_unramified} generalizes as follows.

\begin{remark}\label{t(f)_remarks}
Let $F$ be an algebraically closed extension field of $k$, and let $f:\P^1_F\to S_F$ a morphism.
\begin{enumerate}
\item $f$ is unramified in the sense of Definition~\ref{def:unramified} if and only if $e_p=1$ for all points $p$. This is equivalent to the requirement that $t(f) = 0$.
\item By \eqref{eqn:NormalSheaf}, we have $\sN_f/\sN_f^\tor=\sO_{\P^1}(d-2-t(f))$. 
\end{enumerate}
\end{remark}

\begin{lemma}\label{lem:Ram} Suppose $k$ is a field of characteristic zero. Let $V\subset M^\bir_0(S, D)$ be an integral closed subscheme and let $f$ be a geometric generic point of $V$. Then the composition
\[
T_fV\to T_fM^\bir_0(S, D)\cong H^0(\P^1_F, \sN_f)\to H^0(\P^1_F, \sN_f/\sN_f^\tor)
\]
of the displayed canonical maps is injective. Moreover, either \begin{itemize}

\item $d-1 - \dim V\ge t(f)$ or
\item $\dim V =0$.
\end{itemize}
\end{lemma}

\begin{proof}  We prove the first assertion following the proof of a closely related result by Tyomkin \cite[Proposition 2.4]{Tyomkin}.

Since $\bar{M}_{0,n}(S,D)$ is a separated Artin stack and $M^\bir_0(S, D)$ is an open subscheme of $\bar{M}_{0,n}(S,D)$, there is an \'etale dominant map $\phi:\tilde{V}\to V$ and a morphism $F:\tilde{V}\times\P^1\to \tilde{V}\times S$ over $\tilde{V}$ representing the inclusion $V\to \bar{M}_{0,n}(S,D)$. We consider $f$ as a geometric point of $\tilde{V}$.

Sending a point $v\in \tilde{V}$ to the image curve $F(v,\P^1)\subset v\times S$ defines a morphism $\tilde{\alpha}:\tilde{V}\to |D|\cong \P^N$; if $F(v, \P^1)=F(v', \P^1)$, then since both $F(v,-)$ and $F(v'-)$ are birational maps to $F(v, \P^1)$, there is a unique isomorphism $\phi:\P^1\to \P^1$ (defined over $k(v)\otimes_{k(\tilde{\alpha}(v))}k(v')$) with $F(v',-)=F(v,-)\circ\phi$. Thus, since $\tilde{V}\to V$ is \'etale, $\tilde{\alpha}$ descends to a morphism $\alpha:V\to |D|$ and there is a dense open subscheme $U$ of $V$ over which the map $\alpha$ is an isomorphism with an open subscheme of the image scheme $\overline{\alpha(V)}\subset |D|$. For $v\in V$,  let $C_v$ be the image curve $F(\tilde{v},\P^1)$ for $\tilde{v}\in \tilde{V}$ lying over $v$.

For $v=f$ a geometric generic point of $V$, the map $f:\P^1\to C:=f(\P^1)$ is birational and $C$ has only finitely many singularities. Choose a smooth curve $E$ on $S$ (defined over $k$) such that
\begin{enumerate}[label = (\roman*)]
\item
$H^0(S, \sO_S(D-E))=0$,
\item\label{it:EtC}
$E$ intersects $C$ transversely.
\end{enumerate}
Then~\ref{it:EtC} also holds for $C_v$ for all $v$ in a dense open subset of $V$.  Letting $N=\Deg E\cdot D$, this gives us the morphism $\beta:V_0\to \Hilb_N(E)$, $\beta(v)=C_v\cap E$ for $V_0\subset V$ a dense open subscheme. By~\ref{it:EtC}, we may assume that $\beta(V_0)$ is contained in the open subscheme $\Hilb^0_N(E)$ of $\Hilb_N(E)$ parametrizing reduced closed subschemes of $E$ of length $N$, which is a smooth scheme over $k$.

We claim that after shrinking $V_0$ as necessary, $\beta$ defines isomorphism of $V_0$ with its image in $\Hilb^0_N(E)$. Since the characteristic is zero, it suffices to show that $\beta$ is injective on geometric points of $V_0$. (Indeed, by \cite[8.10.5(i)]{egaIV_3} it suffices to show that the field extension $k(\beta(\eta_V)) \subset k(\eta_V)$ is an isomorphism, where $\eta_V$ denotes the generic point of $V$ or equivalently the image of $f$. Since $k$ is characteristic $0$, this is equivalent to $\Gal( k(\eta_V)/k(\beta(\eta_V))) = 1$.)

Take $v\in V_0$ a geometric point, giving the curve $C_v$ on $S$. We have the exact sheaf sequence
\[
0\to \sO_S(D-E)\to \sO_S(D)\xrightarrow{i_E^*} \sO_E(E\cdot D)\to0;
\]
since $H^0(S, \sO_S(D-E))=0$, $i_E^*$ induces an inclusion of linear systems $i_E^*: |D|\to |E\cap D|$. Thus, for $v, v'$ geometric points of $V_0$,  if $C_v\cap  E= C_{v'}\cap E$ then $C_v=C_{v'}$, and since $\alpha:U\to |D|$ is injective on geometric points and $\beta=i_E^*\circ \alpha$, we see that $\beta(v)=\beta(v')$.

On the other hand, let $W\subset \Hilb^0_N(E)$ be a smooth locally closed subscheme and let $w\in W$ be a geometric point. Then $w$ corresponds to $N$ distinct points $q_1,\ldots, q_N$ of $E$ and $T_w \Hilb^0_N(E)$ is isomorphic to $\oplus_{i=1}^NT_{q_i}E$. Taking $W=\beta(V_0)$ and $w=\beta(f)$, the points $q_1,\ldots, q_N$ are the (transverse) intersection points of $C\cap E$, so at each $p_i$, we have $T_{S, q_i}=T_{C, q_i}\oplus T_{E,q_i}$. Since the points $q_i$ are all smooth points of $C=f(\P^1)$, and  $f:\P^1\to C$ is birational, we have $T_{C,q_i}\cong T_{\P^1,p_i}$, where $p_i=f^{-1}(q_i)$,  and the projection $T_{S, q_i}\to  T_{E,q_i}$ defines an isomorphism $\pi_i:\sN_f\otimes k(p_i)\to T_{E, q_i}$. Since $\beta$ is an isomorphism $V_0\to \beta(V_0)\subset \Hilb_N^0(E)$, sending $T_f(V)=H^0(\P^1, \sN_f)\to \oplus_{i=0}^NT_{E, q_i}=T_{\beta(f)}\Hilb^0_N(E)$ via the composition
\[
T_f(V)=H^0(\P^1, \sN_f)\xrightarrow{\Res_{p_1,\ldots, p_N}}\oplus_{i=1}^N\sN_f\otimes k(p_i)
\xrightarrow{\oplus \pi_i}\oplus_{i=0}^NT_{E, q_i}
\]
is injective. But as all the points $q_i\in C$ are smooth points, $f$ is unramified at each $q_i$, so this latter map factors through $H^0(\P^1, \sN_f)\to H^0(\P^1, \sN_f/\sN_f^\tor)$, so $T_fV\to
H^0(\P^1, \sN_f/\sN_f^\tor)$ is injective, as claimed.

As $H^0(\P^1_F, \sN_f/\sN_f^\tor)\cong F^{d-1-t(f)}$ for  $d-1-t(f)\ge0$ and is zero if $d-1-t(f)<0$ (see Remark~\ref{t(f)_remarks} ), the second assertion follows from the first, which proves the lemma. In the first case, $M^\bir_0(S, D)$ is smooth of $f$ of dimension $d-1$ by obstruction theory because the obstruction $H^1(\P^1_F, \sN_f) \cong H^1(\P^1_F, \sN_f/\sN_f^\tor)=0$ and $H^0(\P^1_F, \sN_f) \cong F^{d-1}$. In the second case, $\dim V = 0$ because $\dim V \leq \dim T_f V = 0.$
\end{proof}

\begin{remark} For $k$ of characteristic $p>0$, the first assertion of Lemma~\ref{lem:Ram} is false: Consider the family of maps $f_a:\P^1\to \P^2$
\[
f_a(t_0, t_1)=(t_0^{p-2}t_1^2+at_0^p, t_1^p, t_0^p),\ a\in \A^1,
\]
Fixing an $a$, take the tangent vector corresponding to the morphism $f_{a+\epsilon}$ over $k(a)[\epsilon]/(\epsilon^2)$. Then $f_a$ and $f_{a+\epsilon}$ have the same defining equation $y^2=x^p-a^p$, so the corresponding section of $\sN_f$ vanishes away from $t=0$.
\end{remark}

\begin{lemma} \label{lem:EvUnram} Let $k$ be  a perfect field, $S$ a del Pezzo surface over $k$ and $D$ an effective Cartier divisor on $S$. Let $n=-\Deg K_S\cdot D-1$. Let $f : (\P^1,p_*) \to S$ be a geometric point of $M_{0,n}(S,D)^{\unr}.$ Then $\bar{M}_{0,n}(S, D)$ is a smooth scheme of dimension $2n$ at $f,$ and $\ev:\bar{M}_{0,n}(S, D)\to S^n$ is \'etale at $f$.
\end{lemma}

\begin{proof} Since $f$ is unramified, $f$ is birational by Lemma~\ref{lm:unramified_maps_are_birational}, so there are no automorphisms of $f$ and $\bar{M}_{0,n}(S, D)$ is a scheme near $(f, p_*)$. Since $f:\P^1\to f(\P^1)$ is unramified, we have $\sN_f\cong \sO_{\P^1}(n-1)$ (Remark~\ref{rmk:computation_normal_sheaf_for_unramified}) and $n\ge0$, so $H^1(\P^1, \sN_f)=0$ and $H^0(\P^1, \sN_f)\cong F^n$. Lemma~\ref{rem:DimModuli} implies that $\bar{M}_{0,n}(S, D)$ is smooth of dimension $2n$ at
$(f, p_*)$. By Remark~\ref{rmk:computation_normal_sheaf_for_unramified}, the kernel and cokernel of $d\ev$ at $f$ are isomorphic to $H^0(\P^1, \sN_f(-\sum_{i=1}^n p_i))$ and $H^1(\P^1, \sN_f(-\sum_{i=1}^n p_i))$, respectively. Since $\sN_f\cong \sO_{\P^1}(n-1)$, it follows that $\sN_f(-\sum_{i=1}^n p_i)\cong \sO_{\P^1}(-1)$ so both of these terms are zero.

\end{proof}

\begin{lemma}\label{lm:dS2}
Assume $d_S = 2.$ Then, the anti-canonical map $\pi : S \to \P^2$ is a 2-1 finite morphism with branch divisor a smooth quartic curve.
\end{lemma}
\begin{proof}
This is~\cite[Theorem III.3.5 and proof]{Kollar}.
\end{proof}

\begin{lemma}\label{lm:dS2d2}
Assume $d_S = 2$ and $\ochar k \neq 2,3.$ Let $\pi : S \to \P^2$ be the anti-canonical map as in Lemma~\ref{lm:dS2}. Let $f : \P^1 \to S$ be a map which is birational onto its image $C = f(\P^1)$. Suppose that we have that $C \cdot (-K_S) = 2.$ Then one of the following holds.
\begin{enumerate}
\item \label{it:iso}
$\pi|_C : C \to \pi(C)$ is an isomorphism, $\pi(C)$ is a smooth conic, and $f : \P^1 \to C$ is an isomorphism.
\item \label{it:d2}
$\pi|_C : C \to \pi(C)$ has degree $2$ and one of the following holds.
\begin{enumerate}
\item
$f$ is unramified and $C$ has a single ordinary double point.
\item \label{it:cusp}
$C$ has a single ordinary cusp and $f$ is ramified at a single point with $t(f) = 1.$ Moreover, $\pi(C)$ is a line tangent to the branch curve of $\pi$ at a flex.
\end{enumerate}
\end{enumerate}
\end{lemma}
\begin{proof}
Since $\pi$ has degree $2,$ it follows that $\pi|_C : C \to \pi(C)$ is either birational or has degree $2.$
If $\pi:C\to \pi(C )$ is birational, then $\pi(C ) \cdot \sO(1) = C \cdot (-K_S) = 2,$ so $\pi(C)$ is a smooth conic. Hence, $f : \P^1 \to C$ and $\pi|_C : C \to \pi(C)$ are both isomorphisms.

If $\pi : C \to \pi(C)$ has degree $2,$ then
\[
2 (\pi(C ) \cdot \sO(1)) = \pi_*(C) \cdot \sO(1) = C \cdot (-K_S) = 2,
\]
so $\pi(C)$ is a line $\ell.$ Let $E \subset \P^2$ be the branch curve of $\pi.$ There are five possible cases:
\begin{gather*}
\ell \cdot E = p_1 + p_2 + p_3 + p_4, \qquad \ell\cdot E=2\cdot p_1 +p_2+p_3, \\
 \ell\cdot E=3\cdot p_1+p_2,\qquad \ell \cdot E = 2 p_1 + 2p_2, \qquad  \ell\cdot E=4\cdot p,
\end{gather*}
with the $p_i$ distinct in the first four cases. If $\ell\cdot E$ were the sum of 4 distinct points, then $C$ would be a smooth curve of genus 1 contrary to the hypothesis. In the last two cases, $C$ would be geometrically reducible contrary to the hypothesis. In the second case $C$ has an ordinary double point, so $f$ must be unramified. In the third case, $C$ has an ordinary cusp, and since $\ochar k \neq 3,$ it follows that $f$ is ramified at a single point with $t(f) = 1.$
\end{proof}

Often we will want to use the following assumption.

\begin{assumption}\label{a:genericunram}
For every effective Cartier divisor $D'$ on $S$, there is a geometric point $f$ in each irreducible component of $M^\bir_0(S, D')$ with $f$ unramified.
\end{assumption}

We prove in Appendix~\ref{Appendix:A} that Assumption~\ref{a:genericunram} holds for $\Char k > 3$ and $d_S \geq 3$. See Theorem~\ref{thm:hyp:pc}. We thank Sho Tanimoto for suggesting the argument.

\begin{lemma}\label{lm:char0_implies_Assumption_a:genericunram}
If $\Char k = 0$ and $d_S \geq 2,$ then Assumption~\ref{a:genericunram} holds.
\end{lemma}
\begin{proof}
Let $f$ be a geometric generic point of $M^\bir_0(S,D')$. By Theorem~\ref{thm:Testa} the scheme $M^\bir_0(S,D')$ is irreducible of dimension $\deg(-K_S\cdot D')-1$. Consequently, Lemma~\ref{lem:Ram} implies that the map $H^0(\P^1_F, \sN_f)\to H^0(\P^1_F, \sN_f/\sN_f^\tor)$ is injective. So, $\sN_f^\tor$ is trivial and $f$ is unramified.
\end{proof}

\begin{proposition} \label{prop:GenODP} Suppose that $k$ is a  field of characteristic $\neq2,3$. Furthermore, suppose $d_S \geq 2$ and Assumption~\ref{a:genericunram} holds.  Let $f\in M^\bir_0(S, D)$ be a geometric generic point.
Then $f$ is in $M^\odp_0(S, D)$.
\end{proposition}

\begin{proof}
Since the condition to be unramified is open, Assumption~\ref{a:genericunram} implies that $f$ is unramified. By Remark~\ref{rmk:computation_normal_sheaf_for_unramified} we have $\sN_f = \calO(d-2).$ Since $d \geq 1,$ it follows that $H^1(\P^1,\sN_f) = 0.$ Therefore, Lemma~\ref{rem:DimModuli} gives $\dim_f M_0^\bir(S,D) = d-1.$

Let $C:=f(\P^1)$. Suppose first that $d\ge 4$. Since $\dim_f M^\bir_0(S, D)=d-1$, we may apply \cite[Theorem 2.8]{Tyomkin}, which gives the result in this case. Since this result is proven under the assumption of characteristic zero, we give a quick sketch of the proof of the relevant portion of the result. Since $f$ is unramified, we have $\sN_f\cong \sO_{\P^1}(d-2)$. We need to show that any point $q$ of $C$ has at most two preimages, and moreover, if two points of $\P^1$ have the same image under $f$, the images of their tangent spaces are distinct. Suppose first that there are three distinct points $p_1, p_2, p_3\in \P^1$ with $f(p_i)=q$ for $i=1,2,3$ for the sake of contradiction. Identify $H^0(\P^1, \sN_f)$ with $T_fM^\bir_0(S, D)$, and consider the first order deformation of $f$ corresponding to $s\in H^0(\P^1, \sN_f)$. Since $f$ is birational, there is an open neighborhood of $q$ such that all other points of the neighborhood have at most one preimage under $f$.  Since $f$ is a geometric generic point, the first order deformation must retain the property that there are three points mapping to one. Thus if $s(p_1)=s(p_2)=0$, then $s(p_3)=0$ as well. On the other hand, since $d-2\ge 2$ and  $\sN_f\cong \sO_{\P^1}(d-2)$, we may find an $s\in H^0(\P^1, \sN_f)$ with $s(p_1)=s(p_2)=0$ but $s(p_3)\neq0$, which yields the desired contradiction.

We are now reduced to eliminating the possibility that we have points $p_1, p_2\in \P^1$ with $f(p_1)=q=f(p_2)$ and $df(T_{\P^1, p_1})=df(T_{\P^1, p_2})$. Suppose that $d\ge5$. Since $f$ is unramified, $\sN_f\cong \sO_{\P^1}(d-2)$,  and since $d-2\ge3$,  we can find a section $s\in H^0(\P^1, \sN_f)$  with $s$ having a 2nd order zero at $p_1$, and a zero of order one at $p_2$. For the associated deformation $f_u$ of $f$ defined over $F[[u]]$, we have, to first order, $f_u(p_1)=f_u(p_2)$, $df_u(T_{\P^1, p_1})=df(T_{\P^1, p_1})$ but $df_u(T_{\P^1, p_2})\neq df(T_{\P^1, p_2})$: if we take analytic coordinates $(x,y)$ on $S$ at $q:=f(p_1)=f(p_2)$ so that the image of the branch of $f$ around $p_2$ is defined by $y=0$, then for a suitable local parameter  $t$ on $\P^1$ at $p_2$, we have $f_u(t)=(t, aut)$ modulo terms of hgher order in $t$ and $u$, with $a\neq0$. Thus $df_u(T_{\P^1, p_2})\subset T_{S,q}\cong \A^2$ is the span of the vector $(1, au)$, while $df_u(T_{\P^1, p_2})$ is the span of $(1,0)$, both modulo $u^2$.   This eliminates the tacnode in $f(\P^1)$ at $q$  by taking the deformation  $f_u(\P^1)$; as above, this implies that there was no tacnode in $f(\P^1)$ to begin with.

Suppose $d=4$ and $d_S\ge3$. The anti-canonical map embeds $S$ in a $\P^{d_S}$, so we may consider $f(\P^1)$ as a degree four rational curve in $\P^{d_S}$ with a tacnode at $q=f(p_1)=f(p_2)$. We claim that $f(\P^1)$ is contained in a $\P^2\subset \P^{d_S}$. Since every degree four rational curve is the linear projection of degree four rational normal curve in $\P^4$, the fact that $f(\P^1)$ is not smooth  implies that $f(\P^1)$ is contained in a $\P^3\subset \P^{d_S}$.  Let $\ell$ be the tangent line to the tacnode of $f(\P^1)$ and consider a plane $\Pi$ containing $\ell$. If $f(\P^1)\not\subset\Pi$, then since $\ell$ is tangent to each of the two branches of $f(\P^1)$ at $q$, the intersection multiplicity at $q$ of $\Pi$ and $f(\P^1)$ in $\P^3$  is at least 4, hence equal to 4 since $f(\P^1)$ has degree 4. Taking a point $q'\in f(\P^1)$, $q'\neq q$, we can take $\Pi'$ to be the plane spanned by $\ell$ and $q'$. But if $f(\P^1)\not\subset \Pi'$, then $\Pi'\cdot f(\P^1)$ has degree $\ge5$, which is impossible, so $f(\P^1)$ is contained in $\Pi'$.

This implies that the intersection multiplicity in $\Pi'$ at $q$ of $\ell$ with $f(\P^1)$ is 4, and thus $\ell$ has intersection multiplicity 2 with each of the two branches of $f(\P^1)$ at $q$. Using the fact that $f(\P^1)$ has arithmetic genus 3, one sees that in local analytic coordinates at $q$ on $\Pi'$, $f(\P^1)$ has equation of the form $(y-x^2)(y-ax^2-bx^3+\ldots)=0$, where either $a\neq1$ or $a=1$ and $b\neq0$ (here we are using the assumption that $\Char k\neq2,3$). Consider now $f(\P^1)$ as a smooth curve on $S$. We claim there is a choice of analytic coordinates at $q$ on $S$ so that $f(\P^1)$ also has equation of the same form in the completion of $\sO_{S,q}$. Indeed, $\Pi'$ must be tangent to $S$ at $q,$ because the Zariski tangent space of $f(\P^1)$ at $q$ has dimension $2$ since $q$ is a singular point, and therefore it is equal to the Zariski tangent space of $S$ at $q$. It follows that a projection from $S$ to $\Pi'$ is a local analytic isomorphism and the form of the equation of $f(\P^1)$ at $q$ is unchanged.
We may assume that the branch through $p_1$ has the equation $y=x^2$ and the branch through $p_2$ the equation $y=ax^2+bx^3+\ldots$. Since $d=4$, we have $\sN_f=\sO_{\P^1}(2)$, so there is a section $s$ of $\sN_f$ having a zero of order 2 at $p_1$ and with $s(p_2)\neq0$. The resulting deformation $f_u$ of $f$ has image curve $f_u(\P^1)$ with local analytic equation at $q$ of the form $(y-x^2)(y-ax^2-bx^3+\ldots-u)=0$, modulo higher order terms in $u$. Thus, intersection of the two local branches of $f_u(\P^1)$ coming from a neighborhood of $p_1$ and a neighborhood of $p_2$ is of the form $(1-a)x^2-bx^3=u$, which in characteristic $\neq 2,3$ shows that the tacnode has separated into two ordinary double points if $a\neq1$,  respectively,  three ordinary double points  if $a=1$, $b\neq0$. As above, this shows that there was no tacnode on $f(\P^1)$ to begin with.

Suppose that $d_S\ge 3$ and $d \leq 3$. We rule out multiple points and tacnodes by a global argument. When $d = 3,$ as above, $f(\P^1)$ is a rational cubic curve in $\P^{d_S}$. Thus, $f(\P^1)$ is either a smooth twisted cubic curve in a $\P^3\subset\P^{d_S}$, or a singular cubic in a $\P^2\subset \P^{d_S}$. In the first case, there is nothing to show, and in the second, since $f$ is unramified, $f(\P^1)$ has a single ordinary double point as singularity. If $d_S\ge 3$ and $d=1,2$, then $f(\P^1)$ is a line ($d=1$) or a smooth conic ($d=2$). This completes the proof for $d_S\ge 3$.

If $d_S=2$, we are in the situation of Lemma~\ref{lm:dS2} with anti-canonical double cover $\pi:S\to \P^2$ branched along a smooth degree four curve $E$. We have handled the case $d\ge5$ above.
We handle the case $d=4$ as follows. Let $C=f(\P^1)$. Then either $\pi:C\to \pi(C)$ is a double cover, with $\pi(C)$ a smooth conic, or $\pi:C\to \pi(C)$ is birational, with $\pi(C)$ a rational quartic curve. In the latter case, we need only eliminate the case of $C$ having a tacnode. If $C$ does have a tacnode, at say $q'\in C$ then   $\pi(C)$ has a tacnode at $q:=\pi(q')$. Since $\pi(C)$ is a quartic curve, the tacnode on $\pi(C)$ has local analytic equation as above: $(y-x^2)(y-ax^2-bx^3+\ldots)$ with $a\neq1$ or $a=1$ and $b\neq0$. If the map  $\pi$ is unramified at $q'$, then $C$ has the same local analytic equation at $q'$ as does $\pi(C)$ at $q$, in suitable analytic coordinates $x',y'$. If $\pi$ is ramified at $q'$, then a local analytic computation shows that $C$ has  local analytic equation at $q'$ of the form $(y'-x^{\prime2})(y'-a'x^{\prime2}+\ldots)$ with $a'\neq1$, again, in suitable analytic coordinates $x',y'$. In either case, we proceed exactly as we did above in the case $d_S\ge3$, $d=4$.

If $C\to \pi(C)$ is a double cover, then $\pi(C)$ is a smooth conic and $\pi(C)\cdot E$ must be of the form
\[
\pi(C)\cdot E=p_1+p_2+2q_1+2q_2+2q_3
\]
with $p_1\neq p_2$ and the $p_i$ distinct from all the $q_j$ (see the proof of Lemma~\ref{lm:dS2d2}). If all the $q_j$ are distinct, then $C$ is smooth outside of ordinary double points at the points $q_j'$ over $q_j$, $j=1,2,3$. If however $q_1=q_2$ then $C$ acquires an ordinary tacnode at $q_1'=q_2'$ and if  $q_1=q_2=q_3$ then $C$ acquires a higher order tacnode at $q_1'=q_2'=q_3'$. Since $d=4$,  we have $\dim_f M^\bir_0(S, D)=3$, so we need only show that the dimension of the space of smooth conics that have $\pi(C)\cdot E$ of the form $p_1+p_2+2q_1+2q_2+2q_3$ with at least two of the $q_j$ equal is at most 2.

For this, fix $q_1=q_2=q$ and consider the linear system on $E$ cut out by degree two curves $C'$ with $C'\cdot E-4q-2q_3>0$. This is the projective space on $H^0(E, \sO_E(2)(-4q-2q_3))$. If  $h^0(E, \sO_E(2)(-4q-2q_3))>0$,  then $C'\cdot E-4q-2q_3=p_1+p_2$ is effective divisor of degree two. By adjunction, the canonical class $K_E$ is $\sO_E(1)$, so $\P(H^0(E, K_E(-p_1-p_2)))$ is the projective space of lines $\ell$ in $\P^2$ with $\ell\cdot E\ge p_1+p_2$, in other words, $\P(H^0(E, K_E(-p_1-p_2)))$ is the line through $p_1$ and $p_2$ if $p_1\neq p_2$, or the line tangent to $E$ at $p$ if $p=p_1=p_2$. Thus $h^1(E, \sO_E(2)(-4q-2q_3))=h^0(E, K_E-p_1-p_2)=1$ and by Riemann-Roch, we have
\[
h^0(E, \sO_E(2)(-4q-2q_3))=2+1-3+h^1(E, \sO_E(2)(-4q-2q_3))=1,
\]
assuming that $h^0(E, \sO_E(2)(-4q-2q_3))>0$.
In other words, for fixed $q, q_3\in E$,  there is at most one smooth conic $C'$ with $C'\cdot E-4q-2q_3=p_1+p_2$ with $p_1\neq p_2$ and the $p_j$ distinct from $q, q_3$. We can then vary the points $q, q_3$ over $E$, to conclude that the space of smooth conics that have $\pi(C)\cdot E$ of the form $p_1+p_2+4q+2q_3$ as above has dimension at most 2. The same argument shows that the space of smooth conics that have $\pi(C)\cdot E$ of the form $p_1+p_2+6q$ with $p_1\neq p_2$ and $p_j\neq q$ for $j=1,2$ has dimension at most 1. This finishes the proof in case $d_S=2$, $d=4$.

It remains to handle the cases $d=1,2,3$, $d_S=2$, and with $\Char k\neq 2,3$.

  If  $d=3$, then $\pi:C\to \pi(C )$ is birational and $\pi(C )$ is a degree 3 integral rational curve in $\P^2$, and hence has a single singularity, which is either an ordinary double point or an ordinary cusp. Thus $C$ itself is either smooth or also has has a single singularity, which is either an ordinary double point or an ordinary cusp. Since $f$ is unramified, $C$ cannot have an ordinary cusp. Thus $f$ is in  $M^\odp_0(S, D)$.

For  $d=1$,  $\pi(C )$ is a line and thus $f : \P^1 \to C$ and $\pi:C\to \pi(C )$ are isomorphisms. If $d=2$, then we are in the situation of Lemma~\ref{lm:dS2d2}. In all cases of the lemma except~\ref{it:d2}\ref{it:cusp} we see immediately that $f$ is in $M^\odp_0(S, D).$ We show that case~\ref{it:d2}\ref{it:cusp} does not occur as follows.
Either $\pi:C\to \pi(C )$ is birational, in which case $\pi(C )$ is a smooth conic and $C$ is smooth, or $\pi:C\to \pi(C )$ has degree two, in which case $\pi(C )$ is a line $\ell$. In this latter case, let $E\subset \P^2$ be the smooth quartic branch curve of the map $\pi:S|to \P^2$. Since $f:\P^1\to C$ is birational, there are three possible cases: either $\ell\cdot E=2\cdot p_1 +p_2+p_3$, $\ell\cdot E=3\cdot p_1+p_2$ or $\ell\cdot E=4\cdot p$, with the $p_i$ distinct in the first two cases; if $\ell\cdot E$ were the sum of 4 distinct points, then $C$ would be a smooth curve of genus 1. In the first case $C$ has an ordinary double point, in the second, an ordinary cusp and in the third a tacnode (in this latter case, $\pi^{-1}(\ell)$ is a union of two -1 curves, intersecting at a single point with multiplicity 2, but we will not need this fact).

By Lemma~\ref{lem:Deg4Tangent} there are only finitely many possibilities for $\ell=\pi(C)$ if $\ell\cdot E=3\cdot p_1+p_2$ or $\ell\cdot E=4\cdot p$. Since $f$ is a generic point and $\dim_f M^\bir_0(S, D)=d-1=1$, $f$ is not of this form. This completes the proof.
\end{proof}

\begin{lemma}\label{lem:Deg4Tangent} Let $k$ be an algebraically closed field of characteristic $\neq 2, 3$. Let $E\subset \P^2$ be a smooth degree four curve. Then for all but finitely many points $p\in E$, the tangent line $\ell_p$ to $E$ at $p$ intersects $E$ at $p$ with multiplicity two. Moreover, $E$ has only finitely many bi-tangents.
\end{lemma}

\begin{proof} Since $E$ is a smooth quartic curve, $E$ has genus two; let $J(E)$ denote the Jacobian of $E$. We first show that for at most finitely many $p$, one has $\ell_p\cdot E=4\cdot p$. Indeed, if this is the case for $p, q$, we have $\sO_E(4\cdot (p-q))\cong \sO_E$ . Thus, if $\ell_p\cdot E=4\cdot p$ for all but finitely many $p\in E$, then    the map $\alpha:E(k)\times E(k)\to J(E)(k)$, $\alpha(p,q)=[\sO_E(p-q)]\in \Pic(E)$, has image in the 4-torsion subgroup of $J(E)(k)$, plus possibly finitely many additional points of $J(E)(k)$. This is impossible, since the image of $\alpha$ generates $J(E)(k)$ and $J(E)$ is an abelian variety of dimension two.

Suppose that for all $p\in E$, we have $\ell_p\cdot E=3\cdot p+p'$ (possibly $p=p'$) and choose $q\in E$ such that $\ell_q\cdot E=3\cdot q+q'$ with $q'\neq q$ and such that a general line through $q$ intersects $E$ in four distinct points. Let $\pi:C\to \P$ be the linear projection from $q$, where $\P$ is the $\P^1$ of all lines in $\P^2$ containing $q$. For $\ell$ such a line, $\pi^{-1}(\ell)=\ell\cdot E-q$. Then $\pi$ has degree three and since we are assuming the characteristic is different from three, $\pi$ is a separable morphism. For $p\neq q$ in $E$, $\pi$ is ramified at  $p$ if and only if the tangent line $\ell_p$ contains $q$; at such $p$,  $\pi$ has ramification index $e_p(\pi)=3$. For $\ell=\ell_q$, $\pi^{-1}(\ell)=\ell_q\cdot E-q=2q+q'$, so $e_q(\pi)=2$; at all other points $x\in E$, $e_x(\pi)=1$. Since the characteristic is $\neq 2,3$, $\pi$ is everywhere tamely ramified. But then the Riemann-Hurwitz formula says
\[
3\cdot (-2)+\sum_{x\in E}(e_x(\pi)-1)=2g(E)-2=4
\]
which is not possible, since $\sum_{x\in E}(e_x(\pi)-1)=(e_q(\pi)-1)+\sum_{p, e_p(\pi)=3}e_p(\pi)-1$ is odd.

We finish by showing that $E$ has only finitely many bi-tangents (we include as a bi-tangent a line $\ell$ with $\ell\cdot E=4p$). By what we have already shown, for all but finitely many points $p\in E$, each line $\ell$ through $p$ that is also a tangent line to $E$ at some point $q\neq p$ satisfies $\ell\cdot E=p+p'+2q$ with $p'\neq q$, and if $\ell$ is the tangent line to $E$ at $p$, then $\ell\cdot E=2p+q+q'$ with $p\neq q$, $p\neq q'$. Taking such a point $p$ and considing the projection from $p$, $\pi:E\to \P^1$ as above, we see that each ramified point $q$ of $\pi$ satisfies $e_q(\pi)=2$. By the Riemann-Hurwitz formula, this says that there are exactly 10 such points, so there are 10 lines $\ell$ through $p$ with $\ell\cdot E=p+p'+2q$ and with $p'\neq q$. At most one of these lines can be the tangent line to $E$ at $p$, so there exists at least nine ponts $q$ on $E$ such that the tangent line to $E$ at $q$ is not a bi-tangent. Since the set of bi-tangents is a closed subset of the dimension one variety of all tangent lines to $E$, $E$ has only finitely many bi-tangents.
\end{proof}

\begin{remark} The Fermat quartic $E\subset \P^2$ defined by $\sum_{i=0}^2X_i^4=0$ is an example of a smooth quartic curve over a field of chararcteristic three such that each tangent line $\ell_p$ has at least a three-fold intersection at $p$: the Hessian matrix is identically zero. We don't know an example in characteristic two.
\end{remark}

\begin{lemma}\label{lem:intersection}
Suppose that $k$ is an algebraically closed field of characteristic $\neq2,3$. Suppose $d_S\ge 2$ and $S$ satisfies Assumption~\ref{a:genericunram}. Let $f\in M^\birf_0(S, D)$ be a geometric generic point over $k$.
Let $C\subset S$ be a reduced curve defined over $k$. Then each point $p\in f(\P^1)\cap C$ is a smooth point of $C$. Moreover, if $d\ge3$, then each point $p\in f(\P^1)\cap C$ is a smooth point of $f(\P^1)$ and $f(\P^1)$ and $C$ intersect transversely at $p$.
\end{lemma}

\begin{remark}
The assumption that $d \geq 3$ is necessary. For example, let $C$ be the image of the map $\P^1 \to \P^1 \times \P^1$ given by $t \mapsto (t, t^p)$ and let $f$ be $\P^1 \times [1,0]$.
\end{remark}

\begin{proof} Let $C'=f(\P^1)$. By Proposition~\ref{prop:GenODP}, the map $f:\P^1\to C'$ is birational and unramified and $C'$ has only ordinary double points as singularities. Since $f$ is free and birational, the degree $d:=C'\cdot(-K_S)$ satisfies $d\ge 2$.

We first show that for $p\in S(k)$ a $k$-point, $p$ is not in $C'$. Let $F$ be an algebraically closed field of definition for $f:\P^1\to S$ and suppose that $p$ is in $f(\P^1)$. There are two cases: $f^{-1}(p)=\{q_1, q_2\}$ with $q_1\neq q_2$ in $\P^1(F)$, or $f^{-1}(p)=q$ is a single $F$-point of $\P^1$. In the first case, $p$ is an ordinary double point of $C'$; let $\ell_1$, $\ell_2$ be the two tangent lines, $\ell_i=df(T_{\P^1, q_i})$. Since $d\ge2$ and $f$ is unramified, $\sN_f=\sO_{\P^1}(d-2)$, so there is a section $s$ of $\sN_f$ with $s(q_i)\neq 0$ for $i=1,2$. There are thus analytic coordinates $x,y$ for $S$ at $p$ such that $C'$ has equation $xy=0$ at $p$ and a deformation $f_u$ corresponding to $s$, and defined over $F[[u]]$,  has image curve $f_u(\P^1)$ with equation $(x-au)(y-bu)$ with $ab\neq 0$, modulo terms of order $\ge2$. Considering $f_u$ as a morphism defined  the algebraic closure $\overline{F((u))}$, this shows that  $p$ is not in $f_u(\P^1)$. Since $f$ was already a geometric generic point of $M^\birf_0(S, D)$ over $k$, $p$ is not in $f(\P^1)$. A similar argument treats the case where $f^{-1}(p)$ is a single point.

In particular, this shows that $f(\P^1)\cap C$ contains no singular point of $C$.

Now suppose $d\ge 3$ and  take $p\in f(\P^1)\cap C$. If $p$ is a singular point of $f(\P^1)$, we have $f^{-1}(p)= \{q_1, q_2\}$ with $q_1\neq q_2$ in $\P^1(F)$. We first show that $df(T_{\P^1,q_i})\neq T_{C,p}$ for $i=1,2$. For this, we already have $df(T_{\P^1,q_1})\neq df(T_{\P^1,q_2})$, so we may assume that $df(T_{\P^1,q_1})=T_{C,p}$, $df(T_{\P^1,q_2})\neq T_{C,p}$, and that in the local anaytic description of $f(\P^1)$ as $xy=0$, $T_{C,p}$ is given by $x=0$ and $df(T_{\P^1,q_2})$ is given by $y=0$. This also identifies $\sN_f\otimes F(q_1)$ with $df(T_{\P^1,q_2})$. Since $d \geq 3$, there is a section $s$ of $\sN_f$ with $s$ having a zero of order one at $q_1$, $s(t)=at+\ldots$, $a\neq0$, where $t$ is the local coordinate at $q_1$ given by the pullback of $y$ and we use a local trivialization of $\sN_f$ at $q_1$ corresponding to a trivialization of $df(T_{\P^1,q_2})$. Letting $f_u$ be a deformation of $f$ over $F[[u]]$ with first order term given by $s$. This gives the equation for $f_u(\P^1)$ of the form $(x-auy)(y-bu)=0$, modulo terms of higher degree, which shows $df_u(T_{\P^1,q_1})\neq T_{C,p}$. As $f$ is already a geometric generic point over $k$, this shows that  $df(T_{\P^1,q_1})\neq T_{C,p}$ to begin with.

A similar argument shows that $df(T_{\P^1,q})\neq T_{C,p}$ if $p$ is a smooth point of $f(\P^1)$ and $f(q)=p$.

Now suppose that there is a point $p\in f(\P^1)\cap C$ that is a singular point of $f(\P^1)$. Write $f^{-1}(p)=\{q_1, q_2\}$. Since $d\ge3$, there is a section $s$ of $\sN_f\cong \sO_{\P^1}(d-2)$ such $s(q_1)=0$, $s(q_2)\neq 0$. Since $f(T_{\P^1,q_i})\neq T_{C,p}$ for $i=1,2$, we have analytic coordinates $x,y$ at $p$ such that $f(\P^1)$ is defined by $xy=0$ and $C$ is defined by $y=x+\ldots$.  We identify $\sN_f\otimes F(q_i)$ with $T_{C,p}$, $i=1,2$ and use the pullback of $y-x$ as a local coordinate at $q_1$. Letting $f_u$ be a deformation of $f$ corresponding to $s$, we have the equation for $f_u(\P^1)$ of the form $(x-au(y-x))(y-bu)=0$ modulo terms of higher order, and with $b\neq0$. This shows that the double point on $f_u(\P^1)$ is $(0, bu)$ in these coordinates, modulo terms of higher order, and thus the tangent vector describing the 1st order movement of the double point of $f(\P^1)$ is non-zero in the normal bundle of $C$ at $p$. This shows that the double point $p$ of $f(\P^1)$ moves away from $C$ in $f_u(\P^1)$  over $\overline{F((u))}$. As above, this shows that each point of $f(\P^1)\cap C$ is smooth on $f(\P^1)$.

Since $df(T_{\P^1, q})\neq T_{C,f(q)}$ if $p=f(q)$ is in $C$, this implies that  $f(\P^1)$ and $C$ intersect transversely at each intersection point $p$.

\end{proof}

\begin{lemma}\label{lemRam2} Let $k$ be a field of characteristic $0$. Let $V\subset M^\birf_0(S, D)$ be an integral closed subscheme, let $f\in V$ be a geometric generic point and let $C:=f(\P^1)\subset S$. Suppose that $C$ has a cusp at $q\in S$ and let  $p\in \P^1$ be the point lying over $q$.  Then:
\begin{enumerate}
\item\label{it:codimV}
 $\codim V\ge 1$.
\item\label{it:qordinary}
If $\codim V=1$ and either $d_S\ge3$ or $d\ge 6$,  then $q$ is an ordinary cusp and  $f$ is unramified on $\P^1\setminus\{p\}$.
\end{enumerate}
\end{lemma}

\begin{proof} \ref{it:codimV} Since $C$ has a cusp, $f$ is ramified, and thus $\sN_f/\sN_f^\tor\cong\sO_{\P^1}(d-2-t(f))$ with  $t(f)\ge1$ as in Remark~\ref{t(f)_remarks}. Since the map $T_fV\to H^0(\P^1, \sN_f/\sN_f^\tor)$ is injective (Lemma~\ref{lem:Ram}) and $M^\birf_0(S, D)$ is smooth of dimension $d-1$, we have $\codim V \ge 1$.

\ref{it:qordinary} Now suppose that $\codim V=1$. By the computation above, we have $t(f)=1=t_p(f)$ $f$ is unramified away from $p$ and $e_p(f)=2$. Suppose that the cusp at $q$ is a higher order cusp; this implies that $C$ is not a component of any $H\in  |-K_S|$. Take a standard system of parameters $t, (x,y)$ for the cusp, so $C$ has local equation $y^2=x^{2n+1}$, $n\ge2$.  If $d_S\ge3$ there is a $\P^1$ of curves $H\in |-K_S|$ passing through $q$ and tangent to the limit tangent line at $q$. All such $H$ have local equation of the form $y=a_2x^2+a_3x^3+\ldots$, so there is at least one such $H$ with $a_2=0$. This yields a multiplicity of at least 6 for $q$ in $H\cdot C$,  so $d\ge6$. Thus, in any case, we have $d\ge6$.

We have  $\sN_f/\sN_f^\tor\cong \sO_{\P^1}(d-3)$, so as $d-3\ge3$,  there is a global section $s$ of  $\sO_{\P^1}(d-3)$ having a zero of order 3 at $p$.
In our standard coordinate system $t, (x,y)$, we have $f(t)=(t^2, t^{2n+1})$ for some $n>1$. Defining the invertible subsheaf $\sL\subset f^*T_S$ as the kernel of $f^*T_S\to \sN_f/\sN_f^\tor$, we have the injective map $df:T_{\P^1}\to \sL$ with image $\sL(-p)\subset \sL$. Thus $\sL\cong \sO_{\P^1}(3)$, $H^1(\P^1, \sL)=0$, so $H^0(\P^1, f^*T_S)\to H^0(\P^1,  \sN_f/\sN_f^\tor)$ is surjective and we may lift $s$ to $\tilde{s}\in H^0(\P^1, f^*T_S)$. With respect to our standard parameters $t, (x,y)$, we have $\sL\otimes F(p )=F\cdot \del/\del x|_p$ and $\del/\del y|_p$  maps to a generator of $\sN_f/\sN_f^\tor\otimes F(p )$. Thus, in $f^*T_S\otimes_{\sO_{\P^1}}\sO_{\P^1, p}^\wedge$, we have
\[
\tilde{s}=a(t)\cdot \del/\del x +b(t)\cdot t^3\del/\del y
\]
with $a(t)\in \sO_{\P^1, p}^\wedge$ and $b(t)$ a unit in $\sO_{\P^1, p}^\wedge$.

The section $\tilde{s}$ defines a 1st order deformation $f_{\epsilon,1}$ of $f$, which one can lift to a deformation $f_\epsilon$ over $F[[\epsilon]]$, since $H^1(\P^1, f^*T_S)=0$. From our description of $\tilde{s}$, we have
 \[
f_{\epsilon}\equiv f_{\epsilon,1}(t)=(t^2, t^{2n+1})+\epsilon(a(t), b(t)\cdot t^3)\mod \epsilon^2
\]
By a translation in $x$ ($\equiv\id\mod \epsilon$) we may assume that $a(t)=0$ and thus
 \[
f_{\epsilon}=(t^2, t^{2n+1}+\epsilon\cdot b(t)\cdot t^3)\mod \epsilon^2
\]
Working over $F((\epsilon))$ we may replace $y$ with $(1/b(0)\epsilon)\cdot y-\sum_{j\ge 2}b_jx^j+y\cdot \sum_{j\ge 1}c_jx^j$ to form a standard coordnate system $t, (x, y_\epsilon)$ with
\[
f_\epsilon^*(x)=t^2, f_\epsilon^*(y_\epsilon)=t^3
\]

Over the field $F((\epsilon))$, the image curve $C_\epsilon:=f_\epsilon(\P^1)$ has an ordinary cusp at $f_\epsilon(p )$, which shows that $V$ is a proper closed subscheme of an irreducible component of $Z_\cusp$; as $\codim Z_\cusp\ge1$ by Lemma~\ref{lem:Ram}, this contradicts $\codim V=1$.
\end{proof}

\begin{lemma} \label{lem:tacnode1} Let $k$ be a field of characteristic zero.  Let $V\subset M^\birf_0(S, D)$ be an integral closed subscheme, let $f\in V$ be a geometric generic point and let $C:=f(\P^1)\subset S$. 
\begin{enumerate}
\item\label{it:codimtac}
Suppose that $d_S\ge 2$ or $d\ge 4$ and that $C$ has a tacnode. Then $\codim V\ge 1$.
\item\label{it:tacord2}
Suppose that $d_S\ge 4$, or  $d_S=3$ and $d\neq 6$, or $d\ge 7$. Suppose  that $C$ has a tacnode of order $\ge 2$.  Then $\codim V\ge 2$.
\end{enumerate}
\end{lemma}

\begin{proof} \ref{it:codimtac} By Lemma~\ref{lem:Ram}, we may assume that $f$ is unramified, so $\sN_f\cong \sO_{\P^1}(d-2)$. Suppose that $C$ has the tacnode at $q$. We note that $C$ is not a component of any $H\in |-K_S|$: if $S_{\bar{k}}\not\cong\P^1\times\P^1$, then $S_{\bar{k}}$ is the blow-up of $\P^2$ at $9-d_S$ points and thus each $H\in |-K_S|$ projects to a cubic curve in $\P^2$ (containing those points). Since no irreducible component of a cubic plane curve has a tacnode, $C$ cannot be a component of $H$. In case $S_{\bar{k}}\cong\P^1\times\P^1$, take an $H\in |-K_S|$ and blow up $S$ at a smooth point of $H$, $\pi:S'\to S$. The proper transform of $H$ to $S'$ is   in $|-K_{S'}|$, so again, no component of $H$ has a tacnode.

If $d_S\ge2$, we may find an $H\in |-K_S|$ containing $q$ and if $H$ is smooth at $q$, with tangent $T_{H,q}$ equal to the common tangent line of the tacnode: this represents at most two linear conditions on $|-K_S|\cong \P^{d_S}$. These conditions imply that $H$ intersects  each of the two branches of $C$ at $q$ with multiplicity at least two, and  thus $d=\Deg H\cdot C\ge 4$, so in any case $d\ge 4$.

 Let $p_1, p_2\in \P^1$ be the pre-images of $q$ under $f$. Since $q$ is a tacnode, we have $df(T_{\P^1, p_1})=df(T_{\P^1, p_2})$, which gives a canonical isomorphism of the normal spaces $N_f\otimes k(p_1)\cong \sN_f\otimes k(p_2)$.

Since the family of maps parametrized by $V$ is equisingular on a dense open subset, $V$ is equisingular on a neighborhood of $f$. This implies that the  tangent map $T_fV\to H^0(\P^1, \sN_f)=T_fM_0(S,D)$ followed by the restriction map
\[
\Res_{p_1, p_2}:H^0(\P^1, \sN_f)\to \sN_f\otimes k(p_1)\oplus \sN_f\otimes k(p_2)\cong \sN_f\otimes k(p_1)^2
\]
has image contained in the diagonal. Since $\Res_{p_1, p_2}$ itself is surjective ($d-2\ge 2$) it follows that $\codim V\ge1$.

For~\ref{it:tacord2}, we first consider the case of a tacnode of order $\ge 2$. Let $f^{-1}(q)=\{p_1, p_2\}$ and choose a standard system of parameters $t_1, t_2, (x,y)$ for the tacnode at $q$. This gives the local  defining equation for $C$,   $y(y-x^{n+1})\in \hat{\sO}_{\P^1, q}\cong F[[x,y]]$, with $n\ge2$.  If $d_S\ge3$, there is a $\P^1$ of curves $H\in |-K_S|$ which contain $q$ and with tangent $T_{H,q}$ equal to the common tangent line of the tacnode (or are singular at $q$). Thus there is an $H\in |-K_S|$ with local defining equation $g=y+bxy+cy^2+\ldots$ and then $q$ appears with multiplicity $\ge 6$ in $H\cdot C$. Thus $d\ge 6$ if $d_S\ge3$.  If $d_S\ge 4$, there is a $\P^2$ of curves $H\in |-K_S|$ which contain $q$ and with tangent $T_{H,q}$ equal to the common tangent line of the tacnode (or are singular at $q$). Arguing as above, there is a $\P^1$ of $H\in |-K_S|$ such that $H$ intersects $C$ at $q$ with multiplicity $\ge 6$, and thus we can find such an $H$ that also intersects $C$ at a point $q'\neq q$, hence $d\ge 7$. Thus, in all cases, we have $d\ge 7$.

Suppose that $f$ is ramified and $d\ge 6$. Then by Lemma~\ref{lem:Ram}, $\codim V\ge1$ and if $\codim V=1$, then  by Lemma~\ref{lemRam2}, $t(f)=1$ and $f$ is ramified at a single point $p_3$, with $f(p_3)$ a simple cusp. Thus $\sN_f/\sN_f^\tor\cong \sO_{\P^1}(d-3)$. Consider the map $df:T_{\P^1}\to f^*T_S$. Since $f$ is ramified to first order at $p_3$, the kernel $\sL$ of $f^*T_S\to \sN_f/\sN_f^\tor$ is an invertible subsheaf  of $f^*T_S$ containing $df(T_{\P^1})$ and with quotient $\sL/df(T_{\P^1})\cong k(p_3)$. Thus $df(T_{\P^1})=\sL(-p)$ and  $\sL\cong \sO_{\P^1}(3)$.

Since $d\ge6$, $\dim_FH^0(\P^1, \sN_f/\sN_f^\tor)\ge 4$, so the restriction map
\[
\Res_{p_1,p_2}: H^0(\P^1, \sN_f/\sN_f^\tor)\to \sN_{f,p_1}/(t_1^2)\oplus \sN_{f,p_2}/(t_2^2)
\]
is surjective. Thus there is a section $s$ of $\sN_f/\sN_f^\tor$ which in the local coordinates $t_1$ and $p$ and $t_2$ at $p_2$ has the form
\[
s(t_1)=at_1+c_1t_1^2+\ldots,\ s(t_2)=c_2t_2^2+\ldots
\]
with $a\neq0$. As in the proof of Lemma~\ref{lemRam2}, this defines a 1st order deformation $f_{\epsilon,1}$ of $f$, which we can lift to a  deformation $f_\epsilon$ of $f$ defined over $F[[\epsilon]]$ with $q_\epsilon:=f_\epsilon(p_1)=f_\epsilon(p_2)$ an ordinary double point (over $F((\epsilon))$). Thus
$C_\epsilon:=f_\epsilon(\P^1)$ is not an equisingular deformation at $q$. We claim that we can modify $f_\epsilon$ without changing the class of $C_\epsilon$ in the local deformation space at $q_\epsilon$, but such that there is an $F[[\epsilon]]$ point $p_{3\epsilon}$ deforming $p_3$ such that $f_\epsilon$ is ramified at $p_{3\epsilon}$. This will exhibit $V$ as a proper closed subscheme of an irreducible component of $Z_\cusp$, hence $\codim V>\codim Z_\cusp\ge1$.

To verify the claim, write $f_\epsilon$ in the standard coordinate system $(s, (x', y'))$ for the ordinary cusp at $q'=f(p_3)$:
\[
f_\epsilon(s)=(s^2+\sum_{i=1}^\infty \epsilon^i\sum_{j=0}^\infty a_{i,j}s^j,
s^3+\sum_{i=1}^\infty \epsilon^i\sum_{j=0}^\infty b_{i,j}s^j)
\]
In the basis $\del/\del x', \del/\del y'$ for $f_\epsilon^*T_{\P^1}$ near $p_3$, the line bundle $\sL$ has generator $\lambda:=(2, 3s)$ with $df(T_{\P^1})\subset \sL$ the $\sO_{\P^1}$-submodule generated by $s\cdot \lambda$. We modify $f_{\epsilon,1}$ first by adding $-\sum_{i=1}^\infty \epsilon^i\cdot(b_{i,1}/3)\cdot \lambda$ to eliminate the linear term in the $y'$ coordinate. Note that in a neighborhood of $p_1$ and $p_2$, $\sL=df(T_{\P^1})$, so modifying by a section of $\sL$ acts by a local automorphism of $\P^1$ in a neighborhood of $p_1, p_2$, which does not affect the class of $C_\epsilon$ in the local deformation theory of $C$ near $q$. Thus we may assume that $f_\epsilon$ is of the form
\[
f_\epsilon(s)=(s^2+\sum_{i=1}^\infty \epsilon^i\sum_{j=0}^\infty a_{i,j}s^j,
s^3+\sum_{i=1}^\infty \epsilon^i\sum_{j=0}^\infty b_{i,j}s^j)
\]
with $b_{i,1}=0$ for all $i$. Making a similar modification by adding $-\sum_{i=1}^\infty \epsilon^i\cdot(a_{i,1}/2)\cdot s\cdot  \lambda$ will eliminate that linear terms in the $x'$ coordinate, so we may assume that $a_{i,1}=0$ for all $i$; this modification corresponds to a translation in $s$, so we have the new origin $p_{3\epsilon}$. Then it is clear that $f_\epsilon$ is ramified at $p_{3\epsilon}$.

Suppose now that $f$ is unramified. In this case, we use the assumption that $d\ge7$. Then $\sN_f\cong \sO_{\P^1}(d-2)$, so $d-2\ge5$ and $\dim_FH^0(\P^1, \sN_f)\ge 6$ and thus the restriction map
\[
\Res_{p_1,p_2}: H^0(\P^1, \sN_f/\sN_f^\tor)\to \sN_{f,p_1}/(t_1^3)\oplus \sN_{f,p_2}/(t_2^3)
\]
is surjective. We take a global section $s$ of $\sN_f$ with a zero of order three at $p_1$ and a zero of order two at $p_2$ and construct as above a deformation $f_\epsilon$ of $f$ with  first-order  deformation corresponding to $s$, and of the form
\[
f_\epsilon(t_1)=(t_1, \epsilon\cdot t_1^3)\mod (\epsilon^2, \epsilon\cdot t_1^4F[[t_1]]),\  f_\epsilon(t_2)=(t_2, t_2^{n+1}+\epsilon\cdot t_2^2)\mod (\epsilon^2, \epsilon\cdot t_2^3F[[t_2]])
\]
and with $f_\epsilon(p_1)=f_\epsilon(p_2)$.
This gives the image curve $C_\epsilon$ an ordinary tacnode at $q_\epsilon=f_\epsilon(p_1)=f_\epsilon(p_2)$, so $V$ is a proper closed subscheme of an integral component of $Z_\tac$, hence $\codim V\ge 2$.
\end{proof}

\begin{lemma}\label{lem:Codim2} Let $S$ be a del Pezzo surface over a field $k$ of characteristic zero and let $V\subset M^\birf_0(S, D)$ be an integral closed subscheme. Let $f\in V$ be a geometric generic point and let $C:=f(\P^1)\subset S$. 
\begin{enumerate}
\item\label{it:qordm}
Suppose that  $d_S\ge2$, $f\in M^\unr_0(S,D)$  and $C$  has a singular point $q$ of order $m$. Then $\codim V\ge m-2$.
\item\label{it:qq'}
Suppose that $d_S\ge2$ and $C$ has singular points $q\neq q'$. Suppose  that $f$ is ramified at a point $p'$ with $f(p')=q'$, that $f$ is unramified at all points $p$ with $f(p )=q$ and that $C$ has multiplicity $m>2$ at $q$. Then $\codim V\ge 2$.
\item\label{it:pp'}
Suppose  that $d_S\ge3$ or $d\ge 6$. Suppose that $C$ has a singular point $q$,  that $f$ is ramified at a point $p$ with $f(p)=q$, and that $f$ is unramified at a point $p'\neq p$ with $f(p' )=q$. Then $\codim V\ge 2$.
\item\label{it:qq'mm'}
Suppose that  $d_S\ge3$, $f\in M^\unr_0(S,D)$ and $C$ has singular points $q$, $q'$ of order $m$, $m'$, respectively, then
$\codim V\ge m+m'-4$.
\item\label{it:qq'mm'dS2}
Suppose that $d_S=2$, $f\in M^\unr_0(S,D)$ and $C$ has singular points $q$, $q'$ of order $m$, $m'$, respectively. Then $\codim V\ge m+m'-5$. If $d\ge 7$ and $m\ge m'\ge 3$, then $\codim V\ge 2$.
\end{enumerate}
\end{lemma}

\begin{proof}  \ref{it:qordm} We refer to the  exact sequence  \eqref{eqn:NormalSheaf}. Since $f$ is unramified, $\sN_f\cong sO_{\P^1}(d-2)$. Since $d\ge1$, we have $H^1(\P^1, \sN_f)=0$, so following Lemma~\ref{rem:DimModuli}, $M_0(S, D)$ is smooth of dimension $d-1$ at $f$.

For a general $H\in |-K_S|$ with $q\in H$, $H$ is integral and does not contain $C$ as a component. Since $|-K_S|$ has dimension $d_S\ge2$, there is an $H\in |-K_S|$ with $H\cap C\supset \{q, q'\}$, with $q\neq q'$, so $d=\deg(H\cdot C)\ge m+ 1$.

Let $F$ be an algebraically closed field over which $f$ is defined. Since $f$ is unramified, $f^{-1}(q)=\{x_1,\ldots, x_m\}$ with $x_i\neq x_j$ for $i\neq j$. Given $v\in T_{V,f}$, we have the corresponding first order deformation $f_v$ of $f$, defined over $F[\epsilon]/\epsilon^2$, the corresponding deformations $x_{i\epsilon}$ of $x_i$ and $q_\epsilon$ of $q$, with $f_\epsilon(x_{i\epsilon})=q_\epsilon$ for all $i$.  Let $L_i\subset T_{S, q}$ be the image $df(T_{\P^1, x_i})$. The deformation $f_v$ corresponds to a section $s_v$ of $N_f$, which gives us the affine subspaces $L_i+s_v(x_i)\subset  T_{S, q}$, and the conditions  $f_\epsilon(x_{i\epsilon})=q_\epsilon$, $i=1,\ldots, m$ implies $\cap_{i=1}^nL_i+s_v(x_i)\neq\0$.  Let $W\subset \oplus_{i=1}^mN_f\otimes k(x_i)$ be the set of $(v_i)\in \oplus_{i=1}^mN_f\otimes k(x_i)$ satisfying $\cap_{i=1}^mL_i+v_i\neq\0$; $W$ is a linear subspace of codimension $\ge m-2$. Since $\sN_f\cong \sO_{\P^1}(d-2)$ and $d-2\ge m-1$, it follows that $H^0(\P^1, \sN_f)\ge m$ and the product of restriction maps
\[
\text{Res}:=\prod_{i=1}^m\res_{x_i}:H^0(\P^1, \sN_f)\to \oplus_{i=1}^m\sN_f\otimes k(x_i)
\]
is surjective. Letting $W'\subset H^0(\P^1, \sN_f)$ be the inverse image of $W$ under $\text{Res}$, we have $\codim W'\ge m-2$ and  $s_v\in W'$  for all $v\in T_{V,f}$. Thus the image of $T_{V,f}$ under the map $v\mapsto s_v$ is a subspace of $H^0(\P^1, \sN_f)$ of codimension $\ge m-2$. Via the identification $H^0(\P^1, \sN_f)=T_{M_0(S,D), f}$, the map $v\mapsto s_v$ is just the inclusion of $T_{V,f}$ in $T_{M_0(S,D), f}$, so $V$ has codimension $\ge m-2$ in $M_0(S,D)$, as claimed.

For~\ref{it:qq'}, the fact that $f$ is ramified at $p'$ implies that $\sN_f^\tor\neq\{0\}$ and thus $\sN_f/\sN_f^\tor\cong \sO_{\P^1}(d-s)$ with $s\ge3$. By Lemma~\ref{lem:Ram}, the map $T_{V,f}\to H^0(\P^1, \sN_f/\sN_f^\tor)$ is injective. If $\codim V\le 1$, then $\dim T_{V,f}\ge d-2$, and as $\dim H^0(\P^1,\sO_{\P^1}(d-s))=d-s+1$, we have $s=3$, $\codim V= 1$ and $T_{V,f}\to H^0(\P^1,\sO_{\P^1}(d-3))$ is an isomorphism.

Since $f$ is ramified at $p'$, and $d_S\ge 2$, there is an $H\in |-K_S|$ with $H\cdot C\ge m\cdot q+2\cdot q'$, so $d\ge m+2$. If $m\ge 4$, then by (1), we have $\codim V\ge 2$, so we may assume that $m=3$, so $d-3\ge 2$. Letting $p_1, p_2, p_3$ be the points of $\P^1$ mapping to $q$, we see that the map
\[
\Res_{p_1, p_2, p_3}: H^0(\P^1, \sN_f/\sN_f^\tor)\to \oplus_{i=1}^3 \sN_f\otimes k(p_i)
\]
is surjective, and thus the composite $V\to \oplus_{i=1}^3 \sN_f\otimes k(p_i)$ is surjective as well. However, from the proof of (1), the condition that the triple point at $q$ deforms along a first order deformation corresponding to a section $s\in H^0(\P^1, \sN_f)$ defines a proper subspace of $ \oplus_{i=1}^3 \sN_f\otimes k(p_i)$, which contradicts the fact that $V\to \oplus_{i=1}^3 \sN_f\otimes k(x_i)$ is surjective.

For~\ref{it:pp'}, we argue as for~\ref{it:qq'} to reduce to the case $m=3$, $t(f)=1$ and $f$ is unramified on $\P^1\setminus\{p\}$.  We consider an analytic neighborhood of $C$ near $q$ as the union of the branch $(C, p)$ corresponding to $p\in \P^1$ and the branch $(C, p')$ corresponding to $p'\in \P^1$. Since $m=3$, $t(f)=1$ and $f$ is unramified on $\P^1\setminus\{p\}$, it follows that the branch
$(C, p)$ is a cusp. Since $d_S\ge 3$, there is an $H\in |-H_K|$ containing $q$ and singular at $q$. As the multiplicity of $q$ in the branch $(C,p)$ is  2, this implies that $q$ occurs with multiplicity $\ge 6$ in $H\cdot C$, so $d\ge6$. Also $\sN_f/\sN_f^\tor\cong \sO_{\P^1}(d-3)$ so $d-3\ge 3$. Moreover,   if $(C,p)$ is not an ordinary cusp, then a small modification of the argument for Lemma~\ref{lemRam2}\ref{it:qordinary} shows that $\codim V\ge2$, so we may assume that $(C,p)$ is an ordinary cusp.

Letting $t\in \sO_{\P^1, p}$ be a local parameter, the restriction map
\[
\Res_{(p,2) p'}: H^0(\P^1, \sN_f/\sN_f^\tor)\to  \sN_f/\sN_f^\tor\otimes \sO_{\P^1, p}/(t^2)\oplus \sN_f\otimes k(p')
\]
is surjective. There is thus a section $s\in H^0(\P^1, \sN_f/\sN_f^\tor)$ mapping to zero in
$\sN_f/\sN_f^\tor\otimes \sO_{\P^1, p}/(t^2)$ and to a non-zero element in $\sN_f\otimes k(p')$. Since $H^1(\P^1, \sN_f)=0$, the corresponding first-order deformation $f_{\epsilon,1}$ is unobstructed. Arguing as for the proof of Lemma~\ref{lemRam2}\ref{it:qordinary}, we may extend $f_{\epsilon,1}$ to a deformation $f_\epsilon$ over $F[[\epsilon]]$ so that $f_\epsilon$ (considered over $F((\epsilon))$) is still ramified at $p$. The conditions on $s$ imply that  $f_\epsilon(p)\equiv p\mod \epsilon^2$, while the branch of $f_\epsilon$ through $p'$ does not pass through $p$, and thus the branch of $f_\epsilon$ through $p'$ does not pass through $f_\epsilon(p )$. Thus $f_\epsilon(p )$ has multiplicity two on $C_\epsilon$. This implies that $V$ is a proper closed subscheme of an integral component of $Z_\cusp$, hence $\codim V\ge2$.

The proof of~\ref{it:qq'mm'} is similar: taking an $H\in |-K_S|$ passing through $q$ and $q'$ and tangent to one of the branches of $C$ at $q$, we see that $d\ge m+m'+1$. We thus have $H^0(\P^1, \sN_f)\ge m+m'$ and for $p_1,\ldots, p_m$ lying over $q$ and $p_1',\ldots, p_{m'}'$ lying over $q'$, the evaluation map
\[
\Res_{p_*, p_*'}:H^0(\P^1, \sN_f)\to \oplus_{i=1}^m\sN_f\otimes k(p_i)\oplus \oplus_{i=1}^{m'}\sN_f\otimes k(p_i')
\]
is surjective. Arguing as in~\ref{it:qordm},  the subspace of
$ \oplus_{i=1}^m\sN_f\otimes k(p_i)\oplus \oplus_{i=1}^{m'}\sN_f\otimes k(p_i')$ corresponding to 1st order deformations of the local germs of $f$ near $p_1,\ldots, p_m, p_1',\ldots, p_{m'}'$ for which $q$ and $q'$ deform to singular points of order $m$, $m'$ respectively has codimension $\ge m+m'-4$,
and thus  $\codim V\ge m+m'-4$.

For~\ref{it:qq'mm'dS2}, we have the estimate $d\ge m+m'$.  In this case, the map $\Res_{p_*, p_*'}$ has image of codimension at most one, and the argument of~\ref{it:qq'mm'} shows that $\codim V\ge  d-5\ge m+m'-5$, in particular, if  $m\ge m'\ge 3$ and $m+m'\ge 7$, then  $\codim V\ge2$. If $m=m'=3$ and $d\ge7$, then as $d-2\ge 5$, the restriction map
\[
\Res_{p_*, p_*'}:H^0(\P^1, \sN_f)\to \oplus_{i=1}^3\sN_f\otimes k(p_i)\oplus \oplus_{i=1}^{3}\sN_f\otimes k(p_i')
\]
is surjective and the argument of~\ref{it:qq'mm'} shows that $\codim V\ge 2$.  \end{proof}

\begin{proposition}\label{prop:SingCodim1} Let $k$ be a field of characteristic $0$ and suppose that $d_S\ge 4$ or $d_S=3$ and $d\neq6$, or $d\ge 7$.  Let $V\subset M^\birf_0(S, D)$ be an integral closed subscheme, let $f\in V$ be a geometric generic point and let $C:=f(\P^1)\subset S$. Suppose that $\codim V=1$. 
\begin{enumerate}
\item\label{it:Ctrip}
If $C$ has a triple point at $q\in S$, then $q$ is the only triple point of $C$, $q$ is an ordinary triple point, $f$ is unramified, and  all other singularities of $C$ are ordinary double points.
\item\label{it:Ccusp}
If $C$ has a cusp at $q\in S$, then $q$ is the only cusp of $C$, $q$ is an ordinary cusp and  all other singularities of $C$ are ordinary double points.
\item\label{it:Ctac}
If $C$ has a tacnode at $q\in S$, then  $q$ is an ordinary tacnode, $f$ is unramified, and  all other singularities of $C$ are ordinary double points.
\end{enumerate}
\end{proposition}

\begin{proof}  \ref{it:Ctrip}  If $f$ is ramified at some point, then by Lemma~\ref{lem:Codim2}\ref{it:qq'}\ref{it:pp'}, $\codim V\ge2$ so $f$ must be unramified. Similarly, if $C$ has a point $q'$ of multiplicity $\ge4$, then by Lemma~\ref{lem:Codim2}\ref{it:qq'mm'} $\codim V\ge2$, so $C$ has only triple points and double points.

We first show that $q$ is an ordinary triple point. Since $f$ is unramified, $f^{-1}(q)$ is three distinct points $p_1, p_2, p_3$ of $\P^1$. If $q$ is not ordinary, then (after reordering) $df(T_{\P^1, p_2})=df(T_{\P^1, p_3})$.

As in the proof of Lemma~\ref{lem:tacnode1}\ref{it:tacord2}, we have $d\ge 6$, and $\sN_f\cong \sO_{\P^1}(d-2)$. Moreover, there is a global  section $s$ with $s(p_1)\neq 0$, and $s$ having a second order zero at  $p_2$ and $p_3$.   To first order, this preserves the tacnode at $q$ corresponding to the branches at $p_2, p_3$, but the branch at $p_1$ no longer passes through $q$. Since $H^1(N_f)=0$, this 1st order deformation is unobstructed, and arguing further as in the proof of Lemma~\ref{lem:tacnode1}\ref{it:tacord2}, we can extend this to a deformation $f_\epsilon$ of $f$ over $F[[\epsilon]]$ with $f_\epsilon(\P^1)$ having a tacnode at $q$. Since $D_\tac$ has codimension one, this implies that $V$ has codimension $\ge2$, contrary to our assumption.

Now suppose $C$ has a double point at $q'$ and  that $q'$ is not an ordinary double point. Since $f$ is unramified, $q'$ must be a tacnode; let $p'_1, p'_2\in \P^1$ be the points lying over $q'$. If $d_S\ge3$,  there is an $H\in |-K_S|$ passing through $q'$ and $q$, and sharing the common tangent line at $q'$. Thus $d\ge 6$. If $d_S=3$, then by assumption, $d\ge7$. If $d_S\ge 4$, then we can find an $H$ as above and passing through an additional point of $C$, so again $d\ge7$, and thus in all cases $d\ge7$.

As $\sN_f\cong \sO_{\P^1}(d-2)$,  $H^0(\P^1, \sN_f)$ has dimension $\ge 6$ and
\[
\Res:H^0(\P^1, \sN_f)\to \oplus_{i=1}^3\sN_f\otimes k(p_i)\oplus \sN_f\otimes k(p'_1) \oplus \sN_f\otimes_{\sO_{\P^1, p_2'}}/\mathfrak{m}_{p_2'}^2
\]
is surjective. Taking a section $s$ of $H^0(\P^1, \sN_f)$ with 1st order zeros at   $p_1,p_2, p_3$, a second order zero at $p_2'$ but with   $s(p'_1)\neq0$, then $s$ defines a first order deformation $f_{\epsilon,1}$ that is equisingular at $q$ but not so at $q'$. As before, we can extend $f_{\epsilon,1}$ to a deformation $f_\epsilon$ over $F[[\epsilon]]$ so that $q$ deforms to a triple point on $f_\epsilon(\P^1)$, but the deformation near $q'$ is not equisingular. Thus $V$ is a proper closed subscheme of $D_\trip$, so $\codim V\ge2$, contrary to assumption.

For~\ref{it:Ccusp}, suppose $f$ is ramified at $p\in \P^1$ with $f(p )=q$. By Lemma~\ref{lemRam2}, $q$ is an ordinary cusp and  $f$ is unramified on $\P^1\setminus\{p\}$. By Lemma~\ref{lem:Codim2}, $C:=f(\P^1)$ has only double points. If $q'\neq q$ is a double point of $C$, then as $f$ is unramified over $q'$, $q'$ must be a tacnode. By Lemma~\ref{lem:tacnode1}, $q'$ is an ordinary tacnode. Assuming $d_S\ge3$, let $H\in |-K_S|$ be chosen so that $\{q, q'\}\subset H$ and that $H$ is tangent to the common tangent line at $q'$; as above, this actually implies that $d\ge7$ and $\sN_f/\sN_f^\tor\cong \sO_{\P^1}(d-3)$, $d-3\ge 4$. Letting $p'_1, p'_2\in \P^1$ be the points over $q'$, we may find a section $s\in H^0(\P^1, \sN_f/\sN_f^\tor)$ with a second order zero at $p$ and at $p_1'$ and   $s(p'_2)\neq 0$. As in the proof of Lemma~\ref{lemRam2}\ref{it:qordinary} and~\ref{it:Ctrip}, the corresponding first order deformation $f_{\epsilon,1}$ of $f$ can be extended to a deformation $f_\epsilon$ over $F[[\epsilon]]$ so that $f_\epsilon$ is ramified at $p$, but the deformation $C_\epsilon:=f_\epsilon(\P^1)$ is not equisingular at $q'$, which yields $\codim V\ge 2$, contrary to assumption.

For~\ref{it:Ctac},  Lemma~\ref{lem:tacnode1}\ref{it:tacord2} implies that  the tacnode at $q$ must be an ordinary tacnode, by~\ref{it:Ccusp} $f$ is unramified and by~\ref{it:Ctrip} and Lemma~\ref{lem:Codim2}, all other singularities are double points. Applying Lemma~\ref{lem:tacnode1}\ref{it:tacord2}, each double point $q'\neq q$ is either an ordinary tacnode or an ordinary double point, so suppose $q'$ is an ordinary tacnode.

Suppose $d_S\ge3$. Taking an $H\in |-K_S|$ passing through $q$ and $q'$ and tangent to the common tangent line at $q$, we see that $d\ge 6$ and $N_f\cong \sO_{\P^1}(d-2)$; as above, this implies that $d\ge7$ in all cases. Let $p_1, p_2\in \P^1$ be the points lying over $q$ and $p_1', p_2'\in \P^1$ be the points lying over $q'$. Let $t_1, t_2,(x,y)$   be a standard system of parameters for $q$ and let  $t_1', t_2',(x',y')$   be a standard system of parameters for $q'$

Consider the restriction map
\begin{multline*}
\Res:H^0(\P^1, \sN_f)\to\\ \sN_f\otimes\hat{\sO}_{\P^1, p_1}/(t_1^2) \oplus
N_f\otimes\hat{\sO}_{\P^1, p_2}/(t_2^2)\oplus \sN_f\otimes k(p_1')\oplus \sN_f\otimes k(p_2')=:W
\end{multline*}
SInce $d\ge7$, $\Res$ is surjective. In particular, we may find a global section $s$ of $\sN_f$ that has a 2nd order zero at $p_1$ and $p_2$, a first order zero at $p_1'$ and is non-zero at $p_2'$. As $H^1(\P^1, \sN_f)=0$, we may extend the corresponding first order deformation of $f$ to a deformation $f_\epsilon$ defined over $F[[\epsilon]]$.

Using the surjectivity of $\Res$, we may take our extension $f_\epsilon$ of the 1st order deformation so that  $q$ deforms to an ordinary tacnode $q_\epsilon$ on the image curve $C_\epsilon$, and  in the coordinates $t_1', t_2',(x',y')$, we have
\[
f_\epsilon(t'_1)=(t'_1+\epsilon\cdot x_1(t'_1, \epsilon), \epsilon\cdot y_1(t'_1, \epsilon)),\
f_\epsilon(t'_2)=(t'_2+\epsilon\cdot x_2(t'_2, \epsilon), t_2^{\prime2}+ \epsilon\cdot y_2(t'_2, \epsilon))
\]
with $y_1(0,\epsilon)=0$,  $y_2(0,\epsilon)\neq 0$.
Translating in $x'$ and then in $t'_2$ (by translations $\equiv0\mod\epsilon$) we may rewrite this as
\[
f_\epsilon(t'_1)=(t'_1, \epsilon\cdot y_1(t'_1, \epsilon)),\
f_\epsilon(t_2)=(t'_2, t_2^{\prime2}+ \epsilon\cdot y_2(t'_2, \epsilon))
\]
still with $y_1(0,\epsilon)=0$,  $y_2(0,\epsilon)\neq 0$. Translating by replacing $y'$ with $y'-\epsilon\cdot y_2(t'_2, \epsilon)$, we reduce to
\[
f_\epsilon(t'_1)=(t'_1, \epsilon\cdot y_2(t'_2, \epsilon)),\
f_\epsilon(t'_2)=(t'_2, t_2^{\prime2})
\]
again with $y_2(0,\epsilon)\neq 0$. The image curve $C_\epsilon=f_\epsilon(\P^1)$ thus has defining equation $(y'-\epsilon\cdot y_2(t'_2, \epsilon))(y'-x^{\prime 2})\in F[[x',y', \epsilon]]$. By Lemma~\ref{lem:TacnodeDef}, $C'_\epsilon$ does not have a tacnode $q'_\epsilon$ specializing to $q'$.  This exhibits $V$ as a proper closed subscheme of an integral component of $Z_\tac$, forcing $\codim V\ge2$, contrary to assumption.
\end{proof}

\begin{lemma}\label{lem:TacnodeDef} Let $f=y(y-x^2)\in F[[x,y]]$ define an ordinary tacnode over an algebraically closed field $F$ of characteristic $\neq2$. let
\[
f_\epsilon(x,y)=(y-\sum_{i=1}^\infty\epsilon^ig_i(x))(y-x^2)\in F[[x,y,\epsilon]]
\]
define a deformation of $f$ over $F[[\epsilon]]$ and suppose that $g_1(0)\neq0$. Then $f_\epsilon$ is not equisingular: the curve $C_\epsilon:=\Spec F[[\sqrt{\epsilon}, x,y]]/(f_\epsilon)[1/\epsilon]$ has two ordinary double points specializing to $(0,0)$ and no other singularities.
\end{lemma}

\begin{proof} Write $\sum_{i=1}^\infty\epsilon^ig_i(x)=\epsilon\cdot \sum_{i,j\ge0}a_{i,j}\epsilon^ix^j$ with $a_{i,j}\in F$. Then $a_{0,0}=g_1(0)\neq0$, so there is an $h(x,\epsilon)\in F[[\epsilon,x]]$ with $h^2=\sum_{i,j\ge0}a_{i,j}\epsilon^ix^j$. The singular locus of $C_\epsilon$ is just the intersection of $y=\sum_{i=1}^\infty\epsilon^ig_i(x)$ with $y=x^2$, that is, the subscheme defined by  $x^2-\epsilon\cdot h^2$ or, over $F[[\sqrt{\epsilon}, x]]$, $(x-\sqrt{\epsilon}\cdot h)(x+\sqrt{\epsilon}\cdot h)$.  Since $F[[\sqrt{\epsilon}, x]]/(x-\sqrt{\epsilon}\cdot h)$ and $F[[\sqrt{\epsilon}, x]]/(x+\sqrt{\epsilon}\cdot h)$ are both reduced,   we have the desired description of $C_\epsilon$
\end{proof}

Note that $Z_\cusp$, $Z_\tac$ and $Z_\trip$ each have only finitely many irreducible components, by Definition~\ref{df:Dcusp_Dtac_Dtrip}.
\begin{definition} We define reduced codimension one subschemes $D_\cusp, D_\tac, D_\trip$ on $\bar{M}_{0}(S, D)$ as follows.
\begin{enumerate}
\item Let $D_\cusp$ be the closure in $\bar{M}_{0}(S, D)$  of the  union of the codimension one integral components   $Z_\cusp\subset M_0^\bir(S,D)$ 
\item Let $D_\tac$ be the closure in $\bar{M}_{0}(S, D)$  of the  union of the codimension one integral components   $Z_\tac\subset M_0^\unr(S,D)$ 
\item Let $D_\trip$ be the closure in $\bar{M}_{0}(S, D)$  of the  union of the codimension one integral components   $Z_\trip\subset M_0^\unr(S,D)$ 
\end{enumerate}
\end{definition}

\section{Non-birational and non-free maps}
Having examined $M^\birf_0(S, D)$, we look more closely at the moduli stack of primary interest, $\bar{M}_{0,n}(S, D)$ with $n=-D\cdot K_S-1$; set $d=-K_S\cdot D$. We have the ``forget the marked points'' map $\pi_{n/0}:\bar{M}_{0,n}(S, D)\to \bar{M}_{0}(S, D)$, which is a composition of the structure maps for the various universal curves $\pi_{i+1/i}:\bar{M}_{0,i+1}(S, D)\to \bar{M}_{0,i}(S, D)$, hence proper and flat. 

\begin{definition} The codimension one subschemes $D_\cusp$, $D_\tac$, $D_\trip$ of $\bar{M}_{0,n}(S, D)$ are given by applying $\pi_{n/0}^{-1}$ to the corresponding closed subschemes of $\bar{M}_{0}(S, D)$.
\end{definition}

The results on $M^\birf_0(S,D)$ of the previous section carry over directly to $M^\birf_{0,n}(S,D)$, for instance, setting $d:=-D\cdot K_S$, $M^\birf_{0,n}(S,D)$ is a smooth finite-type $k$ scheme with $\dim_k M^\birf_{0,n}(S,D)=2d-2$, or $M^\birf_{0,n}(S,D)$ is empty; this follows from Lemma~\ref{rem:DimModuli}. We proceed to study the complement $\bar{M}_{0,n}(S, D)\setminus M^\birf_{0,n}(S,D)$.

 Following the construction of  $\bar{M}_{0,n}(S, D)$ given in \cite{AO}, there is a quasi-projective scheme $\tilde{\bar{M}}_{0,n}(S,D)$ with $\PGL_N$-action, presenting
$\bar{M}_{0,n}(S, D)$ as quotient stack $\PGL_N\backslash\tilde{\bar{M}}_{0,n}(S,D)$. For an $F$-point of $\tilde{\bar{M}}_{0,n}(S,D)$, $F\supset k$ an algebraically closed field, we have the corresponding morphism $f:\P\to S$, where $\P$ is a semi-stable genus 0 curve. This gives us the image Cartier divisor $f_*(\P)$, which we may consider as an $F$-point of the projective space $|D|$; this extends to a morphism $\tilde{\im}:\tilde{\bar{M}}_{0,n}(S,D)\to |D|$. We note that $f_*(\P)=(gf)_*(g\P)$ for $g\in\PGL_N$. It follows that we have the morphism $M_{0,n}(S,D)\to |D|\cong \P^N$,  $N=(D\cdot D+d)/2$, sending the equivalence class $[f]$ of a morphism $f:\P\to S$ to $f_*([\P])$. For $V\subset {\mathbf{M}}_{0,n}(S,D)$ a locally closed substack, we have the constructible subset $\im(V)\subset |D|$ and we may speak of the dimension $\dim\, \im(V)$, which is at most $\dim V$.

\begin{lemma}\label{lm:coverfree}
Let $f : \P^1 \to S$ factor as $f = g \circ q$ where $g : \P^1 \to S$ is birational onto its image and $q : \P^1 \to \P^1$ is a finite map. Then we have a short exact sequence
\[
0 \to \coker(dq) \to \sN_f \to q^* \sN_g \to 0.
\]
Moreover, if $g$ is free then $f$ is free.
\end{lemma}
\begin{proof}
We have the following commutative diagram.
\[
\xymatrix{
0 \ar[r] & T \P^1 \ar[r]^{df} \ar[d]^{dq} & f^* TS \ar[r] \ar[d]^\wr & \sN_f \ar[r] \ar[d] & 0 \\
0 \ar[r] & q^* T\P^1 \ar[r]^{q^*dg} & q^* g^* TS \ar[r] & q^* \sN_g \ar[r] & 0
}
\]
So, the short exact sequence follows by the snake lemma. Since $\ochar k = 0,$ it follows that $\coker(dq)$ is torsion. Thus,
\[
\sN_f/\sN_f^\tor \simeq q^* (\sN_g/\sN_g^\tor).
\]
If $g$ is free, then $\sN_g/\sN_g^\tor \simeq O(m)$ for $m \geq 0,$ so $\sN_f/\sN_f^\tor \simeq O(\deg(q) m).$ So $f$ is free.
\end{proof}

\begin{lemma} \label{lem:Nonfree} Suppose $k$ has characteristic zero. Let $V\subset \bar{M}_{0,n}(S, D)$ be an integral closed substack with geometric generic point $f$. Suppose that $f$ corresponds to a morphism $f:\P^1\to S$  with image curve $C:=f(\P^1)$ and that $f$ is non-free. Then $\dim\, \im V=0$. Moreover, one of the following cases holds:
\begin{enumerate}
\item\label{it:d1C-1}
$d=1$, $n=0$, $C$ is a -1 curve on $S$ and $f:\P^1\to C$ is an isomorphism.
\item\label{it:d2C-1}
$d=2$, $C$ is a -1 curve on $S$ and $f:\P^1\to C$ is a 2-1 cover. In this case, $\dim \ev(V)=1$.
\item\label{it:d2dS2}
$d=2$, $d_S=2$, $f:\P^1\to C$ is birational,  $C$ has an ordinary cusp at $q\in C$, $f^{-1}(q)$ is a single point $p\in \P^1$ and $f:\P^1\setminus\{p\}\to C\setminus \{q\}$ is an isomorphism. Moreover $\dim \ev(V)=1$.
\item\label{it:dge3}
 $d\ge 3$ and $\codim \ev(V)\ge d-1\ge 2$.
\end{enumerate}
In case (3),  $V$ is dense in a component of $D_\cusp$.
\end{lemma}

\begin{proof}
If $f$ is not birational to its image, we factor $f = f' \circ c$ where $c : \P^1 \to \P^1$ has degree $e$ and $f': \P^1 \to S$ is birational to its image. Let $D' = f'_*([\P^1])$ and let $V'$ be the closure of $f'.$ Then, $\im V = e \, \im V'$ and $\dim \im V = \dim \im V'.$ Lemma~\ref{lm:coverfree} implies that $f'$ is not free.

Since $f'$ is non-free, we have $\sN_{f'}/\sN_{f'}^\tor\cong \sO_{\P^1}(m)$ with $m<0$ and thus
\[
H^0(\P^1, \sN_{f'}/\sN_{f'}^\tor)=\{0\}.
\]
Since the tangent map $T_{f'}V'\to H^0(\P^1, \sN_{f'}/\sN_{f'}^\tor)$ is injective (Lemma~\ref{lem:Ram}), this says that $\dim \im V = \dim\, \im V'=0$, so $f_*(\P^1)$ is the unique geometric point of $\im V$. Thus each element of $\ev(V)$ consists of a sequence of $n$ points of $C$. So, $\dim \ev(V) = n$ and $\codim \ev(V)= 2n-n=d-1$. If $d\ge 3$, we are in case~\ref{it:dge3}.

If $d=1$, then $f$ is birational and $C$ is a line and thus a -1 curve, giving us case~\ref{it:d1C-1}.

Suppose $d = 2.$ It follows that $n = 1.$ First assume $d_S \geq 3,$ so $-K_S$ embeds $S.$ So, either $C$ is a conic and $f:\P^1\to C$ is an isomorphism or $C$ is a line and $f:\P^1\to C$ is a double cover. If $f:\P^1\to C$ is an isomorphism, then $\sN_f\cong \sO_{\P^1}$, so $f$ is free contrary to our hypothesis. If $f$ is a double cover,

we apply Lemma~\ref{lm:coverfree} to obtain $\sN_f/\sN_f^\tor\cong \sO_{\P^1}(-2)$, so $f$ is not free and we are in case~\ref{it:d2C-1}.

If $d=2$ and $d_S=2$, then we are in the situation of Lemma~\ref{lm:dS2}. If $f:\P^1\to C$ has degree 2, then $C \cdot (-K_S) = 1,$ so $C$ is a -1 curve on $S$ and we are again in case~\ref{it:d2C-1}.

Suppose $d=d_S=2$ and $f:\P^1\to C$ is birational. We are in the situation of Lemma~\ref{lm:dS2d2}. If $\pi:C\to \pi(C )$ is birational, then $\pi(C)$ is a smooth conic, $C$ is smooth and $f:\P^1\to C$ is an isomorphism. Thus, $\sN_f=\sO_{\P^1}$ and $f$ is free, contrary to the hypothesis.  If $\pi:C\to \pi(C )$ is a double cover, then either $f:\P^1\to C$ is unramified, hence  $\sN_f=\sO_{\P^1}$ and $f$ is free, or $C$ has a single ordinary cusp, and $t(f)=1$ so $\sN_f/\sN_f^\tor\cong \sO_{\P^1}(d-3)=\sO_{\P^1}(-1)$. This is case~\ref{it:d2dS2}.
\end{proof}

\begin{lemma} \label{lem:NonBirat1} Let $k$ be a field. Suppose Assumption~\ref{a:genericunram} holds for $S$. Let $V\subset \bar{M}_{0,n}(S, D)$ be an integral closed substack with geometric generic point $f$. Suppose that $f$ corresponds to a morphism $f:\P^1\to S$  with image curve $C:=f(\P^1)$ and let $d_f$ be the degree of $f:\P^1\to C$. Then $\dim\, \im(V)\le (d/d_f)-1$. \end{lemma}

\begin{proof}
By passing to an extension, we may assume that $k$ is algebraically closed, and in particular perfect. Let $d_C=-C\cdot K_S$, so $d_C\cdot d_f=d$ and $d/d_f=d_C\ge1$; in particular, $d\neq1$ if $d_f>1$. Let $F$ be an algebraically closed field of definition for $C$ and $f$. Let $\tilde{C}$ be the normalization of $C$. Then $\tilde{C}\cong \P^1$ (L\"uroth's theorem) and $f$ factors as $\P^1\xrightarrow{\tilde{f}}\tilde{C}\xrightarrow{g}C$ with $g$ birational and $\tilde{f}$ of degree $d_f$. By Assumption~\ref{a:genericunram}, there is an unramified map $g_0$ in the irreducible component of $M_0^{\bir}(S, C)$ containing $g$. By Lemma~\ref{lem:dimMod_noMarked}, $\dim_{g_0} M_{0}^{\bir}(S, C) = d_C -1$. Thus
\[
\dim\, \im(V)\le \dim_{g_0} M_{0}^{\bir}(S, C) \leq  d_C-1=(d/d_f)-1
\]
\end{proof}

\begin{definition}
Let $U$ be a normal, integral scheme over a field $F$ and let $\P \to U$ be an $n$-marked semi-stable genus zero curve. We say that $\P$ is {\em treelike} if the normalization $\pi : \tilde \P \to \P$ is a disjoint union with each component isomorphic to $\P^1_U.$
\end{definition}

\begin{lemma}\label{lm:constant_tree} Let $\P\to U$ be a treelike family over $U$, with $U$ normal and integral. 
\begin{enumerate}
\item\label{it:dpts}
Suppose the number of irreducible components in the normalization $\tilde\P$ of $\P$ is  $r$. Then for each geometric point $x$ of $U$, the fiber $\P_x$ has exactly $r-1$ double points.
\item\label{it:pairy}
Let $u$ be a geometric generic point of $U$ and let $y$ be a double point of $\P_u$. Then there is a unique pair of components $\tilde{\P}_i, \tilde{\P}_j$ of $\tilde{\P}$ with $y$ equal to the image of $(\tilde{\P}_i\times_\P\tilde{\P}_j)_u$ in $\P_u$. Let $\eta$ be the generic point of $U$ and let
$\overline{(\tilde{\P}_i\times_\P\tilde{\P}_j)_\eta}$ denote the closure of $(\tilde{\P}_i\times_\P\tilde{\P}_j)_\eta$ in $\tilde{\P}_i\times_\P\tilde{\P}_j$. Then the  projection
\[
\overline{(\tilde{\P}_i\times_\P\tilde{\P}_j)_\eta}\to U
\]
is an isomorphism.
\item\label{it:disjoint}
Let $\tilde\P_i$, $\tilde\P_j$ be distinct irreducible components of $\P$ and let $u$ be a geometric generic point of $U$. If $(\tilde\P_i\times_{\P}\tilde\P_j)\times_Uu=\0$, then $\tilde\P_i\times_{\P}\tilde\P_j=\0$.
\item\label{it:sections}
For each pair of components $\tilde{\P}_i, \tilde{\P}_j$ with $\tilde{\P}_i\times_\P\tilde{\P}_j\neq\0$, the projection $\tilde{\P}_i\times_\P\tilde{\P}_j\to U$ is an isomorphism, defining two sections $\sigma_{i,j}^i:U\to \tilde{\P}_i$, $\sigma_{i,j}^j:U\to \tilde{\P}_j$ via the two projections
$\tilde{\P}_i\times_\P\tilde{\P}_j\to\tilde{\P}_i$, $\tilde{\P}_i\times_\P\tilde{\P}_j\to\tilde{\P}_j$.
\end{enumerate}
\end{lemma}

\begin{proof} For~\ref{it:dpts}, recall that each connected genus zero semi-stable curve $P$ over an algebraically closed field $F$ has $\dim_F\chi(\sO_P)=1$. Let $\pi:\tilde{P} \to P$ be the normalization of $P$, and let $r$ be the number of irreducible components of $\tilde{P}$. Then $\chi(\sO_{\tilde{P}})=r$. On the other hand, we have the exact sheaf sequence on $P$,
\[
0\to \sO_P\to \pi_*\sO_{\tilde{P}}\to \oplus_{i=1}^sF(p_i)\to0
\]
where $\{p_1,\ldots, p_s\}$ are the double points of $P$. Then
\[
r=\dim_F\chi(\sO_{\tilde{P}})=\dim_F\chi( \pi_*\sO_{\tilde{P}})=\dim_F\chi(\sO_P)+s=1+s,
\]
so $r=s+1$.

Now, since our family is treelike, each geometric fiber $\P_x$ has normalization $\tilde{P}_x\cong \amalg_{i=1}^r\P^1_x$, so each geometric  fiber $\P_x$ has exactly $r-1$ double points.

For~\ref{it:pairy}, it follows directly from the definition of a semi-stable, genus zero curve, that  there is a unique pair of components $\tilde{\P}_{i u}, \tilde{\P}_{j u}$ of $\tilde{\P}_u$ with $y$ equal to the image of $(\tilde{\P}_{i u}\times_{\P_u}\tilde{\P}_{j u}$ in $\P_u$. The first part of (2) thus follows from the fact that $\P$ is treelike. It is also obvious that the map $\tilde{\P}_{i u}\times_{\P_u}\tilde{\P}_{j u}\to \P_u\to u$ is an isomorphism, so $(\tilde{\P}_i\times_{\P}\tilde{\P}_{j})_u\to u$ is an isomorphism.

The morphism $u\to \eta$ is qcqs and faithfully flat, so the map $(\tilde{\P}_i\times_{\P}\tilde{\P}_{j})_\eta\to \eta$ is an isomorphism. For each geometric point $x$ of $U$, $\P_{ix}\times_{\P_x}\P_{jx}$ is a finite set so $\overline{(\tilde{\P}_i\times_\P\tilde{\P}_j)_\eta}\to U$ is birational, proper and quasi-finite. Since $U$ is normal, and $\overline{(\tilde{\P}_i\times_\P\tilde{\P}_j)_\eta}$ is integral, the projection is  an isomorphism, by Zariski's main theorem.

For~\ref{it:disjoint}, it follows from~\ref{it:dpts} that $\P_u$ has exactly $r-1$ double points, $y_1,\ldots, y_{r-1}$. By (2), there are sections $\sigma_1,\ldots, \sigma_{r-1}:U\to \tilde\P\times_\P\times\tilde\P$ with image sections $\bar\sigma_1,\ldots,\bar\sigma_{r-1}:U\to \P$ with $\bar\sigma_i(\eta)=y_i$. By (1) again, for each geometric point $x$ of $U$, the points $\bar\sigma_1(x),\ldots,\bar\sigma_{r-1}(x)$ are exactly the double points of the fiber $\P_x$.

 If now $(\tilde\P_i\times_{\P}\tilde\P_j)\times_Uu=\0$, but there is a $z\in (\tilde\P_i\times_{\P}\tilde\P_j)_x$, then the image of $z$ in $\P$ is a double point of $\P_x$, so $z=\bar\sigma_l(x)$ for some $i$. But then there is a pair of components $\tilde{P}_{i'}, \tilde{P}_{j'}$ with $\sigma_i$ corresponding to the non-empty fiber product $\tilde{P}_{i'}\times_\P\tilde{P}_{j'}$, and wit $\{i,j,i',j'\}$ having size at least three. But then $z$ is a point of $\P_x$ of multiplicity at least three, a contradiction.

For~\ref{it:sections}, it suffices by~\ref{it:pairy} to show that the inclusion
 \[
 \overline{(\tilde{\P}_i\times_\P\tilde{\P}_j)_\eta}\to  \tilde{\P}_i\times_\P\tilde{\P}_j
 \]
 is an isomorphism. For this, we may take the basechange from $k$ to its algebraic closure, so we may assume that $k$ is itself algebraically closed. Since the $k$-points of $U$ are dense, we need only check over a neighborhood of each closed point $a\in U$. Let $b\in  \tilde{\P}_i\times_\P\tilde{\P}_j$ be the unique point lying over $a$ ($b$ is automatically a closed point); we consider $b$ simultaneously as a closed point of $\tilde\P_i$, $\tilde\P_j$ and $\tilde\P$. We may also pass to the completions $\hat\sA:=\widehat{\sO}_{U,a}$ and $\hat\sB:=\widehat{\sO_{\P,b}}$. As $\sO_{U,a}$ is an excellent normal local domain, $\sA$ is a complete normal local domain. Let $\mathfrak{m}\subset\hat\sA$ be the maximal ideal.

 We  consider the versal deformation space of the singularity $xy=0$, which has base $\Spf(k[[t]])$ and versal family $\Spf(k[[t,x,y]]/(xy-t)$. From this description of the versal family, we find there is an element $f\in \mathfrak{m}$ such that $\hat\sB$ is isomorphic as an $\hat\sA$-algebra to $\sA[[x,y]]/(xy-f)$. In addition, the section $\sigma:U\to \P$  with $\sigma(a)=b$t given by (2) defines a surjection $\psi:\sA[[x,y]]/(xy-f)\to \sA$ splitting the inclusion $\sA\to
 \sA[[x,y]]/(xy-f)$. Moreover, since $\sigma(U)$ is contained in the relative singular locus of $\P\to U$,  the induced map $\Spf\psi:\Spf\sA\to \Spf\sA[[x,y]]/(xy-f)$ is contained in the closed formal subscheme of $\Spf\sA[[x,y]]/(xy-f)$ defined by the vanishing of the section $d(xy-f)$ of  the completed relative K\"ahler differential $\hat\Omega_{\sA[[x,y]]/\sA}=\sA[[x,y]]\cdot dx\oplus \sA[[x,y]]\cdot dy$. Since $d(xy-f)=xdy+ydx$, this says that the kernel of $\psi$ contains the ideal $(x,y)+(xy-f)/(xy-f)$. But since $\sA[[x,y]]/(x,y)=\sA$, this says that $(x,y)\supset (xy-f)$ in $\sA[[x,y]]$, hence $f=0$ and $\hat\sB\cong \sA[[x,y]]/(xy)$. This in turn implies that $\hat\sO_{\tilde{\P}_i,b}\cong \sA[[x,y]]/(x)$, $\hat\sO_{\tilde{\P}_j,b}\cong \sA[[x,y]]/(y)$ and thus
 \[
 \hat\sO_{\tilde{\P}_i\times_\P\tilde{\P}_j,b}\cong  \sA[[x,y]]/(x,y)=\sA=
 \hat\sO_{\overline{(\tilde{\P}_i\times_\P\tilde{\P}_j)_\eta},b},
 \]
 which proves~\ref{it:sections}.
 \end{proof}

Because $(\P, p_*)$ is semi-stable, any marked point $p_i: U \to \P $ lands in the smooth locus of $\P$. Since normalization commutes with smooth base change, $\pi: \tilde{\P} \to \P$ is an isomorphism over the smooth locus and $p_i$ has a unique preimage under $\pi(U): \tilde \P(U) \to \P (U).$

\begin{definition} \label{df:DualTree}
We define a tree associated to $\P$ as follows. Let the vertices $V(\P)$ be the set of components of $\tilde \P.$ Let the half-edges $H(\P) \subset \tilde \P(U)$ be the preimage under $\pi$ of the marked points and nodal points of $\P$,
\[
H(\P) = \{ \pi(U)^{-1}(p_i): i=1,\ldots,n \} \cup \{ \sigma^i_{i,j}, \sigma^i_{i,j} :  \tilde{\P}_i, \tilde{\P}_j \in V(\P) \text{ such that } \tilde{\P}_i\times_\P\tilde{\P}_j\neq\0 \}
\] where the $ \sigma^i_{i,j}, \sigma^j_{i,j}$ are the sections constructed in Lemma~\ref{lm:constant_tree}\ref{it:sections}. Thus, there is a canonical map $\nu : H(\P) \to V(\P).$ Let $i : H(\P) \to H(\P)$ be the involution with orbits of length $1$ corresponding to marked points and orbits of length $2$ corresponding to nodal points. Let the edges $E(\P) \subset H(\P)\times H(\P)$ be the subset consisting of orbits of $i$ of length $2$. Since $\P$ has genus zero, the map $\nu$ induces an inclusion $E(\P) \to V(\P)\times V(\P).$ Thus, we obtain a tree $T(\P).$ Below, by abuse of notation, we may use $\P$ to refer to $T(\P).$
\end{definition}

\begin{lemma}\label{treelike_cover}
Let $\P \to U$ be a semistable genus zero curve with $U$ integral. There exists a dense open subscheme $U_0$ and a surjective finite morphism $W \to U_0$ such that $\P \times_U W \to W$ is treelike and $W$ is smooth.
\end{lemma}
\begin{proof}
Let $\eta$ be the generic point of $U,$ and let $\overline{\eta}: \overline{k(\eta)} \to U$ be a geometric point with image $\eta$ and residue field an algebraic closure of $k(\eta)$. The basechange $\P_{\overline{\eta}}$ of $\P$ to $\overline{\eta}$ has a normalization $\tilde \P_{\overline{\eta}} \to \P_{\overline{\eta}}$. Normal schemes are regular in codimension 1, whence the curve $\tilde \P_{\overline{\eta}}$ is regular. Since $ \overline{k(\eta)}$ is algebraically closed, it follows that $\tilde \P_{\overline{\eta}}$ is smooth \cite[038X]{stacks-project}. A smooth genus $0$ curve over an algebraically closed field is isomorphic to a disjoint union of $\P^1$'s. This isomorphism descends to a finite extension $k(\eta) \to L$, giving a pullback diagram
\[
\xymatrix{  \tilde \P_{\overline{\eta}} \ar[d] \ar[r] & \coprod_{ i=1}^M \P^1_L \ar[d]^{\alpha}\\
 \P_{\overline{\eta}} \ar[d] \ar[r] & \P_{L} \ar[d]\\
\Spec \overline{k(\eta)} \ar[r] & \Spec L}
\]

By enlarging $L$ if necessary, we may assume that $\alpha:  \coprod_{ i=1}^M \P^1_L \to \P_{L}$ is birational and finite, because the property of being an isomorphism and finite is detected after fpqc basechange. Since $ \coprod_{ i=1}^M \P^1_L$ is normal, it follows that $\coprod_{ i=1}^M \P^1_L \to \P_{L}$ is canonically the normalization.

We may choose an open subset $U_0$ of $U$ with $U_0$ affine. Let $W = \Spec \sO(W)$ where $\sO(W)$ is the integral closure of $\sO(U_0)$ in $L$. Since $U_0$ is a finite type $k$-algebra and an integral domain, $\sO (U_0)$ is a Nagata ring and N-2 \cite[p 240 (31.H) Theorem 72 and Def N-2]{MatsumuraCA}. Thus $W \to U_0$ is a finite surjective map giving rise to the field extension $k(\eta) \to L$ on generic points. $\sO(W)$ is an integral domain by construction. It follows that $W$ is reduced, whence geometrically reduced because $k$ is perfect \cite[tag 020I]{stacks-project}. $\sO(W)$ is furthermore a finite type $k$-algebra because it is finite over $\sO(U_0)$. Thus by generic smoothness \cite[tag 056V]{stacks-project}, there is a non-empty open subset $W'$ of $W$ which is smooth over $k$. Replacing $U_0$ by the (non-empty, open) complement of the image of $W-W'$, and then replacing $W$ by its pullback over the new $U_0$, we obtain a finite surjective map $W \to U_0$ with $W$ smooth over $k$ \cite[tag 056V]{stacks-project}. Passing to a further open subset of $U_0$, we may assume that $\alpha$ extends to a map $\alpha': \coprod_{i=1}^M \P^1_{W}  \to \P_{W}$ which we may assume to be finite and birational. As above, it follows that $\alpha'$ is canonically the normalization.

\end{proof}

\begin{definition}
Let $U$ be a normal, integral scheme over a field $F$ and let $\P \to U$ be an $n$-marked semi-stable genus $0$ curve which is treelike. Let $f: (\P,p_*) \to S$ be a stable map of degree $D$. We say that $f$ is {\em simple} if for each geometric point $u$ of $U$, the restriction of $f_u$ to any component of $\P_u$ is either constant or birational and no two components of $\P_u$ have the same image under $f_u$ as a reduced closed subscheme. 
\end{definition}

\begin{lemma}\label{lm:stratification}
Let $\P \to V$ be an $n$-marked semi-stable genus zero curve and let $f : (\P,p_*) \to S$ be a stable map of degree $D.$ There exist
\begin{itemize}
\item
a stratification $V = V_0 \cup \ldots \cup V_N;$
\item
finite covers $W_i \to V_i;$
\item
$n$-marked semi-stable genus zero curves $\P_i \to W_i;$
\item
stable maps $f_i : \P_i \to S;$
\end{itemize}
such that
\begin{enumerate}
\item
$W_i$ is integral and normal, and $\P_i \to W_i$ is treelike.
\item
$f_i$ is simple;
\item \label{item:component_weights}
there exists a function $m : V(\P_i) \to \Z_{>0}$ such that
\[
\sum_{v \in V(\P_i)} m(v) D_v = D,
\] where $D_v = (f_i)_* [v]$ is the Cartier divisor corresponding to the image of $v$ under $f_i$.
\item \label{it:mvgeq2}
Let $a \in V_i$ be a geometric generic point, and let $f_a$ denote the restriction of $f$ to $\P_a.$ If $f_a$ is not simple, then there exists $v \in V(\P_i)$ such that $m(v) \geq 2.$ Otherwise $\P_i$ has the same number of components as $\P_a.$
\item
$\cup_i ev(f_i) = ev(f).$
\end{enumerate}
\end{lemma}

This is an algebraic version of \cite[Proposition 6.1.2 p.~156]{McDuff-Salamon}.

\begin{proof}
Let $a$ be a geometric generic point of $V$. By Noetherian induction it suffices to find an open neighborhood $U$ of $a$, a finite surjective map $W \to U$ with $W$ integral and normal, a $n$-marked treelike semi-stable genus zero curves $\P' \to W$, and a simple stable map $f': \P' \to S$ such that there exists a function $m : V(\P') \to \N$ such that  $\sum_{v \in V(\P')} m(v) D_v = D$ and $\ev (f_U) = \ev(f')$, where $f_U: (\P_U, p_*\vert_U) \to S$ denotes the restriction of $f$. 

Consider the $n$-marked stable map $f_a: (\P_{a}, p_*\vert_{a}) \to S$. $\P_{a}$ splits into a finite number of components $\P_{a1}, \ldots, \P_{ar}$. We aim to rid ourselves of repeated image curves and non-birational components which are not contracted. We may assume that the components $\P_{a1}, \ldots, \P_{as}$ have different (reduced closed) images under $f_a$ and the images of $\P_{a(s+1)}, \ldots, \P_{ar}$ are all the same as one of the images of $\P_{a1}, \ldots, \P_{as}$. For any $i$ such that the restriction $f\vert_{\P_{ai}}:\P_{ai} \to S$  of $f$ to $\P_{ai}$ is not birational, let $\P_{ai}' \to f(\P_{ai})$ be defined to be the normalization of the reduced image curve $f(\P_{ai})$. Since $\P_{ai} \cong \P^1$ is normal, $f \vert_{\P_{ai}}$ factors $$\P_{ai} \stackrel{\pi_i}{\to} \P_{ai}' \stackrel{f'_i}{\to} f(\P_{ai}).$$ For notational simplicity, set $\P_{ai}'=\P_{ai}$ if $f\vert_{\P_{ai}}:\P_{ai} \to S$ is birational.

We will glue together the $\P_{ai}'$  for $i=1,\ldots, s$ potentially with some additional $(\P^1)$'s to form a treelike $n$-marked curve $\P_a'$. (We will not be gluing in the $\P'_{a(s+1)}, \ldots, \P'_{ar}$.) Since the curve $\P_a$ is treelike, we have an associated tree $T(\P_a)$. Removing the vertices $\P_{a(s+1)}, \ldots, \P_{ar}$ (and resulting edges) produces a forest $F$. (By a {\em forest}, we mean a finite disjoint union of trees.) View $\P_{a1}$ as the root of $T(\P_a)$. Traveling out from $\P_{a1}$ defines a root of every tree of $F$. Call the trees of $F$ not containing $\P_{a1}$ the detached trees. Suppose there is a detached tree in $F$ whose root $r$ is attached in $T(\P_a)$ to a component with the same image as a component $c$ on the tree of $F$ containing $\P_{a1}$. Then attach $r$ to $c$. If there is no such root, then the component containing $\P_{a1}$ contains all the vertices and the forest is just a tree. Now $\P_{a1}$ is contained in a potentially larger component of a new forest. We again consider any detached tree of this new forest whose root $r$ is attached to a component with the same image as a component $c$ on the new tree containing $\P_{a1}$. Again attach $r$ to $c$. This process stops when we have a formed a new tree $T'$ whose vertices are in canonical bijection with $\P'_{a1}, \ldots, \P'_{as}$.

We will modify this tree $T'$ to a new tree $T''$ with some extra vertices. It will have an associated treelike $n$-marked semi-stable genus $0$ curve $\P_a'$ over $a$. The extra vertices will correspond to contracted components. For each $i=1,\ldots, s$, let $A_i \subseteq \{1,\ldots, r\}$ denote the subset those indices $j$ such that $\P_{aj}$ has the same reduced closed image curve as $\P_{ai}$. Let $H(\P_a)$ denote the half-edges associated to the tree-like $\P_a$ (Definition~\ref{df:DualTree}). Define $H_i(\P_a) \subset H(\P_a)$ to be those half-edges lying in $\P_{aj}$ for $j\in A_i$. In other words, $H_i(\P_a)$ contains the marked points and the points where two components are attached for every $\P_{aj}$ with $j \in A_i$. Because $v$ is a geometric point, we may choose a preimage under $f_a'$ in $\P_{ai}'$ for every point $f_a(p)$ with $p$ in $H_i(\P_a)$. Let $H'_i$ denote the multiset of these preimages, i.e. the set of these preimages where repeated preimages are contained with the appropriate multiplicity.

We build $\P_a'$ by gluing $(\P^1)$'s together. Start by putting the component $\P'_{a1}$ in $\P_a'$. If $H'_1$ has points with multiplicity greater than $1$, attach a $\P^1$ at the corresponding point $p$ and choose (arbitrarily) smooth points on the new $\P^1$ in bijective correspondence with the multiple copies of $p$. For points of multiplicity equal to $1$, mark the corresponding point on $\P'_{a1}$. This builds a larger marked semi-stable genus $0$ curve. The tree $T''$ has a vertex for $\P'_{a1}$ and each of the attached $(\P^1)$'s. Extend $f_v'$ to this union by sending any attached $(\P^1)$ to the corresponding $f_a(p)$ in $S$.

We continue to build $\P_a'$ and $f_a': \P_a' \to S$. For each edge in $T'$ connected to the first vertex, attach the corresponding component $\P_{ai}'$ for some $i = 1,\ldots, s$ at the appropriate point of the $\P_a'$ under construction. For each point $p$ of $H_i$ with multiplicity greater than one, attach a new $\P^1$ to $\P_a'$ and choose and mark smooth points on the new $\P^1$ in bijective correspondence with the multiple copies of $p$. Extend the definition of $f_a'$ by mapping $\P'_{ai}$ by $f'_{ai}$ and contracting the new $(\P^1)$'s to the corresponding $f_a'(p)$. Add the $\P_{ai}'$ and new $(\P^1)$'s to the tree $T'$ and edges corresponding to the attachment points.

Running through the vertices of $T'$, we obtain a treelike semistable $n$-marked curve $(\P_a', p'_*)$ with tree $T''$ together with a stable map $f_a': (\P_a', p'_*) \to S$.

Define $m_a : V(\P_a') \to \N$ by setting $m(\P_{ai}')$ equal to the sum of the degrees of $f_{a}\vert_{\P_{aj}}:\P_{aj} \to S$ for each $j$ in $A_i$ and setting $m$ to be $0$ on any contracted component. Thus $\sum_{j \in A_i} (f_a)_*[\P_{aj}] = m(\P_{ai}') (f_a')_* [\P_{ai}'] $, and
\begin{equation}\label{msumfv'}
\sum_{v \in V(\P'_a)} m(v)  (f_a')_* [\P_{ai}'] = D.
\end{equation}

Since $\P_a'$ is treelike, there is a corresponding curve $\P'$ over $V$. Spreading out, there is an open neighborhood $W$ of $a$ in $V$ such that the marked points of $\P'_a$ come from sections $W \to \P'_W$. We thus obtain an $n$-marked semi-stable treelike $(\P', p'_*)$ over $W$. By potentially shrinking $W$ further, the stable map $f_a': (\P_a', p'_*) \to S$ spreads out to a stable map $f':(\P', p'_*) \to S$ over $W$. The property \ref{item:component_weights} follows from \eqref{msumfv'}. This completes the proof.

\end{proof}

Let $V$ be an integral finite type $k$-scheme, smooth over $k$, and let $\P\to V$ be a tree-like family with two irreducible components, $\P=\P_1\cup\P_2$. Let $v\in  V$ be a geometric generic point and let $f:\P\to S$ be a morphism. Let $D_i=f_{v*}(\P_{i,v})$ and let $d_i=f_{v*}(\P_{i,v})\cdot(-K_S)$. Let $D=D_1+D_2$, $d=d_1+d_2$.

Sending $x\in V$ to the curve $f_{i,x*}(\P_{i,x})$ defines morphisms
\[
\bar{f}_i:V\to |D_i|,\ i=1,2
\]
and we have
\[
\bar{f}:V\to |D|
\]
with $\bar{f}(x)=\bar{f}_1(x)+\bar{f}_2(x)$.

Let $F$ be an algebraic closed field and let $f : (\P,p_*) \to S$ be a simple stable map over $F.$
Let $\P_{cont}\subset \P$ be the union of the  irreducible components that get mapped to points by $f$ and let $\P_{simp}\subset \P$ be the union of the remaining irreducible components of $\P$. A connected component of $\P_{cont}$ is rigid if it intersects two or more components of $\P_{simp}$ and movable otherwise. Let $\sS_r(f)$ (resp. $\sS_m(f)$) denote the set of rigid (resp. movable) connected components of $\P_{cont}$ containing at least one marked point. Observe that each such component is a subtree of $\P.$ For $T \in \sS_m(f),$ let $n_T$ be the total number of marked points on the components of~$T.$ Let $\sS_s(f)$ be the set of irreducible components of $\P_{simp}.$

\begin{lemma}\label{lem:masterineq} Suppose that Assumption~\ref{a:genericunram} holds for $S$. We take $n=d-1$. Let $p:V\to  \bar{M}_{0,n}(S, D)$ be a map of an integral finite type $k$-scheme $V$ to $\bar{M}_{0,n}(S, D)$, giving the stable map $f_V:(\P_V, p_{V*})\to S$ of an $n$-pointed genus 0 curve $(\P_V, p_{V*})$ over $V$. Suppose $\P_V$ is treelike and  $f_V$ is simple.  Let $v$ be a geometric generic point of $V$,  giving the stable map $f:(\P,p_{*})\to S$. We consider the image $\ev(V)\subset S^n$. Then
\begin{equation}\label{eq:bigineq}
\dim_k \ev(V) \leq d + n - |\sS_s(f)| - |\sS_m(f)| - |\sS_r(f)|,
\end{equation}
or more precisely,
\begin{equation}\label{eq:bigineqprec}
\dim_k \ev(V) \leq d +n - \vert \sS_s(f) \vert - \sum_{T \in \sS_m(f)} (n_T-1)  - |\sS_r(f)|.
\end{equation}
\end{lemma}
\begin{proof}
Let $T \in \sS_m(f).$ We claim that $T$ contains a leaf of $\P.$ Indeed, $T$ is a subtree of $\P$ that only connects to $\P \setminus T$ by a single edge. Therefore, $T$ contains at least two marked points. On the other hand, if $T \in \sS_r(f)$, the point $f(T)$ of $f(\P)$ is in the intersection $f(P)\cap f(P')$ with $P\neq P' \in \sS_s(f).$

Let $\cI \subset |D| \times S$ denote the incidence variety and let $\cV \subset \cI$ denote the preimage of $\im(V)$. Let $\pi_V : \P_V \to V$ denote the projection.  The map  $f_V:\P_V\to S$ induces the morphism
\[
\bar{f}_V:V\to |D|
\]
sending $g\in V$ to the divisor $f_{V*}(\P_{V,g})$ on $S$.
We also have the relative evaluation map
\[
\ev_{V, |D|}=(\bar{f}_V, \ev_V):V\to |D|\times S^n
\]
factoring the evaluation map $\ev_V:V\to S^n$.

For $C$ in $\sS_s(f)$, define $d_C = f_*[C] \cdot (-K_S)$ to be the degree of $f$ restricted to $C$.
By Lemma~\ref{lem:NonBirat1},
\begin{equation}\label{eq:dimimVbd}
\dim\, \im(V)\le \sum_{C \in \sS_s(f)} (d_C -1) \leq d- \vert \sS_s(f) \vert.
\end{equation}
Since $V$ is treelike, the decomposition of $\P_{y}$ into $\P_{y,simp}$ and $\P_{y, cont}$ and the further decomposition into the various trees in $\sS_m(f_y),\sS_r(f_y),$ is constant as $y$ varies over $V$. Namely, there are canonical bijections between, e.g. $\sS_m(f_y)$ and $\sS_m(f_{y'})$ for $y,y'$ in $V$.  For $y\in V$ and $P$ a component of $\P_y$, let $n_P$ denote the number of marked points.

Fix a geometric point $x$ of $\im(V)$ and let $V_x\subset V$ be the fiber $\bar{f}^{-1}_V(x)$. We proceed to give a bound on $\dim_{k(x)}\ev(V_x)$. For $y\in V_x$ and  $P$ a component of $\P_{y, simp}$, we have $n_P$ marked points, each mapping via $f$ to the curve $f_y(P)\subset S$, so over all of $V_x$ these contribute at most $n_P$ to  $\dim_{k(x)}\ev(V_x)$. If $P$ is a component of some rigid tree $T$, then each of the $n_P$ marked points of $P$ map to the intersection of two components of $f(\P_{y,simp})$, so these contribute 0 to $\dim_{k(x)}\ev(V_x)$.  Finally, taking together all the components $P$ in some movable tree $T$, the sum $\sum_{P \in T} n_P \geq 2$ marked points in $T$ all map to the same point of the curve $f(P_T)$, where $P_T$ is the curve in in $\P_{y,simp}$ intersecting $T$. So altogether, these marked points contribute at most 1 to $\dim_{k(x)}\ev(V_x)$. We obtain the bound
\begin{equation}\label{eq:dimevVxbd}
\dim_{k(x)}\ev(V_x)\le \sum_{P \in \sS_s(f)}n_P+\sum_{T \in \sS_m(f)} 1
\end{equation}
Combining the bounds~\eqref{eq:dimimVbd} and~\eqref{eq:dimevVxbd}, since $n_T \geq 2$ for each $T\in \sS_m(f)$, we get
\begin{align}
\dim_k \ev(V)\le \dim_k \ev_{V, |D|}(V) &\le \dim_k\im(V) +  \max_{x \in \im V} \dim_{k(x)}\ev(V_x)   \\
&\leq d-\vert \sS_s(f) \vert +  \sum_{P \in \sS_s(f)}n_P+\sum_{T\in \sS_m(f)}1     \notag \\
&= d +n - \vert \sS_s(f) \vert - \sum_{T \in \sS_m(f)} (n_T-1)  - |\sS_r(f)| \notag \\
&\leq d + n - \vert \sS_s(f) \vert - |\sS_m(f)| - |\sS_r(f)|. \notag
\end{align}
\end{proof}

\begin{lemma}\label{lem:NonBirat2} Suppose that Assumption~\ref{a:genericunram} holds for $S$. We take $n=d-1$. Assume $D$ is not an $m$-fold multiple of a $-1$-curve for $m>1$. Let $p:V\to  \bar{M}_{0,n}(S, D)$ be a map of an integral finite type $k$-scheme $V$ to $\bar{M}_{0,n}(S, D)$, giving the stable map $f_V:(\P_V, p_{V*})\to S$ of an $n$-pointed genus 0 curve $(\P_V, p_{V*})$ over $V$. Let $v$ be a geometric generic point of $V$,  giving the stable map $f:(\P,p_{*})\to S$. We consider the image $\ev(V)\subset S^n$.
\begin{enumerate}
\item \label{lem:NonBirat2:2>_comp-1} Suppose that $\P$ has at least $3$ irreducible components. Then  $\codim \ev(V)\ge 2$.
\item \label{lem:NonBirat2:2.} Suppose that $f$ is non-birational. Then $\codim \ev(V)\ge 2$.
\item \label{lem:NonBirat2:3.} Suppose that $\P=\P_1\cup \P_2$ has 2 irreducible components and $f$ is birational. Then  $\codim \ev(V)\ge1$.
\item \label{lem:NonBirat2:4.} Suppose that $\P=\P_1\cup \P_2$ has 2 irreducible components, $f$ is birational and  $\codim \ev(V)=1$.  Suppose $\Char k=0$, and either
\begin{enumerate}
\item $d_S\ge3$ or
\item$d_S=2$ and $d\neq 2, 4$.
\end{enumerate} Then $f$ is unramified on $\P$  and the image curve $C:=f(\P)$ has only ordinary double points as singularities. In particular, $f$ is an isomorphism to its image in a neighborhood of $\P_1\cap \P_2$.
Moreover, if $\P_i$ has $n_i$ marked points, $i=1,2$, and $d_i=-K_S\cdot f_{*}(\P_i)$, then $d_i-1\le n_i\le d_i$, $i=1,2$.  
\end{enumerate}
\end{lemma}

\begin{proof}

We apply Lemma~\ref{lm:stratification} to  $f_V:(\P_V, p_{V*})\to S$. It suffices to prove the result for each stratum $V_i$ of $V$, so we may assume from the start that there is only a single stratum $V_1=V$. Similarly since $k$ is perfect, we may assume that $W:=W_1$ is smooth over $k$.  Indeed, since $W_1$ is integral, it is reduced. So, since $k$ is perfect, it is geometrically reduced by~\cite[\href{https://stacks.math.columbia.edu/tag/020I}{Tag 020I}]{stacks-project}. Thus, it is generically smooth by~\cite[\href{https://stacks.math.columbia.edu/tag/056V}{Tag 056V}]{stacks-project}, and we can apply Noetherian induction.

Denote the pullback family to $W$ by $(\P_W, p_{W*})$ and the induced map to $S$ by $f_W:(\P_W, p_{W*})\to S$.  Let $w$ be a geometric generic point of $W$ lying over $v$, giving the stable map $f_{W,w}:(\P_{W,w}, p_{W*,w})\to S$, with reduced image curve $f_{W,w}(\P_{W,w})=f(\P)_\red$ and with $\ev(W)=\ev(V)$.

We now prove  \ref{lem:NonBirat2:2>_comp-1} in the case $f$ is simple.
If $|\sS_s(f)| \geq 3,$ we are done by inequality~\eqref{eq:bigineq}. If $|\sS_s(f)| = 2$ then $|\sS_r(f)| + |\sS_m(f)| \geq 1,$ so this case is also covered.
If $|\sS_s(f)| = 1$ then $\sS_r(f) = \emptyset.$ If $|\sS_m| \geq 2,$ we are again done. If $|\sS_s(f)| = 1$ and $|\sS_m(f)| = 1,$ then for the unique $T \in \sS_m$ we have $n_T \geq 3$ since $T$ has at least two vertices. So, we are done by inequality~\eqref{eq:bigineqprec}. This completes the proof of \ref{lem:NonBirat2:2>_comp-1} when $f$ is simple.

It remains to discuss the case of non-simple $f$. In this case, we consider the family $\P_W\to W$ and map $f_W:(\P_W, p_{W*})\to S$, with $\ev(W)=\ev(V)$. If $\P_W$ has at least 3 components, we are done by  \ref{lem:NonBirat2:2>_comp-1} applied to the family $\P_W$. So, we can assume that $\P_W$ has at most  two components. By properties~\ref{item:component_weights} and~\ref{it:mvgeq2} of Lemma~\ref{lm:stratification}, since $S$ is del Pezzo, the curve class $D':=f_{W,w*}(\P_{W,w})$ has degree
\begin{equation}\label{eq:dpineq}
d':=D'\cdot (-K_S)<d.
\end{equation}
If only a single component $P$ of $\P_{W,w}$ is non-collapsed,
then $f(\P)_\red=f_{W,w}(P)$, hence is irreducible. Also, $f_{W,w*}(P) = D'.$ By assumption $f(\P)_\red$ is not a -1 curve, so $D'\cdot (-K_S)\ge 2$. So, by property~\ref{it:mvgeq2} of Lemma~\ref{lm:stratification}, there exists $m \geq 2$ such that $D = m D'.$ Therefore, by property~\ref{item:component_weights} of Lemma~\ref{lm:stratification},
\begin{equation} \label{eq: dprimeBoundsDEqn}
d= m D' \cdot (-K_S) \geq d' + 2.
\end{equation}
So, by inequality~\eqref{eq:bigineq} applied to $f_W,$ we have
\[
\dim_k\ev(W)\le d'+n-1\le d+n-3.
\]
If $\P_{W,w}$ has two non-collapsed components, then by inequalities~\eqref{eq:bigineq} and~\eqref{eq:dpineq}, we obtain
\[
\dim_k ev(W) \leq d' + n - 2 \leq d + n - 3.
\]
This completes the proof of \ref{lem:NonBirat2:2>_comp-1}.

We now prove \ref{lem:NonBirat2:2.}. By~\ref{lem:NonBirat2:2>_comp-1}, we may assume $\P$ has at most $2$ components. So, if one component collapses to a point, that component has at least two marked points. It follows that the image under $\ev$ is contained in a diagonal of $S^n$ whence has codimension at least $\dim S = 2$. Thus, it remains to consider the case when neither component collapses to a point. The assumption implies that there exists a geometric generic point $a \in V$ such that $f_a$ is not simple. Let $d'$ be defined as in~\eqref{eq:dpineq}. Then $d' < d$ by \eqref{eq:dpineq}. Moreover, if $\P_W$ has only one component, then $d' < d-1$ by \eqref{eq: dprimeBoundsDEqn}.  Then apply \eqref{eq:bigineq} to~$f_W.$ This completes the proof of~ \ref{lem:NonBirat2:2.}.

We now prove \ref{lem:NonBirat2:3.}. Define $d'$ as in~\eqref{eq:dpineq}.
If $\P_W$ has only one component, then $d' < d$ by~\eqref{eq:dpineq}. So, \ref{lem:NonBirat2:3.} follows by applying~\eqref{eq:bigineq} to~$f_W.$

We now prove \ref{lem:NonBirat2:4.}. So, we have $\codim \ev(V) = 1$. As we are only interested in the geometric generic point $f$ and every geometric generic point of $V$ lifts to a geometric generic point of $W$, we may assume that the  family $\P_V\to V$ is tree-like with two irreducible components $\P_{V,1}, \P_{V,2}$, and the map $f_V:\P_V\to S$ decomposes as $f_{V,1} \cup f_{V,2}$, with $f_{V,i}:\P_{V,i}\to S$. Fixing a point $v\in V$, let $D_i$ be the curve class $f_{V,i,v*}(\P_{i,v})$, let $d_i=D_i\cdot(-K_S)$ and suppose that   $n_i$ of the $n$ marked points of $\P_V$ are in $\P_{V,i}$, so $d=d_1+d_2$, $n=n_1+n_2$. The families $f_{V,i}:\P_{V,i}\to S$ thus determine a morphism\[
q:V\to M^\bir_{0,n_1}(S, D_1)\times M^\bir_{0,n_2}(S, D_2).
\]

The subset $q(V)$ is a constructible subset of the product $M^\bir_{0,n_1}(S, D_1)\times M^\bir_{0,n_2}(S, D_2)$. Let $V'$ be a dense subset of $q(V)$ such that $V'$ is locally closed in the product. Then $q^{-1}(V')$ is dense and open in $V$. Thus we may replace $V$ with $q^{-1}(V')$ and assume from the start that $V' = q(V)$ is locally closed in the product.

Let $\ev_i: M^\bir_{0,n_1}(S, D_1) \to S^{n_i}$ for $i=1,2$ denote the evaluation maps. The map $\ev: V \to S^n$ factors through $q$ by $\ev = (\ev_1 \times \ev_2) \circ q$. Moreover $M^\bir_{0,n_1}(S, D_1)\times M^\bir_{0,n_2}(S, D_2)$ is a fine moduli space, so the families $\P_{V,i}$ for $i=1,2$ are pulled back from $V'$. Let $g = (g_1,g_2)$ be a geometric generic point of $V'$. To prove \ref{lem:NonBirat2:4.}, it suffices to show that $g_i$ is unramified for $i=1,2$ and the image curves of the $g_i$ have only ordinary double points and intersect transversally in $S$.

We claim that $V'$ is open in $M^\bir_{0,n_1}(S, D_1)\times M^\bir_{0,n_2}(S, D_2)$. By construction, $\codim \ev_1 \times \ev_2(V') = \codim \ev(V) = 1$. Since $f$ is birational, neither component is contracted, so $d_1, d_2 \geq 1$. Thus by Assumption~\ref{a:genericunram} and Lemma~\ref{rem:DimModuli}, we have $\dim  M^\bir_{0,n_i}(S, D_i) = n_i + d_i -1$.
\begin{multline*}
2n-1 = \dim \ev(V) = \dim \ev_1 \times \ev_2(V') \leq \dim V' \leq \\ \dim M^\bir_{0,n_1}(S, D_1)\times M^\bir_{0,n_2}(S, D_2) \leq n_1 + d_1 -1 + n_2 + d_2 -1 = 2n-1
\end{multline*}
Thus the inequalities are equalities. It follows that $V'$ is open in $M^\bir_{0,n_1}(S, D_1)\times M^\bir_{0,n_2}(S, D_2)$ as claimed. 

We now prove the desired bounds on $n_i.$ Indeed,
\[
1 = \codim ev_1 \times ev_2(V') = \codim ev_1(M^\bir_{0,n_1}(S, D_1)) + \codim ev_2(M^\bir_{0,n_2}(S, D_2))
\]
So, for $i = 1,2.$ 

\[
1 \geq \codim ev_i(M^\bir_{0,n_i}(S, D_i)) \geq 2n_i - (d_i + n_i -1) = n_i - d_i +1.
\]
So, $n_i \leq d_i.$ But $n_1 + n_2 = n = d - 1 = d_1 + d_2 - 1,$ so $n_i \geq d_i -1.$

It remains to prove that $C = f(\P)$ has only ordinary double points. Since $V'$ is open in $M^\bir_{0,n_1}(S, D_1)\times M^\bir_{0,n_2}(S, D_2)$, it follows that $(g_1,g_2)$ is a geometric generic point of  $M^\bir_{0,n_1}(S, D_1)\times M^\bir_{0,n_2}(S, D_2)$. By Assumption~\ref{a:genericunram}, $g_i$ is unramified for $i=1,2$. Let $C_i$ denote the image curve of $g_i$. By Proposition~\ref{prop:GenODP}, $C_1$ and $C_2$ have only ordinary double points. Otherwise $d_i =1$ and $g_i$ is an isomorphism to a $-1$ curve, which is smooth. Thus $C_1$ and $C_2$ have only ordinary double points. If $d_1$ or $d_2$ is at least $3$, we apply Lemma~\ref{lem:intersection} to conclude that $C_1$ and $C_2$ intersect transversally in $S$. Otherwise, $d_1,d_2 \leq 2$. We have three cases $(d_1, d_2) = (1,1),(1,2),(2,2)$.

Suppose $(d_1, d_2) = (1,1)$. Then $C_1$ and $C_2$ are distinct $-1$-curves because $f$ is birational. By hypothesis $d_S \geq 3$. Then $C_1$ and $C_2$ are embedded lines in $\P^{d_S}$ and must intersect transversally.

Suppose $d_S\ge3$, so $S$ is anti-canonically embedded in $\P^{d_S}$.
If $d=3$, then we may assume $d_1=1$, $d_2=2$, so $C_1$ is a $-1$ curve and $C_2$ is a smooth conic. By the adjunction formula, we have $C_2\cdot C_2=0$, and thus $C_2$ is a smooth rational curve on $S$ with trivial normal bundle. Noting the exact sequence
\[
0 \to \sO_S \to \sO_S(C_2) \to {i_{C_2}}_* (\sO_{C_2}(C_2\cdot C_2) ) \to 0,
\] this implies that $h^0(S, \sO_S(C_2))=2$. Thus the complete linear system $|C_2|\cong \P^1$ has dimension 1 and has no base-points. Moreover, there is open subset $U\subset |C_2|$ such that each $u\in U$ corresponds to a smooth rational curve $C_u$ in the curve class $|C_2|$. Since $\Char k=0$, Bertini's theorem applied to the linear system $|C_2|\cap C_1$ on $C_1$  implies that for all $u$ in a dense open subset $V$ of $U$, $C_u$ intersects $C_1$ transversely. Since $g_2$ is generic, this implies that $C_1$ and $C_2$ intersect transversely. If $d_1=d_2=2$, the same proof shows that  $C_1$ and $C_2$ are smooth curves intersecting transversely.

The remaining case is $(d_1,d_2) =(1,2)$, $d=3$, $d_S=2$. Let $\pi:S\to \P^2$ be the anti-canonical map, with smooth quartic branch curve $E$. Thus  $C_1$ is a $-1$ curve. There are two possibilities for $C_2$: either $\pi$ induces an isomorphism of $C_2$ with a smooth plane conic, or $C_2\to \pi(C_2)$ is a double cover, with $\pi(C_2)$ a line $\ell$ satisfying $\ell\cdot E=p_1+p_2+2p_3$, with the $p_i$ distinct points of $E$ (see Lemma~\ref{lm:dS2d2} and its proof); indeed, by Lemma~\ref{lem:Deg4Tangent}, such lines are generic in the variety of tangent lines to $E$. In the first case, we again have $C_2\cdot C_2=0$ and we proceed as in the case $d_S\ge3$, $d=3$. Consider the second case. We note that the $-1$ curve  $C_1$ is one of the two components of $\pi^{-1}(\ell')$, where $\ell'$ is a line satisfying $\ell'\cdot E=2p_1+2p_2$, with $p_1, p_2$ points of $E$, not necessarily distinct. Since $g_2$ is generic, the lines $\ell$ and  $\ell'$ intersect at a point $q$ not on $E$, so $\pi$ is \'etale over a neighborhood of $q$ and thus $C_1$ and $C_2$ intersect transversely.
\end{proof}

\begin{remark} If $d_S=2$, then if the branch curve $E$ of the anti-canonical map $\pi:S\to \P^2$ admits line $\ell$ with a 4-fold intersection at a single point $q$ of $E$, then $\pi^{-1}(\ell)=C_1\cup C_2$, where the $C_i$ are -1 curves intersecting with multiplicity 2 at the point $q'$ of $S$ lying over $q$. This is why we need to exclude the case $d_S=d=2$ in \ref{lem:NonBirat2:4.} above.
\end{remark}

\begin{remark}
Lemma~\ref{lem:NonBirat2}~\ref{lem:NonBirat2:2.} is false without the hypothesis that $D$ is not an $m$-fold multiple of a $-1$-curve for $m>1$. For example consider $S = \Bl_0 \P^2$ over a field of characteristic $0$, and let $D$ be twice the exceptional divisor $E$. Then $d=2$. The family over $\A^1 \cong \Spec k[a]$ given by  $f_{\A^1}: \P^1 \stackrel{t \mapsto t^2 - a}{\longrightarrow} \P^1 \cong E \to S$ has $\ev(\A^1)$ of codimension $1$.
\end{remark}

\begin{corollary}\label{Cor:codim_ev(nonodp)>=1} Let $k$ be a field of characteristic not $2$ or $3$ and let $S$ be a del Pezzo surface over $k$ with $d_S \geq 2$. Suppose that Assumption~\ref{a:genericunram} holds for $S$, and that $D$ is a Cartier divisor on $S$ which is not an $m$-fold multiple of a $-1$-curve for $m>1$. Then $\codim~\ev(\bar{M}_{0,d-1}(S, D) \setminus M^{\odp}_{0,d-1}(S,D)) \ge 1$.
\end{corollary}

\begin{proof}
By Lemma~\ref{lem:NonBirat2}~\ref{lem:NonBirat2:2>_comp-1}~\ref{lem:NonBirat2:2.}~\ref{lem:NonBirat2:3.}, $\codim \ev(\bar{M}_{0,d-1}(S, D) \setminus M^{\bir}_{0,d-1}(S,D)) \ge 1$. By Proposition~\ref{prop:GenODP}, the generic point of any irreducible component of $M^\bir_{0,d-1}(S, D)$ is in the ordinary double point locus. Thus the dimension of any irreducible component of $M^\bir_{0,d-1}(S, D)$ is $2(d-1)$. See Lemmas~\ref{rem:DimModuli}~\ref{lm:odp_in_unr_in_M_open} and Remark~\ref{rmk:computation_normal_sheaf_for_unramified}. Moreover, $M^{\odp}_{0,d-1}(S,D)$ is dense in $M^\bir_{0,d-1}(S, D)$. Thus $M^\bir_{0,d-1}(S, D) \setminus M^{\odp}_{0,d-1}(S,D)$ is dimension $<2(d-1)$, whence so is the closure of $\ev(M^\bir_{0,d-1}(S, D) \setminus M^{\odp}_{0,d-1}(S,D))$, which proves the claim.
\end{proof}

\section{Fine structure of the evaluation map}

 Let $k$ be a perfect field, let $S$ be a del Pezzo surface over $k$ with effective Cartier divisor $D$. Let $d=-\Deg K_S\cdot D\ge1$ and let $n=d-1$. We introduce a list of assumptions, which will be convenient for future reference (but which are not running assumptions throughout this section).

\begin{BA}\label{BA} \ 
\begin{enumerate}
\item \label{it:a1charzero}
$\Char k = 0$.
\item\label{it:exlude_multiple_covers_-1curves} $D$ is not an $m$-fold multiple of a $-1$-curve for $m>1$. 
\item \label{it:a3degree}
One of the following holds.
\begin{itemize}
\item $d_S\ge 4$
\item $d_S=3$ and $d\neq 6$
\item $d_S = 2$ and $d\ge 7$
\end{itemize}
\end{enumerate}
\end{BA}

We do assume that $k$ is characteristic $0$ in this section.

\begin{definition}\label{def:bsm}
The locus of \emph{reducible stable maps} is
\[
\bar{M}_{0,n}(S,D)^{\red} = \bar{M}_{0,n}(S,D) \setminus M_{0,n}(S,D).
\]
The locus of \emph{balanced stable maps}
\[
\bar{M}_{0,n}(S,D)^{\bal} \subset \bar{M}_{0,n}(S,D)^{\red}
\]
is the closure of the locus of stable maps $f : (\P,p_*) \to S$ satisfying the following conditions.
\begin{enumerate}
\item
$\P=\P_1\cup\P_2$, with $\P_i\cong \P^1$ and $\P_1 \cap \P_2 = \{p\}$.
\item\label{it:unrtr}
$f$ is unramified and $f|_{\P_1}$ is transversal to $f|_{\P_2}$ at $p.$
\item
Letting $n_i$ denote the number marked points on $\P_i$, and letting
\[
d_i:=-K_S\cdot f_*(\P_i)
\]
denote the degree of $f|_{\P_i},$ we have $d_i-1\le n_i\le d_i$ for $i=1,2$.
\end{enumerate}
\end{definition}
\begin{proposition}\label{pr:evunrambal}
Let $f :  (\P,p_*) \to S$ be a stable map representing a geometric generic point of an irreducible component of $\bar{M}_{0,n}(S,D)^{\bal}.$ Then $\ev$ is unramified at $f.$
\end{proposition}
\begin{proof}
Since $f$ is unramified, $f$ is birational (Lemma \ref{lm:unramified_maps_are_birational}) and has no automorphisms, so $\bar{M}_{0,n}(S,D)$ is a scheme in a neighborhood of $f$. Applying \cite[Tag 0B2G]{stacks-project}, the claim is equivalent to showing $d\ev$ is injective on tangent spaces. By Proposition \ref{Mbar_smooth_unramified_map_from_two_comp_reducible_transverse_tangent_directions}, it thus suffices to show that $d\ev$ is surjective on tangent spaces.

We write
$\P=\P_1\cup\P_2$, with $\P_i\cong \P^1$ and $\P_1 \cap \P_2 = \{p\}$.
Let $n_i$ denote the number of marked points on $\P_i$, and let
\[
D_i =f_*([\P_i]), \qquad  d_i:=-K_S\cdot D_i.
\]
We may assume that $n_1 = d_1$ and $n_2 = d_2 -1$ and $p_j$ lies on $\P_1$ for $j = 1,\ldots,n_1.$

Let $\nu : \bar{M}_{0,n}(S,D) \to \bar{M}_{0,n-1}(S,D)$ denote the map forgetting the first marked point and stabilizing. Let $f' = \nu(f).$ Let $n_1' = n_1 -1$ and $n_2' = n_2$ be the number of marked points on the respective components of $f'.$ We show that
\[
dev_{f'}: T_{f'}\bar{M}_{0,n-1}(S,D)^\red \to T_{ev(f')}S^{n-1}
\]
is surjective. Consider the following diagram.
\[
\xymatrix{
&\bar{M}_{0,n_1'+1}(S,D_1) \times_S \bar{M}_{0,n_2'+1}(S,D_2) \ar[ld]_{\nu_1\times \nu_2} \ar[dd]^{ev_1 \times ev_2}\ar[r]^(.65){\tau} & \bar{M}_{0,n-1}(S,D)^\red\ar[dd]^{ev} \\
\bar{M}_{0,n_1'}(S,D_1) \times \bar{M}_{0,n_2'}(S,D_2)\ar[rd]^{ev_1\times ev_2} \\
&S^{n_1'} \times S^{n_2'} \ar[r]^{\sim} & S^{n-1}.
}
\]
The fiber product over $S$ is taken with respect to the evaluation maps at  the $(n_i'+1)$th marked point on $\bar{M}_{0,n_i'+1}(S,D_i)$ for $i = 0,1.$ The map $\tau$ is the map that attaches the domains at the $(n_i' + 1)$th marked points forming a node. Observe that $f'$ belongs to the image of $\tau.$ Let
\[
f_i' \in \bar{M}_{0,n_i'+1}(S,D_i)
\]
be such that $\tau(f_1',f_2') = f'.$ By commutativity of the diagram, it suffices to show that
$d(ev_1 \times ev_2)_{(f_1',f_2')}$ is surjective. By Lemma~\ref{lem:EvUnram}, $d(ev_i)_{\nu_i(f_i')}$ is surjective for $i = 1,2.$ So, using the commutativity of the diagram, it suffices to show that $d(\nu_1 \times \nu_2)_{(f_1',f_2')}$ is surjective. Indeed, let $\xi_i \in \Gamma(\P_i,\sN_{\nu_i(f_i')})$ represent tangent vectors at $\nu(f_i').$ By condition~\eqref{it:unrtr} of Definition~\ref{def:bsm},
\[
df_1'(T_{p_{n_1'+1}}\P_1)\oplus df_2'(T_{p_{n_1' + 1}}\P_2) = T_{f(p)} S.
\]
So, projecting along the respective summands, we obtain canonical isomorphisms
\[
(\sN_{f_1'})_{p_{n_1'+1}} \cong T_{p_{n_2'+1}}\P_2, \qquad (\sN_{f_2'})_{p_{n_2'+1}} \cong T_{p_{n_1'+1}}\P_1.
\]
Let $v_i \in T_{p_{n_i'+1}}\P_i$ be the tangent vector corresponding to $\xi_j(p_{n_j' + 1})$ for $j = 3-i.$ Then the tangent vectors at $f'_i$ corresponding to $\xi_i$ and $v_i$ lift the tangent vectors corresponding to $\xi_i.$ Thus, we have established the surjectivity of $dev_{f'}.$

We show that there exists a tangent vector $v \in T_f\bar{M}_{0,n}(S,D) \setminus T_f\bar{M}_{0,n}(S,D)^{\red}$ such that $dev_j(v) = 0$ for $j = 2,\ldots n.$
\[
\xymatrix{
T_f \bar{M}_{0,n}(S,D)^\red\ar[d]^{d\nu_f} \ar[r] & T_f \bar{M}_{0,n}(S,D) \ar[r]^(.6){dev_f} & T_{ev(f)} S^n \ar[d]^{d\pi}\\
T_{f'}\bar{M}_{0,n-1}(S,D)^\red \ar[rr]^(.6){dev_{f'}} & & T_{ev(f')}S^{n-1}.
}
\]
Since $f$ is generic, the domain curve of $\nu(f)$ does not undergo stabilization. It follows that $d\nu_f$ is surjective. By the proceeding claim, $dev_{f'}$ is surjective. So, for any $v'\in T_f\bar{M}_{0,n}(S,D) \setminus T_f\bar{M}_{0,n}(S,D)^{\red},$ we may pick $w \in T_f \bar{M}_{0,n}(S,D)^\red$ such that
\[
dev_{f'} \circ d\nu_f(w) = d\pi \circ dev_f (v').
\]
Thus, we take $v = v' - w.$ To complete the construction of $v,$ we show that such $v'$ exists. Indeed, by Proposition~\ref{Mbar_smooth_unramified_map_from_two_comp_reducible_transverse_tangent_directions}
\[
\dim T_f\bar{M}_{0,n}(S,D) = d - 1 +n.
\]
On the other hand, we compute
\[
\dim T_f\bar{M}_{0,n}(S,D)^{\red} = d - 2 +n
\]
as a transverse fiber product over $S.$ The dimension of the factors are given by Lemma~\ref{rem:DimModuli}. The transversality over $S$ follows from Lemma~\ref{lem:EvUnram}.

We claim that to complete the proof of the proposition it suffices to show that
\begin{equation}\label{eq:ev1v}
dev_1(v) \notin \ims df_{p_1}(T_{p_1}\P).
\end{equation}
Indeed, let
\[
V = \operatorname{span}\{v,T_{p_1}\P\} \subset T_f \bar{M}_{0,n}(S,D).
\]
So, by construction of $v,$ we have a diagram with short exact rows as follows.
\[
\xymatrix{
&&T_f \bar{M}_{0,n}(S,D)^\red \ar[d] \ar'[rd]^(.6){dev_{f'}\circ d\nu_f} [rrdd] & \\
V \ar[rr]\ar[d]^{dev_f|_V}& & T_f \bar{M}_{0,n}(S,D) \ar[rr] \ar[d]^{dev_f} && \coker \ar[d]^{c}  & \\
T_{f(p_1)} S \ar[rr] && \oplus_{i = 1}^n T_{f(p_i)} S \ar[rr]^{d\pi}&&\oplus_{i = 2}^n T_{f(p_i)}S.
}
\]
(Here $\coker$ denotes the quotient vector space $T_f \bar{M}_{0,n}(S,D)/V$.) We showed above that $dev_{f'} \circ d\nu_f$ is surjective, so $d\pi \circ dev_f$ is surjective. Thus $c$ is surjective. It follows from~\eqref{eq:ev1v} that $dev_f|_V$ is surjective. So, $dev_f $ is surjective as desired.

Finally, we show~\eqref{eq:ev1v}. Let $F$ be the field of definition of $f.$ Let $\mathfrak{f} : (\mathfrak{P},\mathfrak{p}_*) \to S$ denote a stable map over $F[\epsilon]/(\epsilon^2)$ corresponding to $v.$ We identify $\P$ with the closed fiber of $\mathfrak{P},$ so $\mathfrak{f}|_{\P} = f.$ Choose an open set $p \in U \subset \mathfrak{P}$ with an open immersion
\[
U \subset \Spec(F[s,t,\epsilon]/(st- \epsilon,\epsilon^2)).
\]
Choose an open set $f(p) \in V \subset S$ and maps $x,y: V \to \A^1$ such that $x \times y : V \to \A^2$ is \'etale and
\begin{equation}\label{eq:xy}
x \circ f \in s + (s^2,st,t^2), \qquad y \circ f \in t + (s^2,st,t^2).
\end{equation}

Consider the open subscheme
\[
U' = \Spec(F[s,s^{-1},\epsilon]/(\epsilon^2))\cap U \subset U.
\]
The advantage of $U'$ is that it carries a well-defined vector field $\frac{\partial}{\partial \epsilon}$ along the locus $\{ \epsilon  = 0\}.$ Such a vector field cannot exist on $U$ because the map $U \to \Spec(F[\epsilon]/(\epsilon^2))$ is not smooth.
We claim it suffices to show that the sections
\[
\xi = \left. d\mathfrak{f}\left(\frac{\partial}{\partial\epsilon}\right)\right|_{\epsilon = 0} \in \Gamma(U' \cap \P,f^*TS), \qquad \eta =  df\left(\frac{\partial}{\partial s}\right) \in \Gamma(U' \cap \P,f^*TS),
\]
are generically linearly independent. Indeed, since
\[
dev_1(v) = \xi(p_1), \qquad df_{p_1}(T_{p_1}\P) = \text{span}_F(\eta(p_1)),
\]
and $p_1$ is generic, we obtain~\eqref{eq:ev1v}.

We show the generic linear independence of $\xi$ and $\eta$ as follows. Observe that
\[
x \circ \mathfrak{f}|_{U} \in s + (\epsilon, s^2,st,t^2), \qquad y \circ \mathfrak{f}|_{U} \in t + (\epsilon, s^2,st,t^2).
\]
Let $Q \subset F(s)[\epsilon]/(\epsilon^2)$ denote the regular functions on $U'$, let $Z = U \cap \{t = 0\}$ and let $R \subset F[s]_{(s)}$ denote the regular functions on $Z.$ Observe there is an inclusion map $R \to Q$ induced by the inclusion $F[s]_{(s)} \subset F(s)[\epsilon]/(\epsilon^2).$ For $a_1,a_2 \in Q,$ we denote by $(a_1,a_2)_R \subset Q$ the set of all functions obtained by linear combinations of $a_1$ and $a_2$ with coefficients in $R.$
Restricting to $U'$ amounts to replacing $t \mapsto \epsilon s^{-1},$ so we obtain,
\[
x \circ \mathfrak{f}|_{U'} \in s + (\epsilon, s^2)_{R}, \qquad y \circ \mathfrak{f}|_{U'} \in \epsilon s^{-1} + (\epsilon,s^2)_{R}.
\]
So,
\[
dx(\xi) = d(x \circ \mathfrak{f}|_{U'})\left.\left(\frac{\partial}{\partial\epsilon}\right)\right|_{\epsilon = 0}
\]
is regular at $p,$ and
\[
dy(\xi) = d(y \circ \mathfrak{f}|_{U'})\left.\left(\frac{\partial}{\partial\epsilon}\right)\right|_{\epsilon = 0}
\]
has a simple pole at $p.$ On the other hand, by equation~\eqref{eq:xy} we have
\[
 dx(\eta) = d(x \circ f|_{Z})\left(\frac{\partial}{\partial s}\right) \in 1 + (s), \qquad dy(\eta) = d(y \circ f|_{Z})\left(\frac{\partial}{\partial s}\right) \in  (s).
\]
It follows that
\[
\det\begin{pmatrix}
dx(\xi) & dy(\xi) \\
dx(\eta) & dy(\eta)
\end{pmatrix}
\]
has a simple pole at $p$,
so $\xi,\eta,$ are generically linearly independent as desired.
\end{proof}

\begin{lemma}\label{lm:biratempty}
Suppose Basic Assumptions \ref{BA} hold for $k, S, D$, and suppose $M^\bir_0(S,D)=\0.$ Then $\codim ev(\bar{M}_{0,n}(S,D)) \geq 2.$
\end{lemma}
\begin{proof}
Since $\Char k = 0$, it follows by Lemma~\ref{lm:char0_implies_Assumption_a:genericunram} that Assumption~\ref{a:genericunram} holds. By Lemma~\ref{lem:NonBirat1} and Basic Assumption~\ref{BA}
\ref{it:exlude_multiple_covers_-1curves}, we conclude that $\codim ev(M_{0,n}(S,D)) \geq 2.$ Suppose by way of contradiction that
\[
c: = \codim ev(\bar{M}_{0,n}(S,D) \setminus M_{0,n}(S,D)) \leq 1.
\]
By Lemma~\ref{lem:NonBirat2} and Basic Assumption~\ref{BA}
\ref{it:a3degree} we see that
$
c = 1
$
and $\bar{M}_{0,n}(S,D)^{\bal} \neq \0.$ So, Proposition~\ref{pr:evunrambal} shows that $c = 0,$ which is a contradiction.
\end{proof}

\begin{theorem} \label{prop:Good} Suppose Basic Assumptions \ref{BA} 
hold for $k, S, D$.
Then there is a closed subset $A\subset S^n$ with $\codim ~A\ge 2$ such that the complement of the inverse image $\bar{M}_{0,n}(S,D)^\good:=\bar{M}_{0,n}(S,D) \setminus \ev^{-1}(A)$ satisfies the following.
\begin{enumerate}
\item\label{it:evfinflatdom}
$\bar{M}_{0,n}(S,D)^\good=\0$ if and only if $M^\bir_0(S,D)=\0$. If $M^\bir_0(S,D)\neq\0$, then the moduli space $\bar{M}_{0,n}(S,D)^\good$ is a geometrically irreducible smooth finite-type  $k$-scheme, and the restriction of $\ev$ to $\ev:\bar{M}_{0,n}(S,D)^\good\to S^n\setminus A$ is a finite, flat, dominant morphism.
\item\label{it:evgenetale}
The evaluation map $\ev$ is \'etale in a neighborhood of each $f \in \bar{M}_{0,n}(S,D)^\good$ with $t(f)=0$.\footnote{See Definition \ref{def:torsion_index} for the definition of the torsion index $t(f)$.}
\item\label{it:goododp}
$\bar{M}_{0,n}(S,D)^\good$ contains a dense 
open subset of $M_{0,n}^\odp(S;D)$.
\item \label{it:bir} Geometric points $f$ of $\bar{M}_{0,n}(S,D)^\good$ correspond to birational maps.
\item \label{it:goodnotodp}
Let $f$ be a geometric point of $\bar{M}_{0,n}(S,D)^\good\setminus M_{0,n}^\odp(S;D)$, which we consider as a morphism $f:\P\to S$ for some genus 0 semi-stable curve $\P$. Then $f$ satisfies:
\begin{enumerate}[label=(\roman*)]
\item \label{it:irr}  If $\P=\P^1$ is irreducible, then the image curve $C:=f(\P^1)$ has one singular point $q$ that is not an ordinary double point, and $C$ has either an ordinary cusp, an ordinary tacnode or an ordinary triple point at $q$. Moreover, the marked points do not map to $q$ and $f$ is free. 
\item \label{it:red} If $\P$ is not irreducible, then $\P=\P_1\cup\P_2$, with $\P_i\cong \P^1$. The image curve $C:=f(\P)$ has only ordinary double points as singularities. Moreover, if $n_i$ of the $n$ marked points of $\P$ are in $\P_i$, and $C_i:=f(\P_i)$ has degree $d_i:=-K_S\cdot C_i$, then $d_i-1\le n_i\le d_i$ for $i=1,2$.
\end{enumerate}
\end{enumerate}
\end{theorem}
\begin{remark}
In particular, if $\bar{M}_{0,n}(S,D)^\good \neq \emptyset$ then $\ev:\bar{M}_{0,n}(S,D)^\good\to S^n\setminus A$ is dominant.
\end{remark}

\begin{proof} As the assertions are all detected after a field extension of $k$, we may assume that $k$ is algebraically closed. The moduli stack $\bar{M}_{0,n}(S,D)$ is a proper Artin stack over $k$, so the morphism $\ev:\bar{M}_{0,n}(S,D)\to S^n$ is a proper morphism.

If $M_{0}^\bir(S,D) = \0,$ by Lemma~\ref{lm:biratempty} we can take $A = ev(\bar{M}_{0,n}(S,D))$ and $$\bar{M}_{0,n}(S,D)^\good = \0.$$ Since $M_{0,n}^\odp(S,D) \subset M_{0,n}^\bir(S,D)$ part~\ref{it:goododp} holds. The rest of the Proposition is immediate. Thus, for the remainder of the proof, we assume $M_{0}^\bir(S,D) \neq \0.$

By Lemma~\ref{lm:char0_implies_Assumption_a:genericunram} and Basic Assumption~\ref{BA}\ref{it:exlude_multiple_covers_-1curves}
we may apply Lemma~\ref{lem:NonBirat2}. By Lemma \ref{lem:NonBirat2} \ref{lem:NonBirat2:2.}, we may choose $\bar{M}_{0,n}(S,D)^\good$ so that geometric points $f$ of $\bar{M}_{0,n}(S,D)^\good$ correspond to birational maps, showing~\ref{it:bir}.

We claim that we may choose $\bar{M}_{0,n}(S,D)^\good$ so that geometric points $f: \P^1 \to S$ in $\bar{M}_{0,n}(S,D)^\good\setminus M_{0,n}^\odp(S;D)$ are free. By semi-continuity of cohomology, the non-free locus is closed, and therefore has a finite number of irreducible components. By Lemma~\ref{lem:Nonfree} with $V$ the closure of a geometric point with the reduced substack structure, this will be accomplished by eliminating cases 1,2, and 3 of Lemma \ref{lem:Nonfree}. Case 3 of Lemma \ref{lem:Nonfree} contradicts Basic Assumption \ref{BA} \ref{it:a3degree}. We now eliminate case 2. Since $f: \P^1 \to C$ is a two-to-one cover, we have $D=f_*[\P^1] = 2C$. By
assumption, we have that $M^\bir_0(S,D)\neq\0$. Since $d=2$, it follows from Lemma~\ref{lm:char0_implies_Assumption_a:genericunram}, Remark~\ref{rmk:computation_normal_sheaf_for_unramified}, and the vanishing of $H^1(\P^1, \sO(-1))$ that $M^\birf_0(S,D)\neq\0$. Choose a geometric point $f'$ of $M^\birf_0(S,2C)$. Let $C':=f'(\P^1)$. We can not have $C' = C$ because then $f'$ would not be birational. Thus $C' \cdot C $ must be positive, as it is the intersection of two distinct irreducible curves. On the other hand, $C' \cdot C = 2C \cdot C = -2$ because $C$ is a $-1$ curve. In case 1, the map $f$ belongs to $M_{0,n}^\odp(S;D),$ which we do not consider.

By the preceding claims of birationality and freeness,
\[
(\bar{M}_{0,n}(S,D)^\good \cap M_{0,n}(S,D))\setminus M_{0,n}^\odp(S;D) \subset M_{0,n}(S,D)^\birf.
\]
So Lemma~\ref{prop:SingCodim1} applies and~\ref{it:goodnotodp}\ref{it:irr} follows except for the claim about the marked points. However, the locus in $(\bar{M}_{0,n}(S,D)^\good \cap M_{0,n}(S,D))\setminus M_{0,n}^\odp(S;D)$ where one of the marked points coincides with $q$ is codimension $1$, so we can redefine $A$ to remove it.

Since we have chosen $\bar{M}_{0,n}(S,D)^\good$ so that geometric points are birational, we may apply Lemma \ref{lem:NonBirat2} part~\ref{lem:NonBirat2:4.}. Note that by Basic Assumption \ref{BA} \ref{it:a3degree}, we have $d_S \geq 2$. Thus when the domain curve $\P$ of $f:\P\to S$ is reducible, we have $\P=\P_1\cup\P_2$, with $\P_i\cong \P^1$, and the image curve $C:=f(\P)$ has only ordinary double points as singularities, and $f:\P\to C$ is  an isomorphism in a neighborhood of $\P_1\cap \P_2$. Since $\P$ has arithmetic genus~$0,$ the intersection $\P_1 \cap \P_2$ consists of a single point.

For the bounds $d_i-1\le n_i\le d_i$ in~\ref{it:goodnotodp}\ref{it:red}, proceed as follows. 
Let $V \subset \bar{M}_{0,n}(S,D)^\good\setminus M_{0,n}(S,D)$ be an irreducible component. If $\codim ev(V) \geq 2,$ add $ev(V)$ to $A.$ So, we may assume $\codim ev(V) \leq 1.$ Applying Lemma \ref{lem:NonBirat2} parts~\ref{lem:NonBirat2:3.} and~\ref{lem:NonBirat2:4.} to the generic point of $V$ proves gives the desired bounds and thus completes the proof of~\ref{it:goodnotodp}.

Next, we prove \ref{it:goododp}. By assumption,
we have $M^\bir_0(S,D)\neq\0.$ So, it follows from Proposition~\ref{prop:GenODP} that $\bar{M}_{0,n}(S,D)^\odp \neq \emptyset$. Since $d_S \geq 2$, apply Lemma~\ref{lem:EvUnram} to deduce that $\ev$ is \'etale on $M_{0,n}^\odp(S;D).$ Thus, for any proper closed subset $A\subset S^n$, the complement of the preimage $M_{0,n}^\odp(S;D)\setminus \ev^{-1}(A)$ is dense in $M_{0,n}^\odp(S;D)$. So, $\bar{M}_{0,n}(S,D)^\good = \ev^{-1}(S^n\setminus A)$ contains a non-empty dense open subset of $M_{0,n}^\odp(S;D)$ proving part~\ref{it:goododp}.

Since $\ev$ is unramified on the non-empty space $M_{0,n}^\odp(S;D),$ and $M_{0,n}^\odp(S;D)$  is smooth of dimension $2n$ by Lemma~\ref{lem:EvUnram}, it follows that $\ev:M_{0,n}^\odp(S;D)\to S^n$ is dominant. By part~\ref{it:goododp} it follows that $\ev:\bar{M}_{0,n}(S,D)^\good\to S^n\setminus A$ is dominant as claimed in part~\ref{it:evfinflatdom}.

We now show part~\ref{it:evgenetale}. Let $f: \P \to S$ with $t(f)=0$ represent a point of $\bar{M}_{0,n}(S,D)^\good.$ In the case $\P = \P^1,$ part~\ref{it:evgenetale} follows from Lemma~\ref{lem:EvUnram}. In the case $\P \neq \P^1,$ it follows from \ref{it:goodnotodp}\ref{it:red} that $f$ is balanced, so we may apply Proposition~\ref{pr:evunrambal}.

The geometric irreducibility in~\ref{it:evfinflatdom} is proved as follows. Since birationality is an open condition, property~\ref{it:bir} implies $\bar{M}_{0,n}(S,D)^\good \subset \bMbir{0}{n}(S,D).$ Since the inclusion $\bar{M}_{0,n}(S,D)^\good \subset \bar{M}_{0,n}(S,D)$ is open and dense, so is the inclusion $\bar{M}_{0,n}(S,D)^\good \subset \bMbir{0}{n}(S,D).$  So geometric irreducibility follows from Theorem~\ref{thm:Testa}.
The fact that $\ev^{-1}(S^n\setminus A)$ is a scheme as in~\ref{it:evfinflatdom} follows from the fact that by construction each $f\in \ev^{-1}(S^n\setminus A)$ is birational and hence has no automorphisms.

We now show the smoothness claim of~\ref{it:evfinflatdom}. Let $f \in \bar{M}_{0,n}(S,D)^\good.$ If $f$ is irreducible, by property~\ref{it:goodnotodp}\ref{it:irr} it is free, so Lemma~\ref{rem:DimModuli} asserts that $\bar{M}_{0,n}(S,D)^\good$ is smooth at $f$ of dimension $2n.$ If $f$ is reducible, then property~\ref{it:goodnotodp}\ref{it:red} and Proposition~\ref{Mbar_smooth_unramified_map_from_two_comp_reducible_transverse_tangent_directions} imply that $\bar{M}_{0,n}(S,D)^\good$ is smooth at $f$ of dimension $2n.$

To finish the proof of~\ref{it:evfinflatdom} and thus the proof of the proposition, we need to show that the map $\ev:\ev^{-1}(S^n\setminus A)\to S^n\setminus A$ is finite and flat. By~\ref{it:evgenetale}, it is generically finite. It is proper because $\bar{M}_{0,n}(S,D)$ and $S$ are proper and properness is preserved under pull-back. Since $S$ is smooth, after potentially adding to $A$ a subset of codimension at least $2$ in $S,$ we can assume that $\ev:\ev^{-1}(S^n\setminus A)\to S^n\setminus A$ is finite and flat as desired. See, for example, Proposition~2.9 of~\cite{degree}.
\end{proof}

From Theorem~\ref{prop:Good}, we see that $\ev:\bar{M}_{0,n}(S,D)^\good\to S^n\setminus A$ is unramified on $\bar{M}_{0,n}(S,D)^\good\setminus D_\cusp$. We conclude our discussion of the structure of the evaluation map by computing the ramification index along $D_\cusp$.

Let $f$ be a geometric point of $D_\cusp\cap \bar\sM_{0,n}(S, D)^\good$ with field of definition $F$. We describe a linear map $F \to T_{f}\bar\sM_{0,n}(S, D)^\good$ (which in fact we will show to be injective for any such $f$, even a geometric generic point). First, we construct vectors $\mathcal{V}_a$ in the tangent space $T_{\tilde{f}}\bar\sM_{0,n}(S, D)_F^\good$ of the base change to $F$ of $ \bar\sM_{0,n}(S, D)^\good$ at the canonical point $\tilde{f}$ corresponding to $f$; our map sends $a$ in $F$ to the image of $\mathcal{V}_a$ in $T_{f}\bar\sM_{0,n}(S, D)^\good$. This allows us to use the cohomological description of the tangent space to $\bar\sM_{0,n}(S, D)^\good$ at a closed point while constructing tangent vectors at potentially non-closed points.

By Theorem~\ref{prop:Good}, the point $\tilde{f}$ corresponds to a morphism $\tilde{f}:\P_F^1\to S_F$ which is birational onto its image $C = f(\P_F^1)$ together with $n=d-1$ distinct points $p_1,\ldots, p_n\in \P^1(F)$. Moreover, there is a unique point $p\in \P^1$ where $\tilde{f}$ is ramified, and $\tilde{f}(\P^1)$ has a simple cusp at $q=\tilde{f}(p)$. We may assume $p=0:=[1:0]\in \P^1$. Let $t$ be the standard coordinate on $\P^1 \setminus \{\infty\}.$
\begin{lemma}\label{lm:cuspcoord}
We can choose a system of analytic coordinates $(x,y)$ at the ordinary cusp $q\in C \subset S$ such that $f$ is analytically locally of the form
\[
f(t)=(t^2, t^3).
\]
\end{lemma}
\begin{proof}
Since $q$ is a simple cusp, we can find $x,y \in \hat\calO_{S,q}$ and $u,v \in \hat \calO_{\P^1,p}^*$ such that $f^*(x)=ut^2$ and $f^*(y) = v t^3.$ After rescaling by constants, we can assume $u,v \in 1 + \mathfrak{m}_p.$ So,
\[
f^*(x) = t^2 + \sum_{i \geq 3} a_i t^i, \qquad f^*(y) = t^3 + \sum_{j \geq 4} b_j t^j.
\]
If $a_i \neq 0$ for some $i,$ we proceed by induction. Let $k$ be the minimal integer such that $a_k \neq 0.$ If $k$ is even, we change coordinates on $S$ by
\[
x \mapsto x-a_k x^{k/2}, \qquad y \mapsto y.
\]
If $k$ is odd, we change coordinates by
\[
x \mapsto x- a_k x^{\frac{k-3}{2}} y, \qquad y \mapsto y.
\]
After this procedure, the minimal integer $k$ such that $a_k \neq 0$ increases. The infinite composition of these coordinate changes converges formally. Thus, we may assume that $f^*(x) = t^2.$

Similarly, if $b_j \neq 0$ for some $j,$ we proceed by induction. Let $l$ be the minimal integer such that $b_l \neq 0.$ If $l$ is even, we change coordinates on $S$ by
\[
x \mapsto x, \qquad y \mapsto y - b_l x^{l/2}.
\]
If $l$ is odd, we change coordinates by
\[
x \mapsto x, \qquad y \mapsto y - b_l x^{\frac{l-3}{2}}y.
\]
Again, an infinite composition of such coordinate changes converges formally, so we may assume $f^*(y) = t^3.$
\end{proof}

We proceed with $(x,y)$ as in the preceding lemma.
Fix an $a\in F^\times$ and let  $t'=X_0/X_1=1/t$ be the standard coordinate on $U_1:=\P^1\setminus\{0\}$.   Since $\del/\del t=-t^{\prime2}\del/\del t'$, the (rational) vector field $(1/t)\cdot \del/\del t$ extends to a global section of $T_{\P^1}(p )$. Let
\[
v_a:=(a/t)\del/\del t\in H^0(\P^1,  T_{\P^1}(p)).
\]
Since $d\tilde{f}$ has a zero at $p=0$, we have $d\tilde{f}(v_a)$ in $H^0(\P^1, \tilde{f}^*TS)$. Thus  $d\tilde{f}(v_a)$ induces a tangent vector to $\bar\sM_{0}(S, D)^\good$ corresponding to its image in $H^0(\P^1, \sN_{\tilde{f}})$ along with a first order deformation $\tilde{f}_{1\epsilon}$ of $\tilde{f}$. Let $f_{1 \epsilon}$ denote the first order deformation of $f$ given by projecting $\tilde{f}_{1\epsilon}$ along $S_F \to S$.

\begin{lemma}\label{lem:CuspRam2}  Suppose Basic Assumptions~\ref{BA} hold for $k, S, D$. Let $f$ be a geometric point of $D_\cusp\cap \bar{M}_{0,n}(S, D)^\good$ with field of definition $F$. 
 Then \begin{enumerate}
 \item \label{normal_to_Dcusp_def} The first order deformation $f_{1\epsilon}$ defined above extends to a  deformation $f_\epsilon: \P^1 \otimes F[[\epsilon]] \to S_F$ such that there are local analytic coordinates $(x,y)$ on $S$ such that near $p$ the map $f_\epsilon$ is of the form
\[
f_\epsilon(t)=(t^2, t^3)+\epsilon\cdot (2a, 3at)\mod \epsilon^2.
\]
\item \label{marked_pts_dfrm}There is a deformation of marked points $ (p_{1\epsilon},\ldots,p_{p,n\epsilon})$ such that the tangent vector $\mathcal{V}_a$ corresponding to the deformation $(f_{\epsilon}, p_{1\epsilon},\ldots,p_{n\epsilon})$ is in $\ker d(\ev)$.
\item \label{ev_ramified_over_2} Let $\frak{v}: \Spec( F[[\epsilon]]) \to  \bar\sM_{0,n}(S, D)^\good$ be the morphism corresponding to $(f_{\epsilon}, p_{1\epsilon},\ldots,p_{n\epsilon})$. The composition of $\ev \circ \frak{v}$ has ramification index order $2$, i.e. $e_{0}(\ev \circ f_\epsilon) = 2.$
\end{enumerate}
In particular, the map $\ev:\bar\sM_{0,n}(S, D)^\good\to S^n$ has ramification index two along $D_\cusp$, i.e., for $f\in D_\cusp\cap \bar\sM_{0,n}(S, D)^\good$ a geometric generic point, $e_f(\ev)=2$.
\end{lemma}

\begin{proof}
We may assume that $f$ is a closed point, so $\tilde{f}=f$. Let $C,p,q,t,x,y,p_1,\ldots,p_n$ be as in the above discussion.
Let $\sL\subset f^*T_S$ be the kernel of the quotient map $f^*T_S\to \sN_f/\sN_f^\tor$. Then $\sL$ is an invertible subsheaf of $f^*T_S$ containing the image $df(T_{\P^1})$. By the diagram
\[
\xymatrix{
0 \ar[r]& T_{\P^1}\ar[r]\ar[d] & f^*T_S \ar[r]\ar[d] & \sN_f \ar[r]\ar[d] & 0 \\
0 \ar[r] & \sL \ar[r] & f^*T_S \ar[r] & \sN_f/\sN_f^\tor \ar[r] & 0
}
\]
and the snake lemma, we have $\sN_f^\tor \cong \sL/df(T_{\P^1})$. Since $df$ vanishes to first order at $p$ and nowhere else, the map $df:T_{\P^1}\to \sL$ identifies $\sL$ with $$\sL \cong T_{\P^1}(p )\cong \sO_{\P^1}(3).$$
Since $\det(f^*T_S) \cong \sO(d),$ it follows from the exact sequence
\[
0 \to \sL \to f^*T_S \to \sN_f/\sN_f^\tor \to 0,
\]
that $$\sN_f/\sN_f^\tor\cong \sO_{\P^1}(d-3).$$ If $d_S\ge 3$, then $S$ embeds into $\P^N$ by $-K_S$, and one can choose a hyperplane in $\P^N$ passing through the cusp $q$ and one other point of $C = f(\P^1)$. This hyperplane has intersection multiplicity $d\ge 4$ with $D$. So $d\ge4$ for any $(d_S, d)$ allowed by the hypothesis. Thus $\sN_f/\sN_f^\tor(-1)$ is generated by global sections and $H^1(\P^1, \sN_f(-1))=0$; in particular, $f$ is free.

For a morphism $\phi:Y\to X$ of smooth varieties over $k$ and a geometric point $y\in Y$ with image $x=\phi(y)$, we have the differential
\[
d\phi_y:T_{Y,y}\to T_{X,x}.
\]
The second order differential is the map
\begin{equation}\label{second_order_differential_domain_range}
d^2\phi_y:\Sym^2_{k(y)}(T_{Y,y} \otimes k(y))\to \coker(d\phi_y \otimes k(y)),
\end{equation}
defined as follows.
The map $\phi^*:\sO_{X,x}\to \sO_{Y,y}$ induces the map on jet spaces
\[
\sJ^2\phi^*:\sJ^2\sO_{X,x}=\sO_{X,x}/\mathfrak{m}_x^3\to \sJ^2\sO_{Y,y}=\sO_{Y,y}/\mathfrak{m}_y^3,
\]
 and thus $\sJ^2\phi^*$ induces a map of the kernel of $d\phi^*:\Omega^1_{X,x}\otimes k(x)\to \Omega^1_{Y,y}\otimes k(y)$ to the subspace
\[
\Sym^2\Omega^1_{Y,y}\otimes k(y) \cong \Sym^2(\mathfrak{m}_y^1/\mathfrak{m}_y^2) \cong \mathfrak{m}_y^2/\mathfrak{m}_y^3 \subset  \sJ^2\sO_{Y,y}.
\]
The map $d^2\phi_y$ is the dual of this map.

We show that $d^2 \ev_f$ is nonvanishing. Let $q_i=f(p_i)\in S(F)$. The sheaf sequence
\begin{equation}\label{eqn:SheafSeq}
0\to \sN_f(-\sum_ip_i)\to \sN_f\to \oplus f^*T_{S, q_i}/df(T_{\P^1, p_i})\to 0
\end{equation}
identifies the cokernel of $d\ev_f:T_f(\bar\sM_{0,n}(S, D))\to T_{\ev(f)}(S^n)$ with $H^1(\P^1, \sN_f(-\sum_ip_i))$. See Remark \ref{rem:mod_int_dev_kernel_cokernel_withNfSES}. We will show $d^2 \ev_f \neq 0$, by showing its restriction to $\ker d\ev_f$ is nonvanishing. Since $ \sN_f(-\sum_ip_i)\cong \sN_f^\tor\oplus \sO_{\P^1}(-2)$, the sequence \eqref{eqn:SheafSeq} similarly gives an identification
\begin{equation} \label{eqn:IdentifyTorsion}
	\ker d\ev_f \cong H^{0}(\sN_f^\tor) \cong H^{0}(i_{p*}F).
\end{equation}
 We are interested in showing that the map
\[
d^2\ev_f:\Sym^2 H^0(\sN_f^\tor) \cong F \to H^1(\P^1, \sN_f(-\sum_ip_i))\cong F.
\]
is nonzero. We will consider $d^2\ev_f$ as a quadratic form on $F$, sending $a\in F$ to $d^2\ev_f(\mathcal{V}_a^2)$ where $\mathcal{V}_a$ is the tangent vector $\mathcal{V}_a\in T_f(\bar\sM_{0,n}(S, D)^\good)$ corresponding to $i_{p*}(a)$.

Noting that $H^1(\P^1, \sN_f(-\sum_ip_i))\cong H^1(\P^1, \sN_f/\sN_f^\tor(-\sum_ip_i))$, we will compute $d^2\ev_f$ by composing with the Serre duality pairing
\[
H^1(\P^1, \sN_f/\sN_f^\tor(-\sum_{i=1}^np_i))\times H^0(\P^1, K_{\P^1}\otimes (\sN_f/\sN_f^\tor)^\vee(\sum_{i=1}^np_i))\to F.
\]
In fact, since $\sN_f/\sN_f^\tor\cong \sO_{\P^1}(d-3)=\sO_{\P^1}(n-2)$, we have
\[
K_{\P^1}\otimes (\sN_f/\sN_f^\tor)^\vee(\sum_{i=1}^np_i)\cong \sO_{\P^1}
\]
so there is a unique (up to a non-zero scalar) section $\omega\in H^0(\P^1, K_{\P^1}\otimes (\sN_f/\sN_f^\tor)^\vee(\sum_{i=1}^np_i))$. We will show that $\<\omega,d^2\ev_f(\mathcal{V}^2_a)\> \neq 0$ for nonzero $a$.

We construct a deformation  $(f_\epsilon, p_{1\epsilon},\ldots, p_{n\epsilon})$ of $(f,p_1,\ldots, p_n)$ over $F[[\epsilon]]$ with first order deformation corresponding to $\mathcal{V}_a$ as follows. The invertible sheaf $\sL$ has local generator $\lambda:=2\cdot \del/\del x+3t\cdot \del/\del y$ near $p$, with $t\cdot \lambda=df(\del/\del t)$. Recall above we set
\[
v_a:=(a/t)\del/\del t\in H^0(\P^1,  T_{\P^1}(p ))
\] for $a\in F^\times$,
so $df(v_a)=a\cdot \lambda$ in $\sL\otimes\hat{\sO}_{\P^1,0}$. Let $s_a:=df(v_a)\in H^0(\P^1, \sL)$. In particular, $s_a$ in $H^0(\P^1,\sL) \subseteq H^0(\P^1,\sN_f)$ is in $H^0(\P^1,\sN_f^\tor) \subseteq H^0(\P^1,\sN_f).$
Fixing the isomorphism $i_{p*}F \cong \sN^{\tor}_{f}$ to be given by $a \mapsto a [\lambda],$ we see that
\begin{equation}\label{eq:ip*a=sa}
i_{p*}a = [s_a] \in H^0(\sN^{\tor}_{f}).
\end{equation}
Thus $\mathcal{V}_a$ corresponds to the first order deformation $f_{1\epsilon}$ of $f$ corresponding to the image of $s_a$ in $H^0(\P^1, \sN_f)$ equipped with appropriate marked points.
Since  $H^1(\P^1, \sN_f)=\{0\}$,  the 1st order deformation $f_{1\epsilon}$ extends to a deformation $f_\epsilon$ over $F[[\epsilon]]$. Locally in the coordinate system $t, (x,y)$, the map $f_\epsilon$ is of the form
\[
f_\epsilon(t)=(t^2, t^3)+\epsilon\cdot (2a, 3at)\mod \epsilon^2.
\] This shows \ref{normal_to_Dcusp_def}.

We choose points of $ \P^1(F[[\epsilon]])$ deforming the $p_i$ such that $(f_\epsilon, p_{1\epsilon},\ldots, p_{n\epsilon})$ determines an element of $\ker d(\ev)$ as follows. The global vector field  $v_a$ on $U_1$ gives us the automorphism $\phi_\epsilon$ of $U_1\times\Spec F[[\epsilon]]$
\[
\phi_\epsilon(t')=t'-\epsilon\cdot a t^{\prime 3} .
\]
In coordinates $(t,\epsilon)$, this is
\[
\phi_\epsilon(t)=\frac{1}{1/t-\epsilon\cdot at^{-3}} = t+\epsilon\cdot \frac{a}{t}+\epsilon^2\cdot\frac{a^2}{t^3}\mod \epsilon^3
\]  Let $p_{i\epsilon}= \phi_{-\epsilon}(p_i)\in \P^1(F[[\epsilon]])$. Although the $t, (x,y)$ coordinate system is not necessarily valid near $p_i$, we have
\[
f_\epsilon(p_{i\epsilon})\equiv f(p_i)\mod  \epsilon^2
\]
because \begin{align*}
\frac{\del}{\del\epsilon} f_\epsilon(p_{i\epsilon})\vert_{\epsilon = 0}&=  \frac{\del}{\del\epsilon} f_\epsilon\vert_{\epsilon=0}(p_i) + df(\frac{\del}{\del\epsilon} p_{i\epsilon}\vert_{\epsilon=0})\\
& = s_a(p_i) + df(\frac{\del}{\del\epsilon}\phi_{-\epsilon}(p_i)\vert_{\epsilon=0})\\
& =  s_a(p_i) + df(-v_{a}(p_i))\\
& = s_a(p_i) - s_a(p_i) = 0.
\end{align*}

Thus the $F[[\epsilon]]$ point $\frak{v} \colon \Spec( F[[\epsilon]]) \to \bar{M}_{0,n}(S, D)^\good$ given by $(f_\epsilon, p_{1\epsilon},\ldots, p_{n\epsilon})$ determines a tangent vector which is in $\ker d(\ev).$ It follows from~\eqref{eqn:IdentifyTorsion} that $\ker d(\ev)$ is $1$-dimensional. Therefore, by~\eqref{eq:ip*a=sa} the tangent vector corresponding to $\mathfrak{v}$
equals $\mathcal{V}_a\in T_f(\bar\sM_{0,n}(S, D)^\good).$ This shows \eqref{marked_pts_dfrm}.

We now give a cocycle representing $d^2\ev_f(\mathcal{V}^2_a)$ in $H^1(\P^1,\sN_f/\sN_f^\tor(\sum_i(-p_i)))$ in terms of a morphism $h:U_1[[\epsilon]]\to S$, defined by
\[
h(t', \epsilon):=f_\epsilon\circ \phi_{-\epsilon}(t').
\]

Note that $dh(t',0)(\del/\del\epsilon)=0$ by a chain rule calculation similar to the above, whence the canonical map $\coker (df_{t'}) \to \coker (dh_{(t',0)})$ is an isomorphism. It follows from \eqref{second_order_differential_domain_range} that $d^2h((\del/\del\epsilon|_{\epsilon=0})^2)$ is thus a section of $H^0(U_1, \sN_f)$.\footnote{Another point of view on this is that for any vector field $v$ on $U_1[[\epsilon]]$ extending $\del/\del \epsilon \vert_{\epsilon =0}$, including the examples $v=\del/\del \epsilon$ and $v=\del/\del \epsilon+\epsilon \del/\del t'$, one obtains a section in $H^0(U_1, h^*TS)$ because the section $dh(v) \in H^0(U_1[[\epsilon]], h^*TS)$ vanishes along $\epsilon =0$ and a section of a vector bundle has a first derivative which lies in the same vector bundle restricted to the vanishing locus of the section by taking the derivative in some local trivialization. The image of this section in $H^0(U_1, \sN_f)$ is independent of the choice of $v$.} We let $\del^2h/\del\epsilon^2|_{\epsilon=0}\in H^0(U_1,\sN_f)$ denote this section $d^2h((\del/\del\epsilon|_{\epsilon=0})^2)$.

Let $t_i'$ be the $t'$ coordinate of $p_i$. We compute
\begin{equation}\label{eqn:2ndDeriv1}
\del^2h(t', \epsilon)/\del\epsilon^2|_{\epsilon=0, t'=t_i'}=\frac{d^2f_\epsilon(p_{i\epsilon})}{d\epsilon^2}|_{\epsilon=0}
\end{equation}
in $f^*T_{S, q_i}/T_{\P^1, p_i}=N_f\otimes F(p_i)$.

On the other hand,
\begin{equation}\label{eqn:2ndDeriv0}
d^2\ev_f(\mathcal{V}_a^2)=\del (\ldots, \frac{d^2f_\epsilon(p_{i\epsilon})}{d\epsilon^2}|_{\epsilon=0},\ldots),
\end{equation}
where
\[
\del:\oplus_{i=1}^nf^*T_{S, q_i}/df(T_{\P^1, p_i})\to H^1(\P^1, \sN_f(-\sum_{i=1}^np_i))
\]
is the boundary map in the cohomology sequence associated to \eqref{eqn:SheafSeq}.

Let $\sU$ be the cover of $\P^1$ given by $U_1=\P^1\setminus\{0\}$, $U_0=\P^1\setminus\{p_1,\ldots, p_n\}$. This is indeed a cover because the $p_i$ do not coincide with the point $p = 0$ where $f$ is ramified by Theorem~\ref{prop:Good}~\ref{it:goodnotodp}~\ref{it:irr}. By \eqref{eqn:2ndDeriv0}, \eqref{eqn:2ndDeriv1},
 we can represent $d^2\ev_f(\mathcal{V}_a)$ as the 1-cocycle in $C^1(\sU, \sN_f/\sN_f^\tor(\sum_i(-p_i)))$ given by $$[\del^2h/\del\epsilon^2|_{\epsilon=0}]\in H^0(U_0\cap U_1, \sN_f(\sum_i(-p_i)).$$

Moreover, the trace map $H^1(\P^1, \omega_{\P^1/F})\stackrel{\cong}{\to} F$ sends a class  $[\eta]\in H^1(\P^1, \omega_{\P^1/F})$ represented by some $\eta\in C^1(\sU,\omega_{\P^1/F})=H^0(U_0\cap U_1, \omega_{\P^1/F})$ to the residue $\Res_0\eta$.

For $\omega\in H^0(\P^1, \omega_{\P^1/F}\otimes(\sN_f/\sN_f^\tor)^\vee(\sum_ip_i))$, the pairing  $\<\omega, d^2\ev_f(\mathcal{V}_a)\>$ is therefore given by
\[
\<\omega,d^2\ev_f(\mathcal{V}^2_a)\>= \Res_0\del^2h/\del\epsilon^2|_{\epsilon=0}\cdot \omega
\]
where $\del^2h/\del\epsilon^2|_{\epsilon=0}\cdot \omega$ is to be considered as a section of $\omega_{\P^1/F}$ over $U_0\cap U_1$ via the pairing
\[
\omega_{\P^1/F}\otimes(\sN_f/\sN_f^\tor)^\vee(\sum_ip_i)\otimes
\sN_f/\sN_f^\tor(-\sum_ip_i)\to \omega_{\P^1/F}.
\]

 To make the computation of $\Res_0$, we use a trivialization of $\sN_f/\sN_f^\tor$ in a neighborhood of 0 given as follows:  Since $\sL$ has local generator $\lambda=2\del/\del x+3t\del/\del y$, and
$\sN_f/\sN_f^\tor=f^*T_{\P^2}/\sL$, sending a section $\alpha\del/\del x+\beta\del/\del y$ of $f^*T_{\P^2}$ to $2\beta-3\alpha t$ descends to give an isomorphism of $\sN_f/\sN_f^\tor(-\sum_ip_i)$ with $\sO_{\P^1}$ over $\P^1\setminus\{p_1,\ldots, p_n, \infty\}$. Define $\gamma\in \sO_{\P^1,0}^\times$ so that $\omega$ will transform to a 1-form $\gamma(t)dt$ via this isomorphism.

Combining the previous local expressions for $f_\epsilon(t)$ and $\phi_\epsilon(t)$, we have
\[
h(t, \epsilon)=(t^2, t^3)+\epsilon^2(\frac{3a^2}{t^2}, \frac{3a^2}{t})\mod \epsilon^3
\] in local coordinates $t, (x, y), \epsilon$.
Thus $\frac{\del^2h}{\del\epsilon^2}|_{\epsilon=0}$ maps to $\frac{-6a^2}{t}$ under this local trivialization of $\sN_f/\sN_f^\tor(-\sum_ip_i)$. Thus
\[
\frac{\del^2h}{\del\epsilon^2}|_{\epsilon=0}\cdot\omega=\frac{-6a^2}{t}\cdot \gamma(t)dt,
\]
which yields
\[
\Res_0 \frac{\del^2h}{\del\epsilon^2}|_{\epsilon=0}\cdot\omega=-6\gamma(0)\cdot a^2.
\]
Thus, the quadratic form $d^2\ev_f$ is nonzero as claimed and hence the ramification index is two.
\end{proof}

We now combine our results to compute the relative canonical bundle $\omega_{\ev}$ of $\ev: \bar{M}_{0,n}(S,D)^\good \to S^n \setminus A$ to be $\sO_{\bar{M}_{0,n}(S,D)^\good}(D_\cusp)$. Here, $\ev:\bar{M}_{0,n}(S,D)^\good\to S^n\setminus A$ is as given by Theorem~\ref{prop:Good}. Recall that the relative cotangent sheaf is defined $\Omega_\ev:=\Hom( \ev^*(\Omega_{S^n/k}),\Omega_{\bar{M}_{0,n}(S,D)^\good/k})$ and the relative canonical bundle $\omega_{\ev}$ is given $\omega_{\ev} \cong \omega_{\bar{M}_{0,n}(S,D)^\good/k}\otimes \ev^*(\omega_{S^n/k})^{-1}$.

\begin{theorem} \label{prop:EvRam} Let $k$ be a field of characteristic $0$. Suppose Basic Assumptions \ref{BA} hold for $k, S, D$. We suppose that $\bar{M}_{0,n}(S,D)^\good$ is non-empty.
\begin{enumerate}
\item $\ev$ is ramified along $D_\cusp$ with ramification index two: at each geometric generic point $f$ of $D_\cusp$, there are local analytic coordinates $t_1,\ldots, t_{2n}$ for $\bar{M}_{0,n}(S,D)^\good$ at $f$ and $s_1,\ldots, s_{2n}$ for $S^n$ at $\ev(f)$ such that $D_\cusp$ has local defining equation $t_1$ and $\ev^*(s_1)=t_1^2$, $\ev^*(s_i)=t_i$ for $i=2,\ldots, 2n$. \\
\item The determinant of the section $d(\ev):\sO_{\bar{M}_{0,n}(S,D)^\good}\to \Omega_\ev$ of the relative cotangent bundle has divisor $1\cdot D_\cusp$, and thus determines an isomorphism $$\det d(\ev):\sO_{\bar{M}_{0,n}(S,D)^\good}(D_\cusp)\to \omega_\ev.$$ 
\end{enumerate}
\end{theorem}

\begin{proof} Noting that $\bar{M}_{0,n}(S,D)^\good$ and  $S^n$ are both smooth $k$-schemes, it follows that $\omega_\ev$ is an invertible sheaf. Moreover, $\ev$ is flat and is unramified over $\bar{M}_{0,n}(S,D)^\good\setminus D_\cusp$, so we need only show that the ramification index of $\ev$ along $D_\cusp$ is two, which is  Lemma~\ref{lem:CuspRam2}.
 (2) is an immediate consequence of (1) and the theorem on purity of the branch locus.
\end{proof}

\section{The double point locus}\label{Section:dpl}
We define the double point locus using ideas from \cite[Chapter~9.3]{fulton98}. 
\begin{definition}\label{definition:double_point_locus}
	Given a composition of closed immersions $Z \subset W \subset X$, we define the {\em subscheme of $W$ residual to $Z$} to be the subscheme defined by the ideal sheaf $(I_{W} : I_{Z})$.  Recall that this is the ideal sheaf of all local sections $s$ of $\calO_{X}$ such that $s t$ lies in $I_{W}$ for all local sections $t$ of $I_{Z}$.
\end{definition}

Let $S$ be a smooth del Pezzo surface over a perfect field $k$ equipped with an effective Cartier divisor satisfying Assumption~\ref{BA}~\ref{it:exlude_multiple_covers_-1curves}~\ref{it:a3degree}. We work mostly in characteristic $0$, but also have some analysis in characteristic $p$. Throughout this section, we assume that $\bar{M}_{0,n}(S,D)^\odp$ is non-empty, which implies that $M^\bir_0(S,D)$ is non-empty. 

\begin{remark}\label{rmk:ModpMbirMgood_nonempty}
In characteristic $0$, $\bar{M}_{0,n}(S,D)^\odp$ is non-empty if and only if $M^\bir_0(S,D)$ is non-empty, and these are both equivalent to $\bar{M}_{0,n}(S,D)^\good$ being non-empty by Theorem~\ref{prop:Good}\ref{it:evfinflatdom}\ref{it:goododp}. In characteristic $p$ under Assumption~\ref{a:genericunram}, we also have that $\bar{M}_{0,n}(S,D)^\odp$ is non-empty if and only if $M^\bir_0(S,D)$ is non-empty by Proposition~\ref{prop:GenODP}. So we could equally well assume that $M^\bir_0(S,D)$ is non-empty. When $M^\bir_0(S,D)$ is empty but $(k,S,D)$ satisfies Assumption~\ref{BA}~\ref{it:exlude_multiple_covers_-1curves}~\ref{it:a3degree} and Assumption~\ref{a:genericunram}, the associated Gromov-Witten invariants are defined to be $0$. Note the consistency with Lemma~\ref{lm:biratempty} and Corollary~\ref{Cor:codim_ev(nonodp)>=1}; over a dense open of $S^n$, $\ev$ has empty fiber in this case, so it makes sense to define the degree and the count to be $0$.

\end{remark}

Let $d=-\Deg K_S\cdot D\ge1$. Let $n=d-1$. Let $\uc{n} \to \bar{M}_{0,n}(S,D)$ denote the universal curve, $\uc{n} := \bar{M}_{0,n+1}(S,D)$ and let $\ev: \uc{n} \to S \times\bar{M}_{0,n}(S,D) $ denote the universal map, or in other words, the product of evaluation on the $(n+1)$st marked point with the canonical projection. If the characteristic of $k$ is $0$, we may apply Theorem~\ref{prop:Good} and obtain the smooth $k$-scheme $\M_{0,n}(S,D)^\good$. Let $\uc{n}^\good \to \M_{0,n}(S,D)^\good$ denote the pullback of $\uc{n} \to \bar{M}_{0,n}(S,D)$ to $\M_{0,n}(S,D)^\good$. In particular, $\uc{n}^\good$ is a scheme. In positive characteristic, the map $\uc{n}^\good \to \M_{0,n}(S,D)^\good$ will be replaced by the pullback $\uc{n}^{\odp} \to \M_{0,n}(S,D)^{\odp}$ of $\uc{n} \to \bar{M}_{0,n}(S,D)$ to the locus $\M_{0,n}(S,D)^{\odp}$ of parametrized curves with only ordinary double points .

Work in schemes over the good locus of the moduli stack $\M_{0, n}(S, d)^\good$. So for example, we have schemes $S \times \M_{0, n}(S, d)^\good$ and 
\[
\Delta_S \hookrightarrow (S \times \M_{0, n}(S, d)^\good) \times_{\M_{0, n}(S, d)^\good} (S \times \M_{0, n}(S, d)^\good) \cong S \times S \times \M_{0, n}(S, d)^\good
\] in our category.
Consider the product of the universal maps  
\[
\ev \times \ev \colon \uc{n}^\good \times_{\M_{0,n}(S,d)^\good} \uc{n}^\good \to S \times S \times \M_{0, n}(S, d)^\good.
\] The preimage $(\ev \times \ev)^{-1}(\Delta_S)$ of the diagonal $\Delta_S \subset  S \times S \times \M_{0, n}(S, d)^\good$ contains the diagonal $\Delta_{\uc{n}^\good} \subset \uc{n}^\good \times_{\M_{0,n}(S,d)^\good} \uc{n}^\good$ as one irreducible component.

\begin{definition}\label{def:dpl}
Under the hypotheses of Theorem~\ref{prop:Good}, the {\em double point locus} is the subscheme
\[
\dpl \subset  \uc{n}^\good \times_{\M_{0, n}(S,D)^\good} \uc{n}^\good
\]
defined to be the subscheme of $(\ev \times \ev)^{-1}(\Delta_S)$ residual to $\Delta_{\uc{n}^\good}$. Let
\[
\pi : \dpl \to \M_{0, n}(S,D)^\good
\]
denote the canonical map.

Now drop the assumption that the characteristic of $k$ is $0$. For $S$ a smooth del Pezzo surface over $k$ and $D$ an effective Cartier divisor, define
\[
\dplodp \subset  \uc{n}^\odp \times_{\M_{0, n}(S,D)^\odp} \uc{n}^\odp
\] to be the subscheme of $(\ev \times \ev)^{-1}(\Delta_S)$ residual to $\Delta_{\uc{n}^\odp}$ and let $\pi$ denote the projection map $ \pi: \dplodp \to \M_{0, n}(S,D)^\odp $.
\end{definition}

\begin{lemma}\label{lm:dplodp}
Let $k$ be a perfect field. Let $S$ be a smooth del Pezzo surface over $k$ equipped with an effective Cartier divisor $D$. The projection from the double point locus $\pi: \dplodp \to \M_{0,n}^\odp(S,D)$ is \'etale.
\end{lemma}
\begin{proof}
Over $\M_{0,n}^\odp(S,D)$ the product of universal evaluation maps
\[
\ev \times \ev \colon \uc{n}^\odp \times_{\M_{0,n}(S,D)^\odp} \uc{n}^\odp \setminus \Delta_{\uc{n}^\odp}  \to S \times S \times \M_{0, n}(S, D)^\odp
\]
is transverse to $\Delta_S$ over $\M_{0, n}(S, D)^\odp.$ So, the morphism
\[
\dplodp \setminus  \Delta_{\uc{n}^\odp}  = (\ev \times \ev)^{-1}(\Delta_S)\setminus  \Delta_{\uc{n}^\odp} \overset{\pi}{\longrightarrow} \M_{0, n}(S, d)^\odp
\]
is smooth of relative dimension zero and thus \'etale. A straightforward argument based on Remark~\ref{rem:idealdplpb} shows that
\[
\dplodp \cap  \Delta_{\uc{n}^\odp} = \emptyset.
\] Alternatively, this follows from Lemma~\ref{lemma_unramified_doublelocus}. 
\end{proof}

For the remainder of this section, we assume $k$ has characteristic zero. Note that we have $$\dplodp = \pi^{-1}(\M_{0,n}^\odp(S,D)) \subset \dpl.$$ We will need the following special loci in $\dpl.$
\begin{definition}\label{df:dplodp}
Let $\dplcusp \subset \Delta_{\uc{n}^\good}$ denote the locus with geometric points given by a map $f : \P^1 \to S$ and a point $p \in \P^1$ where $f$ has a simple cusp, together with marked points $(p_1,\ldots,p_n)$ on $\P^1$ such that $(f, p_1,\ldots p_n)$ is in $\M_{0, n}(S,D)^\good$. Let $\dpltac \subset \dpl$ denote the locus with geometric points given by $(f : \P^1 \to S, p_1, \ldots, p_n)$ a geometric point of $\M_{0, n}(S,D)^\good$ and a pair of points $p,q \in \P^1$ where $f$ has a simple tacnode. Let $\dpltrip \subset \dpl$ denote the locus with geometric points given by a geometric point $(f : \P^1 \to S, p_1, \ldots, p_n)$ of $\M_{0, n}(S,D)^\good$ and a pair of points $p,q \in \P^1$ where $f$ has a triple point.
\end{definition}
\begin{remark}\label{rem:idealdplpb}
Let $\fc : B \to \bM{n}^\good$ be a family of stable maps corresponding to a curve $\P \to B$ and a map $f : \P \to S.$ Let
\[
\tilde \fc : \P\times_B \P \to \uc{n}^\good \times_{\M_{0, n}(S,D)^\good} \uc{n}^\good
\]
be the induced map. Let $\I_\dpl$ denote the ideal sheaf of $\dpl.$ Let $(p_1,p_2)$ be a point of $\P \times_B \P$ such that $f(p_1) = f(p_2).$ Let $t_1,t_2,$ be local coordinates on $\P$ at $p_1,p_2,$ respectively. So, locally the ideal sheaf of the diagonal $\Delta_\P \subset \P \times_B \P$ is generated by $t_1 - t_2.$ Let $s = (s_1,s_2)$ be local coordinates at $f(p_1) = f(p_2).$ Then, since $t_1 - t_2$ is not a zero divisor, the pull-back of the colon ideal sheaf $\I_{\dpl}$ is given by
\[
\tilde\fc^*\I_\dpl = \left(\frac{s\circ f(t_1) - s\circ f(t_2)}{t_1 - t_2}\right).
\]
\end{remark}
\begin{lemma}\label{lm:dcuspdpl}
We have $\dplcusp \subset \dpl$.
\end{lemma}
\begin{proof}
Let $f : \P^1 \to S$ be a map with a single simple cusp at $p \in \P^1.$ So $(f,p,p) \in \Delta_{\uc{n}^\good}$ is an $F$ point of $\dplcusp.$ Let $\fc : \Spec(F) \to \bM{n}^\good$ be the corresponding map, so we get an induced map
\[
\tilde \fc : \P^1 \times \P^1 \to \uc{n}^\good \times_{\M_{0, n}(S,D)^\good} \uc{n}^\good.
\]
Choose local coordinates on $\P^1$ at $p$ and on $S$ at $f(p)$ as in Lemma~\ref{lm:cuspcoord}. Let $t_1,t_2,$ be copies of the local coordinate on $\P.$ In particular, $t_i$ vanishes at $p.$ By Remark~\ref{rem:idealdplpb}, we have locally at $p,$
\begin{equation}\label{eq:if}
\tilde \fc^*\I_\dpl = \left(\frac{(t_1^2,t_1^3) - (t_2^2,t_2^3)}{t_1 - t_2}\right) = (t_1 + t_2, t_1^2 + t_1 t_2 + t_2^2) = (t_1+ t_2, t_1^2),
\end{equation}
and it is clear these equations vanish at $t_1 = t_2 = 0.$
\end{proof}

\begin{lemma}\label{lm:dpl*closed}
The loci $\dplcusp,\dpltac,\dpltrip \subset \dpl$ are closed.
\end{lemma}
\begin{proof}
This follows from Theorem~\ref{prop:Good}\ref{it:goodnotodp}\ref{it:irr}.
\end{proof}

In light of Lemma~\ref{lm:dpl*closed}, we equip $\dplcusp,\dpltac,$ and $\dpltrip$ with the reduced induced subscheme structure.

\begin{lemma}\label{lm:dpltrip}
The projection from the double point locus $\pi: \dpl \to \M_{0, n}(S, d)^\good$ is \'etale over a neighborhood of $D_\trip.$
\end{lemma}
\begin{proof}
The proof is the same as that of Lemma~\ref{lm:dplodp}.
\end{proof}

To prove the double point locus~$\dpl$ is smooth at points of $\dplcusp$ and $\dpltac,$ we introduce the following lemma.

\begin{lemma}\label{lm:smtest} Let $k$ be a field, and let $X$ and $Y$ be smooth, integral, finite type $k$-schemes. Let $Z\subset Y$ be a closed subscheme and take $z\in Z$. Suppose that there is a morphism $f:X\to Y$ with $z\in f(X)$, and an integer $\ell$ such that
\begin{itemize}
\item there is an irreducible component $Z_0$ of $Z$ containing $z$ and of codimension $\leq \ell$ on $Y$,
\item The closed subscheme $X\times_YZ$ of $X$ is smooth of pure codimension $\ell$ on $X$.

\end{itemize}
Then $Z$ is smooth of pure codimension $\ell$ in a neighborhood of $z$.
\end{lemma}

\begin{proof}
We may assume $k$ is algebraically closed and $z$ is a $k$-point of $Z$. Since $X$ is smooth, $Z/k$ is smooth in a neighborhood of $z$ if and only if $X\times_k Z$ is smooth over $X$ in a neighborhood of $X \times_k z$. Let $\Gamma_f\subset X\times_kY$ be the graph of $f$. Note that $X\times_YZ$ is isomorphic to the intersection $\Gamma_f \cap X\times_kZ$. So changing notation, we may assume that $f:X\to Y$ is a closed immersion, that is, we may assume that $X$ is a smooth closed subscheme of $Y$.

Since the assertion is local on $Y$ for the \'etale topology, we may assume that $Y$ is an principal open subscheme of $\A^n_k$. Since $X$ is smooth, we may assume that $X$ is a smooth complete intersection in $Y$, with ideal $I_X=(g_1,\ldots, g_m)$, where $m$ is the codimension of $X$ in $Y$.

Let $i:X\cap Z\to Z$ be the inclusion. We have the exact sequence of $\sO_{X\cap Z,z}$-modules
\[
I_{X,z}/I_{X,z}^2\otimes_{\sO_{X,z}}\sO_{X\cap Z,z} \xrightarrow{d} i^*\Omega_{Z/k, z}\to \Omega_{X\cap Z, z}\to 0
\]
Since $X\cap Z$ is smooth of dimension $n-m-\ell$ at $z$, $\Omega_{X\cap Z, z}\cong \sO_{X\cap Z, z}^{n-m-\ell}$. Thus the sequence splits and
\[
i^*\Omega_{Z/k, z}\cong \Omega_{X\cap Z, z}\oplus \im(d).
\]
Since the images of $g_1,\ldots, g_m$ in $d(I_{X,z}/I_{X,z}^2)$ generate $\im(d)$, we have a surjection $\sO_{X\cap Z,z}^m\to \im(d)$. Applying $-\otimes_{ \sO_{X\cap Z, z}}k(z)$, we see that
\begin{align*}
\dim_{k(z)}\Omega_{Z/k, z}\otimes_{\sO_{Z,z}}k(z)= \dim_{k(z)}i^*\Omega_{Z/k, z}\otimes_{\sO_{X\cap Z,z}}k(z)= \\
\dim_{k(z)} \Omega_{X\cap Z, z} \otimes_{\sO_{X\cap Z,z}}k(z)+ \dim_{k(z)} \im(d) \otimes_{\sO_{X\cap Z,z}}k(z)\\
\le (n-m-\ell) + m=n-\ell.
\end{align*}

Choose generators $f_1,\ldots, f_s$ for $I_{Z, z}\subset \sO_{Y,z}$ and let $x_1,\ldots, x_n$ be the standard coordinates on $\A^n$. Then
\[
 \dim_{k(z)}\Omega_{Z/k, z}\otimes_{\sO_{Z,z}}k(z)=n-\text{rank}\begin{pmatrix} \del f_i/\del x_j\end{pmatrix}(z)
\]
Since
\[
n-\ell\ge  \dim_{k(z)}\Omega_{Z/k, z}\otimes_{\sO_{Z,z}}k(z)
\] by the previous, it follows that
\[
\text{rank}\begin{pmatrix} \del f_i/\del x_j\end{pmatrix}(z)\ge \ell
\]
After reordering the $f_i$, we may assume that the matrix
\[
\begin{pmatrix} \del f_i/\del x_j\end{pmatrix}(z)_{1\le i\le \ell}
\]
has rank $\ell$, which implies that (after shrinking $Y$ if necessary) the closed subscheme $Z'\subset Y$ defined by $f_1,\ldots, f_\ell$ is smooth of codimension $\ell$; shrinking $Y$ again if necessary, we may assume that $Z'$ is integral. But then $Z_0$ is a closed subscheme of $Z'$ of the same dimension, so $Z_0=Z'$ and $Z_0\subset Z\subset Z'$, so $Z=Z'$ and $Z$ is smooth of codimension $\ell$ on $Y$.
\end{proof}

\begin{lemma}\label{lm:dplcusp_smooth_ram1_1:1}
The double point locus $\dpl$ is smooth of dimension $d-1 + n$ at the geometric points of the cuspidal locus $\dplcusp$ and the map $\pi : \dpl \to \M_{0, n}(S,D)^\good$ has ramification index $2$ at $\dplcusp.$ The restriction of $\pi$ to a map $\dplcusp \to D_{\cusp} \cap \M_{0,n}(S,D)^\good$ is birational.
\end{lemma}
\begin{proof}
Let $\tf = (f,p) = ((f,p),(f,p))$ represent a geometric point
\[
\Spec{F} \to \dplcusp \subset \uc{n} \times_{\M_{0,n}(S,D)^\good} \uc{n}.
\]
Let $q = f(p).$ We may assume $p = 0 \in \P^1.$
Let $t$ be the standard coordinate on $\P^1 \setminus \{\infty\}$ and let $(x,y)$ be analytic coordinates on $S$ at $q$ as in Lemma~\ref{lm:cuspcoord}.
Using Lemma~\ref{lem:CuspRam2}\ref{normal_to_Dcusp_def}, choose a family of maps $\f : \P^1_{F[[\epsilon]]} \to S$ such that $\f|_{\{\epsilon  = 0\}} = f$ and near $p,$
\[
\f(\epsilon,t) = a_{10}\epsilon + a_{02}t^2 + a_{11} \epsilon t  + a_{03}t^3   \mod{(\epsilon^2,\epsilon t^2,t^4)}
\]
where $a_{ij} \in F^2$ are given by
\begin{equation}\label{eq:aij}
a_{02} = (1,0), \qquad a_{11} = (0,3a), \qquad a_{03} = (0,1).
\end{equation}
Let
\[
\tilde \fc : \P^1_{F[[\epsilon]]} \times_{\Spec(F[[\epsilon]])} \P^1_{F[[\epsilon]]} \to \uc{n}^\good \times_{\M_{0, n}(S,D)^\good} \uc{n}^\good
\]
denote the induced map.
We have local coordinates on $\P^1_{F[[\epsilon]]} \times_{\Spec(F[[\epsilon]])} \P^1_{F[[\epsilon]]}$ at $(p,p,0)$ given by $\epsilon$ and two copies $t_1,t_2,$ of the parameter $t.$ Let $\I$ be the ideal sheaf of $\tf$ on $\uc{n} \times_{\M_{0,n}(S,D)^\good} \uc{n}.$ So, analytically locally $\tilde \fc^*\I = (t_1,t_2,\epsilon).$
By Remark~\ref{rem:idealdplpb}, the pull-back $\tilde\fc^*\I_\dpl$ is generated analytically locally by
\begin{equation}\label{eq:ups}
\Upsilon = \frac{\f(\epsilon,t_1) - \f(\epsilon,t_2)}{t_1 - t_2} = a_{02}(t_1 + t_2) + a_{11}\epsilon + a_{03}(t_1^2 + t_1 t_2 + t_2^2) \mod{(\tilde \fc^*\I)^3 + (\epsilon^2,t_1\epsilon,t_2\epsilon)}.
\end{equation}
Since $a_{02}$ and $a_{11}$ are linearly independent, it follows that the subscheme determined by the ideal sheaf $\fc^*\I_\dpl$ is smooth of codimension $2$ at $(p,p,0).$ Apply Lemma~\ref{lm:smtest} with $l = 2,$ $z = \tf,$
\[
X = \P^1_{F[[\epsilon]]} \times_{\Spec(F[[\epsilon]])} \P^1_{F[[\epsilon]]}, \qquad Z = \dpl, \qquad Y = \uc{n}^\good \times_{\M_{0, n}(S,D)^\good} \uc{n}^\good,
\]
and $X \times_Y Z$ the subscheme of $X$ determined by $\tilde \fc^*\I_\dpl.$ Since $\dpl$ is given by two equations, we can take $Z_0 = Z.$ It follows that $\dpl$ is smooth at $\tf$ of dimension $d-1 + n.$

Now assume that $\tf$ is a geometric generic point of $\dplcusp.$ Let $\tilde \fc$ be as in the proof of Lemma~\ref{lm:dcuspdpl}. That is, we set the parameter $\epsilon$ to zero in the $\tilde \fc$ above. Let $\mathfrak{q} \subset \sO_{(p,p),\P^1 \times \P^1}$ be the ideal sheaf of $\tilde \fc^{-1}(\pi^{-1}(f)).$ By equation~\eqref{eq:if}, the quotient $\sO_{(p,p),\P^1 \times \P^1}/\mathfrak{q}$ is given by $F[t_1,t_2]/(t_1+t_2,t_1^2),$ which has length two and induces an isomorphism to $\Spec F$ after taking the reduced subscheme. Therefore, the ramification index of $\pi$ at $\tf$ is $2$. By Proposition~\ref{prop:SingCodim1}(2), the map $\pi: \dplcusp \to D_{\cusp} \cap \M_{0,n}(S,D)^\good$ is generically a bijection; since we are in characteristic zero, this implies that  $\pi: \dplcusp \to D_{\cusp} \cap \M_{0,n}(S,D)^\good$ is birational.
\end{proof}

\begin{lemma}\label{lm:tacdef}
Let $f : \P^1 \to S$ represent a point of $D_\tac.$ Let $p \neq p' \in \P^1$ such that $q = f(p) = f(p')$ is the tacnode.
\begin{enumerate}
\item\label{it:tt'}
There exist local coordinates $t,t'$ at $p,p',$ respectively, and local analytic coordinates $x,y,$ on $S$ at $q$ such that near $p,$
\[
(x,y) \circ f = (t,t^2)
\]
and near $p',$
\[
(x,y) \circ f  = (t',0).
\]
\item \label{it:tacdef}
There exists a family $\f : \P^1_{F[[\epsilon]]} \to S$ such that $\f|_{\{\epsilon = 0\}} = f,$ and near $p,$
\[
(x,y) \circ \f(t,\epsilon) = (t,t^2 + \epsilon) \mod \epsilon^2,
\]
and near $p',$
\[
(x,y) \circ \f(t,\epsilon) = (t',0) \mod \epsilon^2.
\]
\end{enumerate}
\end{lemma}

\begin{lemma}\label{lm:dpltac_smooth_ram_2:1}
The double point locus $\dpl$ is smooth of dimension $d-1+n$ at the geometric points of the tacnodal locus $\dpltac,$  the map $\pi : \dpl \to \M_{0, n}(S,D)^\good$ has ramification index $2$ at~$\dpltac,$ and the map $\pi|_{\dpltac}: \dpltac \to D_\tac$ is two to one.
\end{lemma}
\begin{proof}
Let $\tf = ((f,p),(f,p'))$ represent a geometric point
\[
\Spec F \to \dpltac \subset \uc{n} \times_{\M_{0,n}(S,D)^\good} \uc{n}.
\]
Let $q = f(p).$ Let $t,t',$ be local coordinates at $p,p',$ respectively, and let $x,y,$ be analytic coordinates at $q$ as in Lemma~\ref{lm:tacdef}\ref{it:tt'}. Let $\f : \P^1_{F[[\epsilon]]} \to S$ be the family of Lemma~\ref{lm:tacdef}\ref{it:tacdef}. Let
\[
\tilde \fc : \P^1_{F[[\epsilon]]} \times_{\Spec(F[[\epsilon]])} \P^1_{F[[\epsilon]]} \to \uc{n}^\good \times_{\M_{0, n}(S,D)^\good} \uc{n}^\good
\]
denote the induced map. We have local coordinates on $\P^1_{F[[\epsilon]]} \times_{\Spec(F[[\epsilon]])} \P^1_{F[[\epsilon]]}$ at $(p,p',0)$ given by $t,t',\epsilon.$
The pull-back $\tilde \fc^* \I_\dpl$ is generated analytically locally by
\begin{equation}\label{eq:tac}
(x,y)\circ \f(\epsilon,t) - (x,y)\circ \f(\epsilon,t') = (t-t',t^2 + \epsilon) \mod \epsilon^2.
\end{equation}
It follows that the subscheme determined by $\tilde \fc^* \I_\dpl$ is smooth of codimension $2$ at $(p,p',0).$ Apply Lemma~\ref{lm:smtest} with $l = 2,$ $z = \tf,$
\[
X = \P^1_{F[[\epsilon]]} \times_{\Spec(F[[\epsilon]])} \P^1_{F[[\epsilon]]}, \qquad Z = \dpl, \qquad Y = \uc{n}^\good \times_{\M_{0, n}(S,D)^\good} \uc{n}^\good,
\]
and $X \times_Y Z$ the subscheme of $X$ determined by $\tilde \fc^*\I_\dpl.$ Since $\dpl$ is given by two equations, we can take $Z_0 = Z.$ It follows that $\dpl$ is smooth at $\tf$ of dimension $d-1+n.$

Now, assume that $\tf$ is a geometric generic point of $\dpltac.$ Let
\[
\tilde \fc|_0 : \P^1 \to S
\]
be the restriction to $\epsilon = 0.$ Let $\mathfrak{q} \subset \sO_{(p,p'),\P^1 \times \P^1}$ be the ideal sheaf of $(\tilde \fc|_0)^{-1}(\pi^{-1}(f)).$ By equation~\eqref{eq:tac}, the quotient $\sO_{(p,p'),\P^1 \times \P^1}/\mathfrak{q}$ is given by $F[t,t']/(t-t',t^2),$ which has length two. Therefore, the ramification index of $\pi$ at $\tf$ is $2.$

Finally, the map $\pi|_{\dpltac}: \dpltac \to D_\tac$ is two to one because
\[
\pi(f,p,p') = \pi(f,p',p) = f.
\]
\end{proof}

\begin{corollary}\label{Cor:dpl_sm_ram_pi}
The double point locus $\dpl$ is smooth of dimension $d-1+n.$  The map $\pi : \dpl \to \M_{0, n}(S,D)^\good$ is finite, flat and generically \'etale.  The ramification of $\pi$ is supported on $\dplcusp$ and $\dpltac$, where it is simply ramified.
\end{corollary}

\begin{proof}
The smoothness and dimension of $\dpl$ follow from Theorem~\ref{prop:Good}~\ref{it:evfinflatdom} and Lemmas~\ref{lm:dplodp},~\ref{lm:dpltrip},~\ref{lm:dplcusp_smooth_ram1_1:1} and~\ref{lm:dpltac_smooth_ram_2:1}. The map $\pi$ is quasi-finite because the fiber over a point of $\M_{0,n}(S,D)^\good$ represented by a map $f : \P \to S$ is contained in the union of the ramification locus of $f$ and the locus where $f$ is not $1-1.$ Since $\pi$ is proper, it follows that $\pi$ is finite. Since the domain and range of $\pi$ are smooth of the same dimension and $\pi$ is quasi-finite, it follows that $\pi$ is flat. Lemmas~\ref{lm:dplodp} and~\ref{lm:dpltrip} show that $\pi$ is \'etale over $\dplodp$ and $\dpltrip.$ In particular, it is generically \'etale. The ramification over $\dplcusp$ and $\dpltac$ was computed in Lemmas~\ref{lm:dplcusp_smooth_ram1_1:1} and~\ref{lm:dpltac_smooth_ram_2:1}.
\end{proof}

\section{Orienting the evaluation map} \label{Section:orienting_ev}
In this section, we continue to assume that $\M_{0, n}(S,D)^\odp$ is non-empty.

 Let $A$ be a Noetherian ring. Let $f:Y\to Z$ be a finite flat morphism of smooth $A$-schemes. It follows that $f_*\sO_Y$ is a locally free $\sO_Z$-module.  The multiplication map on $\sO_Y$ gives the morphism of $\sO_Z$-modules
\[
m:f_*\sO_Y\otimes_{\sO_Z}f_*\sO_Y\to f_*\sO_Y,
\]
and since $f_*\sO_Y$ is a finite locally free $\sO_Z$-module, we have the trace map
\[
\Tr_f:f_*\sO_Y\to \sO_Z
\]
defined by sending $s\in f_*\sO_Y(U)$ to the trace of the multiplication map $\times s: f_*\sO_Y(U)\to f_*\sO_Y(U)$. Rewriting the composition $\Tr\circ m$ as
\[
\delta: f_*\sO_Y\to f_*\sO_Y^\vee
\]
we have the {\em discriminant}  $\disc_f:\det  f_*\sO_Y\to \det f_*\sO_Y^\vee$, given by
\begin{equation}\label{eq:def:disc}
\disc_f = \det \delta.
\end{equation}
Equivalently,
\[
\disc_f:\sO_Z\to (\det f_*\sO_Y)^{-2}
\]

Now  suppose that $f$ is \'etale over each generic point of $Z$, and that $Z$ is reduced. 
Since $\Tr_f$ is a surjection if $f$ is \'etale, we see that $\disc_f$ is generically injective, hence injective since $Z$ is reduced. This gives us the effective Cartier divisor 
 $\Div(\disc_f)$,  supported on the branch locus of $f$.

For $V$ a locally free sheaf of rank $r$ on $Z$, we write $\det^n V$ for the $n$-tensor power over $\sO_Z$ of the invertible sheaf $\det V=\Lambda^rV$. Recall Definition~\ref{def:dpl}.

Let $S$ be a del Pezzo surface of degree $d_S$ over a field $k$ of characteristic $0$, and let $D$ be an effective Cartier divisor. Suppose that Basic Assumptions~\ref{BA}~\ref{it:exlude_multiple_covers_-1curves}~\ref{it:a3degree} holds. We may then apply Theorem \ref{prop:Good} and obtain $\ev: \bar{M}^{\good}_{0,n} \to S^n$. By Definition~\ref{def:dpl}, we have the map $\pi : \dpl \to \M_{0, n}(S,D)^\good$ from the double point locus, which is finite, flat and generically \'etale by Corollary~\ref{Cor:dpl_sm_ram_pi}. We therefore have $\disc_{\pi}: \sO_{\bar{M}^\good_{0,n}(S,D)} \to (\det \pi_*\sO_{\dpl})^{\otimes -2}$ by the above construction. The results of Section~\ref{Section:dpl} compute the divisor of this section.

\begin{theorem}\label{prop:Disc} We have
\[
\Div(\disc_\pi)=1\cdot D_\cusp+2\cdot D_\tac
\]
and thus $\disc_\pi$ defines an isomorphism
\[
\disc_\pi:\sO_{\bar{M}^\good_{0,n}(S,D)}(D_\cusp)\to (\det \pi_*\sO_{\dpl})(-D_\tac)^{\otimes -2}
\]
\end{theorem}

\begin{proof}
A finite, flat map $f: X \to Y$ between smooth schemes over a field has a different $\mathfrak{D}_{f}$, which is an ideal sheaf on an effective Cartier divisor \cite[Tag 0BTC]{stacks-project}. In characteristic $0$, the different is computed \cite[Chapter 5 Theorem 28]{ZariskiSamuelI} to be the product of ideal sheaves $\frak{p}^{e_{\frak{p}}-1}$ where $\frak{p}$ runs over the codimension $1$ ideal sheaves of $X$ where $f$ is ramified and $e_{\frak{p}}$ is the ramification index. By Corollary~\ref{Cor:dpl_sm_ram_pi}, it follows that $\Div \mathfrak{D}_{\pi} = 1 \dplcusp + 1 \dpltac$. By Propositions 14 of Chapter 3 in~\cite{Lang}, $\Div(\disc_\pi) = \pi_* (\Div \mathfrak{D}_{\pi})$. Thus $\Div(\disc_\pi)= \pi_* \dplcusp + \pi_* \dpltac$. By Lemmas~\ref{lm:dplcusp_smooth_ram1_1:1} and \ref{lm:dpltac_smooth_ram_2:1} respectively,
\[
\pi_* \dplcusp =  D_\cusp\text{ and } \pi_* \dpltac = 2\cdot D_\tac .
\]

\end{proof}

We are now in a position to orient the evaluation map in characteristic $0$. 

$\ev: \bar{M}_{0,n}^{\good} \to S^n$ is a map between smooth schemes by Theorem~\ref{prop:Good}~\ref{it:evfinflatdom} and is therefore a local complete intersection morphism \cite[Tag 068E]{stacks-project}. By Theorem \ref{prop:EvRam}, $\ev$ defines an isomorphism
\[
\det d(\ev):\sO_{\bar{M}_{0,n}(S,D)^\good}(D_{\cusp})\to \omega_\ev
\]

\begin{theorem} \label{thm:Orient1} Let $k$ be a field of characteristic $0$ and let $S$ and $D$ be as in Theorem~\ref{into:thm:ram(ev)}. Let $\sL$ be the invertible sheaf on $\bar{M}^\good_{0,n}(S,D)$ given by
\[
\sL=\det^{-1} \pi_*\sO_{\dpl}(-D_\tac)
\]
Then the composition $\det d\ev\circ \disc_\pi^{-1}:\sL^{\otimes 2}\to \omega_\ev$ is an isomorphism on $\bar{M}^\good_{0,n}(S,D)$.
\end{theorem}

\begin{proof} This is a direct consequence of Theorem~\ref{prop:EvRam} and Theorem~\ref{prop:Disc}.
\end{proof}

\section{The symmetrized moduli space}\label{Section:symmetrized_moduli_space_kchar0}

Let $k$ be a field of characteristic $0$. Let $S$ be a del Pezzo surface equipped with an effective Cartier divisor $D$ satisfying Basic Assumptions~\ref{BA}~\ref{it:exlude_multiple_covers_-1curves}~\ref{it:a3degree}  and such that $\Deg K_S\cdot D\ge1$. We continue to assume that $\M_{0, n}(S,D)^\odp$ is non-empty. See Remark~\ref{rmk:ModpMbirMgood_nonempty}.

 The symmetric group $\mfS_n$ acts freely on $\bar{M}_{0,n}(S,D)$ by permuting the marked points, and acting trivially on the underlying curve and the morphism to $S$. This extends to an action on the universal curve $X_{0,n}\to \bar{M}_{0,n}(S,D)$ and the usual permutation action on $S^n$, giving us the following $\mfS_n$-equivariant diagram.
\[
\xymatrix{
\uc{n}\ar[r] &\bar{M}_{0,n}(S,D)\ar[r]^(.6){\ev}&S^n
}
\]
Let $S^n_0$ denote the complement of the pairwise diagonals in $S^n,$ so the restriction of the $\mfS_n$ action to $S^n_0$ is free. Let $n=d-1$. Moreover, $\ev(\M_{0,n}(S,D)^\good) \subset S^n_0$ because by Theorem~\ref{prop:Good} \ref{it:bir} there are no contracted components in the stable maps of $\M_{0,n}(S,D)^\good.$ As above, let $\uc{n}^\good$ denote the inverse image of $\M_{0,n}(S,D)^\good$ under $\uc{n} \to \bar{M}_{0,n}(S,D)$.
Thus, we obtain the following $\mfS_n$-equivariant diagram in which all actions are free.
\begin{equation}\label{eqn:Diagr1}
\xymatrix{
\dpl\ar[r]\ar[dr]^\pi &
\uc{n}^\good \times_{\M_{0, n}(S,D)^\good} \uc{n}^\good\ar[d]\\
&\bar{M}^\good_{0,n}(S,D)\ar[r]^(.6){\ev}&S^n_0
}
\end{equation}
Since $S^n$ is projective over $k$ and $\ev$ and  $\pi$ are quasi-finite (Theorem~\ref{prop:Good}\ref{it:evfinflatdom} and Corollary~\ref{Cor:dpl_sm_ram_pi}), the schemes $\bar{M}^\good_{0,n}(S,D)$, $S^n_0$, and $\dpl$ are quasi-projective. So, one may take their quotients by $\mfS_n$ in the category of quasi-projective $k$-schemes. Since the actions are free, these quotients are smooth over $k.$ We denote the respective quotients by $\dpl_\mfS,$ $\bar{M}^{\good}_{0,n,\mfS}(S,D),$ and $\Sym^n_0S.$ We denote the induced maps by $\pi^\mfS$ and $\ev^\mfS.$ Note that $\Sym^n_0 S$ is an open subscheme of the standard $n$th symmetric product $\Sym^n S.$ Thus we obtain the following diagram of smooth quasi-projective $k$-schemes.
\begin{equation}\label{eqn:Diagr2}
\xymatrix{
\dpl_\mfS\ar[r]^(.3){\pi_\mfS} & \bar{M}^{\good}_{0,n,\mfS}(S,D)\ar[r]^(.6){\ev_\mfS}&\Sym^n_0S
}
\end{equation}
Observe that all squares in the following diagram are Cartesian, where the vertical maps are the quotient maps.
\begin{equation}
\xymatrix{
\dpl\ar[r]^(.35)\pi\ar[d] &
\bar{M}^{\good}_{0,n}(S,D)\ar[d]\ar[r]^(.6){\ev}&S^n_0 \ar[d] \\
\dpl_\mfS\ar[r]^(.35){\pi_\mfS} &
\bar{M}^{\good}_{0,n,\mfS}(S,D)\ar[r]^(.6){\ev_\mfS}&\Sym^n_0S
}
\end{equation}

For $\square\in\{\cusp, \tac, \trip\}$, we let $D^\mfS_\square$ denote the reduced image of $D_\square$ in $\bar{M}^{\good}_{0,n,\mfS}(S,D)$.

\begin{theorem}\label{thm:EvStructure} Let $k,S,D$ be as in Theorem~\ref{into:thm:ram(ev)}.\begin{enumerate}
\item
The canonical section $\det d\ev_\mfS:\sO_{\bar{M}^{\good}_{0,n,\mfS}(S,D)}\to \omega_{\ev_\mfS}$ has divisor $1\cdot D^\mfS_\cusp$ and induces an isomorphism
\[
\det d\ev_\mfS:\sO_{\bar{M}^{\good}_{0,n,\mfS}(S,D)}(D^\mfS_\cusp)\to \omega_{\ev_\mfS}.
\]
\item
The divisor of $\disc_{\pi_\mfS}:\sO_{\bar{M}^{\good}_{0,n,\mfS}(S,D)}\to \det^{-2}\pi^\mfS_*\sO_{\bar\sD^\mfS}$ is $D^\mfS_\cusp+2\cdot D_\tac^\mfS$ and induces an isomorphism
\[
\disc_{\pi^\mfS}:\sO_{\bar{M}^{\good\mfS}_{0,n}(S,D)}(D^\mfS_\cusp)\to [\det\pi^\mfS_*\sO_{\bar\sD^\mfS}(-D_\tac)]^{\otimes - 2}
\]
\item\label{item:thm:EvStructure:orientation_sym}
Letting $\sL^\mfS:=[\det\pi_{\mfS*}\sO_{\dpl^\mfS}(-D_\tac)]^{-1}$, we have the isomorphism
\[
\det d\ev_\mfS\circ \disc_{\pi_\mfS}^{-1}:(\sL^\mfS)^{\otimes 2}\to \omega_{\ev_\mfS}.
\]
\end{enumerate}
\end{theorem}

\begin{proof} This follows from  Theorem~\ref{prop:EvRam} and Theorem~\ref{prop:Disc}, noting that the relative dualizing sheaf $\omega_f$ of a morphism $f$ is compatible with \'etale base-chance, as is the divisor of the discriminant of a morphism, and the divisor of a section of an invertible sheaf is detectible after finite \'etale base-change. Specifically, the fact that $\Div~ d\ev=1\cdot D_\cusp$ and $\Div~ \disc_\pi=1\cdot D_\cusp+2\cdot D_\tac$ implies that
$\Div~ d\ev^\mfS=1\cdot D^\mfS_\cusp$ and $\Div~ \disc_{\pi^\mfS}=1\cdot D^\mfS_\cusp+2\cdot D^\mfS_\tac$; the remaining assertions are direct consequences of these two identities.
\end{proof}

\section{Twisting the degree map}\label{Section:twisting_deg_map_char0}
As before, we let $k$ be a field. Let $S$ be a smooth del Pezzo surface over $k$ equipped with an effective Cartier divisor $D$. Let $k \subseteq \ksep$ denote a separable closure of $k$. For a $k$-scheme $Y$ and field extension $k \subset L$, we write $Y_L$ for $Y \times_k L$. Let $$\sigma= (L_1, \ldots, L_r)$$ be an $r$-tuple of subfields $L_i \subset \ksep$ containing $k$ for $i=1,\ldots, r$ subject to the requirement that $\sum_{i=1}^k [L_i : k] = n$. We think of $\sigma$ as the fields of definition of a list of points of $S$ that our curves will be required to pass through.

The list $\sigma$ is used to define twists $\ev_{\sigma}$ of the evaluation map $\ev: \bar{M}_{0,n}(S, D) \to S^n$ in the following manner. The Galois group $\Gal(\ksep/k)$ acts on the $\ksep$-points of $k$-schemes. Thus $\sigma$ gives rise to a canonical homomorphism $\rho(\sigma) :\Gal(\ksep/k) \to \mfS_{\overline{\sP}(\sigma)},$ where $\overline{\sP}(\sigma)$ denotes the $\ksep$-points of the $k$-scheme $\coprod_{i=1}^r \Spec L_i$ and $\mfS_{\overline{\sP}(\sigma)} \cong \mfS_n$ denotes the symmetric group. For convenience, we fix an identification $\overline{\sP}(\sigma) = \{1,2,\ldots, n\}$ and thus a canonical isomorphism $\mfS_{\overline{\sP}(\sigma)} = \mfS_n$.

There is a canonical inclusion of $\mfS_n$ into $\Aut(S^n)$. We include $\mfS_n$ into $\Aut (\M_{0,n}(S,D))$ by permutation of the marked points, and acting trivially on the underlying curve and the morphism to $S$: for $\tau$ in $\mfS_n$, set $$\tau(u: C \to S, p_1, \ldots, p_n) = (u: C \to S, p_{\tau^{-1}(1)}, \ldots, p_{\tau^{-1}(n)}).$$ Let $X=S^n,X=\M_{0,n}(S,D)$ or the double point locus $\dpl$. The $1$-cocycles \begin{equation}\label{twistcocyclenew}g \mapsto \rho(\sigma)(g) \times g \end{equation}  $$\Gal(\ksep/k) \to \Aut (X_{\ksep})$$ determine twists $X_{\sigma}$ of $X$. Since $\ev_{\ksep}$ and $\pi_{\ksep}:  \dpl_{\ksep} \to \M_{0, n}(S, d)^\good_{\ksep}$ are Galois equivariant for the twisted action, they descend to $k$-maps denoted $\ev_{\sigma}$ and $\pi_{\sigma}$ respectively $$ \ev_{\sigma}: \M_{0,n}(S,D)_\sigma \to (S^n)_\sigma $$
\[
\pi_{\sigma}: \dpl_{\sigma} \to \M_{0,n}(S,D)_\sigma.
\]

The twist $(S^n)_\sigma$ of $S^n$ by $\sigma$ can be expressed as the restriction of scalars
\begin{equation}\label{eq_restriction_of_scalars}
(S^n)_\sigma \cong \prod_{i=1}^r \Res_{L_i/k} S,
\end{equation} allowing us to view $\ev_{\sigma}$ as a map with this codomain.

In this section, we orient an appropriate restriction of $\ev_{\sigma}$ in characteristic $0$. Assume that $k$,$S$, and $D$ are as in the hypotheses of Theorem~\ref{prop:Good}. We continue to assume that $\M_{0, n}(S,D)^\good$ is non-empty. (See Remark \ref{rmk:ModpMbirMgood_nonempty}.) We may assume the set $A\subset S^n$ used to construct $\bar{M}^{\good}_{0,n}(S,D)$ in Theorem~\ref{prop:Good} is stable under the action of $\mfS_n$. This action then restricts to an action on $S^n -A$ defining an open $k$-subscheme $(S^n -A)_{\sigma} \subset S^n_{\sigma}$ whose closed complement $A_{\sigma}$ has codimension $\geq 2$.

Forgetting the marked points determines a $k$-map $\bar{M}_{0,n}(S,D)^{\good}_{\sigma} \to \bar{M}_0(S,D)$ from the twisted good moduli space to the untwisted moduli space of stable curves because $\mfS_n$ acts trivially on the underlying curve. For $\square\in\{\cusp, \tac, \trip\}$, we let $D_{\square,\sigma}$ denote the preimage of $D_\square$ under this map.

\begin{theorem}\label{thm:rel_or_twist_ev_char0}
Let $k,S,D$ be as in Theorem~\ref{into:thm:ram(ev)}.
\begin{enumerate}
\item \label{item:thm:rel_or_twist_ev_char0:smooth} $\ev_{\sigma}: \bar{M}^{\good}_{0,n}(S,D)_{\sigma} \to (S^n)_{\sigma}$ is a map between smooth $k$-schemes.
\item \label{item:thm:rel_or_twist_ev_char0:detdev}
The canonical section $\det d\ev_{\sigma}:\sO_{\bar{M}^{\good}_{0,n}(S,D)_{\sigma}}\to \omega_{\ev_{\sigma}}$ has divisor $1\cdot D_{\cusp, \sigma}$ and induces an isomorphism
\[
\det d\ev_{\sigma}:\sO_{\bar{M}^{\good}_{0,n}(S,D)_{\sigma}}(D_{\cusp, \sigma})\to \omega_{\ev_{\sigma}}.
\]
\item \label{item:thm:rel_or_twist_ev_char0:disc_pi}
The divisor of $\disc_{\pi_{\sigma}}:\sO_{\bar{M}^{\good}_{0,n}(S,D)_{\sigma}}\to [\det(\pi_\sigma)_*\sO_{\dpl_{\sigma}}]^{\otimes -2}$ is $$D_{\cusp,\sigma}+2\cdot D_{\tac,\sigma}$$ and induces an isomorphism
\[
\disc_{\pi_{\sigma}}:\sO_{\bar{M}^{\good}_{0,n}(S,D)_{\sigma}}(D_{\cusp,\sigma})\to [\det (\pi_{\sigma})_*\sO_{\dpl_{\sigma}}(-D_{\tac,\sigma})]^{\otimes -2}
\]
\item \label{item:thm:rel_or_twist_ev_char0:rel_or}
Letting $\sL_{\sigma}:=[\det(\pi_\sigma)_*\sO_{\dpl_{\sigma}}(-D_\tac)]^{-1}$, we have the isomorphism
\[
\det d\ev_{\sigma}\circ \disc_{\pi_{\sigma}}^{-1}:(\sL_{\sigma})^{\otimes 2}\to \omega_{\ev_{\sigma}}.
\]
\end{enumerate}
\end{theorem}

\begin{proof}
Let $L$ be a finite normal extension of $k$ containing $L_i$ for $i=1,\ldots, r$. Then the cocycle \eqref{twistcocyclenew} factors through $\Gal(L/k)$, and there is a canonical isomorphism $X_L \cong X_{\sigma, L}$ for $X = \bar{M}^{\good}_{0,n}(S,D)$, $S^n - A$ or $S^n$. Similarly the base-change  $\ev_{\sigma, L}$ of $\ev_{\sigma}$ is identified with $\ev_L$ via these canonical isomorphisms.

\ref{item:thm:rel_or_twist_ev_char0:smooth} then follows from Theorem~\ref{prop:Good} and the smoothness of $S$ because smoothness is fpqc local and may therefore be checked after base-change to $L$.

Note that $k \subset L$ is \'etale as $k$ is characteristic $0$. The claims  \ref{item:thm:rel_or_twist_ev_char0:detdev} and \ref{item:thm:rel_or_twist_ev_char0:disc_pi} follow from Theorem~\ref{prop:EvRam} and Theorem~\ref{prop:Disc}, respectively, noting that the relative dualizing sheaf $\omega_f$ of a morphism $f$ is compatible with \'etale base-change, as is the divisor of the discriminant of a morphism, and the divisor of a section of an invertible sheaf is detectible after finite \'etale base-change.

\ref{item:thm:rel_or_twist_ev_char0:rel_or} follows from \ref{item:thm:rel_or_twist_ev_char0:smooth}, \ref{item:thm:rel_or_twist_ev_char0:detdev} and \ref{item:thm:rel_or_twist_ev_char0:disc_pi}
\end{proof}

For the comparison of the $\A^1$-degrees corresponding to the orientations of Theorem~\ref{thm:rel_or_twist_ev_char0} \ref{item:thm:rel_or_twist_ev_char0:rel_or} and Theorem~\ref{thm:EvStructure} \ref{item:thm:EvStructure:orientation_sym} in \cite{degree}, we note that there is a pullback diagram
\begin{equation}\label{cd:twist_to_symmetrize}
\xymatrix{\bar{M}^{\good}_{0,n}(S,D)_{\sigma} \ar[d]_{\ev_{\sigma}} \ar[r] & \ar[d]_{\ev_\mfS} \bar{M}^{\good}_{0,n,\mfS}(S,D)\\
S^n_{\sigma} \ar[r] & \Sym^n S}
\end{equation} where the horizontal maps are determined by the quotient maps over the algebraic closure or $L$ (e.g. $S^n_{L} \to \Sym^n S_{L}$), where $L$ is a finite normal extension of $k$ containing the $L_i$. Since the bottom horizontal map is \'etale over $Sym^n_0 S$, the upper horizontal map is \'etale.

\section{Positive characteristic}\label{section:positive_characteristic}

In this section we will extend many of our constructions that up to now have been restricted to characteristic zero to del Pezzo surfaces in positive characteristic. The method is to lift to characteristic zero. 
\subsection{Lifting to characteristic zero}\label{subsection:lifting}

We first recall some basic facts from deformation theory.

\begin{proposition}\label{prop_deformation_theory}\  
Let $\Lambda$ be a complete discrete valuation ring with residue field $k$ and quotient field $K$.
\begin{enumerate}
\item\label{prop_deformation_theory1}  Let $X, Y$ be smooth, proper $\Lambda$-schemes and let $f_0:Y_k\to X_k$ be a morphism. Suppose that $H^1(Y_k, f_0^*T_{X_k/k})=0$. Then there is a $\Lambda$-morphism $f:Y\to X$ with $f_k=f_0$.  If $f_0$ is a closed immersion, then so is $f$. 
\item\label{prop_deformation_theory2} Let $X_0$ be a smooth proper $k$-scheme. Suppose $H^2(X_0, T_{X_0/k})=0$. Then there is a smooth proper $\Lambda$-scheme $X$ with an isomorphism $\phi:X_k\xrightarrow{\sim} X_0$ over $k$. If in addition $H^1(X_0, T_{X_0/k})=0$, then $(X, \phi)$ is unique up to isomorphism over $\Lambda$.
\item\label{prop_deformation_theory3} Let $X$ be a proper $\Lambda$-scheme and let $\sL_0$ be an invertible sheaf on $X_k$. If $H^2(X_k, \sO_{X_k})=0$, there is an invertible sheaf $\sL$ on $X$ and an isomorphism $\psi:\sL_k\xrightarrow{\sim} \sL_0$ of coherent sheaves on $X_k$. If in addition $H^1(X_k, \sO_{X_k})=0$, then $(\sL,\psi)$ is unique up to isomorphism of invertible sheaves on $X$. 
\item\label{prop_deformation_theory4} Let $\sL$ be an invertible sheaf on a proper $\Lambda$-scheme $p:X\to \Spec \Lambda$. If $H^1(X_k,  \sL_k)=0$, then $p_*\sL$ is a free $\Lambda$-module and the natural map $\pi_*\sL\otimes_\Lambda k\to H^0(X_k, \sL_k)$ is an isomorphism. In particular, each section $s_0\in H^0(X_k, \sL_k)$ lifts to a section $s\in H^0(X, \sL)$.
\end{enumerate}
\end{proposition}

\begin{proof}  \ref{prop_deformation_theory1} Let $\hat{X}$, $\hat{Y}$ denote the formal schemes associated to the $\Lambda$-schemes $X, Y$. By \cite[Exp. III, Corollaire 5.6]{sga1}, $f_0$ extends to a morphism of formal schemes $\hat{f}:\hat{Y}\to \hat{X}$. By \cite[Chap. III, Th\'eor\`eme 5.4.1]{EGAIII_1}, there is a unique $\Lambda$-morphism $f:Y\to X$ inducing $\hat{f}$ on the formal schemes. In particular, $f_k=f_0$. 

If moreover $f_0$ is a closed immersion, then it follows that $\hat{f}:\hat{Y}\to \hat{X}$ is a (formal) closed immersion. Then \cite[Chap. III, Corollaire 5.1.8, Th\'eor\`eme 5.4.1]{EGAIII_1} implies that  $f:Y\to X$ is a closed immersion. 

\ref{prop_deformation_theory2} This can be found in \cite[Theorerm 8.5.9(b)]{IllusieFGA}.

\ref{prop_deformation_theory3} Use \cite[Theorerm 8.5.5]{IllusieFGA}.

\ref{prop_deformation_theory4} Apply \cite[Theorerm 8.5.3(a)]{IllusieFGA} with $E=\Lambda_X$, $F=\sL$. 

\end{proof}

\begin{lemma}\label{lm:cohDelPezzo} Let $S$ be a del Pezzo surface over a field $k$, with effective Cartier divisor $D$. As above, let $d_S=\deg_k(K_S\cdot K_S)$ and $d=\deg_k(-K_S\cdot D)$. Then
\begin{enumerate}
\item\label{lm:cohDelPezzo1} $H^1(S,\sO_S)=H^2(S, \sO_S)=0$.
\item\label{lm:cohDelPezzo2} Let $T_{S/k}$ denote the tangent sheaf. Then $H^2(S, T_{S/k}) =0$ and 
\[
\dim_kH^1(S, T_{S/k})=\begin{cases}0&\text{ if }d_S\ge 5\\
5-d_S&\text{ if }1\le d_S< 5
\end{cases}
\]
\item\label{lm:cohDelPezzo3} $H^1(S, \sO_S(D))=0$.
\end{enumerate}
\end{lemma}

\begin{proof} Since cohomology commutes with flat base-change, we may extend from $k$ to its algebraic closure, and assume from the start that $k$ is algebraically closed. Then $S$ is either $\P^1\times\P^1$ or is a blow-up of $\P^2$ at $r:=9-d_S\ge0$ points.

For \ref{lm:cohDelPezzo1}, if $S=\P^1\times\P^2$, then the K\"unneth formula gives $H^1(S, \sO_S)\cong H^1(\P^1, \sO_{\P^1})^2=0$, $H^2(S, \sO_S)\cong H^1(\P^1, \sO_{\P^1})^{\otimes_k2}=0$.

For $S=\P^2$, the vanishing of $H^i(\P^2, \sO_{\P^2})$ for $i>0$ may be found in \cite[Chap. III, \S3, Proposition 8]{FAC}.

 If $\pi:S\to \P^2$ is the blow-up of $\P^2$ at $\{p_1,\ldots, p_r\}$, $r\ge1$,  let $E_i=\pi^{-1}(p_i)$. We compute $H^i(S, \sO_S)$ via the Leray spectral sequence
 \[
 E_2^{p,q}=H^p(\P^2, R^q\pi_*\sO_S)\Rightarrow H^{p+q}(S,\sO_S)
 \]
 Since $\pi_*\sO_S=\sO_{\P^2}$, we need only show that $R^q\pi_*\sO_S=0$ for $q>0$. We use the formal functions thereom
 \[
 (R^q\pi_*\sO_S)_{p_i}=\lim_{\substack{\leftarrow\\n\ge0}}H^q(E_i, \sO_S/\sI_{E_i}^{n+1})
 \]
 As  $E_i\cong\P^1$ and $\sI_{E_i}^n/\sI_{E_i}^{n+1}\cong \sO_{\P^1}(n)$, we find that 
 $(R^q\pi_*\sO_S)_{p_i}=0$; clearly $(R^q\pi_*\sO_S)_p=0$ for $p$ not among the $p_i$, completing the proof of  \ref{lm:cohDelPezzo1}.
 
 For  \ref{lm:cohDelPezzo2}, in case $S=\P^1\times\P^1$, we have
$T_{S/k}=p_1^*\sO_{\P^1}(2)\oplus p_2^*\sO_{\P^1}$, from which \ref{lm:cohDelPezzo2} easily follows. If $\pi:S\to \P^2$ is the blow-up of $\P^2$ at $\{p_1,\ldots, p_r\}$, let $E_i=\pi^{-1}(p_i)$. If $r=0$, we have the Euler sequence
\[
0\to\sO_{\P^2}\to \sO_{\P^2}(1)^3\to T_{\P^2/k}\to0
\]
and   $H^i(\P^2, \sO_{\P^2}(d))=0$ for $i>0$, $d\ge-1$, giving $H^1(\P^2, T_{\P^2/k})=0$ for $i>0$. This also shows that $\dim_kH^0(\P^2, T_{\P^2/k})=8$. 

For $0<r$,  we have the exact sequence
\[
0\to T_{S/k}\xrightarrow{d\pi}\pi^*T_{\P^2/k}\to \oplus_{j=1}^ri_{j*}\sO_{E_j}(-E_j\cdot E_j)\to 0
\]
with $i_j:E_j\to S$ the inclusion. 
Identifying $E_j$ with $\P^1$, we have $\sO_{E_j}(-E_j\cdot E_j)\cong \sO_{\P^1}(1)$, so $H^i(\sO_{E_j}(-E_j\cdot E_j))=0$ for $i>0$.  Using the Leray spectral sequence again, we see that $H^i(S, \pi^*T_{\P^2/k})\cong H^i(\P^2, T_{\P^2/k})=0$ for $i>0$. Thus $H^2(S, T_{S/k})=0$ and we have the exact sequence
\[
H^0(\P^2, T_{\P^2/k})\xrightarrow{\sum_ji_j^*}\oplus_{j=1}^rH^0(E_j, \sO_{E_j}(-E_j\cdot E_j))\to H^1(S, T_{S/k})\to 0. 
\]
Taking parameters $(x, y)$ at $p_j$, we identify $E_i$ with $\P^1$,  $\sO_{E_j}(-E_j\cdot E_j)$ with  $\sO_{\P^1}(1)$, $T_{\P^2, p_j}$ with $k\cdot \del/\del x\oplus k\cdot \del/\del y$,    and we  have $i_j^*(\del/\del x)=-X_1$, $i_j^*(\del/\del y)=X_0$. This identifies $\pi_*(\sO_{E_i}(-E_i\cdot E_i))$ with $T_{\P^2, p_i}$, giving the exact sequence
\[
H^0(\P^2, T_{\P^2/k})\xrightarrow{\sum_ji_{p_j}^*}\oplus_{j=1}^rT_{\P^2,p_j} \to H^1(S, T_{S/k})\to 0.
\]
The automorphism group $\PGL_3$ of $\P^2$ acts 4-transitively on 4-tuples of points, no three of which lie on a line. Since $S$ is a del Pezzo surface, $S$ has no $-a$ curves for $a>1$, so this condition is satisfied for the set $\{p_1,\ldots, p_r\}$, and thus the map $\sum_ji_{p_j}^*$ is surjective for $r\le 4$. For $r=4$,  counting dimensions shows $\sum_ji_{p_j}^*$ is an isomorphism, and for $r>4$, $\sum_ji_{p_j}^*$ is injective, giving
\[
\dim_kH^1(S, T_{S/k})=r-4=5-d_S
\]
as claimed. 

For \ref{lm:cohDelPezzo3}, we have the exact sequence
\[
0\to \sO_S\to \sO_S(D)\to i_{D*}\sO_D(D^{(2))}\to 0
\]
so we reduce to showing $H^1(D, \sO_D(D^{(2))})=0$. Letting $\omega_D$ denote the dualizing sheaf on $D$, Serre duality gives $H^1(D, \sO_D(D^{(2))})\cong H^0(D, \omega_D\otimes \sO_D(-D^{(2))})$. But the adjunction formula says $\omega_D=K_S(D)\otimes_{\sO_S}\sO_D$, so $\omega_D\otimes \sO_D(-D^{(2))}\cong \sO_D(K_S\cdot D)$. Since $-K_S$ is ample, $\deg_D(K_S\cdot D)<0$, so $H^0(D, \sO_D(K_S\cdot D))=0$.
\end{proof}

\begin{lemma}\label{lm:LiftingDelPezzo} Let $S$ be a del Pezzo surface over a field $k$, with effective Cartier divisor $D$. Let $d_S=\deg_k(K_S\cdot K_S)$ and $d=\deg_k(-K_S\cdot D)$. Let $\Lambda$ be a complete discrete valuation ring with residue field $k$ and quotient field $K$. Then
\begin{enumerate}
\item\label{lm:LiftingDelPezzo1} There is a smooth proper $\Lambda$-scheme $\pi:\LiftS\to \Spec\Lambda$ with an isomorphism $\phi:\LiftS_k\xrightarrow{\sim} S$.
\item\label{lm:LiftingDelPezzo2}  For each lifting $(\LiftS, \phi)$ of $S$ over $\Lambda$ as in \ref{lm:LiftingDelPezzo2}, letting $i:S\to \LiftS$ be the closed immersion induced by $\phi$, the restriction map $i^*:\Pic(\LiftS)\to \Pic(S)$ is an isomorphism. 
\item\label{lm:LiftingDelPezzo3} For each lifting $(\LiftS, \phi)$ of $S$ over $\Lambda$ as in \ref{lm:LiftingDelPezzo2}, $\LiftS$ is a del Pezzo surface over $\Lambda$ and the generic fiber $\LiftS_K$ is a del Pezzo surface over $K$. Moreover, we have $d_{\LiftS_K}=d_S$. 
\item\label{lm:LiftingDelPezzo4} For each lifting $(\LiftS, \phi)$ of $S$ over $\Lambda$ as in \ref{lm:LiftingDelPezzo1}, there is an effective Cartier divisor $\LiftDeg$ on $\LiftS$ with $\phi(\LiftDeg_k)=D$.  Moreover, we have $\deg_K(-K_{\LiftS_K}\cdot \LiftDeg_K)=d$.
\end{enumerate}
\end{lemma}

\begin{proof} \ref{lm:LiftingDelPezzo1} follows from Proposition~\ref{prop_deformation_theory}\ref{prop_deformation_theory2} and Lemma~\ref{lm:cohDelPezzo}\ref{lm:cohDelPezzo2}.

For \ref{lm:LiftingDelPezzo2}, we have $H^1(S, \sO_S)=H^2(S, \sO_S)=0$ by  Lemma~\ref{lm:cohDelPezzo}\ref{lm:cohDelPezzo1}. Applying Proposition~\ref{prop_deformation_theory}\ref{prop_deformation_theory3} shows that $i^*$ is an isomorphism.

For \ref{lm:LiftingDelPezzo3}, $-K_S$ is ample, and $K_S$ lifts canonically to the relative dualizing sheaf $\omega_{\LiftS/\Lambda}$, which restricted to $\LiftS_K$ is the canonical sheaf $K_{\LiftS_K}$. By \cite[Th\'eor\`eme 5.4.5]{EGAIII_1}, there is an ample invertible sheaf $\sL$ on $\LiftS$ with $\phi(\sL_k)=-K_S$. But then by  \ref{lm:LiftingDelPezzo2}, $\sL$ is isomorphic to $\omega_{\LiftS/\Lambda}$,  so $-K_{\LiftS_K}$ is ample on $\LiftS_K$. Thus $\LiftS$ is a del Pezzo surface over $\Lambda$ and $\LiftS_K$ is a del Pezzo surface over $K$. The assertion that $d_{\LiftS_K}=d_S$ follows from the conservation of intersection numbers (see e.g. \cite[B18]{Kleinman_The_Picard_Scheme}). 

Finally, to prove \ref{lm:LiftingDelPezzo4}, we have $H^1(S, \sO_S(D))=0$ by Lemma~\ref{lm:cohDelPezzo}\ref{lm:cohDelPezzo3}, and by \ref{lm:LiftingDelPezzo2}, there is an invertible sheaf $\sL$ on $\LiftS$ lifting $\sO_S(D)$.  Let $s_0\in H^0(S, \sO_S(D))$ be the canonical section, so $\div(s_0)=D$.   By Proposition~\ref{prop_deformation_theory} \ref{prop_deformation_theory4} we may lift $s_0$ to a section $s\in H^0(\LiftS, \sL)$; letting $\LiftDeg=\div(s)$, we see that $\phi(\LiftDeg_k)=D$.  The identity $\deg_K(-K_{\LiftS_K}\cdot \LiftDeg_K)=d$ follows aas above by conservation of intersection numbers. 
\end{proof}

\begin{lemma}\label{lm:BA23ImpliesBA123} Let $S$ be a del Pezzo surface over a field $k$, with effective Cartier divisor $D$, and let $\LiftS, \LiftDeg$ be a lifting of $(S,D)$ over $\Lambda$ as in Lemma~\ref{lm:LiftingDelPezzo}.  Suppose that $S, D$ satisfy Basic Assumptions \ref{BA}\ref{it:exlude_multiple_covers_-1curves} \ref{it:a3degree}. Then $\LiftS_K, \LiftDeg_K$ satisfies 
Basic Assumptions \ref{BA}\ref{it:a1charzero}, \ref{it:exlude_multiple_covers_-1curves}, \ref{it:a3degree} and Assumption~\ref{a:genericunram}.
\end{lemma}

\begin{proof} By Lemma~\ref{lm:LiftingDelPezzo}, we have $d_{\LiftS_K}=d_S$ and $d_K=d$. Thus Basic Assumption \ref{BA} \ref{it:a3degree} for $S,D$ implies this assumption for $\LiftS_K, \LiftDeg_K$. Suppose $E\subset S$ is a -1 curve.  Then by Lemma~\ref{lm:LiftingDelPezzo}\ref{lm:LiftingDelPezzo4}, there is a lifting $\tilde{E}$ of $E$ to a relative Cartier divisor on $\LiftS$, and $\tilde{E}_K$ is a -1 curve on $\LiftS_K$. Moreover, by 
Lemma~\ref{lm:LiftingDelPezzo} \ref{lm:LiftingDelPezzo2}, if $D=m\cdot E$, then $D_K=m\cdot \tilde{E}_K$, so  Basic Assumptions \ref{BA}\ref{it:exlude_multiple_covers_-1curves} for $S, D$ implies this assumption for $\LiftS_K, \LiftDeg_K$. Basic Assumption \ref{BA}\ref{it:a1charzero} is trivially satisfied for $\LiftS_K, \LiftDeg_K$ since $K$ has characteristic zero. Similarly, Lemma~\ref{lm:char0_implies_Assumption_a:genericunram} implies that   $\LiftS_K, \LiftDeg_K$ satisfy Assumption~\ref{a:genericunram}.
\end{proof} 

\subsection{The moduli space $\bar{M}_{0,n}(\LiftS, \LiftDeg)^{\good}$ and its first properties}

Let $k$ be a perfect field of characteristic $p>3$. Let $S$ be a del Pezzo surface over $k$ with effective Cartier divisor $D$. We assume that $D$ is not the zero divisor; let $d=\Deg (- K_S\cdot D) \ge1$. 
 
Let $\Lambda$ be a complete discrete  valuation ring with residue field $k$ and quotient field $K$ of characteristic $0$. We fix a lifting $(\LiftS \to \Spec \Lambda, \LiftDeg)$ of $(S, D)$, which exists by Lemma~\ref{lm:LiftingDelPezzo}; also by that result, the generic fiber $\LiftS_K$ is a del Pezzo surface with $d_{\LiftS_K}=d_S$, and the effective Cartier divisor $\LiftDeg_K$ on $\LiftS_K$ has degree $d_K:=\deg_K(-K_{\LiftS_K}\cdot \LiftDeg_K)=d$. 

The following elementary lemma will be used below.

\begin{lemma} \label{lemma:extending_subscheme_to_fiber}
Let $Y \to \Spec(\Lambda)$ be of finite type and let $Z' \subset Y_{K}$ be a closed subscheme. Let $Z$ be the closure of $Z'$ in $Y$. Then  
\begin{enumerate}
\item \label{itlemma:extending_subscheme_to_fiber1} $Z$ is flat over $\Lambda$. In particular,  no irreducible component of $Z$  is contained in the fiber of $Y$ over the closed point of $\Spec\Lambda$. 
\item \label{itlemma:extending_subscheme_to_fiber2} Suppose that $Z'$ is reduced. Then $Z$ is reduced.
\item \label{itlemma:extending_subscheme_to_fiber2.5} Let $W$ be a reduced closed subscheme of $Y$ such that no irreducible component of $W$ is  is contained in the fiber of $Y$ over the closed point of $\Spec\Lambda$. Then $W$ is flat over $\Lambda$.
\item \label{itlemma:extending_subscheme_to_fiber3} Let $W\subset Y$ be a closed subscheme of $Y$ containing $Z'$,   with support of $W$ equal to the support of $Z$ and with $W_K=Z'$. Suppose that  the special fiber $W_k$ is reduced. Then  $W=Z$.
\end{enumerate}
\end{lemma}
\begin{proof}
We claim that the sheaf $\sO_Z$ is $t$-torsion free,  where $t\in \Lambda$ is a generator of the maximal ideal. This implies the result, since a $\Lambda$-module $M$ is flat if and only if $M$ is $t$-torsion free.
 
To see that $\sO_Z$ is $t$-torsion free, we may assume that $Y$ is affine and finite type over $\Lambda$, $Y=\Spec A$. Then $Y$ is a closed subscheme of an affine space over $\Lambda$, $\A^n_\Lambda$, and the closure of $Z'$ in $Y$ is the same as the closure in $\A^n_\Lambda$, so we may assume that $Y=\A^n_\Lambda$. Let $K$ be the quotient field of $\Lambda$ and let $I'\subset K[x_1,\ldots, x_n]$ be the ideal of $Z'$. Then the idea $I$ of $Z$ in $\Lambda[x_1,\ldots, x_n]$  is the maximal ideal $J$ such that $JK[x_1,\ldots, x_n]=I'$. 

Take $\bar{x}\in \Lambda[x_1,\ldots, x_n]/I$ such that $t\bar{x}=0$. Lifting $\bar{x}$ to $x\in \Lambda[x_1,\ldots, x_n]$, we have $tx\in I$. But then the image of $x$ in $K[x_1,\ldots, x_n]$ is in $tI'=I'$, so by maximality of $I$, we have $x\in I$ and $\bar{x}=0$. This proves  the first assertion of \ref{itlemma:extending_subscheme_to_fiber1}.

For \ref{itlemma:extending_subscheme_to_fiber2} suppose that $Z'$ is reduced. Let $\sJ\subset \sO_Z$ be the ideal sheaf of $Z_\red$ in $Z$, then since $Z'$ is reduced and $Z_K=Z'$, we have $\sJ_K=0$. Again by the maximality of $\sI_Z$, we must have $\sJ=0$, so $Z$ is reduced.

Let $W\subset Y$ be as in  \ref{itlemma:extending_subscheme_to_fiber2.5} and let $W'\subset W$ be the closure of $W_K$. By \ref{itlemma:extending_subscheme_to_fiber1}, \ref{itlemma:extending_subscheme_to_fiber2}, $W'$ is flat over  $\Lambda$. But as $W'$ and $W$ have the same support and $W$ is reduced, we must have $W'=W$, so $W$ is flat over $\Lambda$, proving  \ref{itlemma:extending_subscheme_to_fiber2.5}.

For \ref{itlemma:extending_subscheme_to_fiber3}, let $\sI_W\subset \sO_Y$ be the ideal sheaf of $W$. Again by maximality of $\sI_Z$, we have $\sI_W\subset \sI_Z$; let $\sJ\subset \sO_W$ be the image of $\sI_Z$. Since $W_K=Z'=Z_K$, we have $\sJ_K=0$, that is, $\sJ$ is supported on $Y_k$. Applying $-\otimes_\Lambda k$ to the exact sequence
\[
0\to \sJ\to \sO_W\to \sO_Z\to 0
\]
and recalling that $Z$ is flat over $\Lambda$, we have the exact sequence
\[
0\to \sJ/t\sJ\to \sO_{W_k}\to \sO_{Z_k}\to 0
\]
But $Z_k$ is a closed subscheme of $W_k$ with the same irreducible components, and $W_k$ is reduced, so $Z_k=W_k$ and thus $\sJ/t\sJ=0$. Since $\sJ$ is supported on $Y_k$, it follows from  Nakayama's lemma  that $\sJ=0$, so $W=Z$.
\end{proof}

We recall the moduli stack  $\bar{M}_{0,n}(\LiftS, \LiftDeg)$, which was discussed in some more detail in Section~\ref{subsection:ModuliStacksDefs}. We have the evaluation map
\[
\Liftev: \bar{M}_{0,n}(\LiftS, \LiftDeg) \to \LiftS^{n}
\] 
lifting $\ev_k:\bar{M}_{0,n}(S, D)\to S^n$; here we write $ \LiftS^{n}$ for the $n$-fold fiber product of $ \LiftS$ over $\Spec \Lambda$.   

We now extend the constuction of the open subset $\bar{M}_{0,n}(S',D')^\good$ as outlined in Theorem~\ref{prop:Good} to the mixed characteristic case
\begin{construction}\label{const:MixCharGood} Suppose that $(S, D)$ satisfies Basic Assumptions~\ref{BA} \ref{it:exlude_multiple_covers_-1curves}, \ref{it:a3degree} and Assumption~\ref{a:genericunram}.  Then by Lemma~\ref{lm:BA23ImpliesBA123}, $\LiftS, \LiftDeg$ satisy all the Basic Assumptions \ref{BA}. 
Thus,  we may apply Theorem~\ref{prop:Good}, take a closed subset $A_K \subset \LiftS^n_K$ as in Theorem~\ref{prop:Good}, and let $ \overline{A_K} \subset \LiftS^n$ be its closure. By Lemma~\ref{lemma:extending_subscheme_to_fiber}, $\overline{A_K} $ has codimension $\geq 2$ in $\LiftS^n$. Recalling that $M^{\odp}_{0,n}(S,D)$ is open  in $\bar{M}_{0,n}(S,D)$ by Lemma~\ref{lm:odp_in_unr_in_M_open}, we  let $A_k$ be the closed subset  $\ev_k (\bar{M}_{0,n}(S,D) \setminus M^{\odp}_{0,n}(S,D))$ of $S^n_k$. By Corollary~\ref{Cor:codim_ev(nonodp)>=1},  $A_{k}$ has positive codimension in $S^{n}_{k}$, since we are assuming  $S, D$ satisfies Assumption~\ref{a:genericunram}.  Let
\begin{equation}\label{eq:def:liftA}
\LiftA : = \overline{A_K} \cup A_k
\end{equation}
and define
\[
\bar{M}_{0,n}(\LiftS, \LiftDeg)^{\good}  := \bar{M}_{0,n}(\LiftS, \LiftDeg) - \Liftev^{-1}( \LiftA).
\]

We may freely enlarge  $\LiftA$, as long as we ensure that  $\LiftA$ remains closed in $\LiftS^n$, $\LiftA_K$ satisfies the conditions of 
Theorem~\ref{prop:Good} for $K, \LiftS_K, \LiftDeg_K$, and  $\LiftA_k$ has positive codimension in $S^n$.
\end{construction}
For the remainder of \S\ref{section:positive_characteristic}, we will assume that $S, D$ satisfies the conditions of Construction~\ref{const:MixCharGood}, that is Basic Assumptions~\ref{BA} \ref{it:exlude_multiple_covers_-1curves}, \ref{it:a3degree} and Assumption~\ref{a:genericunram} all hold for $S, D$.

\begin{remark} 
We show in the Appendix (Theorem~\ref{thm:hyp:pc}) that del Pezzo surfaces with $d_S\ge3$ in characteristic $\geq 3$  satisfy  Assumption~\ref{a:genericunram}.

\end{remark}

Our next task is to show that $\bar{M}_{0,n}(\LiftS, \LiftDeg)^{\good}$ is smooth over $\Lambda$; we first need a lemma.

\begin{lemma}\label{lm:lifting_ratl_curves} Let $f_0:\P^1\to S$ be a morphism in $M_0(S,D)$.
 If $f_0$ is in $M_0^\unr(S,D)$, then $f_0$ lifts to a morphism $f\in M_0^\unr(\LiftS,\LiftDeg)$.
\end{lemma}

\begin{proof} Since $f_0$ is unramified, we have $\sN_{f_0}\cong \sO_{\P^1}(d-2)$. Since $d\ge1$, we have $H^1(\P^1, \sN_f)=0$. From the exact sequence
\[
0\to T_{\P^1}\to f_0^*T_S\to \sN_f\to0
\]
and the fact that $T_{\P^1}\cong \sO_{\P^1}(2)$, we see that $H^1(\P^1, f_0^*T_S)=0$. Applying Proposition~\ref{prop_deformation_theory}, we see that $f_0$ lifts to a morphism $f:\P^1_\Lambda\to  \LiftS$. By Lemma~\ref{lm:LiftingDelPezzo}\ref{lm:LiftingDelPezzo2}, we see that $f$ is in $M_0(\LiftS, \LiftDeg)$. Since the support of the cokernel of $df:f^*\Omega_{\LiftS/\Lambda}\to \Omega_{\P^1_\Lambda/\Lambda}$ is closed and has empty intersection with the special fiber $\P^1_k$. the fact that $\P^1_\Lambda$ is proper over $\Lambda$ implies that this cokernel is zero, hence $f$ is unramified.
\end{proof}

Let
 $M^\odp_{0,n}(\LiftS, \LiftDeg)^\good:=M^\odp_{0,n}(\LiftS, \LiftDeg)\cap  \bar{M}_{0,n}(\LiftS, \LiftDeg)^{\good}$.

\begin{proposition}\label{pr:MbarLiftgood_smooth}
$\bar{M}_{0,n}(\LiftS, \LiftDeg)^{\good}$ is smooth over $\Lambda$. Moreover, 
$\bar{M}_{0,n}(\LiftS, \LiftDeg)^{\good}$  is non-empty if and only if $\bar{M}_{0,n}(\LiftS_K, \LiftDeg_K)^{\good}$  is non-empty. 
\end{proposition}

\begin{proof} After replacing $\Lambda$ with an unramified extension $\Lambda\to \Lambda'$, with $\Lambda'$ a complete discrete valuation ring with residue field the algebraic closure of $k$, applying the base-change to $\Lambda'$ and changing notation, we may assume that $k$ is algebraically closed. 

We first consider the case in which $\bar{M}_{0,n}(\LiftS_K, \LiftDeg_K)^{\good}$ is empty. We claim that in this case, $\bar{M}_{0,n}(\LiftS, \LiftDeg)^{\good}$ is itself empty. Indeed, by the construction of $\bar{M}_{0,n}(\LiftS, \LiftDeg)^{\good}$, this is the same as asserting that 
$M^{\odp}_0(S,D)$ is empty. If not, take a $k$-point $f_0:\P^1\to S$ of $M^{\odp}_0(S,D)$. By Lemma~\ref{lm:lifting_ratl_curves}, $f_0$ lifts to a morphism $f\in M_0^\unr(\LiftS,\LiftDeg)$, and thus restricts to $f_K\in M_0^\unr(\LiftS_K,\LiftDeg_K)$. But then 
$M_0^\bir(\LiftS_K,\LiftDeg_K)\supset M_0^\unr(\LiftS_K,\LiftDeg_K)\neq\0$, hence by
 Theorem~\ref{prop:Good}\ref{it:goododp}, $\bar{M}_{0,n}(\LiftS_K, \LiftDeg_K)^{\good}\neq\0$, contrary to our assumption.    This proves the second assertion in the statement of the proposition.

Thus, if $\bar{M}_{0,n}(\LiftS_K, \LiftDeg_K)^{\good}$ is empty, then $\bar{M}_{0,n}(\LiftS, \LiftDeg)^{\good}$ is empty and hence is smooth over $\Lambda$, as desired.

We now assume that $\bar{M}_{0,n}(\LiftS_K, \LiftDeg_K)^{\good}$ is non-empty.
Since $\bar{M}_{0,n}$ commutes with base-change, the generic fiber of $\bar{M}_{0,n}(\LiftS, \LiftDeg)^{\good}$ is smooth  by construction and Theorem~\ref{prop:Good}, and is non-empty by assumption. Similarly, the special fiber of $\bar{M}_{0,n}(\LiftS, \LiftDeg)^{\good}$ is contained in $M_{0,n}^{\odp}(S,D)$ by construction, which is smooth by Lemma~\ref{lem:EvUnram}. Thus the structure map $\bar{M}_{0,n}(\LiftS, \LiftDeg)^{\good} \to \Spec \Lambda$ has smooth fibers. By \cite[01V8]{stacks-project}, it is then enough to show that $\bar{M}_{0,n}(\LiftS, \LiftDeg)^{\good}$ is flat over $\Lambda$.

Let $Z$ be the closure of the generic fiber $\bar{M}_{0,n}(\LiftS, \LiftDeg)^{\good}_K$ in 
$\bar{M}_{0,n}(\LiftS, \LiftDeg)^{\good}$. By Lemma~\ref{lemma:extending_subscheme_to_fiber}\ref{itlemma:extending_subscheme_to_fiber1}, it suffices to show that $Z=\bar{M}_{0,n}(\LiftS, \LiftDeg)^{\good}$. 

Clearly $Z$ is a closed subscheme of $\bar{M}_{0,n}(\LiftS, \LiftDeg)^{\good}$. We first show that $Z$ and $\bar{M}_{0,n}(\LiftS, \LiftDeg)^{\good}$ have the same support. 

To show this, it suffices to show that for each point $x_0$ in the special fiber $\bar{M}_{0,n}(\LiftS, \LiftDeg)^{\good}$, there is a point $x\in \bar{M}_{0,n}(\LiftS, \LiftDeg)^{\good}_K$ that specializes to $x_0$. In particular, if $\bar{M}_{0,n}(\LiftS, \LiftDeg)^{\good}_k$ is empty, there is nothing to prove, so assume that $\bar{M}_{0,n}(\LiftS, \LiftDeg)^{\good}_k$ is non-empty.

Choose a point 
 \[
 x_0:=(f_0, p_*) \in \bar{M}_{0,n}(\LiftS, \LiftDeg)^{\good}_k\subset M_{0,n}^{\odp}(S,D).
 \]
 By Lemma~\ref{lm:lifting_ratl_curves}, $f_0$ lifts to a morphism $f\in M_0^\unr(\LiftS,\LiftDeg)$. Since $\Lambda$ is a complete discrete valuation ring, each of the $k$-points $p_1,\ldots, p_n$ of $\P^1_k$ lift to $\Lambda$-points $\mathfrak{p}_1,\ldots, \mathfrak{p}_n$ of $\P^1_\Lambda$, giving us the lifting of $(f_0, p_*)$ to a point $(f, \mathfrak{p}_*)$ of $\bar{M}_{0,n}(\LiftS, \LiftDeg)$.
Because the closure in $\bar{M}_{0,n}(\LiftS, \LiftDeg)$ of the complement of $M_{0,n}(\LiftS_K, \LiftDeg_K)^{\good}$ in $\bar{M}_{0,n}(\LiftS_K, \LiftDeg_K)$ is disjoint from $M_{0,n}(\LiftS, \LiftDeg)^{\good}$ by construction, it follows that $x:=(f, \mathfrak{p}_*)$ is a point of $\bar{M}_{0,n}(\LiftS, \LiftDeg)^{\good}$. Thus $x_0$ is a specialization of $x$, as desired.

Finally, since $\bar{M}_{0,n}(\LiftS, \LiftDeg)^{\good}_k$ is smooth over $k$, the special fiber $\bar{M}_{0,n}(\LiftS, \LiftDeg)^{\good}_k$ of $\bar{M}_{0,n}(\LiftS, \LiftDeg)^{\good}$ is reduced. We have $Z_K=\bar{M}_{0,n}(\LiftS, \LiftDeg)^{\good}_K$ by construction.  Thus by Lemma~\ref{lemma:extending_subscheme_to_fiber}\ref{itlemma:extending_subscheme_to_fiber3}, we have $\bar{M}_{0,n}(\LiftS, \LiftDeg)^{\good}=Z$, completing the proof. 
\end{proof}

For the remainder of \S~\ref{section:positive_characteristic}, we assume that $\bar{M}_{0,n}(\LiftS, \LiftDeg)^{\good}$ is non-empty; equivalently (Proposition~\ref{pr:MbarLiftgood_smooth}), $\bar{M}_{0,n}(\LiftS_K, \LiftDeg_K)^{\good}$ is non-empty.

\subsection{Divisors and the double point locus for $\bar{M}_{0,n}(\LiftS, \LiftDeg)^{\good}$}

The morphism $\Liftev: \bar{M}_{0,n}(\LiftS, \LiftDeg)^{\good} \to \LiftS^n -\LiftA$ is proper because it is the pullback of a proper morphism, and quasi-finite by Theorem~\ref{prop:Good}\ref{it:evfinflatdom} and Lemma~\ref{lem:EvUnram}. Thus $\Liftev: \bar{M}_{0,n}(\LiftS, \LiftDeg)^{\good} \to \LiftS^n -\LiftA$ is finite.

Define $D_{\tac}$ to be the closure of $(D_{\tac})_{K}$ in $\bar{M}_{0,n}(\LiftS, \LiftDeg)^{\good}$. By  Lemma~\ref{lemma:extending_subscheme_to_fiber}, the intersection $D_{\tac} \cap \bar{M}_{0,n}(\LiftS, \LiftDeg)^{\good}_{k}$ has  codimension 1 in $\bar{M}_{0,n}(\LiftS, \LiftDeg)^{\good}_{k}$ and is flat over $\Lambda$. Since $\Liftev$ is finite, $\Liftev(D_{\tac} \cap \bar{M}_{0,n}(\LiftS, \LiftDeg)^{\good}_{k})$ is at least codimension $1$ in $S_k^n$, whence codimension $2$ in $\LiftS^n$.  Adding $\Liftev(D_{\tac} \cap \bar{M}_{0,n}(\LiftS, \LiftDeg)^{\good}_{k})$ to $\LiftA$, we may assume that $D_{\tac}$ has empty intersection with $\bar{M}_{0,n}(\LiftS, \LiftDeg)^{\good}_{k}$. Since $D_{\tac}$ is closed and codimension $1$ in a smooth scheme, $D_{\tac}$ is a relative Cartier divisor. We may similarly define $D_{\cusp}$ to be the closure of $(D_{\cusp})_{K}$ and assume that $D_{\cusp}$ is a Cartier divisor on $\bar{M}_{0,n}(\LiftS, \LiftDeg)^{\good}$ which has empty intersection with $\bar{M}_{0,n}(\LiftS, \LiftDeg)^{\good}_{k}$.

\begin{lemma}\label{lm:lift:divdetev}
The divisor of the section $\det d(\Liftev):\sO_{\bar{M}_{0,n}(\LiftS, \LiftDeg)^{\good}}\to \omega_{\Liftev}$ is $D_{\cusp}$, and thus $\det d(\Liftev)$ defines an isomorphism
\[
\det d(\Liftev): \sO_{\bar{M}_{0,n}(\LiftS, \LiftDeg)^{\good}} (D_\cusp) \to \omega_{\Liftev}
\] on $\bar{M}_{0,n}(\LiftS, \LiftDeg)^{\good} $.
\end{lemma}

\begin{proof}
The evaluation map is compatible with base change. Suppose that ${M}^{\odp}_{0,n}(\LiftS_k, \LiftDeg_k)$ is non-empty. On the special fiber, $\bar{M}_{0,n}(\LiftS, \LiftDeg)^{\good}_k$ is contained in ${M}^{\odp}_{0,n}(\LiftS_k, \LiftDeg_k)$. By Lemma~\ref{lem:EvUnram}, $\ev \cong \Liftev_k$ is \'etale on ${M}^{\odp}_{0,n}(\LiftS_k, \LiftDeg_k)$, so $\Liftev_k: \bar{M}_{0,n}(\LiftS_k, \LiftDeg_k)^{\good} \to S^{n}$ is flat and unramified. Since ramification can be checked on fibers of a smooth $\Lambda$-scheme, $\Liftev$ is unramified at points of $\bar{M}_{0,n}(\LiftS_k, \LiftDeg_k)^{\good}$ (\cite[tag 02G8]{stacks-project}). Since $\bar{M}_{0,n}(\LiftS, \LiftDeg)^{\good}$ is flat over $\Lambda$ (Proposition~\ref{pr:MbarLiftgood_smooth}), it follows from \cite[tag 039B]{stacks-project} that $\Liftev$ is flat on $\bar{M}_{0,n}(\LiftS_k, \LiftDeg_k)^{\good}$.  Thus $\Liftev$ is \'etale on $\bar{M}_{0,n}(\LiftS_k, \LiftDeg_k)^{\good}$, whence $\Liftev$ is \'etale over an open neighborhood $U$ of the special fiber $\bar{M}_{0,n}(\LiftS,\LiftDeg)_k^\good \subset \bar{M}_{0,n}(\LiftS,\LiftDeg)^\good.$ Thus $\Div\det d(\Liftev)\cap U = 0$.

We recall that $\LiftS_K,\LiftDeg_K$ satisfy Basic Assumptions~\ref{BA} by Lemma~\ref{lm:BA23ImpliesBA123}. Over the open set of $\bar{M}_{0,n}(\LiftS_K,\LiftDeg_K)^\good$ given by the generic fiber of 
$\bar{M}_{0,n}(\LiftS,\LiftDeg)^\good$, the proposition follows from Theorem~\ref{prop:EvRam}

If ${M}^{\odp}_{0,n}(\LiftS_k, \LiftDeg_k)$ is empty, then we need only check on 
$$\bar{M}_{0,n}(\LiftS,\LiftDeg)^\good_K\subset \bar{M}_{0,n}(\LiftS_K,\LiftDeg_K)^\good,$$ which follows as above from Theorem~\ref{prop:EvRam}.
\end{proof}

Forgetting the last marked point defines a map $\bar{M}_{0,n+1}(\LiftS, \LiftDeg) \to \bar{M}_{0,n}(\LiftS, \LiftDeg)$ from the universal curve to the moduli space. Define $\bar{X}_{0,n}(\LiftS, \LiftDeg)^{\good} \subset \bar{M}_{0,n+1}(\LiftS, \LiftDeg)$ to be the inverse image of $\bar{M}_{0,n}(\LiftS, \LiftDeg)^{\good}$.

Define the double point locus $\tilde{\pi}: \Liftdpl \to \bar{M}_{0,n}(\LiftS, \LiftDeg)^{\good}$ using the natural analogue of Definition~\ref{def:dpl}.

\begin{lemma}\label{lm:Liftdpl_good} Let $M^\odp_{0,n}(\LiftS, \LiftDeg)^\good:=\bar{M}_{0,n}(\LiftS, \LiftDeg)^{\good}\cap M^\odp_{0,n}(\LiftS, \LiftDeg)$. The double point locus $\Liftdpl$ satisfies the following.
\begin{enumerate}
\item \label{item:Liftdpl_smooth} By possibly enlarging $\LiftA$, we may take $\Liftdpl$ to be smooth over $\Lambda$.
\item \label{item:tildepi_has_disc:lm:Liftdpl_good} The map $\tilde{\pi}$ is finite and flat.
\item \label{item:tildepi_etale_special_fiber} The map $\tilde{\pi}$ is  \'etale over $M^\odp_{0,n}(\LiftS, \LiftDeg)^\good$ and
$M^\odp_{0,n}(\LiftS, \LiftDeg)^\good$ is an open neighborhood of $\bar{M}_{0,n}(\LiftS, \LiftDeg)^{\good}_k$ in $\bar{M}_{0,n}(\LiftS, \LiftDeg)^{\good}$. \end{enumerate}
\end{lemma}

\begin{proof}
At points of the generic fiber  $\Liftdpl_K$, it follows from  Corollary~\ref{Cor:dpl_sm_ram_pi} that, after enlarging $\LiftA$ if necessary,  $\Liftdpl$ is smooth over $\Lambda$; similarly, \ref{item:tildepi_has_disc:lm:Liftdpl_good} holds for 
\[
\tilde{\pi}_K: \Liftdpl_K \to \bar{M}_{0,n}(\LiftS, \LiftDeg)^{\good}_K.
\]
The fact that $M^\odp_{0,n}(\LiftS, \LiftDeg)^\good$ is an open neighborhood of $\bar{M}_{0,n}(\LiftS, \LiftDeg)^{\good}_k$ in $\bar{M}_{0,n}(\LiftS, \LiftDeg)^{\good}$  follows from the construction of $\bar{M}_{0,n}(\LiftS, \LiftDeg)^{\good}$.

Let $M^\odp_{0,n}(\LiftS, \LiftDeg)^\good:=\bar{M}_{0,n}(\LiftS, \LiftDeg)^{\good}\cap M^\odp_{0,n}(\LiftS, \LiftDeg)$ and let $\Liftdpl^\odp$ be the restriction of $\Liftdpl$  over the open subscheme $M^\odp_{0,n}(\LiftS, \LiftDeg)^\good$. Define $\uc{n}^{\odp, \good}\to M^\odp_{0,n}(\LiftS, \LiftDeg)^\good$ similarly. Since $M^\odp_{0,n}(\LiftS, \LiftDeg)^\good$ is an open neighborhood of the special fiber $\bar{M}_{0,n}(\LiftS, \LiftDeg)^{\good}_k$ in $\bar{M}_{0,n}(\LiftS, \LiftDeg)^{\good}$, to complete the proof,  it suffices to  prove 
\ref{item:Liftdpl_smooth}, \ref{item:tildepi_has_disc:lm:Liftdpl_good} and \ref{item:tildepi_etale_special_fiber}  for the restriction 
\[
\tilde{\pi}^\odp: \Liftdpl^\odp \to M^\odp_{0,n}(\LiftS, \LiftDeg)^{\good} 
\]
of  $\tilde{\pi}$.  

Since  $\Liftdpl^\odp$ is closed in $\uc{n}^{\odp,\good}\times_{M^\odp_{0,n}(\LiftS, \LiftDeg)^\good}\uc{n}^{\odp,\good}$ and
$$\uc{n}^{\odp,\good}\times_{M^\odp_{0,n}(\LiftS, \LiftDeg)^\good}\uc{n}^{\odp,\good}\to M^\odp_{0,n}(\LiftS, \LiftDeg)^\good$$ is proper,  this shows that $\Liftdpl^\odp\to M^\odp_{0,n}(\LiftS, \LiftDeg)^\good$ is proper.

By Lemma~\ref{lemma_unramified_doublelocus}, 
\[
\Liftdpl^\odp\cap \Delta_{\uc{n}^{\odp,\good}}=\0
\]
the intersection taking place in $\uc{n}^{\odp,\good}\times_{M^\odp_{0,n}(\LiftS, \LiftDeg)^\good}\uc{n}^{\odp,\good}$.
Thus 
\[
\Liftdpl^\odp= (\ev^{\odp,\good} \times_{M^\odp_{0,n}(\LiftS, \LiftDeg)^\good} \ev^{\odp,\good})^{-1}(\Delta_{\LiftS/M^\odp_{0,n}(\LiftS, \LiftDeg)^\good})]\setminus 
\Delta_{\uc{n}^{\odp,\good}}
\]
where $\ev^{\odp,\good}:\uc{n}^{\odp,\good}\to M^\odp_{0,n}(\LiftS, \LiftDeg)^\good\times_\Lambda \LiftS$ is the universal map and $\Delta_{\LiftS/M^\odp_{0,n}(\LiftS, \LiftDeg)^\good}$ is the relative diagonal.

Next we recall that $M^\odp_{0,n}(\LiftS, \LiftDeg)$ is smooth over $\Lambda$ and $\uc{n}^\odp$ is smooth over $M^\odp_{0,n}(\LiftS, \LiftDeg)$, hence  $\uc{n}^{\odp,\good}\times_{M^\odp_{0,n}(\LiftS, \LiftDeg)^\good}\uc{n}^{\odp,\good}$ is smooth over $\Lambda$. To prove 
\ref{item:Liftdpl_smooth}, it thus suffices to show that $\ev^{\odp,\good} \times_{M^\odp_{0,n}(\LiftS, \LiftDeg)^\good} \ev^{\odp,\good}$ is transverse to the inclusion $\Delta_{\LiftS/M^\odp_{0,n}(\LiftS, \LiftDeg)^\good}\hookrightarrow M^\odp_{0,n}(\LiftS, \LiftDeg)^\good\times_\Lambda \LiftS\times_\Lambda \LiftS$, at points away from $\Delta_{\uc{n}^{\odp,\good}}$. This tranversality follows immediately from the definition of $M^\odp_{0,n}(\LiftS, \LiftDeg)$.

Similarly, this transversality implies that $\Liftdpl^\odp\to M^\odp_{0,n}(\LiftS, \LiftDeg)^\good$ is  a smooth morphism and that  $\Liftdpl^\odp$ has codimension two in $\uc{n}^{\odp,\good}\times_{M^\odp_{0,n}(\LiftS, \LiftDeg)}\uc{n}^{\odp,\good}$. Since $\Liftdpl^\odp\to  M^\odp_{0,n}(\LiftS, \LiftDeg)^\good$ is thus smooth, proper and of relative dimension zero, it follows that 
$\Liftdpl^\odp\to  M^\odp_{0,n}(\LiftS, \LiftDeg)^\good$ is finite and \'etale. This completes the proof.

\end{proof}

As before, we have $\disc_{\tilde{\pi}}: \sO_{\bar{M}_{0,n}(\LiftS, \LiftDeg)^{\good}} \to (\det \tilde{\pi}_* \sO_{\Liftdpl})^{\otimes -2}$.  (See Section~\ref{Section:orienting_ev}. By Lemma~\ref{lm:Liftdpl_good}~\ref{item:tildepi_has_disc:lm:Liftdpl_good},\ref{item:tildepi_etale_special_fiber},  $\tilde{\pi}$ admits the claimed discriminant, and $\det \tilde{\pi}_* \sO_{\Liftdpl}$ is a line bundle.)

\begin{lemma}\label{lm:lift:Disc} The divisor of $\disc_{\tilde{\pi}}$ is computed
\[
\Div~\disc_{\tilde{\pi}}=1\cdot D_\cusp+2\cdot D_\tac
\]
and thus $\disc_{\tilde{\pi}}$ defines an isomorphism
\[
\disc_\pi:\sO_{\bar{M}_{0,n}(\LiftS, \LiftDeg)^{\good}} (D_\cusp)\to (\det \tilde{\pi}_*\sO_{\Liftdpl})(-D_\tac)^{\otimes -2}
\]
\end{lemma}

\begin{proof}
By Lemma~\ref{lm:Liftdpl_good}~\ref{item:tildepi_etale_special_fiber}, $\tilde{\pi}$ is \'etale over an open neighborhood $U$ of the special fiber $\bar{M}_{0,n}(\LiftS,\LiftDeg)_k^\good \subset \bar{M}_{0,n}(\LiftS,\LiftDeg)^\good.$ Thus $\Div~\disc_{\tilde{\pi}}\cap U = 0$. Over the open set of $\bar{M}_{0,n}(\LiftS_K,\LiftDeg_K)^\good$ given by the generic fiber, we may apply Theorem~\ref{prop:Disc}, which proves the claim.
\end{proof}

\begin{theorem} \label{thm:Orient_pos_char_1} Let $\sL$ be the invertible sheaf on $\bar{M}_{0,n}(\LiftS,\LiftDeg)^\good$ given by
\[
\sL=[\det \tilde{\pi}_*\sO_{\Liftdpl}(-D_\tac)]^{\otimes -1}
\]
Then the composition $\det d\Liftev~\circ \disc_{\tilde{\pi}}^{-1}:\sL^{\otimes 2}\to \omega_{\Liftev}$ is an isomorphism on $\bar{M}_{0,n}(\LiftS, \LiftDeg)^{\good} $.
\end{theorem}

\begin{proof}
Follows immediately from Lemmas~\ref{lm:lift:divdetev} and \ref{lm:lift:Disc}.
\end{proof}

\subsection{The symmetrized evaluation map in positive characteristic}\label{subsection:pos_char_symm_ev_map}
We continue to assume that $\bar{M}_{0,n}(\LiftS,\LiftDeg)^\good$ is non-empty.

Just as in Section~\ref{Section:symmetrized_moduli_space_kchar0}, let $\LiftS^{n}_{0}$ denote the complement of the relative diagonals in $\LiftS^{n}$, where the product is taken over $\Lambda$. The  symmetric group $\mfS_n$ acts freely on $\LiftS^{n}_{0}$, $\bar{M}_{0,n}(\LiftS,\LiftDeg)$, and the universal curve $\bar{M}_{0,n+1}(\LiftS,\LiftDeg)$. By enlarging the closed subset $\LiftA$ of \eqref{eq:def:liftA} to be invariant under $\mfS_n$, we likewise obtain a free action on $ \bar{M}_{0,n}(\LiftS, \LiftDeg)^{\good}$, $\bar{X}_{0,n}(\LiftS, \LiftDeg)^{\good}$ and the double point locus $\tilde{\pi}: \Liftdpl \to \bar{M}_{0,n}(\LiftS, \LiftDeg)^{\good}$.

Proceedings as in Section~\ref{Section:symmetrized_moduli_space_kchar0}, we take quotients and form the following commutative diagram with Cartesian squares and vertical maps finite \'{e}tale quotient maps:

\begin{equation}
\xymatrix{
\Liftdpl\ar[r]^(.35)\Liftpi\ar[d] &
\bar{M}_{0,n}(\LiftS,\LiftDeg)^{\good}\ar[d]\ar[r]^(.6){\Liftev}&\LiftS^n_0 \ar[d] \\
\Liftdpl_\mfS\ar[r]^(.35){\Liftpi_\mfS} &
\bar{M}_{0,n,\mfS}(\LiftS,\LiftDeg)^{\good}\ar[r]^(.6){\Liftev_\mfS}&\Sym^n_0\LiftS.
}
\end{equation}

For $\square\in\{\cusp, \tac, \trip\}$, we let $D^\mfS_{\square}$ denote the reduced image of $D_{\square}$ in $\bar{M}^{\good}_{0,n,\mfS}(\LiftS,\LiftDeg)$.

\begin{theorem}\label{thm:orient_sym_ev_pos_char}
Let $k$ be a perfect field of characteristic $p>3$. Let $S$ be a del Pezzo surface over $k$ with effective Cartier divisor $D$, satisfying Basic Assumptions~\ref{BA} \ref{it:a3degree} and Assumption~\ref{a:genericunram}.

\begin{enumerate}
\item
The canonical section $\det d\Liftev_\mfS:\sO_{\bar{M}^{\good}_{0,n,\mfS}(\LiftS,\LiftDeg)}\to \omega_{\Liftev_\mfS}$ has divisor $1\cdot D^\mfS_\cusp$ and induces an isomorphism
\[
\det d\Liftev_\mfS:\sO_{\bar{M}^{\good}_{0,n,\mfS}(\LiftS,\LiftDeg)}(D^\mfS_\cusp)\to \omega_{\Liftev_\mfS}.
\]
\item
The divisor of $\disc_{\Liftpi_\mfS}:\sO_{\bar{M}^{\good}_{0,n,\mfS}(\LiftS,\LiftDeg)}\to [\det\Liftpi_{\mfS *}\sO_{\bar\sD_\mfS}]^{\otimes -2}$ is $D^\mfS_\cusp+2\cdot D_\tac^\mfS$ and induces an isomorphism
\[
\disc_{\Liftpi_\mfS}:\sO_{\bar{M}^{\good}_{0,n, \mfS}(\LiftS,\LiftDeg)}(D^\mfS_\cusp)\to [\det\Liftpi_{\mfS *}\sO_{\bar\sD_{\mfS}}(-D_\tac)]^{\otimes -2}
\]
\item
Letting $\sL^\mfS:=\det^{-1}\Liftpi_{\mfS *}\sO_{\bar\sD_\mfS}(-D_\tac)$, we have the isomorphism
\[
\det d\Liftev_\mfS\circ \disc_{\Liftpi_\mfS}^{-1}:(\sL^\mfS)^{\otimes 2}\to \omega_{\Liftev_\mfS}.
\]
\end{enumerate}
\end{theorem}
\begin{proof}
	The proof is essentially the same as the proof of Theorem~\ref{thm:EvStructure}. Replace the uses Theorem~\ref{prop:EvRam} and Theorem~\ref{prop:Disc} with Lemmas \ref{lm:lift:divdetev} and \ref{lm:lift:Disc}.
\end{proof}

\subsection{Twists of the evaluation map in positive characteristic}\label{subsection:twists_of_evaluation_map_positive_char}
We continue to assume that $\bar{M}_{0,n}(\LiftS,\LiftDeg)^\good$ is non-empty.

As in Section~\ref{Section:twisting_deg_map_char0}, let $\sigma= (L_1, \ldots, L_r)$ be an $r$-tuple of subfields $L_i \subset \ksep$ containing $k$ for $i=1,\ldots, r$ subject to the requirement that $\sum_{i=1}^k [L_i : k] = n$. The reduction map defines an equivalence between the category of finite \'{e}tale extensions of $\Lambda$ and the analogous category over $k$ \cite[Expos\'e IX 1.10]{sga1}. Thus the twisting construction from Section~\ref{Section:twisting_deg_map_char0} lifts over $\Lambda$.

Let $\Lambda \subset \Lambda^{\text{unr}}$ be the extension corresponding to the separable closure $\ksep$ of $k$. ($\ksep$ is ind-finite.) Let $\LiftA \hookrightarrow  \LiftS^n$ be a closed set as constructed in \eqref{eq:def:liftA}. By potentially enlarging $\LiftA$ we may assume that $\LiftA$ is invariant under the action of symmetric group $\mfS_n$.  Proceeding as in Section~\ref{Section:twisting_deg_map_char0}, we obtain a $\Lambda$-map
\[
	 \Liftev_{\sigma}: \M_{0,n}(\LiftS,\LiftDeg)^{\good}_\sigma \to (\LiftS^n\setminus \LiftA)_\sigma
\] with special fiber $\ev_{\sigma}$ and which is canonically identified with $\Liftev$ after base change to $\Lambda^{\text{unr}}$. We similarly twist the double point locus $\tilde{\pi}: \Liftdpl \to \bar{M}_{0,n}(\LiftS, \LiftDeg)^{\good}$ producing a $\Lambda$-map
\[
\Liftpi_{\sigma}: \Liftdpl_{\sigma} \to \M_{0,n}(\LiftS,\LiftDeg)^{\good}_\sigma
\] We again have a forgetful map $\bar{M}_{0,n}(\LiftS,\LiftDeg)^{\good}_{\sigma} \to \bar{M}_{0}(\LiftS,\LiftDeg)$. For $\square\in\{\cusp, \tac \}$, we let $D_{\square,\sigma}$ denote the preimage of $D_\square$ under this map.

\begin{theorem}
Let $k$ be a perfect field of characteristic $p>3$. Let $S$ be a del Pezzo surface over $k$ with effective Cartier divisor $D$, satisfying Assumptions~\ref{BA} \ref{it:a3degree} and Assumption~\ref{a:genericunram}.
\begin{enumerate}
\item \label{item:thm:rel_or_twist_ev_charp:smooth} $\ev_{\sigma}: \bar{M}_{0,n}(\LiftS,\LiftDeg)^{\good}_{\sigma} \to (\LiftS^n)_{\sigma}$ is a map between smooth $\Lambda$-schemes.
\item \label{item:thm:rel_or_twist_ev_charp:detdev}
The canonical section $\det d\Liftev_{\sigma}:\sO_{\bar{M}^{\good}_{0,n}(\LiftS,\LiftDeg)_{\sigma}}\to \omega_{\Liftev_{\sigma}}$ has divisor $1\cdot D_{\cusp, \sigma}$ and induces an isomorphism
\[
\det d\Liftev_{\sigma}:\sO_{\bar{M}^{\good}_{0,n}(\LiftS,\LiftDeg)_{\sigma}}(D_{\cusp, \sigma})\to \omega_{\Liftev_{\sigma}}.
\]
\item \label{item:thm:rel_or_twist_ev_charp:disc_pi}
The divisor of $\disc_{\Liftpi_{\sigma}}:\sO_{\bar{M}^{\good}_{0,n}(\LiftS,\LiftDeg)_{\sigma}}\to [\det(\Liftpi_\sigma)_*\sO_{\Liftdpl_{\sigma}}]^{\otimes -2}$ is $$D_{\cusp,\sigma}+2\cdot D_{\tac,\sigma}$$ and induces an isomorphism
\[
\disc_{\Liftpi_{\sigma}}:\sO_{\bar{M}^{\good}_{0,n}(\LiftS,\LiftDeg)_{\sigma}}(D_{\cusp,\sigma})\to [\det (\Liftpi_{\sigma})_*\sO_{\Liftdpl_{\sigma}}(-D_{\tac,\sigma})]^{\otimes -2}
\]
\item \label{item:thm:rel_or_twist_ev_charp:rel_or}
Letting $\sL_{\sigma}:=[\det(\Liftpi_\sigma)_*\sO_{\Liftdpl_{\sigma}}(-D_\tac)]^{-1}$, we have the isomorphism
\[
\det d\Liftev_{\sigma}\circ \disc_{\Liftpi_{\sigma}}^{-1}:(\sL_{\sigma})^{\otimes 2}\to \omega_{\Liftev_{\sigma}}.
\]
\end{enumerate}
\end{theorem}
\begin{proof}
	The proof is parallel to the proof of Theorem~\ref{thm:rel_or_twist_ev_char0}.
\end{proof}

\appendix

\section{Unramified maps in positive characteristic}\label{Section:unramified_maps_in_any_char}\label{Appendix:A} In this section we 
let $S$ be a del Pezzo surface over a field $k$ of characteristic greater than $3$ with $d_S := K_S \cdot K_S \geq 3$, and we prove the following result.

\begin{theorem}\label{thm:hyp:pc}
Let $D \in \Pic(S)$ be effective. If $M^\bir_0(S, D)$ is non-empty, then $M^\bir_0(S, D)$ is irreducible,  and there is a geometric point $u\in M^\bir_0(S, D)$ with $u$ unramified. 
\end{theorem}

\begin{remark}\label{rem:hyp:pc} It follows from Lemma~\ref{lm:odp_in_unr_in_M_open} that if $M^\bir_0(S, D)$ is irreducible  and there is a geometric point $u\in M^\bir_0(S, D)$ with $u$ unramified, then there is a dense open subset of $M^\bir_0(S, D)$ consisting of unramified maps.

In what follows, we say that a general $f\in M_0(S,D)$ has property $P$ to mean that property $P$ holds for all geometric points in dense open subset of $M_0(S,D)$.
\end{remark}

Recall that $D\in \Pic(S)$ is {\em nef} if for every reduced, irreducible curve $C$ on $S$, the intersection degree $D\cdot C$ is non-negative.

\begin{lemma}\label{lm:nef} Let $D \in Pic(S)$ be effective and let $d=-K_S\cdot D$. If  $d\ge2$ and $M^\bir_0(S, D)$ is non-empty, then $D$ is nef.
\end{lemma}

\begin{proof} Take $f:\P^1\to S$ a geometric point of $M^\bir_0(S, D)$, let $D_0\subset S$ be the image curve $f(\P^1)$ and let $C$ be a  reduced, irreducible curve on $S$. Since $f$ is birational, $D_0$ is also reduced and  irreducible and the class $[D_0]\in \Pic(S)$ is $f_*([\P^1])=D$. Thus, if $C\neq D_0$, then $C\cdot D=C\cdot D_0\ge0$. Also, since $\P^1\to D_0$ is the normalization of $D_0$, it follows that $D_0$ has arithmetic genus $p_a(D_0)\ge g(\P^1)=0$. By the adjunction formula, we have
\[
D_0\cdot(D_0+K_S)=2p_a(D_0)-2\ge -2.
\]
Since $D\cdot (-K_S)\ge2$, we thus have
\[
D_0\cdot D=D_0\cdot D_0\ge -2+(-K_S\cdot D)\ge 0
\]
so $D$ is nef.
\end{proof}

 Following \cite{BLRT}, we say that  a reduced  irreducible curve $D_0$ on $S$ is a $-K_S$-conic if $-K_S\cdot D_0=2$. Since $d_S\ge 3$, $-K_S$ embeds $S$ in a projective space $\P^{d_S}$ and under this embedding a $-K_S$-conic is an irreducible degree two curve, hence a smooth conic in some plane $\P^2\subset \P^{d_S}$.

The following is a consequence of Theorem~1.5 of~\cite{BLRT}. 
\begin{theorem}\label{thm:BLRT1.5}
Let $S$ be a del Pezzo surface of degree $d_S \geq 3$ over a  field $k$ of characteristic $p > 3.$ Let $D \in \Pic(S)$ satisfy $d := -K_S \cdot D \geq 2$ and suppose that  $M_0(S,D)^\bir\neq\0$. Then $M_0(S,D)$  is irreducible, the general point is a free, birational map, and the evaluation map $\ev : M_{0,1}(S,D) \to S$ is dominant.
\end{theorem}

\begin{proof}
By Lemma~\ref{lm:nef}, $D$ is nef. By \cite[Theorem 1.5]{BLRT}, if $d\ge 3$ and $D$ is not a multiple of a $-K_S$-conic, then $M_0(S,D)$  is irreducible and a general point $f:\P^1\to S$ is a free, birational map. In particular, the normal sheaf $\sN_f$ has torsion-free quotient isomorphic to $\sO(e)$ with $e\ge0$.  This in turn implies that the evaluation map from the universal curve $\ev : M_{0,1}(S,D) \to S$ has surjective differential at a general point $x\in M_{0,1}(S,D)$ lying over $f$, hence $\ev : M_{0,1}(S,D) \to S$ is dominant.

If $d = 2$ and $M_0(S,D)^\bir\neq\0$, it follows that $D$ is the class of a $-K_S$-conic. So, it remains to consider the case $D = m[D_0]$ for a $-K_S$-conic $D_0$ with $m \geq 1.$ By the proof of~\cite[Theorem 1.5]{BLRT}, the space $M_0(S,D)$ is irreducible and a general point $f:\P^1\to S$ is an $m$-fold cover of a smooth conic. If $m \geq 2,$ then $M_0(S,D)^\bir = \0,$ so we are done. If $m = 1,$ the general point $f$ is an isomorphism onto its image. So, $\sN_f \cong \sO_{\P^1}$, whence $f$ is free and the same argument as above shows that $\ev : M_{0,1}(S,D) \to S$ is dominant.
 \end{proof}

\begin{lemma}\label{lm:lowdegree} Suppose that $M_0(S, D)^\bir\neq\0$. Let $d:=-K_S\cdot D$ and suppose that $2\le d\le 3$. Then a general $f\in M_0(S, D)^\bir$ is unramified.
\end{lemma}

\begin{proof}  By Lemma~\ref{lm:odp_in_unr_in_M_open} and Theorem~\ref{thm:BLRT1.5}, we need only find a single unramified $f$ in $M_0(S, D)^\bir$.

For $d=2$, $f$ being birational implies that $f(\P^1)$ is a $-K_S$-conic on $S$. Thus $f(\P^1)$ is smooth and $f:\P^1\to f(\P^1)$ is an isomorphism.

For $d=3$, suppose that $f:\P^1\to S$ is ramified and birational to its image $C:=f(\P^1)$. Let $p\in \P^1$ be a point of ramification of $f$ and let $q=f(p)$. We take the canonical embedding $S\subset \P^{d_S}$, so $C$ is a  degree 3 rational curve in $\P^{d_S}$ with singular point $q$. We claim that $C$ spans a plane $P$ and as a curve on $P\cong \P^2$, $C$ is a cubic curve with an ordinary cusp. Indeed, we may choose two additional points $p_1, p_2$ on $\P^1\setminus\{p\}$ so that $q$, $f(p_1)$, $f(p_2)$ do not lie on a line in $\P^{d_S}$ Let $\Pi$ be a hyperplane in $\P^{d_S}$ containing $q$, $f(p_1)$ and $f(p_2)$. If $\Pi$ does not contain $C$, then since $q$ is a singular point of $C$, these three points of intersection contribute at least 4 to the total intersection degree $\Pi\cdot C=3$, which is impossible. Thus $C\subset \Pi$. We repeat this argument, finding in the end that $C$ is contained in a plane $P$, as claimed. Since $f$ is ramified at $p$, it follows that $q$ is a cusp on $C$, and since $C$ is a plane cubic and the characteristic is $>3$, it follows that $q$ is an ordinary cusp on $C$, with local equation of the form $y^2=x^3$ in  an open affine plane $\A\subset P$ containing~$q$. 

Since $P$ intersects $S$ in the curve $C$ with singular point $q$, it follows that $P$ is tangent to $S$ at $q$, so there is a local analytic isomorphism $(S,q)\cong (\A^2,(0,0))$. This gives us  local analytic coordinates $x, y$ on $S$ at $q$ such that $C$ has the equation $y^2=x^3$ and $f$ is given in these coordinates by $f(t)=(t^2, t^3)$ in an analytic neighborhood of $p=0\in \A^1\subset \P^1$, with local analytic coordinate $t$. By Lemma~\ref{lem:CuspRam2}, there is a deformation $f_\epsilon$ of  $f:\P^1\to S$  such that 
\[
f_\epsilon(t)=(t^2+2a\epsilon, t^3+3a\epsilon t) \mod \epsilon^2
\]
Thus 
\[
df_\epsilon(d/dt)=(2t+\epsilon^2g)\cdot\del/\del x +(3t^2+3a\epsilon+\epsilon^2h)\cdot \del/\del y
\]
for some $g,h\in k[[t,\epsilon]]$. Thus $df_\epsilon((d/dt)=0\Rightarrow\epsilon=0$, i.e.,   $f_\epsilon$ is unramified in an $\epsilon$-adic neighborhood of $p$, for $\epsilon\neq0$.  

Alternatively, letting $x_\epsilon=x-2a\epsilon$, $y_\epsilon=y$, the curve $C_\epsilon:=f_\epsilon(\P^1)$ has local equation near $q$ of the form
\[
y_\epsilon^2=x_\epsilon^3-6a\epsilon x_\epsilon^2\mod \epsilon^2
\]
which has an ordinary double point at $(x_\epsilon, y_\epsilon)=(0,0)$, since $\Char k>3$.

Since $M_0(S,D)^\bir$ is open in $M_0(S,D)$,  we have found an unramified map in $M_0(S,D)^\bir$, which completes the proof. \end{proof}

\begin{lemma}\label{lm:transverse}
Suppose that for $i=1,2$,  $M_0(S, D_i)^\bir\neq\0$ and a general $f_i\in M_0(S, D_i)^\bir$ is unramified. For general  $f_i$ in $M_0(S, D_i)^\bir$,  and $d_i= -K_S \cdot D_i \geq 2$,  the maps $f_1, f_2$  intersect transversely, that is for any $p_1,p_2$ such that $f_1(p_1) = f_2(p_2)=q$, we have 
 $df_1(T_{p_1}\P^1) +  df_2(T_{p_2}\P^1) = T_{S,q}$.
\end{lemma}
\begin{proof} By Theorem~\ref{thm:BLRT1.5}, we may assume that both maps $f_i$ are free. As the corresponding evaluation maps are dominant, we may assume that $f_1(\P_1)\cap f_2(\P_2)$ is a finite set.  

Suppose first that $d_1\ge 3$ and that $df_1(T_{p_1}\P^1) =  df_2(T_{p_2}\P^1)$. Since $f_1$ is free and unramified, we have $\sN_{f_1}\cong \sO_{\P^1}(d_1-2)$. Since $d_1\ge3$, we have $H^1(\P_1, \sN_{f_1})=0$ and there is a section $s$ of $\sN_{f_1}$ that has a zero of order one at $p_1\in \P_1$. Letting $f_{1\epsilon}$ be the deformation of $f_1$ (modulo automorphisms of $\P_1$) corresponding to $s$, we see that modulo $\epsilon^2$, we have $f_{1\epsilon}(p_1)=f_1(p_1)=q$, but $df_{1\epsilon}(T_{p_1}\P^1)\neq df_1(T_{p_1}\P^1)$. Since we have taken $f_1$, $f_2$ general, this implies that we had  $df_1(T_{p_1}\P^1) \neq  df_2(T_{p_2}\P^1)$ to begin with, which completes the proof in case one of $d_1, d_2$ is at least $3$.

Suppose $d_1=d_2=2$. Since both $f_i$ are unramified, this implies that $C_i:=f_i(\P_i)$ are both  $-K_S$-conics, hence are plane conic curves on $S\subset \P^{d_S}$, after taking the anti-canonical embedding of $S$. Let $P_i\subset  \P^{d_S}$ be the plane spanned by $C_i$, $i=1,2$.

Suppose first that $P_1\neq P_2$. We may therefore take a general hyperplane $\Pi_1\subset \P^{d_S}$ containing $P_1$, but not containing $C_2$. Since $\Pi_1$ is general, we have
\[
\Pi_1\cdot S=C_1+D_1
\]
for some effective 1-dimensional algebraic cycle $D_1$ on $S$; since $\Pi_1$ does not contain $C_2$, $D_1\cap C_2$ is a finite set of points of $S$. By the associativity of intersection product, we have
\[
(C_1+D_1)\cdot_S C_2=\Pi_1\cdot_{\P^{d_S}} C_2=(\Pi_1\cdot P_2)\cdot_{P_2}C_2
\]
Since $\Pi_1\cdot P_2$ is a line in $P_2$, this implies that $(C_1+D_1)\cdot_S C_2=2$, so $0\le C_1\cdot C_2\le 2$. Since $q$ is a point on $C_1\cap C_2$, we thus have $1\le C_1\cdot C_2\le2$.

If $C_1\cdot C_2=1$, then $C_1$ and  $C_2$ intersect transversely at $q$, and we are done. If $C_1\cdot C_2=2$, and there are two points in the intersection, again $C_1$ and $C_2$ intersect transversally at $q.$ Otherwise, we may choose local analytic coordinates $x,y$ on $S$ at $q$ so that $C_1$ is defined by $y=0$ and $C_2$ is defined by $y=x^2+...$. We have $\sN_{f_1}\cong \sO_{\P^1}$, so we may take a nowhere vanishing section $s$ of $\sN_{f_1}$ and deform $f_1$ according to this section to the map $f_{1\epsilon}$. In our local coordinates, this gives the equation of $C_{1\epsilon}:=f_{1\epsilon}(\P_1)$ as $y=\lambda\epsilon+...$ for some constant $\lambda\neq 0$. Thus $C_{1\epsilon}$ and $C_2$ intersect transversely at two points, since $\Char k\neq 2$. Again, since $f_1, f_2$ were assumed to be general, this means that $C_1$ and $C_2$ intersected transversely at $q$ to begin with, which settles the case $P_1\neq P_2$.

To finish, suppose that $P_1=P_2$; call this common plane $P$. Then $C_1, C_2$ are two smooth conics in the plane $P$ intersecting at a finite set of points, so the intersection multiplicity $m(C_1\cdot C_2, q)$ satisfies $1\le m(C_1\cdot C_2, q)\le 4$. If $m(C_1\cdot C_2, q)=1$ we are done. In case $2\le m(C_1\cdot C_2, q)\le 4$, we again take local analytic coordinates $x, y$ at $q$ on $S$ so that $C_1$ is defined by $y=0$ and $C_2$ is defined by $y=x^m+...$ with $m\in \{2,3,4\}$. We make a deformation of $f_1$ as above, and noting that $\Char k>4$, we find that $C_{1\epsilon}$ intersects $C_2$ transversely at all intersection points near (in the $\epsilon$-adic topology) to $q$,  which completes the proof.
\end{proof}

\begin{lemma}\label{lm:smoothing} Let $(U, u)$ be a smooth pointed curve over $k$, and let $(\sC, c)$ be a reduced pointed surface over $k$. Let $\pi:\sC\to U$ be a proper, flat, surjective morphism such that $\pi(c)=u$, and $\sC\setminus \{c\}$ is smooth over $U$. In addition, we assume that the fiber $\sC_u:=\pi^{-1}(u)$ is the union of two smooth curves $C_1, C_2$ joined at the single point $c$,  $\sC_u:=C_1\cup_{c}C_2$, and that $\sC_u$ has an ordinary double point at $c$. 

Let $T$ be a smooth finite-type  $k$-scheme and let $f:\sC\to T$ be a morphism. Suppose that the respective restrictions of $f$, $f_1:C_1\to T$, and   $f_2:C_2\to T$,   are  unramified, and 
\[
df_1(T_{c}C_1) \cap  df_2(T_{c}C_2) =0, 
\]
the intersection taking place in $T_{T,f(c)}$. Then there is an open neighborhood $U'\subset U$ of $u$ such for all $v\in U'\setminus\{u\}$, the restriction $f_v:\sC_v\to T$ of $f$ to the fiber $\sC_v:=\pi^{-1}(v)$ is unramified.
\end{lemma}

\begin{proof} Let $\bar{f}:C_1\cup_{c}C_2\to T$ denote the restriction of $f$ to $\sC_u$. We have the map
\[
df:f^*\Omega_{T/k}\to \Omega_{\sC/U}
\]
and our assumption on the maps $f_1, f_2$ implies that on $\sC_u$, the map
\[
df\otimes k(u)=d\bar{f}:\bar{f}^*\Omega_{T/k}\to \Omega_{C_1\cup_{c}C_2/k}
\]
is surjective. Nakayama's lemma implies that $df$ is surjective over an open neigborhood $\sC'$ of $\pi^{-1}(u)$ in $\sC$ and since $\pi$ is proper, there is an open neighborhood $U'$ of $u$ such that $\pi^{-1}(U')\subset \sC'$, which gives us the open neighborhood we wanted.
\end{proof}

\begin{proof}[Proof of Theorem~\ref{thm:hyp:pc}]
We proceed by induction on $d = -K_S \cdot D.$ If $d = 1,$ the moduli space contains a unique map, which is an isomorphism onto a $-1$ curve. 

If $d \geq 2$, then $M_0(S,D)^\bir$ is irreducible by Theorem~\ref{thm:BLRT1.5}, so we need only find an unramified map in $M_0(S,D)^\bir$ to finish the proof.

If $d = 2,3$, it follows   from Lemma~\ref{lm:lowdegree} that the general map in $M_0(S,D)^\bir$ is unramified. 

For $d\ge 4$, Theorem 1.1 and the following paragraph of~\cite{BLRT} show that the hypotheses for~\cite[Lemma 5.1]{BLRT} are satisfied.  By  Theorem~\ref{thm:BLRT1.5}, the closure $\overline{M_0(S,D)}$ of $M_0(S,D)$ in $\M_0(S,D)$ is an irreducible component of $\M_0(S,D)$ with a dense open subset 
$M_0(S,D)$ parametrizing a dominant family of birational maps of irreducible curves.  Thus, we may apply 
 \cite[Lemma 5.1]{BLRT} to   $\overline{M_0(S,D)}$. This shows that there is a smooth irreducible pointed curve $(U,u)$, a proper, flat, surjective pointed map  $\pi:(\sC, p) \to (U,u)$ defining a  semi-stable family of genus 0 curves over $U$,   and  map $f:\sC\to S$ such that
 \begin{itemize}
 \item $\sC_t:=\pi^{-1}(t)$ is a smooth $\P^1$ for $t\in U\setminus\{u\}$,
 \item the map $f_t:\sC_t\to S$ is birational for $t\in U\setminus\{u\}$.
 \item  the fiber $\bar{f}:\sC_u\to S$ is a reducible stable map $\bar{f} : \P=\P_1\cup_p\P_2 \to S$ in $\bar M_0(S,D)$ with two irreducible components $f_i : \P_i \to S$, $i=1,2$, 
 \item Each $f_i$  is a  general member of a dominant family of birational stable maps in $M_0(S,D_i)$.
 \end{itemize}
  By induction and Remark~\ref{rem:hyp:pc}, each $f_i$ is unramified. Write $D_i = (f_i)_*([\P_i])$ and let $d_i = -K_S \cdot D_i$. Since the families are dominant, $d_i \geq 2.$ 

 By Lemma~\ref{lm:transverse} the maps $f_1,f_2$ intersect transversally at the point $q = f(p).$ By Lemma~\ref{lm:smoothing} there is a neighborhood $U'$ of $u$ such that map $f_v:\sC_t=\P^1\to S$ is unramified for all $v\in U'\setminus \{u\}$, that is, $f_v$ is an unramified map in $M_0(S,D)^\bir$ so the theorem follows. 
\end{proof}

\bibliographystyle{alpha}

\bibliography{Bibli}

\begin{thebibliography}{JPMPR25}

\bibitem[AO01]{AO}
Dan Abramovich and Frans Oort.
\newblock Stable maps and {H}urwitz schemes in mixed characteristics.
\newblock In {\em Advances in algebraic geometry motivated by physics
  ({L}owell, {MA}, 2000)}, volume 276 of {\em Contemp. Math.}, pages 89--100.
  Amer. Math. Soc., Providence, RI, 2001.

\bibitem[BLRT23]{BLRT}
Roya Beheshti, Brian Lehmann, Eric Riedl, and Sho Tanimoto.
\newblock Rational curves on del {P}ezzo surfaces in positive characteristic.
\newblock {\em Trans. Amer. Math. Soc. Ser. B}, 10:407--451, 2023.

\bibitem[BRW25]{BRW-WWI}
Erwan Brugall\'e, Johannes Rau, and Kirsten Wickelgren.
\newblock {W}elschinger--{W}itt invariants.
\newblock {\em Preprint}, available at \url{https://arxiv.org/abs/2509.04172},
  2025.

\bibitem[BW25]{BW-ABQ}
Erwan Brugall\'e and Kirsten Wickelgren.
\newblock A quadratic {A}bramovich-{B}ertram formula.
\newblock {\em Preprint}, available at \url{https://arxiv.org/abs/2506.17854},
  2025.

\bibitem[CH98]{Caporaso_Harris_CoutingPlaneCurves}
Lucia Caporaso and Joe Harris.
\newblock Counting plane curves of any genus.
\newblock {\em Invent. Math.}, 131(2):345--392, 1998.

\bibitem[Cho08]{Cho}
Cheol-Hyun Cho.
\newblock Counting real {$J$}-holomorphic discs and spheres in dimension four
  and six.
\newblock {\em J. Korean Math. Soc.}, 45(5):1427--1442, 2008.

\bibitem[CK99]{CoxKatz-Mirror_symmetry}
David~A. Cox and Sheldon Katz.
\newblock {\em Mirror symmetry and algebraic geometry}, volume~68 of {\em
  Mathematical Surveys and Monographs}.
\newblock American Mathematical Society, Providence, RI, 1999.

\bibitem[Dol12]{Dolgachev}
Igor~V. Dolgachev.
\newblock {\em Classical algebraic geometry}.
\newblock Cambridge University Press, Cambridge, 2012.
\newblock A modern view.

\bibitem[Ful98]{fulton98}
William Fulton.
\newblock {\em Intersection theory}, volume~2 of {\em Ergebnisse der Mathematik
  und ihrer Grenzgebiete. 3. Folge. A Series of Modern Surveys in Mathematics
  [Results in Mathematics and Related Areas. 3rd Series. A Series of Modern
  Surveys in Mathematics]}.
\newblock Springer-Verlag, Berlin, second edition, 1998.

\bibitem[GP98]{GoettschePand}
L.~G\"{o}ttsche and R.~Pandharipande.
\newblock The quantum cohomology of blow-ups of {${\bf P}^2$} and enumerative
  geometry.
\newblock {\em J. Differential Geom.}, 48(1):61--90, 1998.

\bibitem[Gro61]{EGAIII_1}
A.~Grothendieck.
\newblock \'{E}l\'ements de g\'eom\'etrie alg\'ebrique. {III}. \'{E}tude
  cohomologique des faisceaux coh\'erents. {I}.
\newblock {\em Inst. Hautes \'Etudes Sci. Publ. Math.}, (11):167, 1961.

\bibitem[Gro66]{egaIV_3}
A.~Grothendieck.
\newblock \'{E}l\'ements de g\'eom\'etrie alg\'ebrique. {IV}. \'{E}tude locale
  des sch\'emas et des morphismes de sch\'emas. {III}.
\newblock {\em Inst. Hautes \'Etudes Sci. Publ. Math.}, (28):255, 1966.

\bibitem[Gro85]{Gromov}
M.~Gromov.
\newblock Pseudo holomorphic curves in symplectic manifolds.
\newblock {\em Invent. Math.}, 82(2):307--347, 1985.

\bibitem[Ill05]{IllusieFGA}
Luc Illusie.
\newblock Grothendieck's existence theorem in formal geometry.
\newblock In {\em Fundamental algebraic geometry}, volume 123 of {\em Math.
  Surveys Monogr.}, pages 179--233. Amer. Math. Soc., Providence, RI, 2005.
\newblock With a letter (in French) of Jean-Pierre Serre.

\bibitem[JPMPR24]{JPMPR-tropPlaneCurves}
Andr\'es Jaramillo~Puentes, Hannah Markwig, Sabrina Pauli, and Felix R\"ohrle.
\newblock Arithmetic counts of tropical plane curves and their properties.
\newblock {\em Adv. Geom.}, 24(4):553--576, 2024.

\bibitem[JPMPR25]{JPMPR-quadenr}
Andr\'es Jaramillo~Puentes, Hannah Markwig, Sabrina Pauli, and Felix R\"ohrle.
\newblock Quadratically enriched plane curve counting via tropical geometry.
\newblock 2025.
\newblock ArXiv 2502.02569.

\bibitem[JPP25a]{JPP-quadenr}
Andr\'es Jaramillo~Puentes and Sabrina Pauli.
\newblock A quadratically enriched correspondence theorem.
\newblock {\em Journal of the European Mathematical Society}, 2025.
\newblock ArXiv 2309.11706.

\bibitem[JPP25b]{PuentesPauli-Bezout}
Andr\'es Jaramillo~Puentes and Sabrina Pauli.
\newblock Quadratically enriched tropical intersections.
\newblock {\em J. Reine Angew. Math.}, 821:151--193, 2025.

\bibitem[Kle05]{Kleinman_The_Picard_Scheme}
Steven~L. Kleiman.
\newblock The {P}icard scheme.
\newblock In {\em Fundamental algebraic geometry}, volume 123 of {\em Math.
  Surveys Monogr.}, pages 235--321. Amer. Math. Soc., Providence, RI, 2005.

\bibitem[KLSW23]{degree}
Jesse~Leo Kass, Marc Levine, Jake~P. Solomon, and Kirsten Wickelgren.
\newblock A quadratically enriched count of rational curves.
\newblock {\em Preprint}, 2023.

\bibitem[KM94]{Kontsevich-Manin}
M.~Kontsevich and Yu. Manin.
\newblock Gromov-{W}itten classes, quantum cohomology, and enumerative
  geometry.
\newblock {\em Comm. Math. Phys.}, 164(3):525--562, 1994.

\bibitem[Kol96]{Kollar}
J\'{a}nos Koll\'{a}r.
\newblock {\em Rational curves on algebraic varieties}, volume~32 of {\em
  Ergebnisse der Mathematik und ihrer Grenzgebiete. 3. Folge. A Series of
  Modern Surveys in Mathematics [Results in Mathematics and Related Areas. 3rd
  Series. A Series of Modern Surveys in Mathematics]}.
\newblock Springer-Verlag, Berlin, 1996.

\bibitem[KW19]{KWA1degree}
Jesse Kass and Kirsten Wickelgren.
\newblock The class of {E}isenbud--{K}himshiashvili--{L}evine is the local
  {A}1-brouwer degree.
\newblock {\em Duke Mathematical Journal}, 168(3):429--469, 2019.

\bibitem[KW21]{cubicsurface}
Jesse~Leo Kass and Kirsten Wickelgren.
\newblock An arithmetic count of the lines on a smooth cubic surface.
\newblock {\em Compos. Math.}, 157(4):677--709, 2021.

\bibitem[Lan02]{Lang}
Serge Lang.
\newblock {\em Algebra}, volume 211 of {\em Graduate Texts in Mathematics}.
\newblock Springer-Verlag, New York, third edition, 2002.

\bibitem[Lev19]{Levine-Witt}
Marc Levine.
\newblock Motivic {E}uler characteristics and {W}itt-valued characteristic
  classes.
\newblock {\em Nagoya Math. J.}, 236:251--310, 2019.

\bibitem[Lev20]{Levine-EC}
Marc Levine.
\newblock Aspects of enumerative geometry with quadratic forms.
\newblock {\em Doc. Math.}, 25:2179--2239, 2020.

\bibitem[Lev21]{Levine-IntrinsicStable}
Marc Levine.
\newblock The intrinsic stable normal cone.
\newblock {\em Algebr. Geom.}, 8(5):518--561, 2021.

\bibitem[LPLS24]{LLV-Eulerchar}
Marc Levine, Simon Pepin~Lehalleur, and Vasudevan Srinivas.
\newblock Euler characteristics of homogeneous and weighted-homogeneous
  hypersurfaces.
\newblock {\em Adv. Math.}, 441:Paper No. 109556, 86, 2024.

\bibitem[LT98]{LiTian}
Jun Li and Gang Tian.
\newblock Virtual moduli cycles and {G}romov-{W}itten invariants of algebraic
  varieties.
\newblock {\em J. Amer. Math. Soc.}, 11(1):119--174, 1998.

\bibitem[LV25]{LV-DTP13}
Marc Levine and Anna~M. Viergever.
\newblock Quadratic donaldson-thomas invariants for $(\mathbb{P}^1)^3$ and some
  other smooth proper toric threefolds.
\newblock {\em Preprint}, available at \url{https://arxiv.org/abs/2503.14420},
  2025.

\bibitem[Mat80]{MatsumuraCA}
Hideyuki Matsumura.
\newblock {\em Commutative algebra}, volume~56 of {\em Mathematics Lecture Note
  Series}.
\newblock Benjamin/Cummings Publishing Co., Inc., Reading, Mass., second
  edition, 1980.

\bibitem[McK21]{McKean-Bezout}
Stephen McKean.
\newblock An arithmetic enrichment of {B}\'{e}zout's {T}heorem.
\newblock {\em Math. Ann.}, 379(1-2):633--660, 2021.

\bibitem[MS94]{McDuff-Salamon}
Dusa McDuff and Dietmar Salamon.
\newblock {\em {$J$}-holomorphic curves and quantum cohomology}, volume~6 of
  {\em University Lecture Series}.
\newblock American Mathematical Society, Providence, RI, 1994.

\bibitem[MV99]{morelvoevodsky1998}
F.~Morel and V.~Voevodsky.
\newblock {${\mathbb A}^1$}-homotopy theory of schemes.
\newblock {\em Inst. Hautes \'Etudes Sci. Publ. Math.}, (90):45--143 (2001),
  1999.

\bibitem[Pau22]{Pauli-quadratic_types_quintic_3-fold}
Sabrina Pauli.
\newblock Quadratic types and the dynamic {E}uler number of lines on a quintic
  threefold.
\newblock {\em Adv. Math.}, 405:Paper No. 108508, 37, 2022.

\bibitem[PP22]{PajwaniPal-YZ}
Jesse Pajwani and Ambrus P\'al.
\newblock An arithmetic {Y}au-{Z}aslow formula.
\newblock {\em Preprint}, available at \url{https://arxiv.org/abs/2210.15788},
  2022.

\bibitem[RT95]{Ruan-Tian}
Yongbin Ruan and Gang Tian.
\newblock A mathematical theory of quantum cohomology.
\newblock {\em J. Differential Geom.}, 42(2):259--367, 1995.

\bibitem[Ser55]{FAC}
Jean-Pierre Serre.
\newblock Faisceaux alg\'{e}briques coh\'{e}rents.
\newblock {\em Ann. of Math. (2)}, 61:197--278, 1955.

\bibitem[sga03]{sga1}
{\em Rev\^etements \'etales et groupe fondamental ({SGA} 1)}.
\newblock Documents Math\'ematiques (Paris) [Mathematical Documents (Paris)],
  3. Soci\'et\'e Math\'ematique de France, Paris, 2003.
\newblock S{\'e}minaire de g{\'e}om{\'e}trie alg{\'e}brique du Bois Marie
  1960--61. [Algebraic Geometry Seminar of Bois Marie 1960-61], Directed by A.
  Grothendieck, With two papers by M. Raynaud, Updated and annotated reprint of
  the 1971 original [Lecture Notes in Math., 224, Springer, Berlin; MR0354651
  (50 \#7129)].

\bibitem[Sol06]{Solomon-thesis}
Jake~P. Solomon.
\newblock Intersection theory on the moduli space of holomorphic curves with
  {L}agrangian boundary conditions.
\newblock {\em Thesis}, available at \url{https://arxiv.org/abs/math/0606429},
  2006.

\bibitem[{Sta}18]{stacks-project}
The {Stacks Project Authors}.
\newblock \textit{Stacks Project}.
\newblock \url{https://stacks.math.columbia.edu}, 2018.

\bibitem[Tes09]{Testa}
Damiano Testa.
\newblock The irreducibility of the spaces of rational curves on del {P}ezzo
  surfaces.
\newblock {\em J. Algebraic Geom.}, 18(1):37--61, 2009.

\bibitem[Tyo07]{Tyomkin}
Ilya Tyomkin.
\newblock On {S}everi varieties on {H}irzebruch surfaces.
\newblock {\em Int. Math. Res. Not. IMRN}, (23):Art. ID rnm109, 31, 2007.

\bibitem[Wel05]{Welschinger-invtsReal4mflds}
Jean-Yves Welschinger.
\newblock Invariants of real symplectic 4-manifolds and lower bounds in real
  enumerative geometry.
\newblock {\em Invent. Math.}, 162(1):195--234, 2005.

\bibitem[ZS75]{ZariskiSamuelI}
Oscar Zariski and Pierre Samuel.
\newblock {\em Commutative algebra. {V}ol. {II}.}
\newblock Springer-Verlag, New York-Heidelberg, 1975.
\newblock Reprint of the 1960 edition.

\end{thebibliography}

\end{document}